%% file: 03.11.2025.tex
\documentclass[11pt]{article}

\usepackage[margin=1in]{geometry}
\usepackage{amsmath}
\usepackage{geometry}
\usepackage{indentfirst}
\usepackage{bm}
\usepackage{accents}
\usepackage{changepage}
\usepackage{color}
\usepackage{bigints}
\usepackage{yhmath}
\usepackage{dirtytalk}
\usepackage{tensor}

\usepackage{pifont}
\usepackage{multirow}
\usepackage{enumerate}
\usepackage{enumitem}
\usepackage{mathtools}
\usepackage{latexsym,amsthm,amsfonts,amssymb,amsmath,amsxtra,mathrsfs,stmaryrd,bm}
\usepackage{figsize,subfigure,graphicx,ifthen,calc,psfrag}
\usepackage{upgreek}
\usepackage[greek,english]{babel}
\usepackage[numbers,sort&compress]{natbib}

\usepackage{comment}
\usepackage[backref=page]{hyperref}
\usepackage{cleveref} 
\usepackage{caption}
\usepackage{float}
\usepackage{svg}

\DeclareMathOperator*{\argmin}{arg\,min}
\newtheorem{thm}{Theorem}[section]
\newtheorem{prop}[thm]{Proposition}
\newtheorem*{defi*}{Definition}
\newtheorem{defi}[thm]{Definition}

\newtheorem{lem}[thm]{Lemma}
\newtheorem{rem}[thm]{Remark}

\makeatletter
\newcommand{\mylabel}[2]{#2\def\@currentlabel{#2}\label{#1}}
\makeatother

\numberwithin{equation}{section}

\allowdisplaybreaks

\usepackage{nomencl}
\makenomenclature

\usepackage{etoolbox}
\providetoggle{nomsort}
\settoggle{nomsort}{true} % true = sort by use, false = sort as usual

\makeatletter
\iftoggle{nomsort}{%
    \let\old@@@nomenclature=\@@@nomenclature        
        \newcounter{@nomcount} \setcounter{@nomcount}{0}%
        \renewcommand\the@nomcount{\two@digits{\value{@nomcount}}}% Ensure 10>01
        \def\@@@nomenclature[#1]#2#3{% Taken from package documentation
          \addtocounter{@nomcount}{1}%
        \def\@tempa{#2}\def\@tempb{#3}%
          \protected@write\@nomenclaturefile{}%
          {\string\nomenclatureentry{\the@nomcount\nom@verb\@tempa @[{\nom@verb\@tempa}]%
          \begingroup\nom@verb\@tempb\protect\nomeqref{\theequation}%
          |nompageref}{\thepage}}%
          \endgroup
          \@esphack}%
      }{}
\makeatother

\usepackage{multicol}
\usepackage{algorithm}
\usepackage[noend]{algpseudocode}

\usepackage{algorithm}
\usepackage{algpseudocode}

\newcounter{example}
\newcommand{\example}{%
  \refstepcounter{example}%
  \par\medskip
  \noindent\textbf{Example~\theexample.}\quad
}

\crefname{example}{Example}{Examples}
\Crefname{example}{Example}{Examples}

\usepackage{float}

\title{
Extended Walk-on-Spheres Algorithm\\ for Linear and Nonlinear Elliptic Problems of Divergence-type \\[0.5 cm]
}
\author{
Iulian C\^{i}mpean$^{1,3,}$\thanks{Corresponding author. E-mails: \texttt{iulian.cimpean@unibuc.ro}},
Andreea Grecu$^{2,}$\thanks{E-mails: \texttt{andreea.grecu@ismma.ro}},
Arghir Zarnescu$^{4,5,3,}$ \thanks{E-mails: \texttt{azarnescu@bcamath.org}}
}

\date{\small
$^1$Department of Mathematics, Faculty of Mathematics and Computer Science, University of Bucharest, 
14~Academiei, 010014~Bucharest, Romania\\
$^2$``Gheorghe Mihoc~--~Caius Iacob" Institute of Mathematical Statistics and Applied Mathematics of the Romanian Academy, 13~Calea 13 Septembrie, 050711~Bucharest, Romania\\
$^3$``Simion Stoilow" Institute of Mathematics of the Romanian Academy, 21~Calea Grivi\c{t}ei,  010702~Bucharest, Romania \\
$^4$ BCAM, Basque Center for Applied Mathematics, Mazarredo 14, 48009 Bilbao, Bizkaia, Spain\\
$^5$IKERBASQUE, Basque Foundation for Science, Maria Diaz de Haro 3, 48013 Bilbao, Bizkaia, Spain
}

\begin{document}

\maketitle
%%%%%%%%%%%%%%%%%%%%%%%%%%%%%%%%%%%%%%%%%%%%%%%%%%%%%
\begin{abstract}
The Walk-on-Spheres algorithm, introduced by M. E. Muller in 1956, is a well known Monte Carlo method that leverages Brownian exit distributions from spheres to solve the Laplace equation with Dirichlet boundary conditions. 
Its mesh-free nature, robustness on complex geometries, favorable scaling with dimension, and intrinsic parallelism distinguish it from mesh-based solvers. However, its efficient applicability has been essentially limited to operators that admit explicit probabilistic exit laws, excluding most variable-coefficient and nonlinear elliptic operators.

We propose a general framework that aims to overcome this limitation by using the classical Dirichlet Laplacian and harmonic extension as universal building blocks. Rather than seeking a custom stochastic representation for each operator, we employ Walk-on-Spheres to precompute a reusable numerical operator toolbox that approximates the inverse Dirichlet Laplacian, the harmonic extension operator, and their gradients. 
These precomputed operators are then used to represent candidate solutions and to transform arbitrary Dirichlet boundary value problems into a finite-dimensional algebraic system/optimization problem for an unknown source term. 
Solving the resulting algebraic system/optimization problem and substituting back yields an approximate solution to the original PDE.
Even more, for a general linear second order elliptic operator, the above mentioned precomputed toolbox can be directly used to obtain not just an approximation of a certain solution of the corresponding generalized Dirichlet problem, but an estimator of both the Green's integral operator and the elliptic measure operator.

Numerical experiments on a range of benchmarks, including non-symmetric and anisotropic linear elliptic equations, semilinear and quasilinear problems, demonstrate the method’s flexibility and efficiency. %promising performance.

\end{abstract}



% 65N21  	Numerical methods for inverse problems for boundary value problems involving PDEs
%65N75  	%Probabilistic methods, particle methods, etc. for boundary value problems involving PDEs
% 35J25  	Boundary value problems for second-order elliptic equations
%65C05  	%Monte Carlo methods
%65N80  	%Fundamental solutions, Green's function methods, etc. for boundary value problems involving PDEs
%65N35  	%Spectral, collocation and related methods for boundary value problems involving PDEs
% 60J65  	Brownian motion
% 65C40  	Numerical analysis or methods applied to Markov chains
% 60J60  	Diffusion processes
% 65N12  	Stability and convergence of numerical methods for boundary value problems involving PDEs
% 65N15  	Error bounds for boundary value problems involving PDEs
%65N20  	Numerical methods for ill-posed problems for boundary value problems involving PDEs
%65N21  	Numerical methods for inverse problems for boundary value problems involving PDEs
%65N22  	Numerical solution of discretized equations for boundary value problems involving PDEs [See also 65Fxx, 65Hxx]
%65N25  	Numerical methods for eigenvalue problems for boundary value problems involving PDEs

{\footnotesize{\tableofcontents}}

%%%%%%%%%%%%%%%%%%%%%%%%%%%%%%%%%%%%%%%%%%%%%%%%%%%%%%%%%%%%%%%%%%%%%%%%%%%%%%%%%%%%%%%%%%
%introduction

\input{Introduction.tex}

%%%%%%%%%%%%%%%%%%%%%%%%%%%%%%%%%%%%%%%%%%%%%%%%%%%%%%

% General framework

\input{General_framework.tex}

\input{numerical}

%%%%%%%%%%%%%%%%%%%%%%%%%%%%%%%%%%%%%%%%%%%%%%%%%%%%%%

%%%%%%%%%%%%%%%%%%%%%%%%%%%%%%%%%%%%%%%%%%%%%%%%%%%%%%
 %\appendix

\input{Appendix}

%%%%%%%%%%%%%%%%%%%%%%%%%%%%%%%%%%%%%%%%%%%%%%%%%%%%%

\bigskip
\noindent \textbf{Acknowledgements.} I.C. acknowledges support from the Ministry of Research, Innovation and Digitization (Romania), grant CF-194-PNRR-III-C9-2023.

The work of A.Z. has been partially supported by the Basque Government through the BERC 2026-2029 program and by the Spanish State Research Agency through BCAM Severo Ochoa excellence accreditation Severo Ochoa CEX2021-00114 and through project PID2023-146764NB-I00
funded by MICIU/AEI/10.13039/501100011033. A.Z. was also partially supported by a grant of the Ministry of Research, Innovation and Digitization, CNCS-UEFISCDI, project number PN-IV-P2-2.1-T-TE-2023-0219, within PNCDI IV.

\addcontentsline{toc}{section}{\refname}
% Add references to ToC (and bookmarks)
% \bibliographystyle{plain}
\bibliographystyle{abbrv}

\end{document}

%% file: Introduction.tex
\section{Introduction}\label{sec:Introduction}

The body of literature on numerical methods for partial differential equations is vast, and it lies well beyond our scope to attempt even a partial survey here. 
We simply point to a few representative monographs - such as \cite{smith1985numerical, quarteroni1994numerical, braess2001finite, strikwerda2004finite, bathe2006finite, thomee2007galerkin, johnson2009numerical} - together with the references therein.
For the sake of comparison with the methods developed in the present work, we would like to point out from the very beginning that most  {\it deterministic } numerical methods for approximating the solution of partial differential equations comprise three core components:

\medskip
\noindent{$\circ\quad$}\textit{A finite-dimensional representing function space \(\chi\).} 
The approximate solution is sought in a suitably chosen finite-dimensional subspace \(\chi\). Typical choices include piecewise-polynomial finite-element spaces such as \(P_1\) and \(P_2\) Lagrange elements, discontinuous Galerkin spaces, structured or adaptive finite-difference grids, spectral bases based on Fourier or Chebyshev polynomials, spline spaces, neural networks, etc.
  
\medskip
\noindent{$\circ\quad$}\textit{Efficient evaluation of functions in \(\chi\), as well as their derivatives and integrals.} 
Practical implementations require computationally effective procedures for evaluating basis functions and their gradients and for computing integrals that arise in weak formulations. Representative techniques include reference-element evaluation with Jacobian mappings for finite elements, numerical quadrature rules, finite-difference, fast Fourier transforms, automatic differentiation, etc.
  
\medskip
\noindent{$\circ\quad$}\textit{Reduction to and solution of algebraic systems/optimization problems.} 
The discretization and evaluation steps yield systems of linear or nonlinear algebraic equations/optimization problems that must be solved numerically. Common approaches comprise assembly of sparse linear systems and solution by direct sparse factorizations or iterative Krylov methods, gradient methods for nonlinear problems, multilevel or domain-decomposition techniques, etc.

\medskip
The selection of the function space, the evaluation strategy, and the solver for the finite dimensional problem, is typically done in a mutually consistent manner so as to balance accuracy, stability, and computational efficiency for the problem at hand.

\medskip
%\noindent{$\diamond\quad$} 
A traditionally different numerical approach for PDEs is provided by {\it the probabilistic methods} (see e.g. \cite{beck2021nonlinear, giles2015multilevel, graham2013stochastic, kroese2013handbook, barth2011multi, talay2006probabilistic, heinrich2001multilevel, cheridito2007second, fahim2011probabilistic} and the references therein) which are typically based on representing the solution $u$ of a given PDE, at a given point $x$ in the (possibly space-time) domain of definition of $u$ by
\begin{equation}\label{eq:probabilistic_method}
    u(x)=\mathbb{E}\left\{X_x\right\},
\end{equation}
where $X_x$ is a random variable that is obtained from an underlying stochastic process that corresponds to the operator involved in the formulation of the PDE, and depends on the given data (Dirichlet, Neumann, source terms, etc) and the point $x$.
Then, for each $x$ fixed, one uses a Monte Carlo scheme that approximates the average in \eqref{eq:probabilistic_method}, thus rendering an estimator of the value $u(x)$.

Monte-Carlo methods that directly represent $u(x)$ as in \eqref{eq:probabilistic_method}, typically do not follow the three-bullet scheme from above. 
When available, they have several remarkable features that distinguish them from the traditional deterministic methods: are mesh-free, scale at most polynomially with the dimension of the space, and are pleasingly parallel. 

Usually, deterministic methods that follow the three-bullet scheme are preferred for low dimensional PDEs, whilst probabilistic methods based on representations like \eqref{eq:probabilistic_method} tend to be significantly more efficient for high-dimensional problems, or when the unknown solution is aimed to be approximated only at a few number of points in the given domain.

Both of the approaches presented above can nevertheless be combined, as in the case of {\it Deep Kolmogorov methods} (see e.g. \cite{berner2020numerically,beck2021solving,cimpean2025error}), that essentially follow the three-bullet structure from above, but at the same time, rely on probabilistic representations like \eqref{eq:probabilistic_method} derived from the theory of stochastic differential equations, in order to construct a meaningful loss function to be optimized.

In this paper, our approach builds upon perhaps one of the most efficient schemes that falls in the class of probabilistic methods for PDEs, namely the classical Walk-on-Spheres (WoS) method.
WoS, introduced by M. E. Muller in 1956 \cite{Muller}, leverages Brownian exit distributions from spheres to efficiently represent (in the spirit of \eqref{eq:probabilistic_method}) the solution $u$ to the Laplace equation with Dirichlet boundary conditions:
\begin{equation*}
    \Delta u =f  \mbox{ in } D, \quad u=g \mbox{ on } \partial D,\quad D\subset \mathbb{R}^d.
\end{equation*}
For more recent developments of this method we refer to \cite{deaconu2006random, KyOs18, CiGrMaI,BCPZ, sabelfeld1995integral,Bossy_Champagnat_Maire_Talay_2010, deaconu2017walk, sawhney2020monte, shalimova2020random, sabelfeld2022global,ShalimovaSabelfeld+2023+79+93,miller2023boundary,WoB23} and the references therein.

The WoS method is distinguished from traditional mesh/grid-based solvers by its mesh-free nature, robustness on complex geometries and irregular data, favorable scaling with dimension, and intrinsic parallelism. 
However, despite its  impressive capabilities,  the use of this method has been limited  for a long time to operators that admit explicit probabilistic exit laws, excluding most variable-coefficient and nonlinear elliptic operators.
Recently, the applicability of the WoS algorithm to more general operators has been numerically investigated 
in \cite{sawhney2022grid} for second order elliptic operators with scalar and $C^2$-smooth diffusion coefficients, through a fixed point integral equation solved by a recursive application of the standard WoS algorithm; this method has been further used in \cite{viswanath2026operator} in order to train a neural network that aims at learning the solution operator of the same PDE. 

The aim of this work is to take a significant step forward and develop a general (yet simple) robust and versatile framework that uses the classical WoS algorithm in conjunction with a cheap domain discretization in order to construct a Laplacian Operator Toolbox (see \eqref{eq:LOT}) that, once precomputed and stored, can be used to solve much more general classes of boundary value problems of divergence-type, such as, for instance:
\begin{enumerate}
    \item[-] {\it linear elliptic PDEs with (possibly non-smooth) variable coefficients} of the type
    \begin{equation}\label{linear_intro}
        \nabla \cdot (A \nabla u) + b  \cdot \nabla u - cu = f  \mbox{ in } D, \quad
        u = g  \mbox{ on } \partial D,
    \end{equation}
    or (non)linear
    \item[-] {\it Euler-Lagrange equations} satisfied by the critical points of energy functionals $\mathcal{F}$ of the form
    \begin{equation}\label{energy_intro}
        \mathcal{F}(u)=\frac{1}{\lambda(D)}\int_D F(x,u(x),\nabla u(x)) \; dx, \quad u\in {\sf D}(\mathcal{F}),
    \end{equation} 
    where $D\subset \mathbb{R}^d$ is a bounded connected open set, $\lambda(D)$ is its volume, whilst $F:D \times \mathbb{R} \times \mathbb{R}^d\rightarrow \mathbb{R}$ is a sufficiently smooth function, e.g. of class $C^1$.
\end{enumerate}

\paragraph{Informal description of the proposed method.}
Given $D\subset \mathbb{R}^d$ a bounded, open and connected set (i.e. bounded domain), 
let $\mathcal{U}:=\left(-\frac{1}{2}\Delta_0\right)^{-1}$ be the usual inverse of $-\frac{1}{2}\Delta_0$, where $\Delta_0$ is the homogeneous Dirichlet Laplacian on $D$, denote by $U$ its corresponding integral operator, and let $H$ denote the operator that maps a bounded and measurable function on $\partial D$ to its harmonic extension in $\overline{D}$ (see \eqref{eq:mathcal U}, \Cref{defi:U_0}, \Cref{defi:H}).

\medskip
Our proposed method can be briefly described as follows:

\medskip
\noindent{\scalebox{0.7}{\ding{117}}$\quad$} The representing space $\chi$ in which we seek for an approximate solution is given by
    \begin{equation*}
        \chi \coloneqq {U}\left(L^\infty(D)\right)+H\left(\mathcal{B}_b(\partial D)\right),
    \end{equation*}
    where 
    \begin{equation*}
        \mathcal{B}_b(\partial D):=\{g:\partial D \rightarrow \mathbb{R} : g \mbox{ is measurable and bounded}\}.
    \end{equation*}
    Note that $\chi$ is a very rich space, in particular $\chi\cap H^1(D)$ is dense in $H^1(D)$ under minimal assumptions on the regularity of $D$, e.g. if $D$ is Lipschitz.

\medskip
\noindent{$\scalebox{0.7}{\ding{117}}\quad$} Both operators ${U}$ and $H$, together with their derivatives $\nabla {U}$ and $\nabla H$, when applied to bounded functions, admit continuous representatives given by expected values as in \eqref{eq:probabilistic_method}, and these representations render efficient Monte Carlo estimators that are easily implemented through parallel algorithms.
Furthermore, given a (possible unstructured) set of points in $D\cup \partial D$, the operators ${U}, \nabla {U},$ and $H, \nabla H$, are approximated by random matrices that are derived from the above mentioned Monte Carlo estimators.

The above mentioned point locations and the resulting matrices form what we call the {\it Laplacian Operator Toolbox (LOT)} associated to the domain $D$; see \eqref{eq:LOT}.

For computational efficiency, these matrices are retrieved by a simultaneous GPU-parallel Monte Carlo computation of the occupation time of the Voronoi cells corresponding to the given discretization points. 

\medskip
\noindent{\scalebox{0.7}{\ding{117}}$\quad$} Once the above mentioned Laplacian Operator Toolbox associated to a domain $D$ has been precomputed and stored, one can easily use it to solve various types of elliptic PDEs of divergence-type, like  generalized Dirichlet problem associated to the linear elliptic PDE with variable coefficients \eqref{linear_intro}, or Euler-Lagrange equations associated to the energy functionals\eqref{energy_intro}; we do this in \Cref{s:solving PDEs}.

In \Cref{sec:numerical examples}, we test the method described above on a range of benchmarks, including non-symmetric and anisotropic linear elliptic equations, semilinear and quasilinear problems with Dirichlet boundary conditions such as:
\begin{enumerate}[itemsep=-2pt]
    \item[-] Laplace and Poisson problem $\bf{-\Delta u = 0}$, $\bf{-\Delta u = f}$ (Examples \ref{example Laplace} - \ref{example Poisson});
    \item[-] second order linear elliptic equation $\bf{ \mathbf{div} (A(x) \nabla u ) +   b(x) \cdot \nabla u - c u  = f}$ (Examples \ref{example elliptic div} - \ref{example elliptic div min} );
    \item[-] $p-$Laplace problem $\bf{ -  \mathbf{div} (|\nabla u|^{p-2} \nabla u) = f}$ (Examples \ref{example plaplace disk} - \ref{example plaplace square});
    \item[-] minimal surface equation $ \bf{\mathbf{div} \left( \frac{\nabla u}{\sqrt{1+ |\nabla u|^2}} \right) = 0}$ (Example \ref{example minimal surface});
    \item[-] semilinear elliptic equation $\bf{\Delta u + g(u) = f}$(Example \ref{example semilinear}).
\end{enumerate}

\paragraph{Organization of the paper.}

In \Cref{s:1}, given merely a bounded, open and connected set $D\subset \mathbb{R}^d$, we rigorously derive the Laplacian Operator Toolbox needed in the next section in order to solve the general PDEs \eqref{linear_intro} and \eqref{energy_intro} from above.
More precisely:
\begin{enumerate}[itemsep=-2pt]
    \item[-]  In \Cref{ss1} we recall some crucial regularity properties for the Poisson and harmonic extension operators $U$ and $H$, respectively, and then we revisit the probabilistic representations for $U,\nabla U,H,\nabla H$ (see \Cref{prop:U_0 and DU_0} and \Cref{prop:H and DH}, as well as \Cref{prop:removing singularities}). 
    \item[-] In \Cref{ss:MC_estimators} we derive the Monte-Carlo estimators for $U,\nabla U,H,\nabla H$, rendered by the WoS algorithm, and we prove their consistency in \Cref{prop:lim_MN}.
    \item[-] In \Cref{ss:matrix_MC}, by relying on a cheap partitioning of the domain, we derive the announced Laplacian Operator Toolbox \eqref{eq:LOT} which essentially consists of random matrices that estimate $U,\nabla U,H,\nabla H$. 
    The consistency of these estimators is proved in \Cref{prop:general_convergence}.
    A basic example of partitioning is given by Voronoi diagrams, briefly explained at the end of \Cref{ss:matrix_MC}; this partitioning is also the main one used in our implementations in \Cref{sec:numerical examples}.
    \item[-] Finally, in \Cref{ss:pseudocode} we present the pseudocode algorithms that enable the implementation of the Laplacian Operator Toolbox \eqref{eq:LOT} mentioned above.
\end{enumerate}
Further, in \Cref{s:solving PDEs} we use the Laplacian Operator Toolbox introduced in \Cref{s:1} in order to solve general PDEs as those in \eqref{linear_intro} and \eqref{energy_intro}. 
We do this systematically, namely 
\begin{enumerate}[itemsep=-2pt]
    \item[-] In \Cref{ss:linearPDE} we treat general linear elliptic PDEs of the type \eqref{linear_intro}.
    \item[-] In \Cref{ss:EL} we deal with Euler-Lagrange equations associated to \eqref{energy_intro}.
    \item[-] In \Cref{sec:numerical examples} we present the numerical simulations for Examples \ref{example Laplace} - \ref{example semilinear} listed above. In the case of the Laplace and Poisson problems in Examples \ref{example Laplace} - \ref{example Poisson} we also make comparisons with the finite element method. 
\end{enumerate}
In Appendix \ref{appendix numerical} are provided supplementary numerical details for Examples \ref{example elliptic div} - \ref{example semilinear}, and in Appendix \ref{appendix proofs} we give the proofs for several statements from \Cref{s:1}.

%% file: General_framework.tex
\section{A Laplacian Operator Toolbox (LOT) through WoS}\label{s:1}
In this section we present systematically the construction of the Laplacian Operator Toolbox mentioned in Introduction \ref{sec:Introduction}; the resulting toolbox is summarized in \eqref{eq:LOT}.

\subsection{Probabilistic representations of functions and their gradients}\label{ss1}

Let $D\subset \mathbb{R}^d$ be a bounded domain (i.e. a bounded, open and connected set) and consider $\left(\Delta_0,{\sf D}(\Delta_0)\right)$ be the homogeneous Dirichlet Laplacian on $D$, with domain ${\sf D}(\Delta_0) = \{ u \in H^1_0(D), \, \Delta u \in L^2(D) \}$.
Recall that
\begin{equation*}
    {\sf D}(\Delta_0)\subset H_0^1(D) \mbox{ densely}, \quad \mbox{and the extension\;} \Delta_0:H_0^1(D)\rightarrow H^{-1}(D) \mbox{ is a continuous linear bijection}.
\end{equation*}
The bounded linear operator $\mathcal{U}$ defined by
\begin{equation}\label{eq:mathcal U}
    \mathcal{U}:H^{-1}(D)\rightarrow H_0^1(D), \quad  \mathcal{U}(f):=\left(-\frac{1}{2}\Delta_0\right)^{-1}
\end{equation}
satisfies
\begin{equation*}
\mathcal{U}\left(L^2(D)\right)={\sf D}(\Delta_0).
\end{equation*}
Note that unless the domain $D$ is sufficiently smooth, ${\sf D}(\Delta_0)$ is larger than $H^2(D)\cap H_0^1(D)$ \cite{Grisvard85}. 
However, by e.g. \cite[Theorem 8.8 \& Theorem 8.10]{GiTr01}, we have
\begin{equation*}
    \mathcal{U}f\in W^{k+2,2}_{\sf loc}(D)\cap H_0^1(D)\quad \mbox{for every } f\in W^{k,2}(D), \quad k\geq 0.
\end{equation*}
Moreover, by Calderón–Zygmund estimates we have
\begin{equation*}
    \mathcal{U}f\in W^{2,p}_{\sf loc}(D)\cap H_0^1(D)\quad \mbox{for every } f\in L^p(D), \quad p> 1.
\end{equation*}

\begin{defi}\label[defi]{defi:U_0}
Let $G_D(x,y)$ denote the Green function of $\Delta_0$ on $D$. 
We set
\begin{equation*}
    Uf(x) \coloneqq - 2\int_D f(y) G_D(x,y )\;dy, \quad \mbox{for every } x\in D \mbox{ and } f\in L^p, \, p>d/2.
\end{equation*}
\end{defi}
We have that
\begin{equation*}
    Uf\in C(D)\cap L^\infty(D), \quad \mbox{and} \quad \mathcal{U}f=Uf \quad dx\mbox{-a.e.}, \,f\in L^p, \, p>d/2,
\end{equation*}
see \cite[Theorem 8.31]{GiTr01} showing that $\mathcal{U}f\in  C(\bar{D})$
and by e.g. \cite[Theorem 8.17, Theorem 8.24, \&  Theorem 8.25]{GiTr01}, we deduce that
\begin{equation*}
   Uf \mbox{ is locally Holder,} \quad f\in L^p, p>d/2.
\end{equation*}
Moreover, by e.g. \cite[Theorem 8.31]{GiTr01} we deduce:
\begin{equation*}
    \mbox{If } D \mbox{ is regular, then}\quad Uf\in C(\overline{D}) \mbox{ and } Uf(x)=0, \, x\in \partial D, \quad \mbox{ for every }f\in L^p, \, \,p>d/2.
\end{equation*}
Finally, by \cite[Corollary 8.36]{GiTr01}, we have
\begin{equation*}
    Uf\in C^{1,\alpha}_{\sf loc}(D) \quad \mbox{for every } \alpha<1, \quad \mbox{whenever } f\in L^\infty(D).
\end{equation*}

Throughout, we set
\begin{equation}\label{eq:r}
    r(x):=\inf_{y\in \partial D}\{\|x-y\| : y\in \partial D\}, \quad x\in \overline{D}.
\end{equation}
Also, let us recall the definition of the the Walk-on-Spheres (WoS) chain  $(X^x_k)_{k \geq 0}$, for $x \in D$: 
\begin{equation}\label{eq:wos chain}
    \begin{cases}
        X^x_0=x\\
        X^x_{k+1}=X^x_k + r(X^x_k) \mathbf{U}_k, \quad k \geq 0,
    \end{cases}
\end{equation}
where $(\mathbf{U}_k)_{k \geq 0}$ are independent and identically distributed random variables on a probability space $(\widetilde{\Omega}, \widetilde{\mathcal{F}},\widetilde{\mathbb{P}})$, with
\begin{equation}\label{eq:Uk iid}
    \mathbf{U}_k \sim {\rm Uniform}(\partial B (0,1)), \quad k \geq 0,
\end{equation}
i.e. uniformly distributed on the $(d-1)$-dimensional unit sphere from $\mathbb{R}^d$ centered at the origin, $\partial B (0,1)$. 

Further, we let $(B_t)_{t\geq 0}$ be a standard $d$-dimensional Brownian motion defined on a filtered probability space $(\Omega,\mathcal{F},\mathcal{F}_t,\mathbb{P})$ satisfying the usual hypotheses, and consider the realization of the Brownian motion started at $x\in \mathbb{R}^d$ given by
\begin{equation*}
    B_t^x:=x+B_t,\quad t\geq 0.
\end{equation*}
Moreover, if $A\in \mathcal{B}(\mathbb{R}^d)$, we set
\begin{equation*}
    \tau^x_{A}:=\inf\{t>0 : B_t^x\in A\}.
\end{equation*}

In the next two results we present the key probabilistic representations for $U,\nabla U, H,\nabla H$; they are essentially known ( \cite{sawhney2020monte,BCPZ,cimpean2025quantitative} are some recent works where they have been used), however, for completeness, we also provide rigorous proofs.

\begin{prop}[Probabilistic representation for $U$ and $\nabla U$]\label[prop]{prop:U_0 and DU_0}
    The following assertions hold:
    \begin{enumerate}
        \item[(i)] Let $Y$ be a real valued random variable defined on $(\widetilde{\Omega}, \widetilde{\mathcal{F}},\widetilde{\mathbb{P}})$, independent of $(\mathbf{U}_k)_{k \geq 0}$ \eqref{eq:Uk iid}, with probability density function
        \begin{equation}\label{eq:density Green}
            f_Y(y) = -2 d G_{B(0,1)}(0,y), \quad y \in B(0,1),
        \end{equation}
        where $G_{B(0,1)}$ is the Green function for the Laplace operator in $B(0,1)$ with zero boundary conditions, hence $G_{B(0,1)}(0,y)$ has the explicit representation
        \begin{equation*}
            G_{B(0,1)}(0,y)
            =\begin{cases}
                \frac{1}{2\pi}\ln{\|y\|},& \quad d=2 \\
               - \frac{1}{d(d-2)V_d}\left[1-\frac{1}{\|y\|^{d-2}}\right],&\quad d\geq 3
            \end{cases},
            \quad y \in B(0,1).
        \end{equation*}
where $V_d = \dfrac{\pi^{d/2}}{\Gamma(d/2+1)}$ the volume of the unit $d$-ball.
        For every $f\in L^p(D)$, $p>d/2$, we have
        \begin{align}\label{eq:U0_rep_intro}
        U f (x) 
        &= \mathbb{E} \left\{ \int_{0}^{\tau^x_{\partial D}} f(B^x_t) \ d t \right\}=\dfrac{1}{d} \mathbb{E} \left\{ \sum_{k \geq 0} r^2(X^x_k) f(X^x_k + r(X^x_k) Y)   \right\}, \quad x \in D.
        \end{align}
        \item[(ii)] Let $\widetilde{Y} \sim {\rm Uniform}(B (0,1))$, i.e. uniformly distributed in the $d$-dimensional unit ball $B(0,1)$. For every $f\in L^\infty(D)$, we have
        \begin{equation}\label{eq:grad_U0_rep_intro}
        \begin{aligned}
        \nabla U f (x) 
        & = \dfrac{d}{r(x)^2} \ \mathbb{E} \left\{ \int_{\tau^x_{\partial B(x,r(x))}}^{\tau^x_{\partial D} }f\left(B^{B^x_{\tau^x_{\partial B (x,r(x))}}}_t\right) \ d t \  \left( 
         B^x_{\tau^x_{\partial B (x,r(x))}} -x \right)\right\}\\
        & \quad + 2 \dfrac{r(x)}{d} \ \mathbb{E} \left\{ \widetilde{Y} \left( \dfrac{1}{|\widetilde{Y}|^d} - 1 \right) f(x+r(x) \widetilde{Y} )\right\}\\
        &=\dfrac{1}{r(x)^2} \ \mathbb{E} \left\{ \left( 
         X_1^{x} -x \right)\sum\limits_{k\geq 1} r^2 \left(X^x_k \right) f \left(X^x_k + r\left(X^x_k\right)Y \right) \  \right\} \\
         &\quad + 2 \dfrac{r(x)}{d} \ \mathbb{E} \left\{ \widetilde{Y} \left( \dfrac{1}{|\widetilde{Y}|^d} - 1 \right) f(x+r(x) \widetilde{Y} )\right\}, \quad x \in D.
        \end{aligned}
        \end{equation}
    \end{enumerate}
\end{prop}
\begin{proof}
  See Appendix \ref{proof:prop:U_0 and DU_0}.
\end{proof}

We focus next on the harmonic extension operator.
\begin{defi}\label[defi]{defi:H}
    For each $x\in D$ let $\mu_x$ denote the harmonic (probability) measure (at $x$) on $\partial D$.
    Recall that, by Harnack inequality \cite[Corollary 8.21]{GiTr01} (see also \cite[Proposition 2.1]{CiGrMaI}), $\mu_x$ and $\mu_y$ are mutually absolutely continuous for $x,y\in D$, and the space $L^1(\partial D; \mu_x)$ is independent of $x\in D$.
Set
\begin{equation*}
    Hg(x) \coloneqq \int_{\partial D} g(z)\;\mu_x(dz), \quad g\in L^1(\mu_{x}), \quad x\in D.
\end{equation*}
\end{defi}
We have that 
\begin{equation*}
    H\left(L^1(\mu_{z})\right)\subset C^\infty(D) \quad \mbox{and}\quad \Delta Hg=0 \mbox{ in } D, \quad g\in L^1(\mu_{z}).
\end{equation*}
Moreover, if $D$ is a regular domain (i.e. all boundary points are regular in the sense of Wiener \cite{Wiener24}, \cite[Section 2.8-2.9]{GiTr01}), then
\begin{equation*}
    \lim\limits_{D\ni x\to y} Hg(x)=g(y), \quad \mbox{for every } y\in \partial D \mbox{ and } g\in C(\partial D).
\end{equation*}

\begin{rem}
If $D$ is Lipschitz then for $g\in H^{1/2}(\partial D)$ we have that $Hg$ uniquely solves weakly in $H^1(D)$
\begin{equation*}
    \Delta Hg=0 \mbox{ in } D, \quad g={\sf Tr}(Hg) \mbox{ on } \partial D,
\end{equation*}
where ${\sf Tr}$ denotes the trace operator ${\sf Tr}:H^1(D)\rightarrow H^{1/2}(\partial D)$; this follows by the classical existence result \cite[Theorem 8.3]{GiTr01} and the fact that ${\sf Tr}$ has a (continuous) right-inverse \cite[Theorem 3.37]{mclean2000strongly}.
Moreover, we can use this extension of the boundary data to decompose uniquely any element $u\in H^1(D)$ as
\begin{equation*}
    u=u_0+h, \quad u_0\in H_0^1(D), \, h\in H^1(D), \mbox{ with } \langle \nabla u_0, \nabla h\rangle_{L^2(D)}=0,
\end{equation*}
and in fact this decomposition is given by
\begin{equation}
    u_0=u-H\left(u|_{\partial D}\right), \quad h=H\left(u|_{\partial D}\right).
\end{equation}
Moreover, if $\Delta u:=f\in L^2(D)$ then
$
    u_0=\mathcal{U} f.
$
\end{rem}

\begin{prop}[Probabilistic representation for $H$ and $\nabla H$]\label[prop]{prop:H and DH}
    The following assertions hold:
    \begin{enumerate}
        \item[(i)]  Recalling that $r$ is given by \eqref{eq:r}, we have
        \begin{align}\label{eq:H_rep_intro}
        H g (x) 
        &= \mathbb{E} \left\{ g\left(B^x_{\tau^x_{\partial D}}\right) \right\}, \quad x \in D, \quad g\in L^1(\mu_z)\\
        \nabla H g (x) 
        &= \dfrac{d}{r(x)^2} \ \mathbb{E} \left\{ g \left( B^x_{\tau^x_{\partial D}} \right) \ \left( B^x_{\tau^x_{\partial B (x,r(x))}} -x\right) \right\}, \quad x \in D, \quad g\in \mathcal{B}_b(\partial D).
        \end{align}
        \item[(ii)] Let $g:\overline{D}\rightarrow \mathbb{R}$ be bounded and measurable such that $\lim\limits_{D\ni x\to y} g(x)=g(y)$ for $y\in \partial D\setminus \mathcal{N}$, where $\mathcal{N}$ satisfies $\mu_x(\mathcal{N})=0$ for one (hence all, as it follows from Definition~\ref{defi:H}) $x\in D$; recall that the latter condition is trivially satisfied if $\mathcal{N}$ is at most countable, whilst if $\mathcal{N}=\emptyset$, then we must have $g|_{\partial D}\in C(\partial D)$, see e.g. \cite[p. 100]{doob1984classical}, \cite{Beznea_potentza}).
        Then for every $x\in D$ we have
        \begin{align}
            Hg(x)
            &=\lim\limits_{M\to \infty}\mathbb{E}\left\{g(X_M^x)\right\}.\\
            \nabla Hg (x) &=\lim\limits_{M\to \infty} \dfrac{d}{r(x)^2} \ \mathbb{E} \left\{ g \left( X^x_M \right) (X^x_1-x) \right\}.
        \end{align}
    \end{enumerate}
\end{prop}
\begin{proof}
    See Appendix \ref{proof:prop:H and DH}.
\end{proof}

\paragraph{Taming the singular terms in \eqref{eq:density Green}, \eqref{eq:grad_U0_rep_intro}.}\label{ss2}
In this paragraph we highlight alternative representations that help avoid division‑by‑zero issues arising from numerical approximations, and that also reduce the sampling variance of two quantities appearing in \Cref{prop:U_0 and DU_0}, more precisely:
\begin{enumerate}
    \item[$\bullet$] The random variable $Y$ which is sampled according to the singular density $f_Y$ given by \eqref{eq:density Green};
    \item[$\bullet$] The expectation of the singular term $\widetilde{Y} \left( \dfrac{1}{|\widetilde{Y}|^d} - 1 \right) f(x+r(x) \widetilde{Y} )$ in \eqref{eq:U0_rep_intro} (\Cref{prop:U_0 and DU_0} (ii)) which is going to be computed by Monte Carlo simulation. 
\end{enumerate}

\begin{prop}\label[prop]{prop:removing singularities}
\begin{enumerate}
    \item[(i)] Let $Y$ be a random variable with values in $B(0,1)$ with density $f_Y$ given by \eqref{eq:density Green}.
        \begin{enumerate}
            \item[(i.1)](Case $d=2$) If $U\sim {\rm Uniform}([0,1])$ independently of $Z\sim { \rm Uniform}(\partial B(0,1))$, and $W_{-1}$ denotes the {\it Lambert $W$ function} (the $-1$ branch) which has the property that $W_{-1}(y)e^{W_{-1}(y)}=y$, then $R:=e^{\frac{1+W_{-1}(-U/e)}{2}}\sim \rho$ and
            \begin{equation*}
                RZ\sim f_Y.
            \end{equation*}
            Since an efficient PyTorch implementation of the Lambert function seems to be missing at the moment, we adopt the following convenient representation: Let $U_1,U_2,U_3 \sim {\rm Uniform([0,1])}$ be independent, and let $R:=U_{(2)}$ their middle order statistics. 
            It is known that $R\sim {\sf Beta}(2,2)$.
            Independently of $U_{(2)}$, let $Z\sim {\rm Uniform}(\partial B(0,1))$.
            Then for every bounded and measurable $\psi : B(0,1)\rightarrow \mathbb{R}$ we have
            \begin{equation*}
                \mathbb{E}\left[\psi(Y)\right]
                =-\frac{2}{3}\mathbb{E}\left\{\frac{\ln{R}}{(1-R)}\psi(RZ)\right\}
            \end{equation*}
            \item[(i.2)](Case $d\geq 3$) Let $R=U_{(2)}\sim \mbox{\sf Beta}\left(\frac{2}{d-2},2\right)$ and $Z\sim {\rm Uniform}(\partial B(0,1))$ be independent. 
            Then
            \begin{equation*}
                R^{\frac{1}{d-2}}Z\sim f_Y.
            \end{equation*}
        \end{enumerate}
        \item[(ii)] Let $\widetilde{Y} \sim {\rm Uniform} (B(0,1))$. 
        Further, let $U\sim {\rm Uniform}([0,1])$ and $Z\sim {\rm Uniform}(\partial B(0,1))$ such that $U$ and $Z$ are independent, and set $R:=1-\sqrt{U}$.
Then for every $f\in L^\infty(D)$ and $x\in D$ we have
\begin{equation*}
    \mathbb{E} \left\{ \widetilde{Y} \left( \dfrac{1}{|\widetilde{Y}|^d} - 1 \right) f(x+r(x) \widetilde{Y} )\right\}
    =\frac{d}{2}\mathbb{E}\left\{\sum\limits_{0\leq k\leq d-1}R^k Z  f(x+r(x) RZ )\right\}.
\end{equation*}
\end{enumerate}
\end{prop}

\begin{proof}
    See Appendix \ref{proof:prop 2.6}.
\end{proof}

% \subsection{Monte Carlo estimators for {$U$, $\nabla U$, $H$, $\nabla H$} in $D \subset \mathbb{R}^d$}\label{ss:MC_estimators}
\subsection{\texorpdfstring{Monte Carlo estimators for $U$, $\nabla U$, $H$, $\nabla H$ in $D \subset \mathbb{R}^d$}{Monte Carlo estimators for U, ∇U, H, ∇H in D ⊂ ℝᵈ}}\label{ss:MC_estimators}
For every $x\in D$ we consider
$\left(X_k^{x,l}\right)_{k\geq 0},l\geq 1$, independent and identically distributed (i.i.d) copies of $\left(X_k^{x}\right)_{k\geq 0}$ \eqref{eq:wos chain}.
Moreover, with $Y,\widetilde{Y}$ as in  \Cref{prop:U_0 and DU_0}, let $Y^l,\widetilde{Y}^l,l\geq 1$, be independent such that $Y^l\sim Y$ and $\widetilde{Y}^l \sim \widetilde{Y}$.

Based on \Cref{prop:U_0 and DU_0} and \Cref{prop:H and DH}, we consider the following Monte Carlo estimators parametrized by the number of i.i.d. samples $N\geq 1$, and by the number of steps $M\geq0$ performed by the WoS chain.

\medskip
\noindent{$\bullet\quad$}For every $f:D\rightarrow \mathbb{R}$ measurable such that $f\in L^p(D), \, p>d/2$, and $x\in D$, we consider the estimator
\begin{equation}\label{eq:u MC}
 U^f_{M,N}(x) \coloneqq \dfrac{1}{N} \sum_{l=1}^N \dfrac{1}{d} \sum_{k = 0}^{M} r^2\left(X^{x,l}_k\right) f\left(X^{x,l}_k + r(X^{x,l}_k) Y^l\right).
\end{equation}

\medskip
\noindent{$\bullet\quad$} For every $f:D\rightarrow \mathbb{R}$ bounded and measurable, and $x\in D$, we consider the estimator
\begin{equation}\label{eq: gradu U MC}
    \begin{aligned}
    {}^\nabla U^{f}_{M,N}(x) & \coloneqq \dfrac{1}{r(x)^2} \dfrac{1}{N} \sum_{l=1}^N   (X^{x,l}_1 - x)  \sum_{k = 1}^{M} r^2 \left(X^{x,l}_k \right) f \left(X^{x,l}_k + r\left(X^{x,l}_k\right) Y^l \right) \\
    & \quad + 2 \dfrac{r(x)}{d} \ \dfrac{1}{N} \sum_{l=1}^N \left\{ \widetilde{Y}^l \left( \dfrac{1}{|\widetilde{Y}^l|^d} - 1 \right) f(x+r(x) \widetilde{Y}^l )\right\}.
    \end{aligned}
\end{equation}

\medskip
\noindent{$\bullet\quad$} For every $g$ as in \Cref{prop:H and DH} and $x\in D$ we consider the estimator
    \begin{equation}\label{eq:H MC}
        \left(Hg(x)\approx\right)\quad {H}_{ M,N}^g(x) \coloneqq \frac{1}{N}\sum_{l=1}^{N}g(X_M^{x,l}). 
    \end{equation}

\medskip
\noindent{$\bullet\quad$} For every $g$ as in \Cref{prop:H and DH} and $x\in D$ we consider the estimator
\begin{equation}\label{eq:dh MC}
     \left(\nabla Hg(x)\approx\right)\quad {}^\nabla {H}_{ M,N}^g (x) \coloneqq \dfrac{d}{r(x)^2} \dfrac{1}{N} \ \sum_{l=1}^N \left\{ g \left( X^{x,l}_M \right)  (X^{x,l}_1-x) \right\}.
\end{equation}

\begin{prop}\label[prop]{prop:lim_MN}
For every $f\in L^\infty(D)$ and every $g:\overline{D}\rightarrow \mathbb{R}$ as in \Cref{prop:H and DH}, we have $\mathbb{P}$-a.s. that    
\begin{equation*}
    \lim\limits_{M\to\infty} \, \lim\limits_{N \to \infty} \left(|U^f_{M,N}(x)-Uf(x)|+|{}^\nabla U^f_{M,N}(x)-\nabla Uf(x)|+|{H}_{ M,N}^g(x)-Hg(x)|+|{}^\nabla {H}_{ M,N}^g (x)-\nabla Hg(x)|\right)=0.
\end{equation*}
\end{prop}
\begin{proof}
    It follows immediately by applying the law of large numbers for the limit over $N$, and \Cref{prop:U_0 and DU_0} and \Cref{prop:H and DH} for the limit over $M$.
\end{proof}

\begin{rem} \label[rem]{rem:MC_no_singular}
    Let $R^l\sim {\sf Beta}(2,2)$ if $d=2$, $R^l\sim {\sf Beta}(\frac{2}{d-2},2)$ if $d\geq 3$, $S^l\sim 1-\sqrt{{\rm Uniform}(0,1)}$, and $Z^l\sim {\rm Uniform}(\partial B(0,1))$, all independent, as well as independent from $\left(X_k^{x,l}\right)_{k\geq 0},l\geq 1$.
    Then, in view of \Cref{prop:removing singularities}, the Monte Carlo estimators $U^f_{M,N}(x)$ \eqref{eq:u MC} and $\nabla U^f_{M,N}(x)$ \eqref{eq: gradu U MC} are equal in distribution with the following refined versions:
\begin{equation*}
    \begin{aligned}
    U^f_{M,N}(x) 
    & \stackrel{\mathcal{D}}{=}
    \begin{cases}
        -\dfrac{2}{3N} \sum\limits_{l=1}^N\frac{\ln{R^l}}{(1-R^l)}\dfrac{1}{d} \sum\limits_{k = 0}^{M} r^2(X^{x,l}_k) f(X^{x,l}_k + r(X^{x,l}_k) R^lZ^l),\quad & d=2\\
        \dfrac{1}{N} \sum\limits_{l=1}^N \dfrac{1}{d} \sum\limits_{k = 0}^{M} r^2(X^{x,l}_k) f(X^{x,l}_k + r(X^{x,l}_k) (R^l)^{\frac{1}{d-2}}Z^l),\quad & d\geq 3
    \end{cases},\\
    {}^\nabla U^f_{M,N}(x) & \stackrel{\mathcal{D}}{=} 
    -\dfrac{1}{r(x)^2} \dfrac{2}{3N} \sum_{l=1}^N   \frac{\ln{R^l}}{(1-R^l)}(X^{x,l}_1 - x)  \sum_{k = 1}^{M} r^2 \left(X^{x,l}_k \right) f \left(X^{x,l}_k + r(X^{x,l}_k) R^lZ^l \right) \\
    & \quad + r(x) \ \dfrac{1}{N} \sum_{l=1}^N \left\{ (1+S^l)Z^l f(x+r(x) S^lZ^l )\right\}, \quad d=2,\\
    {}^\nabla U^f_{M,N}(x) & \stackrel{\mathcal{D}}{=} 
    \dfrac{1}{r(x)^2} \dfrac{1}{N} \sum_{l=1}^N   (X^{x,l}_1 - x)  \sum_{k = 1}^{M} r^2 \left(X^{x,l}_k \right) f \left(X^{x,l}_k + r\left(X^{x,l}_k\right) (R^l)^{\frac{1}{d-2}}Z^l \right) \\
    & \quad +  r(x) \ \dfrac{1}{N} \sum_{l=1}^N \left\{ \left[\sum\limits_{0\leq k\leq d-1}(S^l)^k\right] Z^l f(x+r(x) S^lZ^l )\right\}, \quad d\geq 3.
    \end{aligned}
\end{equation*}
\end{rem}
    
% \subsection{Matrix estimators for {$U$, $\nabla U$, $H$, $\nabla H$} in $D \subset \mathbb{R}^d$}\label{ss:matrix_MC}
\subsection{\texorpdfstring{Matrix estimators for $U$, $\nabla U$, $H$, $\nabla H$ in $D \subset \mathbb{R}^d$}{Matrix estimators for U, grad U, H, grad H in D subset ℝᵈ}}\label{ss:matrix_MC}

In this subsection we use the Monte Carlo estimators introduced in \Cref{ss:MC_estimators} in order to numerically estimate the linear operators $U$, $\nabla U$, $H$, $\nabla H$ by random matrices which we describe systematically in the following.

\paragraph{Spatial discretization.}
For a given bounded domain $D\subset \mathbb{R}^d$, we consider the following partitions/locations that are aimed to discretize $D$, as well as its boundary $\partial D$:
\begin{enumerate}
    \item[$\bullet$] $\mathbf{W}_n:=\{W_1,\dots, W_n\}$ is a measurable partition of $D$ and $x_i\in W_i$ represents some {\it centers} of the {\it cell} $W_i$, $1\leq i\leq n$,  used to discretize the {\bf outputs} of the operators $U$, $\nabla U$, $H$, $\nabla H$ which are, in particular, continuous functions.
    \begin{enumerate}
        \item[$\circ$] For $u\in C(D;\mathbb{R})$ we consider 
        \begin{equation}\label{eq:u approx}
        \left(u(x)\approx\right) \quad u_{\mathbf{W}_n}(x):=\sum \limits_{i=1}^{n} \bar{u}_i 1_{W_i}(x)\; x\in D, \quad \tilde{u}_i \coloneqq u(x_i), \quad 1 \leq i \leq n,
        \end{equation}
    and note that $u_{\mathbf{W}_n}$ (tacitly) depends also on the centers $(x_i)_{1\leq i\leq n}$.
        
    \end{enumerate}
    \item[$\bullet$] $\mathbf{V}_m:=\{V_1,\dots, V_m\}$ is a measurable partition of $D$, used to discretize the {\bf inputs} of the operators $U$ and $\nabla U$ (which need not be continuous so we use averages instead of point-values as for outputs).
    \begin{enumerate}
        \item[$\circ$] For $f \in L^1(D)$, we consider
        \begin{equation}\label{eq:f approx}
           \left(f(x)\approx\right)\quad f_{\mathbf{V}_m}(x):=\sum \limits_{j=1}^{m} \bar{f}_j 1_{V_j}(x),\; x\in D, \quad \bar{f}_j \coloneqq \frac{1}{\lambda(V_j)}\int_{V_j} f(x) dx, \quad 1 \leq j \leq m.
        \end{equation} 
    \end{enumerate}
    \item[$\bullet$] $\mathbf{\Gamma}_p:=\left\{\Gamma_1,\dots,\Gamma_p\right\}$ is a measurable partition of $\overline{ D}$.
    Further, let  $z_i\in \Gamma_i\cap \partial D$ some {\it boundary center} of the {\it cell} $\Gamma_i$, $1\leq i\leq p$, used to discretize the input of the operators $H$ and $\nabla H$. 
    \begin{enumerate}
        \item[$\circ$] For $g:\partial D\rightarrow\mathbb{R}$, typically (piecewise) continuous on $\partial D$, we consider  $g_{\mathbf{\Gamma}_p}:\overline{D}\to\mathbb{R}$
        \begin{equation}\label{eq:g approx}
            g_{\mathbf{\Gamma}_p}(x):=\sum \limits_{j=1}^{p} \tilde{g}_j 1_{\Gamma_j}(x), \;x\in \overline{D}, \quad \tilde{g}_j \coloneqq g(z_j), \quad 1 \leq j \leq p,
        \end{equation}
        and note that $g_{\mathbf{\Gamma}_p}$ (tacitly) depends also on the centers $(z_j)_{1\leq i\leq p}$.
    \end{enumerate}
\end{enumerate}

\begin{rem}
    Clearly, one can always consider the simpler case $n=m$ and choose $\mathbf{W}_n=\mathbf{V}_m$.
    However, as already announced, the role of the partition $\mathbf{W}_n$ is to discretize the image of the operators $U$, $\nabla U$, $H$, $\nabla H$, which have a smoothing effect. In contrast, the role of the partition $\mathbf{V}_m$ is to discretize the domains of the smoothing operators $U$, $\nabla U$, which typically consist of less regular functions. 
    In this light, it is justified to have the freedom to choose $\mathbf{W}_n$ and $\mathbf{V}_m$ independently. 
    Moreover, this could be extremely useful when one aims to speed up and/or reduce the memory cost of the resulting algorithms.
\end{rem}

For a collection $\mathbf{S}_m:=\left\{S_1,\dots,S_m\right\}\subset 2^{\mathbb{R}^d}$ we define 
\begin{equation*}
    {\sf diam}(\mathbf{S}_m):=\max_{1\leq i\leq m}\sup_{x,y\in S_i}\|x-y\|.
\end{equation*}

\begin{lem}\label[lem]{lem:weak_convergence} 
  Let us consider a sequence of partitions  $\mathbf{V}_m=\{V^1_1,\dots, V^1_m\},\mathbf{\Gamma}_p=
\{\Gamma^1_1,\dots, \Gamma^1_p\}$, 
$m,p\geq 1$, as above, such that 
    \begin{equation}\label{eq:diam_to_zero}
        \lim\limits_{m \to \infty}{\sf diam}(\mathbf{V}_m)=0=\lim\limits_{p\to\infty}{\sf diam}(\mathbf{\Gamma}_p \cap \partial D).
    \end{equation}
    Furthermore, let $f\in L^\infty(D)$ and $g:\partial D\rightarrow \mathbb{R}$ be bounded and measurable such that 
    $g$ is continuous at all but a finite set of points $A\subset\partial D$.
    Then 
    \begin{equation*}
        \lim\limits_{m\to\infty}f_{\mathbf{V}_m}=f,\quad \mbox{weakly in } L^r(D), \, 1<r<\infty,\quad \mbox{whilst}\quad  \lim\limits_{p\to\infty}g_{\mathbf{\Gamma}_p}=g,\quad \mbox{pointwise on } \partial D\setminus A. 
    \end{equation*} 
\end{lem}
\begin{proof}
    See Appendix \ref{proof:lem 2.10}.
\end{proof}

\begin{prop}\label[prop]{prop:lim_mp}
Consider the settings and hypotheses of  Lemma~\ref{lem:weak_convergence}.
Then for every $x\in D$ and $1\leq j\leq d$, we have
\begin{align}
    \lim\limits_{m\to\infty} \left[|Uf_{\mathbf{V}_m}(x)-Uf(x)|+\|\nabla Uf_{\mathbf{V}_m}(x)-\nabla Uf(x)\|_{\mathbb{R}^d}\right]
    &=0,\\
    \lim\limits_{p\to\infty} \left[|Hg_{\mathbf{\Gamma}_p}(x)-Hg(x)|+\|\nabla Hg_{\mathbf{\Gamma}_p}(x)-\nabla Hg(x)\|_{\mathbb{R}^d}\right]
    &=0.
\end{align}
\end{prop}
\begin{proof}
    See Appendix \ref{proof:prop 2.11}.
\end{proof}
Let $M, N \in \mathbb{N}^*$. 
For $x \in D$, let $\left(X^{x,l}_k\right)_{1\leq k\leq M, 1\leq l \leq N}$ i.i.d copies of $\left(X_k^{x}\right)_{k\geq 0}$ \eqref{eq:wos chain}.
Moreover, with $Y,\widetilde{Y}$ as in  \Cref{prop:U_0 and DU_0}, let $Y^l,\widetilde{Y}^l,1 \leq l \leq N$ be independent such that $Y^l\sim Y$ and $\widetilde{Y}^l \sim \widetilde{Y}$.

\paragraph{Approximate matrix for $U$.} In view of \eqref{eq:u MC} and \eqref{eq:u approx}-\eqref{eq:f approx}, we define the approximate of $U f(x_i)$ 
\begin{equation*}
    U f(x_i)\approx \sum_{j=1}^{m} \bar{f}_j \left[\mathbf{A}_{M,N} \right]_{ij}
    =\left[\mathbf{A}_{M,N}\bar{f}\right]_i, \quad 1\leq i\leq n,
\end{equation*}
with $\mathbf{A}_{M,N} \in \mathbb{R}^{n \times m}$ given by
\begin{align}\label{eq:Estimator_A}
     \left[\mathbf{A}_{M,N} \right]_{ij} 
     &\coloneqq \dfrac{1}{N} \sum_{l=1}^N \dfrac{1}{d} \sum_{k = 0}^{M} r^2(X^{x_i,l}_k) 1_{V_j}\left(X^{x_i,l}_k + r(X^{x_i,l}_k)Y^l\right).
\end{align}

\paragraph{Approximate matrix for $\nabla U$.} In view of \eqref{eq: gradu U MC} and \eqref{eq:u approx}-\eqref{eq:f approx}, we consider,
\begin{equation*}
     \partial_{j} U f(x_i)  \approx \sum_{k=1}^{m} \bar{f}_k \left[\mathbf{B}^{(j)}_{M,N}\right]_{ik}=\left[\mathbf{B}^{(j)}_{M,N}\bar{f}\right]_{i}, \quad 1\leq i\leq n,\; 1\leq j\leq d,
\end{equation*}
with $\mathbf{B}^{(j)}_{M,N} \in \mathbb{R}^{n \times m}$, $1\leq j\leq d$, given by 
\begin{equation}\label{eq:Estimator_B}
\begin{aligned}
    \left[\mathbf{B}^{(j)}_{M,N}\right]_{ik} 
     & \coloneqq \dfrac{1}{r(x_i)^2}  \dfrac{1}{N} \sum_{l=1}^N ((X^{x_i,l}_1 - x_i))_{j}  \sum_{s = 1}^{M} r^2 \left(X^{x,l}_s \right) 1_{V_k} \left(X^{x,l}_s + r\left(X^{x,l}_s\right) Y^l \right) \\
     & \quad + 2 \dfrac{r(x_i)}{d} \ \dfrac{1}{N} \sum_{l=1}^N \left\{ (\widetilde{Y}^l)_{j} \left( \dfrac{1}{|\widetilde{Y}^l|^d} - 1 \right) 1_{V_k} \left(x_i+r(x_i) \widetilde{Y}^l \right)\right\}, \quad 1\leq i\leq n, \ 1\leq k \leq m.
\end{aligned}
\end{equation}

\begin{rem}\label[rem]{rem:MC_matrices_no_singular}
The random matrices $\mathbf{A}_{M,N}$ and $\mathbf{B}^{(j)}_{M,N}$, $1\leq j\leq d$, can be replaced by the numerically more convenient matrices obtained using Remark~\ref{rem:MC_no_singular} in a straightforward manner.
\end{rem}

\paragraph{Approximate matrix for $H$.} In view of \eqref{eq:H MC} and \eqref{eq:g approx}, we consider
\begin{equation*}
    H g(x_i) \approx \sum_{j=1}^{p} \tilde{g}_j \left[\mathbf{H}_{M,N} \right]_{ij}
    =\left[\mathbf{H}_{M,N} \tilde{g}\right]_{i}, \quad 1\leq i\leq n,
\end{equation*}
with $\mathbf{H}_{M,N} \in \mathbb{R}^{n \times p}$ given by
\begin{equation}\label{eq:Estimator_H}
     \left[\mathbf{H}_{M,N} \right]_{ij} \coloneqq \dfrac{1}{N} \sum_{l=1}^N 1_{\Gamma_j}\left(X^{x_i,l}_M \right), \quad 1\leq i\leq n, \ 1\leq j\leq p.
\end{equation}

\paragraph{Approximate matrix for $\nabla H$.} 
In view of \eqref{eq:dh MC} and \eqref{eq:g approx}, we consider
\begin{equation*}
    \partial_{j} H g(x_i)\approx \sum_{k=1}^{p} \tilde{g}_k  \left[\mathbf{K}^{(j)}_{M,N}\right]_{ik}
    =\mathbf{K}^{(j)}_{M,N}\tilde{g}(i), \quad 1\leq i\leq n,\;1\leq j\leq d,
\end{equation*}
with $\mathbf{K}^{(j)}_{M,N} \in \mathbb{R}^{n \times p}$, $1\leq j\leq d$, given by
\begin{equation}\label{eq:Estimator_K}
    \left[\mathbf{K}^{(j)}_{M,N}\right]_{ik} \coloneqq \dfrac{d}{r(x_i)^2} \dfrac{1}{N} \ \sum_{l=1}^N \left\{ 1_{\Gamma_k} \left( X^{X^{x_i,l}_1}_M \right) ((X^{x_i,l}_1 -x_i))_j \right\}, \quad 1\leq i\leq n, \ 1\leq k \leq p.
\end{equation}

\begin{prop}\label[prop]{prop:general_convergence}
    Let $f\in L^\infty(D)$, $g:\overline{D}\rightarrow \mathbb{R}$ as in \Cref{prop:H and DH}.
    Furthermore, assume that the partitions $\mathbf{V}_m,\mathbf{\Gamma}_p$, $m,p\geq 1$, satisfy \eqref{eq:diam_to_zero} and $g|_{\partial D}$ is continuous at all but a finite number of points of $\partial D$. Then $\mathbb{P}$-a.s. we have for every $1\leq i\leq n$, $1\leq j\leq d$,
    \begin{align*}
        &\lim\limits_{m\to\infty}\,\lim\limits_{M\to\infty}\,\lim\limits_{N\to\infty} \left[\mathbf{A}_{M,N} \bar{f}\right]_i=Uf(x_i),\quad \lim\limits_{m\to\infty}\,\lim\limits_{M\to\infty}\,\lim\limits_{N\to\infty} \left[\mathbf{B}^{(j)}_{M,N} \bar{f}\right]_i=\partial_jUf(x_i),\\
        &\lim\limits_{p\to\infty}\,\lim\limits_{M\to\infty}\,\lim\limits_{N\to\infty} \left[\mathbf{H}_{M,N} \tilde{g}\right]_i=Hg(x_i),\quad \,\lim\limits_{p\to\infty}\,\lim\limits_{M\to\infty}\lim\limits_{N\to\infty} \left[\mathbf{K}^{(j)}_{M,N} \tilde{g}\right]_i=\partial_jHg(x_i).
    \end{align*}
\end{prop}
\begin{proof}
    The proof follows by corroborating \Cref{prop:lim_MN} with \Cref{prop:lim_mp}.
\end{proof}

\begin{rem}
\begin{enumerate}
    \item[(i)] The convergences entailed by \Cref{prop:general_convergence} do not require any kind of regularity/geometrical structure of the partitions $\mathbf{V}_m$ and $\mathbf{\Gamma}_p$, the minimal condition \eqref{eq:diam_to_zero} being enough.
    \item[(ii)] The limits involved in \Cref{prop:general_convergence} cannot be interchanged in general.
\end{enumerate}
\end{rem}

Let us summarize this section by retaining the following toolbox of partitions (together with their centers) and matrices that are associated to the prescribed bounded domain $D\subset \mathbb{R}^d$, the Dirichlet Laplacian, and the harmonic extension operator, which have been introduced above:
\begin{center}
\fbox{
\begin{minipage}{0.9\linewidth}
    \centering\textbf{Laplacian Operator Toolbox (LOT)}
\begin{equation*}
    (x_i\in W_i)_{1\leq i\leq n}\subset D,\quad (V_i)_{1\leq i\leq m},\quad (z_i\in \Gamma_i)_{1\leq i\leq p}\in \partial D
\end{equation*}    
\begin{equation}\label{eq:LOT}
U\approx \mathbf{A}_{M,N}\in \mathbb{R}^{n\times m},
\quad \partial_j U\approx \mathbf{B}^{(j)}_{M,N}\in \mathbb{R}^{n\times m},
\quad H\approx\mathbf{H}_{M,N}\in \mathbb{R}^{n\times p},
\quad \partial_j H\approx\mathbf{K}^{(j)}_{M,N}\in \mathbb{R}^{n\times p}.
\end{equation}

\end{minipage}
}
\end{center}
Given such a toolbox, we can easily proceed and solve numerically more general classes of PDEs.
We detail this final step in the next section, by considering two benchmark classes of PDEs. 

\begin{rem}\label[rem]{rem:no voronoi}
    From the numerical point of view, for the above toolbox we only need to store and manipulate the centers $(x_i),(z_i)$ and the matrices $\mathbf{A}_{M,N}, \mathbf{B}^{(j)}_{M,N}, \mathbf{H}_{M,N}, \mathbf{K}^{(j)}_{M,N}$. 
    In other words, the elements of the partitions play no active role in the algorithms. This will become more clear in \Cref{ss:pseudocode}. 
\end{rem}

\paragraph{Example: Voronoi diagrams.} A natural and effective choice of the partitions $\mathbf{W}_n,\mathbf{V}_m,\mathbf{\Gamma}_p$, is to use Voronoi diagrams. They are extensively studied, widely applied, and supported by numerous efficient open‑source implementations.
A Voronoi diagram is a partition of the underlying space, generated by some prescribed centers (seed points).
In our case, in order to construct $\mathbf{W}_n,\mathbf{V}_m,\mathbf{\Gamma}_p$, we proceed as follows (see \Cref{Example_Voronoi} for a visualization in a square domain):
\begin{enumerate}
    \item[$\circ$] Let $(x_i)_{1\leq i\leq n}\subset D$ be the prescribed centers for $\mathbf{W}_n$ and set
    \begin{equation*}
        W_i:=\{x\in D : \|x-x_i\|\leq \|x-x_j\|, \quad \forall\; 1\leq j\leq n\},\quad 1\leq i\leq n.
    \end{equation*}
    \item[$\circ$] Let $(y_i)_{1\leq i\leq m}\subset D$ be the prescribed centers for $\mathbf{V}_m$ and set
    \begin{equation*}
        V_i:=\{y\in D : \|y-y_i\|\leq \|y-y_j\|, \quad \forall\; 1\leq j\leq m\},\quad 1\leq i\leq m.
    \end{equation*}
    \item[$\circ$] Let $(z_i)_{1\leq i\leq p}\subset \partial D$ be the prescribed centers for $\mathbf{\Gamma}_p$ and set
    \begin{equation*}
        \Gamma_i:=\{z\in \overline{D} : \|z-z_i\|\leq \|z-z_j\|, \quad \forall\; 1\leq j\leq p\},\quad 1\leq i\leq p.
    \end{equation*}
\end{enumerate}
As already stressed in \Cref{rem:no voronoi}, in our algorithms we do not use the partitions $\mathbf{W}_n,\mathbf{V}_m,\mathbf{\Gamma}_p$, except for the theoretical analysis of the method; instead, we rely on their centers and on an efficient function that returns the closest center to a given query point $x\in D$.

\begin{figure}[H]
    \centering
    \begin{minipage}{0.22\textwidth}
        \centering
        \includegraphics[width=\textwidth]{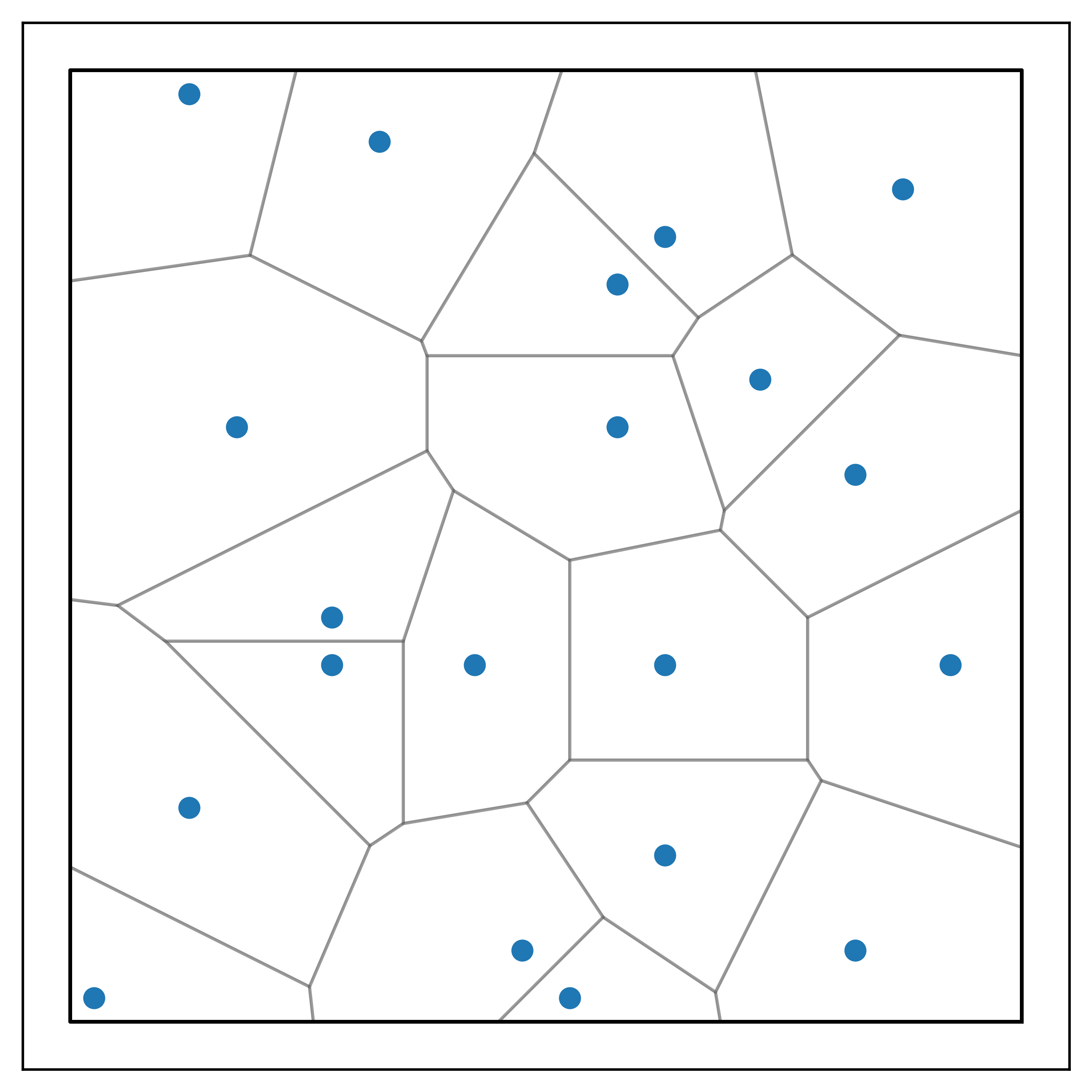}
        \captionof{subfigure}{$n=20$}
    \end{minipage}
    \begin{minipage}{0.22\textwidth}
        \centering
        \includegraphics[width=\textwidth]{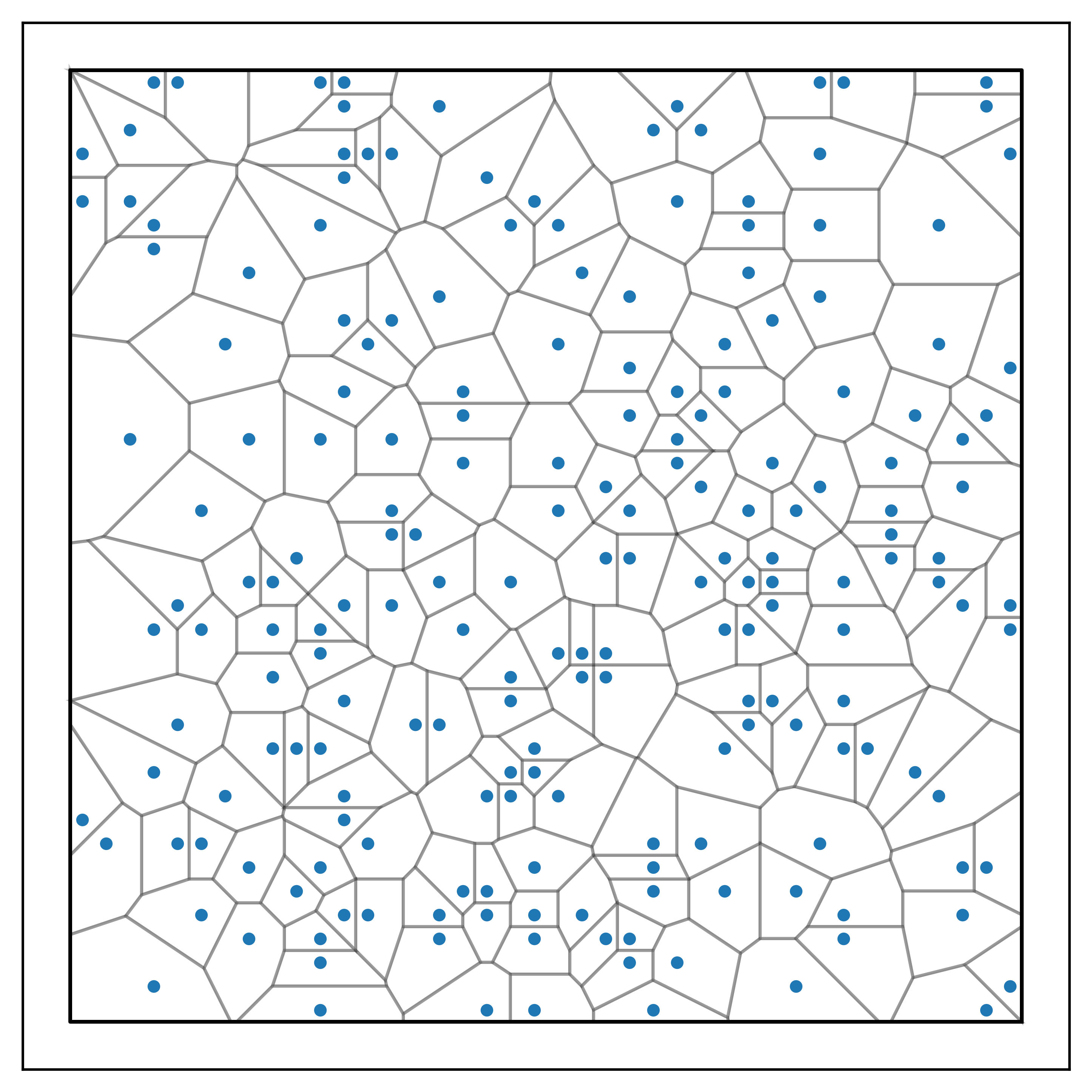}
        \captionof{subfigure}{$n=200$}
    \end{minipage}
    \begin{minipage}{0.22\textwidth}
        \centering
        \includegraphics[width=\textwidth]{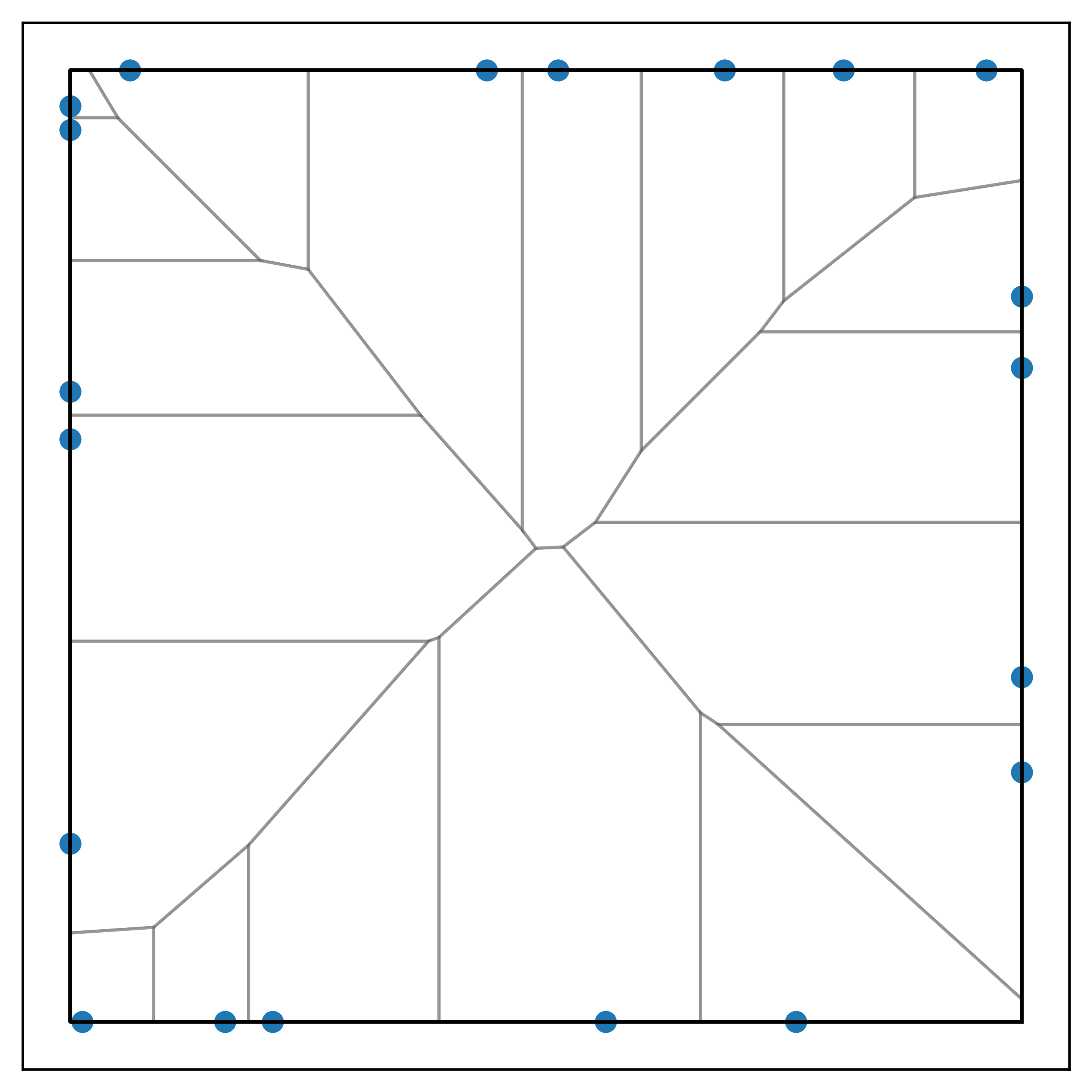}
        \captionof{subfigure}{$p=20$}
    \end{minipage}
    \begin{minipage}{0.22\textwidth}
        \centering
        \includegraphics[width=\textwidth]{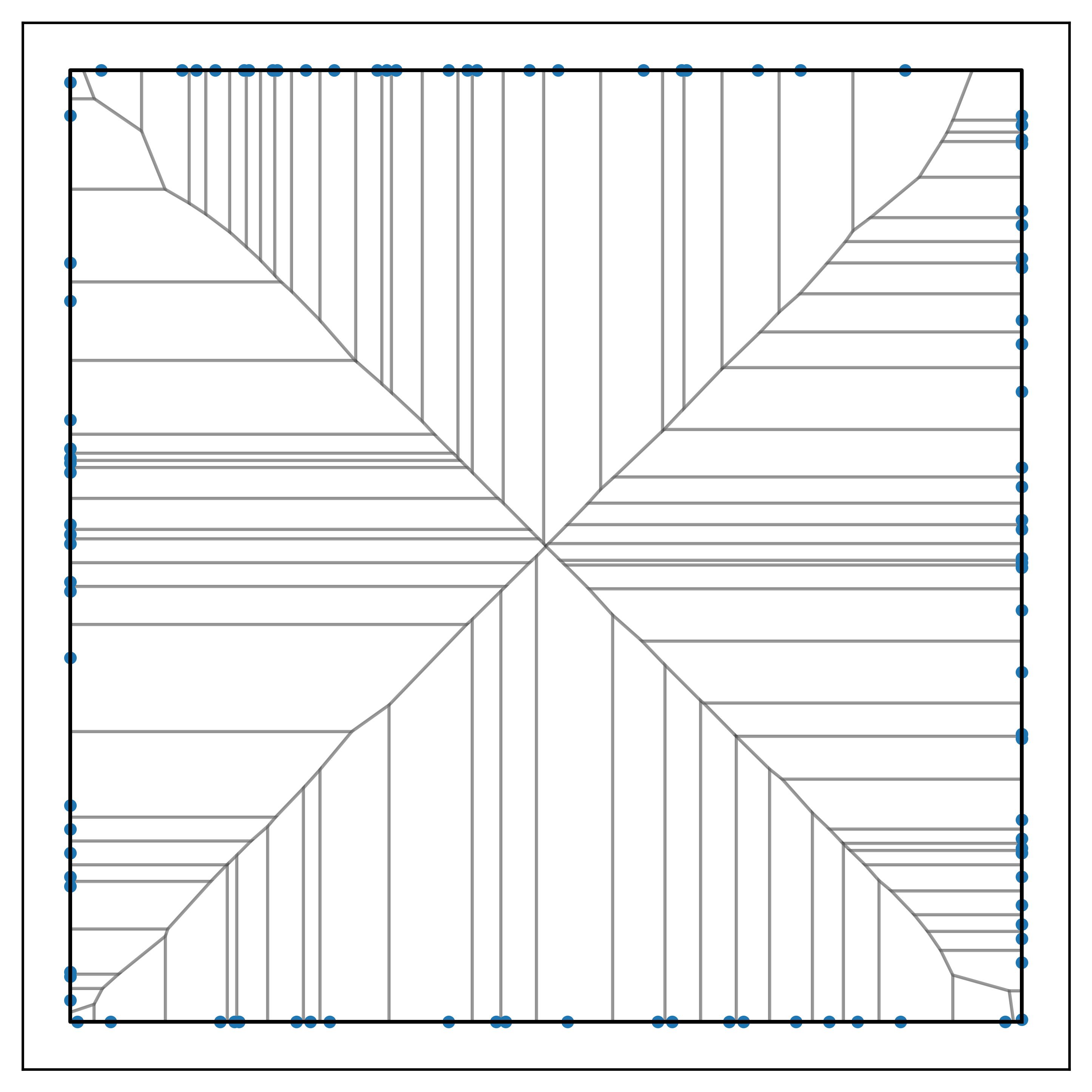}
        \captionof{subfigure}{$p=100$}
    \end{minipage}
    \caption{Example of Voronoi diagrams: (a)-(b) $\mathbf{W}_n$ for $n$ points $x_i \in D$; (c)-(d) $\mathbf{\Gamma}_p$ for $p$ points $z_i \in \partial D$.}
    \label{Example_Voronoi}
\end{figure}

\subsection{Laplacian Operator Toolbox: Pseudocode algorithms}\label{ss:pseudocode}

In this subsection we provide the pseudocode algorithm that enables the implementation of the Laplacian Operator Toolbox \eqref{eq:LOT} introduced in \Cref{ss:matrix_MC}.
Before doing so, we would like to mention that, in \eqref{eq:LOT}, the parameters $M$ (the number of steps performed by the WoS chain \eqref{eq:wos chain}) and $N$ (the number of i.i.d. samples used in the Monte Carlo estimator) can be distinct for each of the estimators $\mathbf{A}_{M,N}, \mathbf{B}^{(j)}_{M,N},\mathbf{H}_{M,N},\mathbf{K}^{(j)}_{M,N}$.
For this reason, but also for the ease of implementation and, if necessary, for personal adaptation, we include below a separate pseudocode algorithm for computing each of the above four estimators, namely \Cref{alg:A,alg:B,alg:H,alg:K}.
Moreover, in \Cref{alg:LOT} we aggregate the four pseudocode algorithms into a unified one that fully implements the Laplacian Operator Toolbox \eqref{eq:LOT} at once.

\input{Algorithm}

\begin{rem}\label[rem]{rem:implementation_tips}
    \begin{enumerate}[label=\roman*)] We would like to point out three aspects that are of crucial importance for the efficient implementation of \Cref{alg:A,alg:B,alg:H,alg:K}/\Cref{alg:LOT}:
        \item  For many types of geometries, the distance function to the boundary is not exactly available, and in these cases, some  approximation of it is in fact employed. For a variety of domains, these can be easily computed \cite[Section 5.1]{sawhney2020monte}; see also \cite{BCPZ} for theoretical quantifiable guarantees that using such an approximation is essentially harmless for the WoS algorithm.
        \item For clarity, we preferred to write explicitly the updates for each entry of the four output matrices; however, in practice, the resulting nested "for" loops are highly time-consuming.
        Thus, we strongly recommend that the "for" loops updating the matrices $\mathbf{A}_{M,N},\mathbf{B}^{(j)}_{M,N},\mathbf{H}_{M,N}$,
        $\mathbf{K}^{(j)}_{M,N}$ to be implemented in vectorized form ("for" loops over $(i,k,j)$) and the computations to be performed in parallel ("for" loop over $l$).
        \item In principle, the most time consuming part of the above algorithms is represented by the computation of the binary matrices $(\alpha_{ik})_{ik}$, $(\delta_{ik})_{ik}$, and $(\gamma_{ik})_{ik}$. To make these computations efficient, one should heavily rely on the Voronoi structure of the partitions. In this case, a generic binary matrix of the type $(1_{V_k}(z_l))_{1\leq k\leq m,1\leq l\leq N}$ with $V_k$ having center $y_k$ satisfies
        \begin{align*}
            \left(1_{V_k}(\xi_l)\right)_{1\leq k\leq m,1\leq l\leq N}
            &=\left(1_{\{k\}}\left(\argmin_{1\leq s\leq m} \|y_s - \xi_l \|\right)\right)_{1\leq k\leq m,1\leq l\leq N}
        \end{align*}
        whilst the computation of the distances $\left(\|y_k-\xi_s\|\right)_{ks}$ from above can be performed in vectorized form by employing an already existing and optimized function such as {\rm torch.cdist} available in PyTorch.
        As a matter of fact, the computation of the above binary matrices is clearly a problem of computing the $1$-nearest neighborhood which can be efficiently tackled using data structures such as {\it $k$-$d$ trees}. As shown in \cite{friedman1977algorithm}, in low-dimensional spaces, the nearest point computation of a query point with respect to $n$ (randomly distributed) points in the space structured by a $k$-$d$ tree has complexity $\mathcal{O}(\log n)$ on average. 
        Such computational tools have already been adopted in the context of WoS in e.g. \cite{sawhney2020monte}.
    \end{enumerate}
\end{rem}

\section{Solving general linear and nonlinear elliptic PDEs through LOT}\label{s:solving PDEs}

In this section we assume that, given a bounded open set $D\subset \mathbb{R}^d$, some partitions $\mathbf{W}_n,\mathbf{V}_m,\mathbf{\Gamma}_p$, and some corresponding centers $(x_i)_{1\leq i\leq n}$, $(y_i)_{1\leq i\leq m}$, and $(z_i)_{1\leq i\leq p}$ respectively, a sample of the Laplacian Operator Toolbox \eqref{eq:LOT} has been stored. 
The role of the following two subsections is to show that once we have such a toolbox, we can easily attempt to solve general classes of boundary value problems of divergence type. 

\subsection{Second order linear elliptic PDEs}\label{ss:linearPDE}

Let $D\subset \mathbb{R}^d$ be a bounded Lipschitz domain, whilst 
\begin{equation*}
    A=(a_{ij})_{i,j}:D\rightarrow \mathbb{R}^{d\times d},\quad b=(b_i)_i:D\rightarrow \mathbb{R}^{d}, \quad c:D\rightarrow [0,\infty)
\end{equation*}
are bounded and measurable, such that $A$ is symmetric and strictly elliptic, namely
\begin{equation*}
  a_{ij}(x)=a_{ji}(x),\; 1\leq i,j\leq d, \quad\mbox{and}\quad  \sum\limits_{1\leq i,j\leq d}a_{ij}(x)\xi_i\xi_j\geq \gamma_0\|\xi\|^2,\quad \xi\in\mathbb{R}^d, \mbox{ for some } \gamma_0>0. 
\end{equation*}
By e.g. \cite[Theorem 8.3]{GiTr01}, if $f\in L^2(D), \, g\in H^{1/2}(\partial D)$, then there is a unique weak solution $u\in H^1(D)$ of the {\it generalized} Dirichlet problem
\begin{eqnarray}\label{eq:linearPDE}
\left\{
\begin{array}{rl}
    \mathcal{L}u = f & \mbox{ in } D, \\[4pt]
    u = g & \mbox{ on } \partial D
\end{array}
\right.\quad, 
\mbox{ where } \mathcal{L}u:=\nabla \cdot (A \nabla u) + b  \cdot \nabla u - cu.
\end{eqnarray}

As announced in the previous sections, the approach we propose in order to numerically find an approximation of the solution $u$, is to represent it as
\begin{equation*}
u\approx Uq+Hg,\quad \mbox{for some unknown } q\in L^\infty(D).    
\end{equation*}

Note that the boundary condition is already included in the representation.
Moreover, the above approximation may be taken with an $H^1(D)$-accuracy that can be made arbitrarily small, more precisely we have:
\begin{lem}\label[lem]{lem:arbitrary_accuracy}
    Let $u\in H^1(D)$ be the unique solution to \eqref{eq:linearPDE} with $f,g$ as above. 
    Then for every $\varepsilon>0$ there exists $q=q_\varepsilon\in L^\infty(D)$ such that 
    \begin{equation*}
        \|u-U q - Hg\|_{H^1(D)}\leq \varepsilon.
    \end{equation*}
\end{lem}
\begin{proof}
    The claim is immediate since $u-Hg\in H_0^1(D)$ and  $U(L^\infty(D))\subset H_0^1(D)$ is dense.
\end{proof}
Furthermore, by taking $L^2(D)$-inner products in the first equality in \eqref{eq:linearPDE}, with respect to test functions of the type $U\varphi$,  $\varphi\in L^\infty(D)$, and using the approximate representation $u\approx Uq+Hg$, we get
\begin{equation*}
    \int_D -\langle A\nabla (Uq+Hg),\nabla U \varphi\rangle_{\mathbb{R}^d} + b\cdot \nabla (Uq+Hg)\;  U\varphi - c (Uq+Hg) U\varphi \;dx
    \approx \int_D f U\varphi \; dx 
\end{equation*}
hence
\begin{equation*}
    \int_D -\langle A\nabla Uq,\nabla U \varphi\rangle_{\mathbb{R}^d} + b\cdot \nabla Uq\;  U\varphi - c UqU\varphi \;dx
    \approx \int_D f U \varphi + \langle A\nabla Hg,\nabla U \varphi\rangle_{\mathbb{R}^d}- b\cdot \nabla Hg\;  U\varphi + cHg U\varphi\; dx 
\end{equation*}

\begin{rem}
    In view of Lemma~\ref{lem:arbitrary_accuracy}, the above approximation can be made arbitrarily accurate, uniformly with respect to $\varphi$ in bounded subset of $H^{-1}$ (hence $U\varphi$ lies in a bounded set in $H^1_0$).
\end{rem}

Now we use the partition $(W_i)_{1\leq i\leq n}$ and their centers $(x_i)_{1\leq i\leq n}$ to discretize the above integrals.

\begin{rem}\label[rem]{rem:discretization}
The above mentioned discretization is done by considering that the integrands are constant on each cell $W_i$.
In fact, on each $W_i$, we approximate the integrands by their values in the center $x_i$ of $W_i$.
In particular, we assume that values $f(x_i),b(x_i),c(x_i)$ are well defined and represent a feasible approximation of $f|_{W_i},b|_{W_i},c|_{W_i}$, respectively. 
Implicitly, this assumption requires a certain regularity of $f,b,c$.
Of course, the values $f(x_i),b(x_i),c(x_i)$ can be replaced by a certain average of $f,b,c$ on $W_i$.
\end{rem}

Based on the Remark~\ref{rem:discretization}, we are motivated to continue as follows:
\begin{align*}
    &\sum\limits_{1\leq i\leq n} \frac{\lambda(W_i)}{\lambda(D)} \left[-\langle A\nabla Uq(x_i),\nabla U \varphi(x_i)\rangle_{\mathbb{R}^d}+ \langle b(x_i), \nabla Uq(x_i)\rangle_{\mathbb{R}^d}\;  U\varphi(x_i) - c(x_i) Uq(x_i)U\varphi(x_i)\right]\\
    &\quad \approx \sum\limits_{1\leq i\leq n} \frac{\lambda(W_i)}{\lambda(D)} \left[f(x_i) U\varphi(x_i)\right.\\ 
    &\hspace{1cm} + \left.\langle A(x_i)\nabla Hg(x_i),\nabla U \varphi(x_i)\rangle_{\mathbb{R}^d} 
    - \langle b(x_i), \nabla Hg(x_i)\rangle_{\mathbb{R}^d}\;  U\varphi(x_i) 
    + c(x_i)Hg(x_i) U\varphi(x_i)\right]
\end{align*}

Using in the above expression the matrix estimators \eqref{eq:Estimator_A},\eqref{eq:Estimator_B},\eqref{eq:Estimator_H},\eqref{eq:Estimator_K} from \Cref{ss:matrix_MC} for the operators $U,\nabla U,H,\nabla H$, 
and setting
\begin{align*}
    &\bar{q}=\left(\frac{1}{\lambda(V_j)}\int_{V_j}q(x)\;dx\right)_{1\leq j\leq m}\in \mathbb{R}^m,\quad \bar{\varphi}=\left(\frac{1}{\lambda(V_j)}\int_{V_j}\varphi(x)\;dx\right)_{1\leq j\leq m}\in \mathbb{R}^m\\
    &\tilde{f}=\left(f(x_i)\right)_{1\leq i\leq n}\in \mathbb{R}^n,\quad \tilde{g}=\left(g(z_j)\right)_{1\leq j\leq p}\in \mathbb{R}^p,\quad \tilde{c}=(c(x_i))_{1\leq i\leq n}\in \mathbb{R}^n,
\end{align*}
we get
\begin{align}
    &\sum\limits_{1\leq i\leq n} \frac{\lambda(W_i)}{\lambda(D)} \left\{-\left\langle A(x_i)\left[\mathbf{B}^{(\cdot)}_{M,N}\bar{q}\right]_{i},\left[\mathbf{B}^{(\cdot)}_{M,N}\bar{\varphi}\right]_{i}\right\rangle_{\mathbb{R}^d} 
    + b(x_i)\cdot \left[\mathbf{B}^{(\cdot)}_{M,N}\bar{q}\right]_{i}\;  \left[\mathbf{A}_{M,N}\bar{\varphi}\right]_i 
    - c(x_i) \left[\mathbf{A}_{M,N}\bar{q}\right]_i\left[\mathbf{A}_{M,N}\bar{\varphi}\right]_i\right\}\nonumber\\
    &\quad \approx \sum\limits_{1\leq i\leq n} \frac{\lambda(W_i)}{\lambda(D)} \left\{f(x_i) \left[\mathbf{A}_{M,N}\bar{\varphi}\right]_i 
    + \left\langle A(x_i)\left[\mathbf{K}^{(\cdot)}_{M,N}\tilde{g}\right]_{i},\left[\mathbf{B}^{(\cdot)}_{M,N}\bar{\varphi}\right]_{i}\right\rangle_{\mathbb{R}^d}\right.\nonumber\\
    &\left.\quad\quad\quad 
    - b(x_i)\cdot \left[\mathbf{K}^{(\cdot)}_{M,N}\tilde{g}\right]_{i}\;  \left[\mathbf{A}_{M,N}\bar{\varphi}\right]_i 
    + c(x_i)\left[\mathbf{H}_{M,N}\tilde{g}\right]_i \left[\mathbf{A}_{M,N}\bar{\varphi}\right]_i\right\}.\label{eq:matrix_test_varphi}
\end{align}

Let us introduce the following notations
\begin{align}
    \lambda_i
    &:=\frac{\lambda(W_i)}{\lambda(D)}\in [0,1],\quad 1\leq i\leq n,\\
    \mathcal{B}_i
    &:=\left(\left[\mathbf{B}^{(k)}_{M,N}\right]_{ij}\right)_{1\leq k\leq d,\;1\leq j\leq m}\in \mathbb{R}^{d\times m}, 
    \quad \mathcal{K}_i:=\left(\left[\mathbf{K}^{(k)}_{M,N}\right]_{ij}\right)_{1\leq k\leq d,\;1\leq j\leq p}\in \mathbb{R}^{d\times p}, \quad 1\leq i\leq n,\label{eq:B,K_i}\\
    \mathcal{B}^{b}
    &:=\left(\lambda_i\left\langle \left[\mathbf{B}^{(\cdot)}_{M,N}\right]_{ij} , b(x_i)\right\rangle_{\mathbb{R}^d}\right)_{1\leq i\leq n, 1\leq j\leq m}\in \mathbb{R}^{n\times m},\nonumber\\
    \mathcal{K}^{b}
    &:=\left(\lambda_i\left\langle \left[\mathbf{K}^{(\cdot)}_{M,N}\right]_{ij} , b(x_i)\right\rangle_{\mathbb{R}^d}\right)_{1\leq i\leq n, 1\leq j\leq p}\in \mathbb{R}^{n\times p},\nonumber\\
    \mathbf{A}_{M,N}^\lambda
    &:=\left(\lambda_i\left[\mathbf{A}_{M,N}\right]_{ij} \right)_{1\leq i\leq n,1\leq j\leq m}\in \mathbb{R}^{n\times m},\nonumber\\
    \mathcal{A}^{c}
    &:=\mathbf{A}^T_{M,N}{\sf diag}(\lambda \tilde{c})\mathbf{A}_{M,N}\in \mathbb{R}^{m\times m},
    \quad \mathcal{H}^{c}:=\left(\lambda_ic(x_i)\left[\mathbf{H}_{M,N}\right]_{ij}\right)_{1\leq i\leq n, 1\leq j\leq p}\in \mathbb{R}^{n\times p}.\nonumber
\end{align}

Moreover, consider the matrices
\begin{align*}
    \Lambda
    &:=\sum\limits_{1\leq i\leq n} \lambda_i \mathcal{B}_i^T A(x_i)\mathcal{B}_i
    - \mathbf{A}_{M,N}^T \mathcal{B}^{b} 
    + \mathcal{A}^c\in \mathbb{R}^{m\times m},\\
    \Sigma
    &:=\sum\limits_{1\leq i\leq n} \lambda_i \mathcal{B}^T_{i}A(x_i)\mathcal{K}_i
    -\mathbf{A}_{M,N}^T\mathcal{K}^b 
    +\mathbf{A}_{M,N}^T\mathcal{H}^c\in \mathbb{R}^{m\times p}.
\end{align*}

Using the above notations,
the approximate identity \eqref{eq:matrix_test_varphi} recasts as
\begin{align*}
    &-\left\langle \sum\limits_{1\leq i\leq n} \lambda_i \mathcal{B}_i^T A(x_i)\mathcal{B}_i\bar{q},\;\bar{\varphi}\right\rangle_{\mathbb{R}^m} 
    + \left\langle\mathbf{A}_{M,N}^T \mathcal{B}^{b}\bar{q},\; \bar{\varphi}\right\rangle_{\mathbb{R}^m} 
    - \left\langle\mathcal{A}^c\bar{q},\; \bar{\varphi}\right\rangle_{\mathbb{R}^m}\\
    &\quad \approx \sum\limits_{1\leq i\leq n} \lambda_i \left\{f(x_i) \left[\mathbf{A}_{M,N}\bar{\varphi}\right]_i 
    + \left\langle A(x_i)\mathcal{K}_i\tilde{g},\mathcal{B}_i\bar{\varphi}\right\rangle_{\mathbb{R}^d}\right.\\
    &\left.\quad\quad\quad 
    - \left\langle b(x_i), \mathcal{K}_i\tilde{g}\right\rangle_{\mathbb{R}^d}\;  \left[\mathbf{A}_{M,N}\bar{\varphi}\right]_i 
    + c(x_i)\left[\mathbf{H}_{M,N}\tilde{g}\right]_i \left[\mathbf{A}_{M,N}\bar{\varphi}\right]_i\right\}\\
    &\quad 
    =\left\langle \mathbf{A}_{M,N}^\lambda\tilde{f}, \bar{\varphi}\right\rangle_{\mathbb{R}^m}
    +\left\langle \sum\limits_{1\leq i\leq n} \lambda_i \mathcal{B}^T_{i}A(x_i)\mathcal{K}_{i}\tilde{g},\bar{\varphi}\right\rangle_{\mathbb{R}^m} -\left\langle\mathbf{A}_{M,N}^T\mathcal{K}^b \tilde{g},\bar{\varphi}\right\rangle_{\mathbb{R}^m}
    +\left\langle\mathbf{A}_{M,N}^T\mathcal{H}^c \tilde{g},\bar{\varphi}\right\rangle_{\mathbb{R}^m}\\
    % &\quad \phantom{=\left\langle \mathbf{A}_{M,N}^\lambda\tilde{f}, \bar{\varphi}\right\rangle_{\mathbb{R}^m}}
    % -\left\langle\mathbf{A}_{M,N}^T\mathcal{K}^b \tilde{g},\bar{\varphi}\right\rangle_{\mathbb{R}^m}
    % +\left\langle\mathbf{A}_{M,N}^T\mathcal{H}^c \tilde{g},\bar{\varphi}\right\rangle_{\mathbb{R}^m}.
\end{align*}
Since $\varphi$ was arbitrarily chosen, we are motivated to consider
\begin{align*}
    \underbrace{\left[\sum\limits_{1\leq i\leq n} \lambda_i \mathcal{B}_i^T A(x_i)\mathcal{B}_i
    - \mathbf{A}_{M,N}^T \mathcal{B}^{b} 
    + \mathcal{A}^c\right]}_{\Lambda}\bar{q}
    \approx
    -\mathbf{A}_{M,N}^\lambda\tilde{f}
    -\underbrace{\left[\sum\limits_{1\leq i\leq n}\lambda_i \mathcal{B}^T_{i}A(x_i)\mathcal{K}_{i}
    -\mathbf{A}_{M,N}^T\mathcal{K}^b 
    +\mathbf{A}_{M,N}^T\mathcal{H}^c \right]}_{\Sigma}\tilde{g}.
\end{align*}

Consequently, we are naturally lead to consider and solve the following finite dimensional linear system 
\begin{equation}\label{eq:linear_q}
    \boxed{ \mbox{Find } \bar{q}\in \mathbb{R}^m \mbox{ such that:}
    \quad
    \Lambda \bar{q}= -\mathbf{A}^\lambda_{M,N}\tilde{f}-\Sigma \tilde{g}.
    \quad
    }
\end{equation}

Let $\bar{q}$ be a solution to the linear system \eqref{eq:linear_q}.
Then, the Monte Carlo estimator $\bar{u}_{M,N,V,W,\Gamma}$ of the solution $u$ to problem \eqref{eq:linearPDE}, retrieved at the points $(x_i)_{1\leq i\leq n}\subset D$, using the discretization $(V,W,\Gamma)$ of $\overline{D}$ settled in \Cref{ss:matrix_MC}, is obtained through
\begin{equation*}
    \boxed{
    \quad
    \mathbb{R}^n\ni \bar{u}_{M,N,V,W,\Gamma}:=\mathbf{A}_{M,N}\bar{q}+\mathbf{H}_{M,N}\tilde{g}.
    \quad}
\end{equation*}
Moreover, if $\Lambda^{-1}$ represents a (possibly regularized) generalized inverse of $\Lambda$, e.g. the pseudoinverse (see e.g. \cite{Ha87}), and $\bar{q}$ is obtained from \eqref{eq:linear_q} by multiplication with $\Lambda^{-1}$, then $\bar{u}_{M,N,V,W,\Gamma}$ can be written directly in terms of $\tilde{f}$ and $\tilde{g}$ as follows:
\begin{equation}\label{eq:linear_inversion}
    \boxed{
    \quad
    \mathbb{R}^n\ni \bar{u}_{M,N,V,W,\Gamma}:=-\mathbf{A}_{M,N} \Lambda^{-1}\mathbf{A}^\lambda_{M,N}\tilde{f}+\left[\mathbf{H}_{M,N}-\mathbf{A}_{M,N}\Lambda^{-1}\Sigma\right] \tilde{g}.
    \quad}
\end{equation}
Furthermore, let us denote by $G^\mathcal{L}$ and $\mu^\mathcal{L}$ the Green function and the elliptic measure, respectively, associated to the operator $\mathcal{L}$ given in \eqref{eq:linearPDE}, and the bounded domain $D$, and consider the associated integral operators
\begin{equation*}
    \mathcal{G}^\mathcal{L}f(x):=\int_D f(y) G^\mathcal{L}(x,y)\;dy, \;f\in L^\infty, \quad \mbox{and} \quad H^\mathcal{L}g(x):=\int_{\partial D} g(y)\;\mu_x^\mathcal{L}(dy), \; g\in C(\partial D), \quad x\in D.
\end{equation*}
Thus, for $f\in L^\infty(D)$ and $g\in C(\partial D)\cap H^{1/2}(\partial D)$, we have that the solution $u$ to \eqref{eq:linearPDE} admits the representation
\begin{equation*}
u(x)=\mathcal{G}^\mathcal{L}f(x)+H^\mathcal{L}g(x),\quad x\in D,
\end{equation*}
and consequently, corroborating with the estimate \eqref{eq:linear_inversion}, we obtain the following Monte Carlo operator estimates
\begin{equation*}
    \boxed{
    \quad
    \left[\mathcal{G}^\mathcal{L}f(x_i)\right]_{1\leq i\leq n}\approx -\mathbf{A}_{M,N} \Lambda^{-1}\mathbf{A}^\lambda_{M,N}\tilde{f}, \quad \left[H^\mathcal{L}g(x_i)\right]_{1\leq i\leq n}\approx \left[\mathbf{H}_{M,N}-\mathbf{A}_{M,N}\Lambda^{-1}\Sigma\right]\tilde{g}.}
\end{equation*}

\subsection{Euler-Lagrange boundary value problems}\label{ss:EL}

Let $D\subset \mathbb{R}^d$ be a bounded domain, $F:D \times \mathbb{R} \times \mathbb{R}^d\rightarrow \mathbb{R}$ be a continuous function,
and consider the energy functional
\begin{equation}\label{eq:F}
    \mathcal{F}(u)=\frac{1}{\lambda(D)}\int_D F(x,u(x),\nabla u(x)) \; dx, \quad u\in {\sf D}(\mathcal{F}),
\end{equation} 
where $\lambda$ denotes the Lebesgue measure on $\mathbb{R}^d$.
Here, ${\sf D}(\mathcal{F})$ is typically a Banach space of functions $u$, such that $\mathcal{F}(u)$ is well defined.
We assume that $F$ is such that $\mathop{\bigcap}\limits_{p\geq 1} W^{1,p}(D)\subset {\sf D}(\mathcal{F})$, and we denote by ${\sf D}(\mathcal{F})_0$ the closure of $C_c^\infty(D)$ in ${\sf D}(\mathcal{F})$.
The first aim is to minimize $\mathcal{F}$ over all elements that match a prescribed boundary data $g$, namely to find 
\begin{equation}\label{eq:minimization_pb}
    u^\ast = \argmin_{u\in {\sf D}(\mathcal{F}),\; u|_{\partial D}=g} \mathcal{F}(u), \quad \mbox{where }u|_{\partial D}=g \mbox{ is understood in the sense that } u-Hg\in {\sf D}(\mathcal{F})_0,
\end{equation}
assuming that such a minimal solution exists. 
Recall that if $F\in C^1$ in $(u,p)\in \mathbb{R}\times \mathbb{R}^d$, then  any $u^\ast$ satisfying \eqref{eq:minimization_pb} also satisfies the weak Euler-Lagrange equation
\begin{equation}\label{eq:EL}
    \quad\int_{D}\frac{\partial F}{\partial u}(\cdot,u^\ast,\nabla u^\ast)\varphi-\frac{\partial F}{\partial p}(\cdot,u^\ast,\nabla u^\ast)\cdot\nabla \varphi \, dx=0, \quad \varphi\in C_c^\infty(D),
\end{equation}
subject to $u^\ast|_{\partial D}=g, \mbox{ i.e. } u^\ast-Hg\in{\sf D}(\mathcal{F})_0$.
Solving \eqref{eq:EL} is the second aim that we also consider in our numerical tests.

Precisely as in \Cref{ss:linearPDE}, our starting point for the numerically treatment of \eqref{eq:minimization_pb} or \eqref{eq:EL}, is to search for solutions $u^\ast$ that can be approximated as
\begin{equation}\label{eq:representation_u*}
    u^\ast\approx U f^\ast + Hg, \quad \mbox{for some unknown } f^\ast\in L^\infty(D) \mbox{ that remains to be estimated}. 
\end{equation}

Let us assume that $U(L^\infty(D))+Hg\subset{\sf D}(\mathcal{F})$,
so that in light of \eqref{eq:representation_u*}, we are motivated to introduce a new functional, $\mathcal{J}$, given by
\begin{align}
    &\mathcal{J}:{\sf D}(\mathcal{J}):=L^\infty(D)\subset L^2(D)\rightarrow \mathbb{R},\nonumber\\
    &\mathcal{J}(f)
    =\mathcal{F}(U f+Hg)\nonumber\\
    &\phantom{\mathcal{J}(f)}=\frac{1}{\lambda(D)}\int_D F(x,U f(x)+Hg(x),\nabla U f(x)+\nabla Hg(x)) \; dx, \quad f\in L^\infty(D).\label{eq:functional J}
\end{align}

\begin{rem}\label[rem]{rem:C1alpha}
\begin{enumerate}
    \item[(i)] Assume that $D$ satisfies a uniform exterior sphere condition. 
    If $g$ is Lipschitz on $\partial D$, we have that $\nabla Hg$ explodes at most logarithmically at the boundary (see e.g. \cite[Remark 3.1 \& Proposition 3.2]{cimpean2025convergence}), hence $Hg\in W^{1,p}(D)$ for every $1\leq p<\infty$.
    Furthermore, if $f\in L^\infty(D)$ then $Uf$ is globally Lipschitz; indeed, by the already invoked \cite[Corollary 8.36]{GiTr01} we have that  $U f\in C^{1,\alpha}_{\sf loc}(D)$, hence the global Lipschitz regularity is ensured by e.g. \cite[Theorem 1.5, (ii)]{cimpean2025quantitative}.
    Consequently, we have
        \begin{equation*}
        U(L^\infty(D))+Hg\subset {\sf D}(\mathcal{F}).
    \end{equation*}
      \item[(ii)] If $D$ is of class $C^{1,\alpha}$ we can say more: If $g\in C^{1,\alpha}(\overline{D})$, by \cite[Theorem 8.34]{GiTr01} we deduce that 
    \begin{equation*}
        U(L^\infty(D))+Hg\subset C^{1,\alpha}(\overline{D})\subset W^{1,\infty}(D)\subset {\sf D}(\mathcal{F}).
    \end{equation*}
\end{enumerate}
\end{rem}

\paragraph{Discretization of the functional $\mathcal{J}$.}

The first step is to consider the discretization of the integral in \eqref{eq:functional J}, assuming that the integrand is constant and equal to its value in $x_i$, on each cell $W_i, 1\leq i\leq n$ introduced in \Cref{ss:matrix_MC}, namely
\begin{equation*}
    \mathcal{J}(f)\approx\mathcal{J}_n(f):=\sum_{1\leq i\leq n}\frac{\lambda(W_i)}{\lambda(D)} F(x_i,U f(x_i)+Hg(x_i),\nabla U f(x_i)+\nabla Hg(x_i)), \quad f\in L^\infty(D).\label{eq:functional J_n}
\end{equation*}
Note that by the regularity results recalled in \Cref{ss1}, for $f\in L^\infty(D)$ and $g$ bounded and measurable on $\partial D$,
\begin{equation*}
    D\ni x\mapsto F(x,U f(x)+Hg(x),\nabla U f(x)+\nabla Hg(x)) \mbox{ is continuous},
\end{equation*}
so the point-wise evaluation is well-defined.

The next step is to use in the expression of $\mathcal{J}_n(f)$ the (Monte Carlo) matrix estimators $\mathbf{A}_{M,N}$ \eqref{eq:Estimator_A}, $ \mathbf{B}^{(j)}_{M,N}$ \eqref{eq:Estimator_B}, $\mathbf{H}_{M,N}$ \eqref{eq:Estimator_H}, $\mathbf{K}^{(j)}_{M,N}$ \eqref{eq:Estimator_K} from \Cref{ss:matrix_MC} for the operators $U,\nabla U,H,\nabla H$, which leads to
\begin{align}
    \mathcal{J}(f)
    &\approx\mathcal{J}_n(f)\nonumber\\
    &\approx \mathcal{J}_{n,m,p,M,N}(\bar{f})
    :=\sum_{1\leq i\leq n}\frac{\lambda(W_i)}{\lambda(D)} F(x_i,\ \mathbf{A}_{M,N}\bar{f}(i)+\mathbf{H}_{M,N} \tilde{g}(i),\ \mathcal{B}_i\bar{f}+\mathcal{K}_i\tilde{g}),\label{eq:gradJ}
\end{align}
where $\bar{f}\in \mathbb{R}^m$ and $\tilde{g}\in \mathbb{R}^p$ are obtained from $f$ and $g$ as in \eqref{eq:f approx} and \eqref{eq:g approx}, respectively, whilst $\mathcal{B}_i$ and $\mathcal{K}_i$ are given by \eqref{eq:B,K_i}.

As a matter of fact, $\mathcal{J}_{n,m,M,N}$ is now well defined as a functional
\begin{equation}
\mathcal{J}_{n,m,p,M,N}:\mathbb{R}^m\rightarrow\mathbb{R}, 
\end{equation}
and hence the original problems 
\eqref{eq:minimization_pb}, \eqref{eq:EL} are reduced numerically as follows:
\begin{equation}\label{eq:min_el_discrete}
    \boxed{
    \quad
    \eqref{eq:minimization_pb} \mbox{ reduces to } \bar{f}^\ast
    :=\argmin\limits_{\bar{f}\in \mathbb{R}^m}\mathcal{J}_{n,m,p,M,N}(\bar{f}),\quad \mbox{whilst } \eqref{eq:EL} \mbox{ reduces to}\quad  \nabla \mathcal{J}_{n,m,p,M,N}(\bar{f}^\ast)=0.
    \quad
    }
\end{equation}
Consequently, by solving problems \eqref{eq:min_el_discrete}, we end up with the following approximation of the solution $u^\ast$ to problem \eqref{eq:minimization_pb}, or \eqref{eq:EL}, respectively:
\begin{equation*}
    \boxed{
    \quad
    u^\ast(x_i)
    \approx \mathbf{A}_{M,N}\bar{f}^\ast(i)+\mathbf{H}_{M,N}\tilde{g}(i),
    \quad
    \nabla u^\ast(x_i)
    \approx \mathcal{B}_i\bar{f}^\ast+\mathcal{K}_i\tilde{g}, \quad 1\leq i\leq n.
    \quad
    }
\end{equation*}
As a consequence, we can estimate $u^\ast,\nabla u^\ast$ in $D$ by a piecewise constant interpolation on the partition $(W_i)_{1\leq i\leq n}$, namely
\begin{equation*}
    u^\ast(x)\approx \sum\limits_{1\leq i\leq n}\left[\mathbf{A}_{M,N}\bar{f}^\ast(i)+\mathbf{H}_{M,N}\tilde{g}(i)\right]1_{W_i}(x),
    \quad
    \nabla u^\ast(x)\approx \sum\limits_{1\leq i\leq n}\left[\mathcal{B}_i\bar{f}^\ast+\mathcal{K}_i\tilde{g}\right]1_{W_i}(x), \quad x\in D.
\end{equation*}
As a matter of fact, we can use the estimate for $\nabla u^\ast(x_i)$ in order to enhance the approximation for $u^\ast$ through
\begin{equation*}
    u^\ast(x)
    \approx \mathbf{A}_{M,N}\bar{f}^\ast(i)+\mathbf{H}_{M,N}\tilde{g}(i) + \left\langle\left[\mathcal{B}_i\bar{f}^\ast+\mathcal{K}_i\tilde{g}\right], x-x_i\right\rangle, \quad x\in W_i, \quad 1\leq i\leq n.
\end{equation*}

\paragraph{The gradient of \texorpdfstring{$\mathcal{J}_{n,m,p,M,N}$}{}.}
As $\mathcal{J}_{n,m,p,M,N}:\mathbb{R}^m\rightarrow \mathbb{R}$ is an 
explicit ((non)linear) mapping whose regularity is given by the regularity of $F(x,u,p)$ in the $(u,p)$-variable, problems \eqref{eq:min_el_discrete} could in principle be tackled by any preferred method. 
In our numerical tests we use gradient methods like Adam optimizer or High-index Saddle Dynamics (HiSD).
Consequently, the gradient of $\mathcal{J}_{n,m,p,M,N}$ is required.
Its existence is ensured by further imposing
\begin{enumerate}
    \item[${\sf (H_F)}$] The partial derivatives $\partial_uF,\partial_pF$ are continuous.
\end{enumerate}
Under ${\sf (H_F)}$, it is straightforward that $ \mathcal{J}_{n,m,p,M,N}:\mathbb{R}^m\rightarrow\mathbb{R}$ is continuously differentiable.
Moreover, using the explicit expression \eqref{eq:gradJ} of $ \mathcal{J}_{n,m,M,N}$, and setting
\begin{equation*}
    \left[\mathbf{A}_{M,N}\right]_{i}:=\left(\left[\mathbf{A}_{M,N}\right]_{ij}\right)_{1\leq j\leq m}\in \mathbb{R}^m, \quad 1\leq i \leq n,
\end{equation*}
a straightforward differentiation gives:
\begin{equation*}
\boxed{
\quad
\begin{aligned}
   &
   \nabla_{\bar f} \;\mathcal{J}_{n,m,p,M,N}( \bar f)
   = \sum\limits_{i=1}^{n} \frac{\lambda(W_i)}{\lambda(D)} \Big\{ \left[\mathbf{A}_{M,N}\right]_{i} \ \partial_{u} F(x_i, \mathbf{A}_{M,N} \bar f(i) + \mathbf{H}_{M,N} \bar g(i),\mathcal{B}_i \bar f + \mathcal{K}_i \tilde{g})\\
    & \;\qquad\qquad\qquad \qquad \qquad + \mathcal{B}^T_i \ \partial_{p} F(x_i, \mathbf{A}_{M,N} \bar f(i) + \mathbf{H}_{M,N} \bar g(i), \mathcal{B}_i\bar{f} + \mathcal{K}_i \tilde{g}) \Big\}\in \mathbb{R}^m, \quad \bar f\in \mathbb{R}^m.
\end{aligned}
\quad
}
\end{equation*}

%% file: Algorithm.tex
\begin{algorithm}[H]
\footnotesize
\caption{({\footnotesize WoS})-LOT: $\mathbf{A}_{M,N}$}
\label{alg:A}
\begin{algorithmic}
\Require $(\mathbf{x}_i)_{1\leq i\leq n},(\mathbf{y}_i)_{1\leq i\leq m}$
\Comment{Centers of the Voronoi diagrams $\mathbf{W}$ and $\mathbf{V}$}\\ 
\vspace{-10pt}
\State \hspace{\algorithmicindent} \hspace{\algorithmicindent} $r$ \Comment{Distance function to the boundary $\partial D$}\\ 
\vspace{-10pt}
\State \hspace{\algorithmicindent} \hspace{\algorithmicindent} $M,N$ \Comment{Number of WoS steps, number of i.i.d. samples of WoS trajectories} \\
\vspace{-10pt}
\Ensure $\mathbf{A}_{M,N}$ \\

\State $\mathbf{A}_{M,N} \gets \mathbf{0} \quad \in \mathbb{R}^{n \times m}$ \Comment{Initialize matrix}

\vspace{0.1cm}

\For{$l\in \{1,\dots,N\}$} \Comment{Implement the estimator from \eqref{eq:Estimator_A}}

\vspace{0.1cm}

\State $X^{\rm WoS}_{1:n} \gets \mathbf{x}_{1:n}$ \Comment{Initialize coordinates matrix}

\vspace{0.1cm}

\State $Y \sim -2d G_{B(0,1)}(0,y)dy\quad\in B(0,1)\subset \mathbb{R}^d$ \Comment{Sample $Y$ as in \Cref{prop:U_0 and DU_0}; see also \Cref{prop:removing singularities} (i)}

\vspace{0.1cm}

\For{$s\in\{1,\dots,M\}$}
\Comment{Iterate the $M$ steps of the WoS according to \eqref{eq:Estimator_A}}
\vspace{0.1cm}

\State $\mathbf{r}^{\rm WoS}_{1:n} \gets  r\left(X^{\rm WoS}_{1:n}\right)$

\vspace{0.1cm}

\For{$(i,k) \in \{1,\dots,n\} \times \{1,\dots,m\}$} \Comment{Compute $\mathbf{A}_{M,N}$ according to \eqref{eq:Estimator_A}}

\vspace{0.1cm}

\State{$\delta_{ik}\gets 1_{V_k} \left(X^{\rm WoS}_i + \mathbf{r}^{\rm WoS}_{i} Y \right)$} 
\Comment{$V_k$ is the Voronoi cell with center $\mathbf{y}_k$; see also \Cref{rem:implementation_tips} (iii)}

\vspace{0.1cm}

\State $ \left[ \mathbf{A}_{M,N} \right]_{ik} \gets \left[ \mathbf{A}_{M,N} \right]_{ik}+ \frac{1}{N} \ \frac{ \left(\mathbf{r}^{\rm WoS}_{i}\right)^2}{d} \delta_{ik}$

\EndFor

\vspace{0.1cm}

\State $\mathbf{U} \sim {\rm Uniform}(\partial B(0,1))\quad \in \mathbb{R}^d$

\vspace{0.1cm}

\State $X^{\rm WoS}_{1:n} \gets X^{\rm WoS}_{1:n} + \mathbf{r}^{\rm WoS}_{1:n} \mathbf{U} \quad\in \mathbb{R}^{n\times d}$\Comment{Update WoS position }

\EndFor

\EndFor
\end{algorithmic}
\end{algorithm}

\begin{algorithm}[H]
\footnotesize
\caption{({\footnotesize WoS})-LOT: $\mathbf{B}_{M,N}$}
\label{alg:B}
\begin{algorithmic}
\Require $(\mathbf{x}_i)_{1\leq i\leq n},(\mathbf{y}_i)_{1\leq i\leq m}$, \Comment{Centers of the Voronoi diagrams $\mathbf{W}$ and $\mathbf{V}$}\\ 
\vspace{-10pt}
\State \hspace{\algorithmicindent} \hspace{\algorithmicindent} $r$ \Comment{Distance function to the boundary $\partial D$}\\ 
\vspace{-10pt}
\State \hspace{\algorithmicindent} \hspace{\algorithmicindent} $M,N$ \Comment{Number of WoS steps, number of i.i.d. samples of WoS trajectories} \\
\vspace{-10pt}
\Ensure $\left(\mathbf{B}^{(j)}_{M,N}\right)_{1\leq j\leq d}$ \\

\State $\mathbf{B}^{(j)}_{M,N} \gets \mathbf{0} \quad\in \mathbb{R}^{n \times m},\quad 1\leq j\leq d$ \Comment{Initialize matrix}

\vspace{0.1cm}

\State $\mathbf{r}_{1:n} \gets r(\mathbf{x}_{1:n})\quad\in \mathbb{R}^n$
\Comment{Compute the (vectorized) distance from $\mathbf{x}_{1:n}:=(\mathbf{x}_i)_{1\leq i\leq n}$ to $\partial D$}
\vspace{0.1cm}

\For{$l\in \{1,\dots,N\}$} \Comment{Implement the estimator from \eqref{eq:Estimator_B}}

\vspace{0.1cm}

\State $\widetilde{Y} \sim {\rm Uniform}(B({0},1))\quad\in \mathbb{R}^d$ \Comment{Sample uniformly distributed point on unit ball}

\vspace{0.1cm}

\For{$(i,k) \in \{1,\dots,n\} \times \{1,\dots,m\}$} \Comment{Compute the second term in the RHS of \eqref{eq:Estimator_B}}

\vspace{0.1cm}

\State{$\alpha_{ik} \gets 1_{V_k} \left(\mathbf{x}_i + \mathbf{r}_i \widetilde{Y} \right)$}
\Comment{$V_k$ is the Voronoi cell with center $\mathbf{y}_k$; see also \Cref{rem:implementation_tips} (iii)}

\vspace{0.1cm}

\For{$j \in \{1,\dots,d\}$} 

\vspace{0.1cm}

\State $ \left[ \mathbf{B}^{(j)}_{M,N} \right]_{ik} \gets \left[ \mathbf{B}^{(j)}_{M,N} \right]_{ik} +  \frac{2\mathbf{r}_i}{Nd}  \widetilde{Y}_{j} \left( \frac{1}{|\widetilde{Y}|^d} -1 \right) \ \alpha_{ik} $

\EndFor
\EndFor

\vspace{0.1cm}

\State $\mathbf{U}_1 \sim {\rm Uniform}(\partial B(0,1))\quad\in \mathbb{R}^d$ \Comment{Sample uniformly distributed point on unit sphere}

\vspace{0.1cm}

\State $X^{\rm WoS, \;1}_{1:n} \gets \mathbf{x}_{1:n} + \mathbf{r}_{1:n} \mathbf{U}_1 \quad\in \mathbb{R}^{n\times d}$\Comment{Record the first step of WoS starting from the centers $\mathbf{x}_{1:n}$}

\vspace{0.1cm}

\State $X^{\rm WoS}_{1:n} \gets X^{\rm WoS, \;1}_{1:n}$ \Comment{Initialize coordinates matrix}

\vspace{0.1cm}

\State $Y \sim -2d G_{B(0,1)}(0,y)dy\quad\in B(0,1)\subset \mathbb{R}^d$ \Comment{Sample $Y$ as in \Cref{prop:U_0 and DU_0}; see also \Cref{prop:removing singularities} (i)}

\vspace{0.1cm}

\For{$s\in\{1,\dots,M\}$}
\Comment{Iterate the $M$ steps of the WoS according to \eqref{eq:Estimator_B}}
\vspace{0.1cm}

\State $\mathbf{r}^{\rm WoS}_{1:n} \gets  r\left(X^{\rm WoS}_{1:n}\right)$

\vspace{0.1cm}

\For{$(i,k) \in \{1,\dots,n\} \times \{1,\dots,m\}$} \Comment{Compute $\mathbf{B}^{(j)}_{M,N}$ according to \eqref{eq:Estimator_B}}

\vspace{0.1cm}

\State{$\delta_{ik}\gets 1_{V_k} \left(X^{\rm WoS}_i + \mathbf{r}^{\rm WoS}_{i} Y \right)$}

\vspace{0.1cm}

\For{$j \in \{1,\dots,d\}$} 

\vspace{0.1cm}

\State $ \left[ \mathbf{B}^{(j)}_{M,N} \right]_{ik} \gets \left[ \mathbf{B}^{(j)}_{M,N} \right]_{ik}+ \frac{1}{\mathbf{r}_i^2 N}  \left(X^{\rm WoS, \;1}_{i}-\mathbf{x}_i\right)_j (\mathbf{r}^{\rm WoS}_i)^2 \delta_{ik} $
\EndFor
\EndFor

\vspace{0.1cm}

\State $\mathbf{U} \sim {\rm Uniform}(\partial B(0,1))\quad \in \mathbb{R}^d$

\vspace{0.1cm}

\State $X^{\rm WoS}_{1:n} \gets X^{\rm WoS}_{1:n} + \mathbf{r}^{\rm WoS}_{1:n} \mathbf{U} \quad\in \mathbb{R}^{n\times d}$\Comment{Update WoS position }

\EndFor
\EndFor

\end{algorithmic}
\end{algorithm}

\begin{algorithm}[H]
\footnotesize
\caption{({\footnotesize WoS})-LOT: $\mathbf{H}_{M,N}$}
\label{alg:H}
\begin{algorithmic}
\Require $(\mathbf{x}_i)_{1\leq i\leq n}, (\mathbf{z}_i)_{1\leq i\leq p}$, \Comment{Centers of the Voronoi diagrams $\mathbf{W}$ and $\mathbf{\Gamma}$}\\ 
\vspace{-10pt}
\State \hspace{\algorithmicindent} \hspace{\algorithmicindent} $r$ \Comment{Distance function to the boundary $\partial D$}\\ 
\vspace{-10pt}
\State \hspace{\algorithmicindent} \hspace{\algorithmicindent} $M,N$ \Comment{Number of WoS steps, number of i.i.d. samples of WoS trajectories} \\
\vspace{-10pt}
\Ensure $\mathbf{H}_{M,N}$; \\

\State $\mathbf{H}_{M,N} \gets \mathbf{0} \quad\in \mathbb{R}^{n \times p}$ \Comment{Initialize matrix}

\vspace{0.1cm}

\For{$l\in \{1,\dots,N\}$} \Comment{Implement the estimator from \eqref{eq:Estimator_H}}

\vspace{0.1cm}

\State $X^{\rm WoS}_{1:n} \gets \mathbf{x}_{1:n}$ \Comment{Initialize coordinates matrix}

\vspace{0.1cm}

\For{$s\in\{1,\dots,M\}$}
\Comment{Iterate the $M$ steps of the WoS according to \eqref{eq:Estimator_H}}
\vspace{0.1cm}

\State $\mathbf{U} \sim {\rm Uniform}(\partial B(0,1))\quad \in \mathbb{R}^d$

\vspace{0.1cm}

\State $X^{\rm WoS}_{1:n} \gets X^{\rm WoS}_{1:n} + r\left(X^{\rm WoS}_{1:n}\right) \mathbf{U} \quad\in \mathbb{R}^{n\times d}$\Comment{Update WoS position }

\EndFor

\vspace{0.1cm}

\For{$(i,k) \in \{1,\dots,n\}\times\{1,\dots,p\}$} \Comment{Compute $\mathbf{H}_{M,N}$ according to \eqref{eq:Estimator_H}}

\vspace{0.1cm}

\State{$\gamma_{ik}\gets 1_{\Gamma_k} \left(X^{\rm WoS}_i \right)$}
\Comment{$\Gamma_k$ is the Voronoi cell with center $\mathbf{z}_k$; see also \Cref{rem:implementation_tips} (iii)}

\vspace{0.1cm}

\State $ \left[ \mathbf{H}_{M,N} \right]_{ik} \gets \left[ \mathbf{H}_{M,N} \right]_{ik}+ \frac{1}{N}\gamma_{ik}  $

\EndFor
\EndFor
\end{algorithmic}
\end{algorithm}

\begin{algorithm}[H]
\footnotesize
\caption{({\footnotesize WoS})-LOT: $\mathbf{K}_{M,N}$}
\label{alg:K}
\begin{algorithmic}
\Require $(\mathbf{x}_i)_{1\leq i\leq n}, (\mathbf{z}_i)_{1\leq i\leq p}$ \Comment{Centers of the Voronoi diagrams $\mathbf{W}$ and $\mathbf{\Gamma}$}\\ 
\vspace{-10pt}
\State \hspace{\algorithmicindent} \hspace{\algorithmicindent} $r$ \Comment{Distance function to the boundary $\partial D$}\\ 
\vspace{-10pt}
\State \hspace{\algorithmicindent} \hspace{\algorithmicindent} $M,N$ \Comment{Number of WoS steps, number of i.i.d. samples of WoS trajectories} \\
\vspace{-10pt}
\Ensure $\left(\mathbf{K}^{(j)}_{M,N}\right)_{1\leq j\leq d} $; \\

\State $\mathbf{K}^{(j)}_{M,N} \gets \mathbf{0} \quad\in \mathbb{R}^{n \times p},\quad 1\leq j\leq d$ \Comment{Initialize matrix}

\vspace{0.1cm}

\State $\mathbf{r}_{1:n} \gets r(\mathbf{x}_{1:n})\quad\in \mathbb{R}^n$
\Comment{Compute the (vectorized) distance from $\mathbf{x}_{1:n}:=(\mathbf{x}_i)_{1\leq i\leq n}$ to $\partial D$}
\vspace{0.1cm}

\For{$l\in \{1,\dots,N\}$} \Comment{Implement the estimator from \eqref{eq:Estimator_K}}

\vspace{0.1cm}

\State $\mathbf{U}_1 \sim {\rm Uniform}(\partial B(0,1))\quad\in \mathbb{R}^d$ \Comment{Sample uniformly distributed point on unit sphere}

\vspace{0.1cm}

\State $X^{\rm WoS, \;1}_{1:n} \gets \mathbf{x}_{1:n} + \mathbf{r}_{1:n} \mathbf{U}_1 \quad\in \mathbb{R}^{n\times d}$\Comment{Record the first step of WoS starting from the centers $\mathbf{x}_{1:n}$}

\vspace{0.1cm}

\State $X^{\rm WoS}_{1:n} \gets X^{\rm WoS, \;1}_{1:n}$ \Comment{Initialize coordinates matrix}

\vspace{0.1cm}

\For{$s\in\{1,\dots,M\}$}
\Comment{Iterate the $M$ steps of the WoS according to \eqref{eq:Estimator_K}}
\vspace{0.1cm}

\State $\mathbf{r}^{\rm WoS}_{1:n} \gets  r\left(X^{\rm WoS}_{1:n}\right)$

\vspace{0.1cm}

\State $\mathbf{U} \sim {\rm Uniform}(\partial B(0,1))\quad \in \mathbb{R}^d$

\vspace{0.1cm}

\State $X^{\rm WoS}_{1:n} \gets X^{\rm WoS}_{1:n} + \mathbf{r}^{\rm WoS}_{1:n} \mathbf{U} \quad\in \mathbb{R}^{n\times d}$\Comment{Update WoS position }

\EndFor

\vspace{0.1cm}

\For{$(i,k) \in \{1,\dots,n\}\times\{1,\dots,p\}$} \Comment{Compute $\mathbf{K}^{(j)}_{M,N}$ according to \eqref{eq:Estimator_K}}

\vspace{0.1cm}

\State{$\gamma_{ik}\gets 1_{\Gamma_k} \left(X^{\rm WoS}_i \right)$}
\Comment{$\Gamma_k$ is the Voronoi cell with center $\mathbf{z}_k$; see also \Cref{rem:implementation_tips} (iii)}

\For{$j\in \{1,\dots,d\}$}

\State{$\left[ \mathbf{K}^{(j)}_{M,N} \right]_{ ik} \gets \left[ \mathbf{K}^{(j)}_{M,N} \right]_{ik}+ \ \frac{d}{N\mathbf{r}_i^2} \left(X^{\rm WoS, \;1}_{i}-\mathbf{x}_i\right)_j \ \gamma_{ik}$}

\EndFor

\EndFor
\EndFor
\end{algorithmic}
\end{algorithm}

\begin{algorithm}[!hb]
\footnotesize
\caption{({\footnotesize WoS})-LOT}
\label{alg:LOT}
\begin{algorithmic}
\Require $(\mathbf{x}_i)_{1\leq i\leq n},(\mathbf{y}_i)_{1\leq i\leq m}, (\mathbf{z}_i)_{1\leq i\leq p}$ \Comment{Centers of the Voronoi diagrams $\mathbf{W}$, $\mathbf{V}$, and $\mathbf{\Gamma}$}\\ 
\vspace{-10pt}
\State \hspace{\algorithmicindent} \hspace{\algorithmicindent} $r$ \Comment{Distance function to the boundary $\partial D$}\\ 
\vspace{-10pt}
\State \hspace{\algorithmicindent} \hspace{\algorithmicindent} $M,N$ \Comment{Number of WoS steps, number of i.i.d. samples of WoS trajectories} \\
\vspace{-10pt}
\Ensure $\mathbf{A}_{M,N},\; \left(\mathbf{B}^{(j)}_{M,N}\right)_{1\leq j\leq d},\; \mathbf{H}_{M,N}, \;\left(\mathbf{K}^{(j)}_{M,N}\right)_{1\leq j\leq d} $; \\

\State $\mathbf{A}_{M,N} \gets \mathbf{0} \quad \in \mathbb{R}^{n \times m}$ \Comment{Initialize matrices}

\vspace{0.1cm}

\State $\mathbf{B}^{(j)}_{M,N} \gets \mathbf{0} \quad\in \mathbb{R}^{n \times m},\quad 1\leq j\leq d$

\vspace{0.1cm}

\State $\mathbf{H}_{M,N} \gets \mathbf{0} \quad\in \mathbb{R}^{n \times p}$ 

\vspace{0.1cm}

\State $\mathbf{K}^{(j)}_{M,N} \gets \mathbf{0} \quad\in \mathbb{R}^{n \times p},\quad 1\leq j\leq d$

\vspace{0.1cm}

\State $\mathbf{r}_{1:n} \gets r(\mathbf{x}_{1:n})\quad\in \mathbb{R}^n$
\Comment{Compute the (vectorized) distance from $\mathbf{x}_{1:n}:=(\mathbf{x}_i)_{1\leq i\leq n}$ to $\partial D$}
\vspace{0.1cm}

\For{$l\in \{1,\dots,N\}$} \Comment{Implement the estimators from \eqref{eq:Estimator_A},\eqref{eq:Estimator_B},\eqref{eq:Estimator_H},\eqref{eq:Estimator_K}}

\vspace{0.1cm}

\State $\widetilde{Y} \sim {\rm Uniform}(B({0},1))\quad\in \mathbb{R}^d$ \Comment{Sample uniformly distributed point on unit ball}

\vspace{0.1cm}

\For{$(i,k) \in \{1,\dots,n\} \times \{1,\dots,m\}$} \Comment{Compute the second term in the RHS of \eqref{eq:Estimator_B}}

\vspace{0.1cm}

\State{$\alpha_{ik} \gets 1_{V_k} \left(\mathbf{x}_i + \mathbf{r}_i \widetilde{Y} \right)$}
\Comment{$V_k$ is the Voronoi cell with center $\mathbf{y}_k$; see also \Cref{rem:implementation_tips} (iii)}

\vspace{0.1cm}

\For{$j \in \{1,\dots,d\}$} 

\vspace{0.1cm}

\State $ \left[ \mathbf{B}^{(j)}_{M,N} \right]_{ik} \gets \left[ \mathbf{B}^{(j)}_{M,N} \right]_{ik} +  \frac{2\mathbf{r}_i}{Nd}  \widetilde{Y}_{j} \left( \frac{1}{|\widetilde{Y}|^d} -1 \right) \ \alpha_{ik} $

\EndFor
\EndFor

\vspace{0.1cm}

\State $\mathbf{U}_1 \sim {\rm Uniform}(\partial B(0,1))\quad\in \mathbb{R}^d$ \Comment{Sample uniformly distributed point on unit sphere}

\vspace{0.1cm}

\State $X^{\rm WoS, \;1}_{1:n} \gets \mathbf{x}_{1:n} + \mathbf{r}_{1:n} \mathbf{U}_1 \quad\in \mathbb{R}^{n\times d}$\Comment{Record the first step of WoS starting from the centers $\mathbf{x}_{1:n}$}

\vspace{0.1cm}

\State $X^{\rm WoS}_{1:n} \gets X^{\rm WoS, \;1}_{1:n}$ \Comment{Initialize coordinates matrix}

\vspace{0.1cm}

\State $Y \sim -2d G_{B(0,1)}(0,y)dy\quad\in B(0,1)\subset \mathbb{R}^d$ \Comment{Sample $Y$ as in \Cref{prop:U_0 and DU_0},\Cref{prop:removing singularities}}

\vspace{0.1cm}

\For{$(i,k) \in \{1,\dots,n\} \times \{1,\dots,m\}$} \Comment{Update $\mathbf{A}_{M,N}$ with the first term of the 2nd sum in \eqref{eq:Estimator_A} }

\vspace{0.1cm}

\State{$\delta_{ik}\gets 1_{V_k} \left(\mathbf{x}_i + \mathbf{r}_{i} Y \right)$}

\vspace{0.1cm}

\State $ \left[ \mathbf{A}_{M,N} \right]_{ik} \gets \left[ \mathbf{A}_{M,N} \right]_{ik}+ \frac{1}{N} \ \frac{ \mathbf{r}_{i}^2}{d} \delta_{ik}$

\EndFor

\vspace{0.1cm}

\For{$s\in\{1,\dots,M\}$}
\Comment{Iterate the $M$ steps of the WoS according to \eqref{eq:Estimator_A},\eqref{eq:Estimator_B},\eqref{eq:Estimator_H},\eqref{eq:Estimator_K}}
\vspace{0.1cm}

\State $\mathbf{r}^{\rm WoS}_{1:n} \gets  r\left(X^{\rm WoS}_{1:n}\right)$

\vspace{0.1cm}

\For{$(i,k) \in \{1,\dots,n\} \times \{1,\dots,m\}$} \Comment{Compute $\mathbf{A}_{M,N}$ and $\mathbf{B}^{(j)}_{M,N}$ according to \eqref{eq:Estimator_A} - \eqref{eq:Estimator_B}}

\vspace{0.1cm}

\State{$\delta_{ik}\gets 1_{V_k} \left(X^{\rm WoS}_i + \mathbf{r}^{\rm WoS}_{i} Y \right)$}

\vspace{0.1cm}

\State $ \left[ \mathbf{A}_{M,N} \right]_{ik} \gets \left[ \mathbf{A}_{M,N} \right]_{ik}+ \frac{1}{N} \ \frac{ \left(\mathbf{r}^{\rm WoS}_{i}\right)^2}{d} \delta_{ik}$

\vspace{0.1cm}

\For{$j \in \{1,\dots,d\}$} 

\vspace{0.1cm}

\State $ \left[ \mathbf{B}^{(j)}_{M,N} \right]_{ik} \gets \left[ \mathbf{B}^{(j)}_{M,N} \right]_{ik}+ \frac{1}{\mathbf{r}_i^2 N}  \left(X^{\rm WoS, \;1}_{i}-\mathbf{x}_i\right)_j (\mathbf{r}^{\rm WoS}_i)^2 \delta_{ik} $
\EndFor
\EndFor

\vspace{0.1cm}

\State $\mathbf{U} \sim {\rm Uniform}(\partial B(0,1))\quad \in \mathbb{R}^d$

\vspace{0.1cm}

\State $X^{\rm WoS}_{1:n} \gets X^{\rm WoS}_{1:n} + \mathbf{r}^{\rm WoS}_{1:n} \mathbf{U} \quad\in \mathbb{R}^{n\times d}$\Comment{Update WoS position }

\EndFor

\vspace{0.1cm}

\For{$(i,k) \in \{1,\dots,n\}\times\{1,\dots,p\}$} \Comment{Compute $\mathbf{H}_{M,N}$ and $\mathbf{K}^{(j)}_{M,N}$ according to \eqref{eq:Estimator_H} - \eqref{eq:Estimator_K}}

\vspace{0.1cm}

\State{$\gamma_{ik}\gets 1_{\Gamma_k} \left(X^{\rm WoS}_i \right)$}
\Comment{$\Gamma_k$ is the Voronoi cell with center $\mathbf{z}_k$; see also \Cref{rem:implementation_tips} (iii)}

\vspace{0.1cm}

\State $ \left[ \mathbf{H}_{M,N} \right]_{ik} \gets \left[ \mathbf{H}_{M,N} \right]_{ik}+ \frac{1}{N}\gamma_{ik}  $

\For{$j\in \{1,\dots,d\}$}

\State{$\left[ \mathbf{K}^{(j)}_{M,N} \right]_{ ik} \gets \left[ \mathbf{K}^{(j)}_{M,N} \right]_{ik}+ \ \frac{d}{N\mathbf{r}_i^2} \left(X^{\rm WoS, \;1}_{i}-\mathbf{x}_i\right)_j \ \gamma_{ik}$}

\EndFor

\EndFor
\EndFor
\end{algorithmic}
\end{algorithm}

%% file: numerical.tex
\subsection{Numerical examples}\label{sec:numerical examples}

In this section we present numerical simulations for some elliptic partial differential equations: Laplace Dirichlet equation (Example \ref{example Laplace}), Poisson problem (Example \ref{example Poisson}), second order linear elliptic equation in divergence form (Examples \ref{example elliptic div} - \ref{example elliptic div min} ), $p-$Laplace problem (Examples \ref{example plaplace disk} - \ref{example plaplace square}), minimal surface equation (Example \ref{example minimal surface}), semilinear elliptic equation (Example \ref{example semilinear}), in bounded domains $D \subset \mathbb{R}^d$, $d=2$ (disk, square, L-shaped domain). Before presenting the numerical experiments, we briefly recall and fix notations, comment on the used optimization methods and on the types of numerical data plots. Supplementary numerical details are included in Appendix \ref{appendix numerical}. \\

\noindent \textbf{Notations}. We keep the notations for points
\begin{equation*}
    \big(x_i\big)_{1\leq i\leq n} \subset D, \,\, \big(y_i\big)_{1\leq i\leq m} \subset D, \,\, \big(z_i\big)_{1\leq i\leq p} \subset \partial D,
\end{equation*}
and for the estimate matrices 
\begin{equation*}
    \mathbf{A}_{M,N} \in \mathbb{R}^{n \times m}, \,\, \mathbf{B}_{M,N} \in \mathbb{R}^{d \times n \times m}, \,\, \mathbf{H}_{M,N} \in \mathbb{R}^{n \times p}, \,\, \mathbf{K}_{M,N} \in \mathbb{R}^{d \times n \times p},
\end{equation*}of the operators $U_0$, $\nabla U_0$, $H$ and $\nabla H$, respectively, in $D$, with $N$ Monte Carlo samples and $M$ steps of $WoS$. We refer to 
\vspace{-10pt}
\begin{eqnarray*}
    \text{Exact solution of the PDE, if available: } u_{\text{ex}},
\end{eqnarray*}
and denote, the estimate of the solution and the estimate of its gradient proposed in this work by
\begin{equation*}
    \begin{aligned}
    \widehat{u}  & \coloneqq \mathbf{A}_{M,N} \bar{f}^* + \mathbf{H}_{M,N} \tilde{g} \; \in \mathbb{R}^n\\
     \widehat{\nabla u} & \coloneqq \mathbf{B}_{M,N} \bar{f}^* + \mathbf{K}_{M,N} \tilde{g} \; \in \mathbb{R}^{d,n}\\
    \end{aligned}
\end{equation*}
with $\tilde{g}=\big(g(z_i)\big)_{1 \leq i \leq p}$ where $g$ is the given Dirichlet boundary data and 
\begin{eqnarray*}
\bar{f}^*=
\left\{
\begin{array}{ll}
     \big( f(y_i) \big)_{1 \leq i \leq m},  \text{ with given data $f$, as in Example \ref{example Laplace} ($f=0$)  and Example \ref{example Poisson}}, \\[4pt]
    \bar{f}_{alg}  \in \mathbb{R}^m \text{ found by solving the linear algebraic system \eqref{eq:linear_q}, as in Example \ref{example elliptic div} (see also Appendix \ref{appendix numerical})}, \\[4pt]
    \bar{f}_{opt} \in \mathbb{R}^m \text{ found by optimization of $\mathcal{J}_{n,m,p,N,M} $ \eqref{eq:min_el_discrete} as in Examples \ref{example elliptic div min} - \ref{example semilinear} (see the paragraph below)}.
\end{array}
\right.
\end{eqnarray*}
We consider the piecewise constant functions on Voronoi cells $W_i$ associated to points $x_i$, $1\leq i \leq n$,
\begin{equation*}
\begin{aligned}
    & \mathbf{\widehat{u}}(x) \coloneqq \mathbf{\widehat{u}}(x_i) = \widehat{u}(i), \, \quad &&\mathbf{\widehat{\nabla u}}(x)  \coloneqq \mathbf{\widehat{\nabla u}}(x_i) = \widehat{\nabla u}(i), \\
    & \mathbf{{u}}_{\text{ex}}(x) \coloneqq  u_{\text{ex}}(x_i), \quad && \mathbf{\nabla {u}}_{\text{ex}}(x) \coloneqq  \nabla u_{\text{ex}}(x_i) \qquad \qquad x \in W_i, \; 1\leq i\leq n,
\end{aligned}
\end{equation*}
and recall the functional \eqref{eq:F} and its discrete approximation, respectively,
\begin{equation*}
\begin{aligned}
    \mathcal{F} ({u})& = \int_D  F(x,{u}(x),\nabla {u}(x)) \; dx\\
    \mathcal{F}_n ({u})& \coloneqq \sum_{1 \leq i \leq n} \lambda(W_i) \;F(x_i,{u}(x_i),\nabla {u}(x_i)).
\end{aligned}
\end{equation*}
Note that $\mathcal{F} (\mathbf{\widehat{u}}) = \mathcal{F}_n (\mathbf{\widehat{u}})$ and that $\mathcal{F}_n (u_{\text{ex}}) = \mathcal{F}_n ( \mathbf{{u}}_{\text{ex}})$. When shown comparisons (table of errors), we compute the integral in $\mathcal{F}(u)$ or $L^p$-norms using Monte Carlo estimate with $10^7$ samples.\\

\noindent \textbf{Optimization}. In view of \eqref{eq:gradJ}, we recall
\begin{equation*}
    \mathcal{J}_{n,m,p,N,M} (\bar{f}) = \sum_{1 \leq i \leq n} \lambda(W_i) \;F \left( x_i, \, [\mathbf{A}_{M,N} \bar{f} + \mathbf{H}_{M,N} \tilde{g}]_i, \,[\mathbf{B}_{M,N} \bar{f} + \mathbf{K}_{M,N} \tilde{g}]_i \right).
\end{equation*}
For the (local) minimization of $\mathcal{J}_{n,m,p,N,M}$ \eqref{eq:min_el_discrete} we used Adam optimization method \cite[Algorithm~1 (Section 2)]{Adam17} for all examples \ref{example elliptic div min} - \ref{example semilinear}. In the case of multiple solutions in Example \ref{example semilinear}, for the search of saddle points of $\mathcal{J}_{n,m,p,N,M}$ we relied on the p-HiSD (preconditioned high saddle dynamics) method proposed very recently in \cite[(3.3)]{huang2026computingsaddlepointsstiff}. See additional numerical details in Appendix \ref{appendix numerical}. \\

\noindent \textbf{Plots}. We visualize numerical vector data as:
\vspace{-5pt}
\begin{itemize}[itemsep=-2pt]
    \item[--] vector indices or coordinates against their values;
    \item[--] heatmap of values constant on Voronoi cells;
    \item[--] surface, heatmap or contour plot of an interpolation (linear/cubic) of point values.
\end{itemize}

\subsubsection{WoS vs. FEM: A Toy comparison (Laplace and Poisson problems)}

In the first two examples treated below we consider the Laplace and Poisson problems, respectively, we compute the gradients of the solutions obtained by the probabilistic approach used in this paper, and we compare with those obtained by the Finite Element Method using FreeFEM++. In other words, we compare the point-wise estimates for $\nabla H$ and $\nabla U$ from the Laplacian Operator Toolbox \eqref{eq:LOT}, obtained by the two approaches just mentioned. 
We mention at this point that a decisive comparison between our method and FEM is hard to achieve and beyond the scope of the two tests presented below. This is mainly because of two aspects: (i) the two approaches are based on different techniques and require different discretizations; (ii) the algorithms and their numerical implementation resulting from the method developed in this paper are far from being as mature and as optimized as the ones available for the FEM.
In order to make the comparison with FEM as relevant as possible, we paired the number of Voronoi cells used in our discretization with the number of vertices in the mesh required by FEM. Moreover, since in our methods we use a piece-wise constant representation of the input of the operators $\nabla H$ and $\nabla U$, a fair comparison is to consider FEM with piece-wise linear basis functions ($P1$); nevertheless, we compared our results with FEM using $P_2$ elements (quadratic basis functions), as well.

{\example \label{example Laplace}}
\textbf{Comparison with FEM - $\mathbf{\nabla H}$. }
Following \cite{Ai02}, let us consider
\begin{eqnarray*}
\left\{
\begin{array}{ll}
      \Delta u = 0  & \mbox{ in } D = B(0,1) \subset \mathbb{R}^2, \\[4pt]
    u = g & \mbox{ on } \partial D.
\end{array}
\right.
\end{eqnarray*}
with $g(x,y) =  |y| $. For $x \in (-1,1), \, y=0$, the exact solution is $u_{\text{ex}}(x,0) = \frac{1-x^2}{\pi x }\log \frac{1+x}{1-x}$.
The choice of this particular solution is motivated as follows: As it is well known, if a domain $D$ satisfies the exterior sphere condition, and $g$ is $\alpha$-Holder continuous ($0<\alpha<1$), then it's harmonic extension $u$ (i.e. the solution to the above problem) is also $\alpha$-H\"older continuous; see also \cite{cimpean2025quantitative} for a recent study on this topic. This preservation of the boundary regularity is no longer valid for Lipschitz data, as the above example provides a counterexample; one can see that $\nabla u_{\sf ex}$ explodes logarithmically near $(1,0)$, and this is in fact the worst possible behavior of the gradient of a harmonic function with Lipschitz trace on $\partial D$, if $D$ admits an exterior sphere condition (see again \cite{cimpean2025quantitative}).

\input{Aikawa_fem}

%%%%%%%%%%%%%%%%%%%%%%%%%%%%%%%%%%%%%%%%%%%%%%%%%%%%%%
{\example \label{example Poisson}}
    \textbf{Comparison with FEM - $\mathbf{U}$, $\mathbf{\nabla U}$.}
Now we compare LOT vs FEM on an L-shape domain, for a solution $u$ and its gradient whose singularity at the corner is significantly higher than in the previous example. More precisely, consider
\begin{eqnarray*}
\left\{
\begin{array}{ll}
     - \frac{1}{2}\Delta u = f  & \mbox{ in } D  = (-1,1)^2 \setminus [0,1) \times (-1,0], \\[4pt]
    u = 0 & \mbox{ on } \partial D.
\end{array}
\right.
\end{eqnarray*}
Consider $u_{ex}(r,\phi) \coloneqq r^{2/3} \sin{\dfrac{2\phi}{3}}(1-r^2 \cos^2{\phi})(1-r^2\sin^2{\phi}), \; x=r \cos{\phi},\; y=r\sin \phi, \; \phi \in [0,2 \pi), \; (x,y) \in D$, and $f$ taken accordingly.
Note that $\nabla u_{\sf ex}$ has a singularity at the corner of magnitude $r^{-1/3}$, which is in fact typical for such a domain (see \cite{gr11} or \cite{cimpean2025quantitative}). 
The next figures reveal that the Monte Carlo estimators used to compute $\nabla U$ through \eqref{eq:LOT} are quite competitive with the ones provided by FEM, even when using $P_2$ elements.

\input{comparison_fem}

As a conclusion of the first two tests, the Monte Carlo estimators used to produce the Laplacian Operator Toolbox \eqref{eq:LOT} seem to be quite stable and accurate, mesh-free, easy to implement (as they require essentially vector calculus), and still competitive with FEM. 
In order to assess the use of LOT for general elliptic linear and nonlinear problems, in what follows we use it on various benchmark PDEs and compare with the existing numerical results.\\

\newpage
\noindent \textbf{Examples \ref{example elliptic div} - \ref{example semilinear}.} In all of the following numerical simulations, whenever the domain $D$ is the unit disk or the unit square, we use the same estimate  matrices $\mathbf{A}_{M,N} \subset \mathbb{R}^{n,m}, \,\, \mathbf{B}_{M,N} \subset \mathbb{R}^{d,n,m}, \,\, \mathbf{H}_{M,N} \subset \mathbb{R}^{n,p}, \,\, \mathbf{K}_{M,N} \subset \mathbb{R}^{d, n,p},$ for the operators $U$, $\nabla U$, $H$, $\nabla H$, respectively, in $D$. Below we list the coordinates of the discretization (plotted in \Cref{domain_points}) and the parameters considered for the approximations: the number of points $n$, $m$ and $p$, the number of Monte Carlo samples $N$ and the number of $WoS$ steps, $M$. The matrices can be accessed from the Appendix \ref{appendix numerical}. \\

1. \textbf{Unit disk} $D = B(0,1)$: $n=m=300$, $p=60$; $N=10^6$, $M=30$. 
% \vspace{-5pt}
\begin{equation}\label{eq:unit disk points}
\begin{aligned}
\big(x_i\big)_{1\leq i\leq n} &=
\left\{(r_k \cos{\theta_{k,l}}, r_k \sin{\theta_{k,l}})\;\middle|\; r_k = \frac{R}{C} \Big( k-\tfrac{1}{2}\Big), \theta_{k,l} = \frac{2 \pi}{6k-3} l, \; 1 \leq k \leq C, \;1 \leq l \leq 6k-3,\;R=1,C=10 \right\} \\
\big(y_i\big)_{1\leq i\leq m} &= \big(x_i\big)_{1\leq i\leq n} \\
\big(z_i\big)_{1\leq i\leq p} &=
\left\{(R \cos{\theta_k}, R \sin{\theta_k})\;\middle|\; R=1,\;\theta_k=\frac{2\pi}{p}(k-1),\; 1\leq k\leq p \right\}
\end{aligned}
\end{equation}

2. \textbf{Unit square} $D = (0,1)^2$: $n=m=400$, $p=80$; $N=10^7$, $M=30$.
% \vspace{-5pt}
\begin{equation}\label{eq:unit square points}
\begin{aligned}
\big(x_i\big)_{1\leq i\leq n} &=
\left\{(t_k,t_l)\;\middle|\;
t_k=\frac{1}{M_1}\Big(k-\tfrac{1}{2}\Big),\;
t_l=\frac{1}{M_1}\Big(l-\tfrac{1}{2}\Big),\;
1 \leq k \leq M_1, \; 1\leq l \leq M_2, \; M_1=M_2=20
\right\} \\
\big(y_i\big)_{1\leq i\leq m} &= \big(x_i\big)_{1\leq i\leq n} \\
\big(z_i\big)_{1\leq i\leq p} &=
\left\{(t_k,0),\; (1,t_l),\; (t_k,1),\;(0,t_l) \;\middle|\;
t_k=\frac{1}{M_1}\Big(k-\tfrac{1}{2}\Big),\;
t_l=\frac{1}{M_1}\Big(l-\tfrac{1}{2}\Big),\;
1 \leq k \leq M_1, \; 1\leq l\leq M_2,\; \right\}. \\
 & \quad \quad \quad \quad \quad \quad \quad \quad \quad \quad \quad \quad \quad \quad \quad \quad \quad \quad \quad \quad \quad \quad \quad \quad \quad \quad \quad \quad \quad \quad \quad \quad \quad \quad \quad \quad M_1=M_2=20
\end{aligned}
\end{equation}

\begin{figure}[H]
    \centering
    \begin{minipage}{0.22\textwidth}
        \centering
        \includegraphics[width=\textwidth]{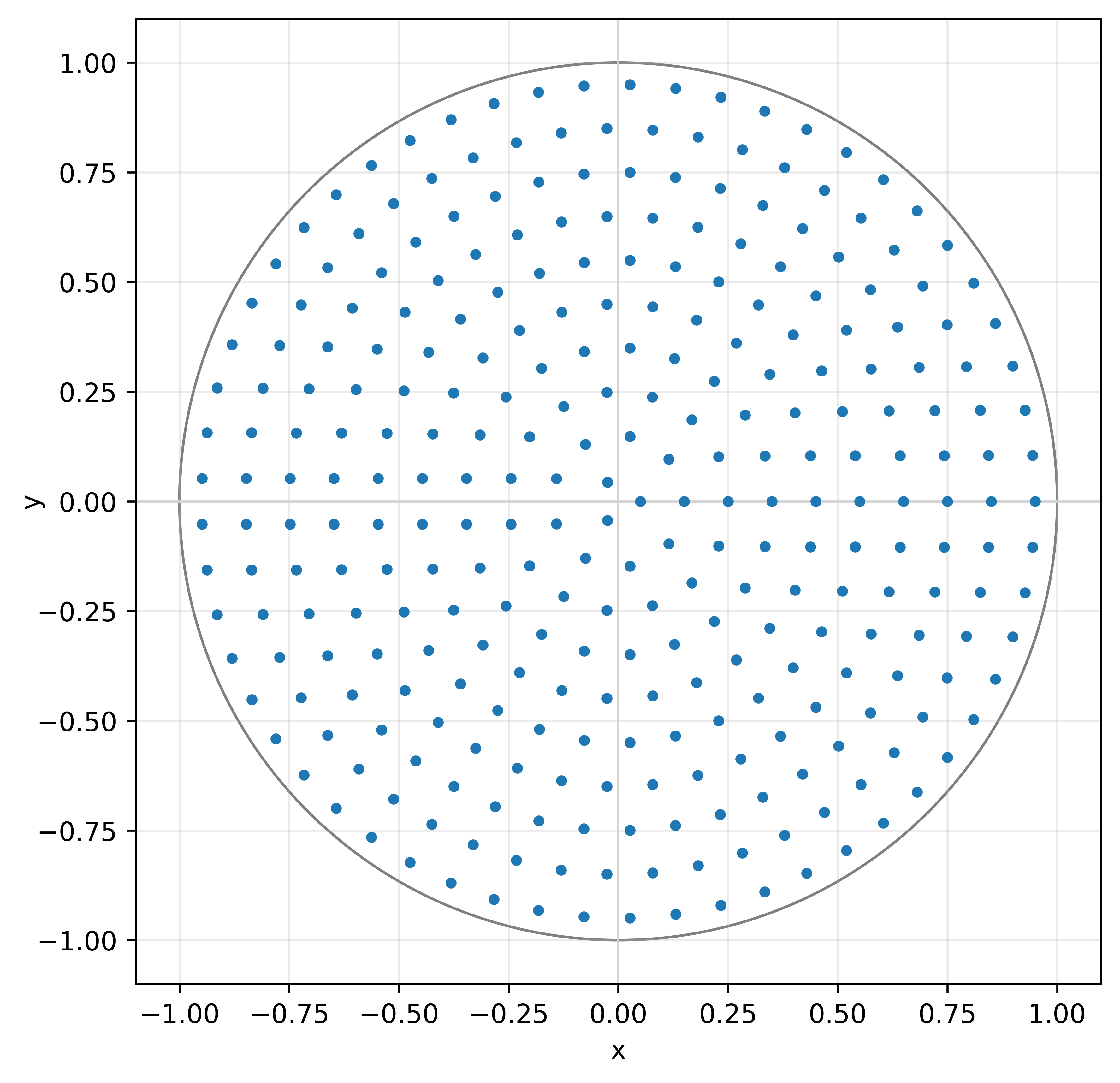}
        \captionof{subfigure}{$n=300$}
    \end{minipage}
    \begin{minipage}{0.22\textwidth}
        \centering
        \includegraphics[width=\textwidth]{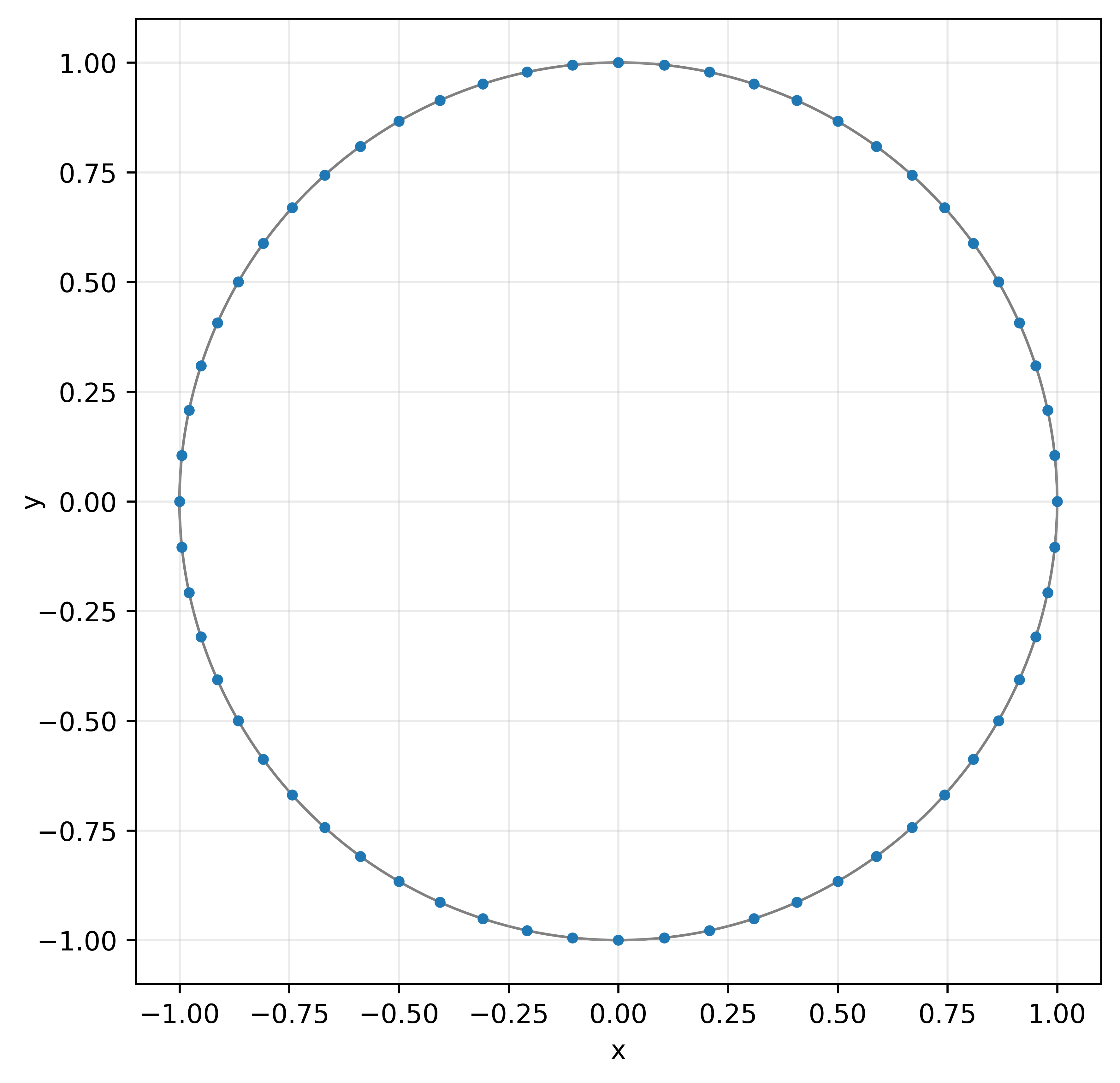}
        \captionof{subfigure}{$p=60$}
    \end{minipage}
    \begin{minipage}{0.22\textwidth}
        \centering
        \includegraphics[width=\textwidth]{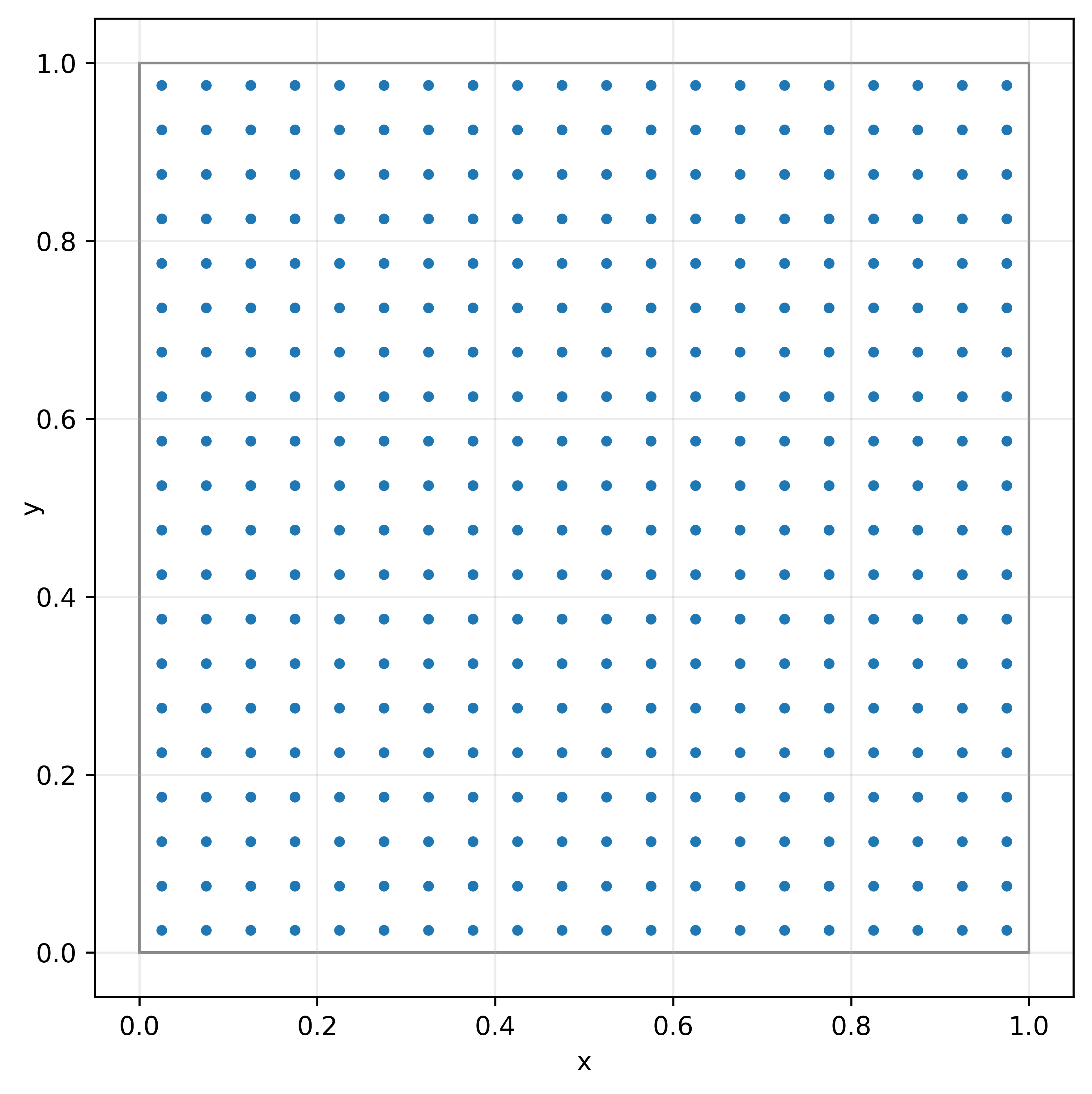}
        \captionof{subfigure}{$n=400$}
    \end{minipage}
    \begin{minipage}{0.22\textwidth}
        \centering
        \includegraphics[width=\textwidth]{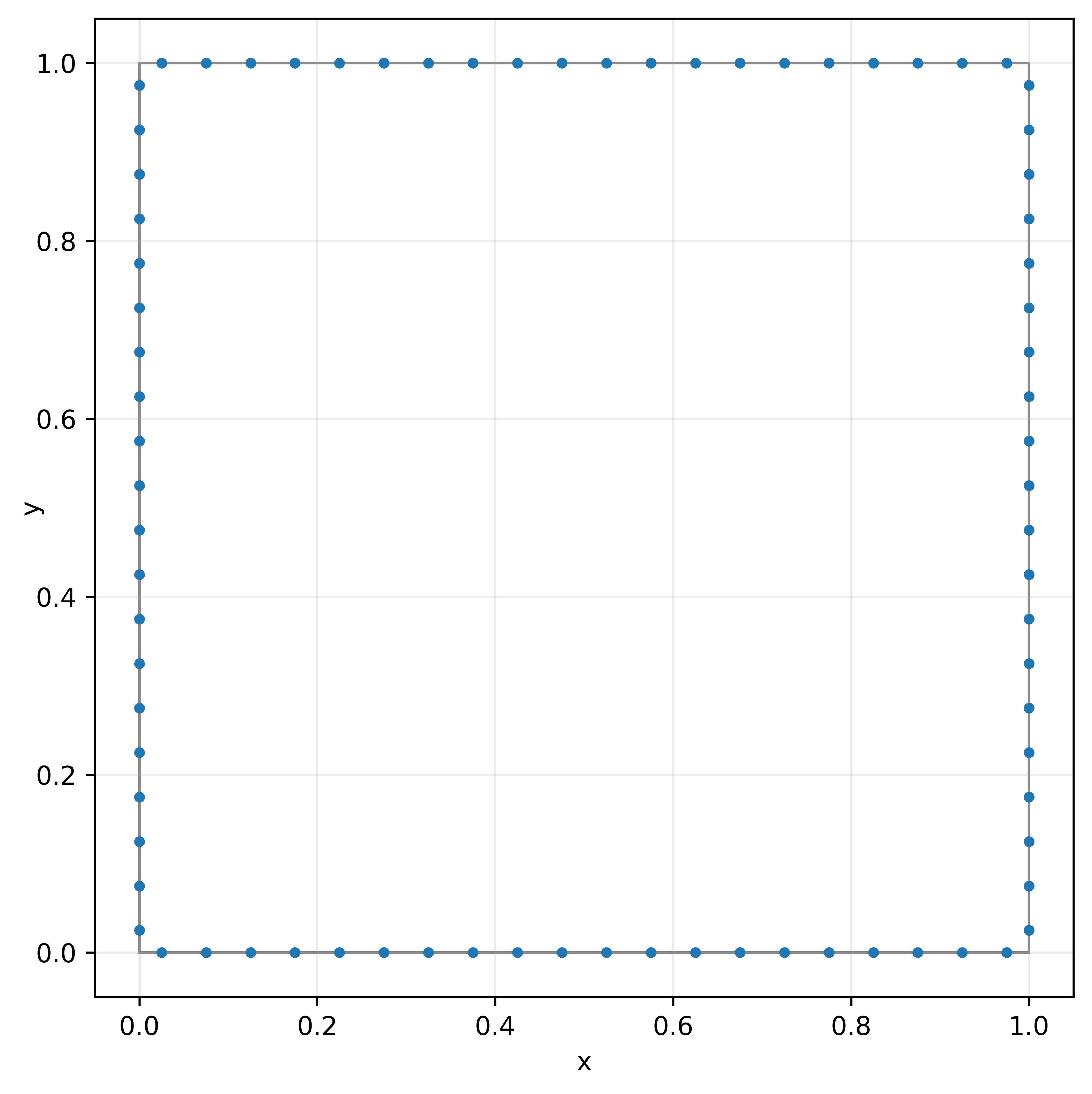}
        \captionof{subfigure}{$p=80$}
    \end{minipage}
    \caption{Points $(x_i)_{1 \leq i \leq n} \subset D$ (a), (c) and points $(z_i)_{1 \leq i \leq p} \subset \partial D$ (b), (d), for $D = B(0,1)$ and $D=(0,1)^2$.}
    \label{domain_points}
\end{figure}

\subsubsection{Solving linear elliptic PDEs as in Section \ref{ss:linearPDE}}

{\example \label{example elliptic div}}
\textbf{Second order elliptic $\bf{\mathrm{div} (A(x) \nabla u ) +   b(x) \cdot \nabla u - c u  = f}$}. Consider
\begin{eqnarray*}
\left\{
\begin{array}{ll}
       \mathrm{div} (A(x) \nabla u ) +   b(x) \cdot \nabla u - c u  = f  & \mbox{ in } D = B(0,1) \subset \mathbb{R}^2, \\[4pt]
    u = g & \mbox{ on } \partial D.
\end{array}
\right.
\end{eqnarray*}
with $A(x,y)= \begin{pmatrix}2+x^2 & x y\\[6pt] x y & 2+y^2\end{pmatrix} $, $ b(x,y)=(-y,x) $, $c=1$. We consider $u_{\text{ex}}(x,y) = (1-x^{2}-y^{2}) \sin{\pi x}$ and $f$, $g$ taken accordingly.

\input{elliptic-solution-operator}

\subsubsection{Solving linear elliptic PDEs as in Section \ref{ss:EL}}
{\example \label{example elliptic div min}}
\textbf{Second order elliptic equation $\bf{- \mathrm{div} (A(x) \nabla u ) +  b(x) \cdot \nabla u + c u  = f, \, b= A \nabla \phi}$}. Let
\begin{eqnarray*}
\left\{
\begin{array}{ll}
      - \mathrm{div} (A(x) \nabla u ) +  ( A(x) \nabla \phi) \cdot \nabla u + c u  = f  & \mbox{ in } D = B(0,1) \subset \mathbb{R}^2, \\[4pt]
    u = g & \mbox{ on } \partial D.
\end{array}
\right.
\end{eqnarray*}
with $A(x,y)= \begin{pmatrix}1 & x y\\[6pt] x y & 1\end{pmatrix} $, $ \phi(x,y)=\tfrac{1}{2}\bigl(x^2+y^2\bigr) $, $c=1$. We consider $u_{\text{ex}}(x,y) = (1-x^{2}-y^{2})(x^3 - 3xy^2)$ and $f$, $g$ taken accordingly.
This is the Euler-Lagrange equation of the functional $\mathcal{F} : \{ u \in H^1(D): \, u = g \mbox{ on } \partial D\} \to \mathbb{R}$,
\begin{equation*}
    \mathcal{F} (u) = \int_D \mathrm{e}^{-\phi} \left( \frac{1}{2} \nabla u \cdot A \nabla u + \frac{1}{2} c u^2 - f u  \right) \; d \textrm{x}.
\end{equation*}

\input{elliptic-disk-new}

\subsubsection{Solving nonlinear elliptic PDEs as in Section \ref{ss:EL}}
\paragraph{p-Laplace equation.} Consider the Poisson problem for the p-Laplace operator $\Delta_p u \coloneqq  \mathrm{div} (|\nabla u|^{p-2} \nabla u)$, $p>1$ constant,
% \vspace{-10pt}
\begin{eqnarray*}
\left\{
\begin{array}{ll}
     - \Delta_{p} u = f  & \mbox{ in } D, \\[4pt]
    u = g & \mbox{ on } \partial D.
\end{array}
\right.
\end{eqnarray*}
This is the Euler-Lagrange equation of the functional $\mathcal{F} : \{ u \in W^{1,p}(D): \, u = g \mbox{ on } \partial D\} \to \mathbb{R}$,
\begin{equation}\label{eq:plaplace}
    \mathcal{F} (u) = \int_D  \frac{1}{p} | \nabla u|^{p}  - f u \; d \textrm{x}.
\end{equation}

{\example\label{example plaplace disk} }\textbf{$p-$Torsion problem on unit disk.} Let $D = B(0,1) \subset \mathbb{R}^2$, $f (x,y) = 1$ (torsion problem), $g=0$ in \eqref{eq:plaplace}. In the case of the unit disk, we know $u_{\text{ex}}(x,y) = \frac{p-1}{p} \left(\frac{1}{2}\right)^{\frac{1}{p-1}} \left(1-(x^2+y^2)^\frac{p}{2 p -2} \right)$.

\input{p-laplace-disk}
\vspace{-5pt}
{\example\label{example plaplace square}} \textbf{$p-$Torsion problem on unit square.} Let $D = (0,1)\times(0,1) \subset \mathbb{R}^2$, $f (x,y) = 1$ (torsion problem), $g=0$ in \eqref{eq:plaplace}. In the case of the unit square, $u_{\text{ex}}$ is not available. In the linear case, i.e. $p=2$, it can be represented via Fourier series. 

\input{p-laplace-square}

\paragraph{Minimal surface equation.}\
{\example \label{example minimal surface}} Consider
% \vspace{-20pt}
\begin{eqnarray*}
\left\{
\begin{array}{ll}
     \mathrm{div} \left( \frac{\nabla u}{\sqrt{1+ |\nabla u|^2}} \right) = 0  & \mbox{ in } D = B(0,1) \subset \mathbb{R}^2, \\[4pt]
    u = g & \mbox{ on } \partial D.
\end{array}
\right.
\end{eqnarray*}
with $g(x,y) = 0.3 \sin(5 \phi), \; x= \cos{\phi},\; y=\sin \phi, \;\phi \in [0,2 \pi)$, example from \cite[Section 5.3]{Sakakibara2024}. In this case, $u_{\text{ex}}$ not available. This is the Euler-Lagrange equation of the functional $\mathcal{F} : \{ u \in W^{1,1}(D): \, u = g \mbox{ on } \partial D\} \to \mathbb{R}$,
\begin{equation*}
    \mathcal{F} (u) = \int_D  \sqrt{ 1+ | \nabla u|^{2}} \; d \textrm{x}.
\end{equation*}

\input{ms-disk-2}

\paragraph{Semilinear elliptic equation.}\
% \vspace{-10pt}
{\example\label{example semilinear}}
\textbf{ Problem $\bf{\Delta u + u^2 = f}$ studied by Breuer, McKenna, Plum \cite{BreuerKennaPlum03}}. Consider
\begin{eqnarray} \label{eq:multiple}
\left\{
\begin{array}{ll}
     \Delta u + u^2 = s \, \sin(\pi x) \sin (\pi y)  & \mbox{ in } D = (0,1)\times(0,1) , \\[4pt]
    u = 0 & \mbox{ on } \partial D.
\end{array}
\right.
\end{eqnarray}
with $s$ constant. This is the Euler-Lagrange equation of the functional $\mathcal{F} : H^1_0(D) \to \mathbb{R}$,
\begin{equation*}
    \mathcal{F} (u) = \int_D  \frac{1}{2} | \nabla u|^{2} -  \frac{1}{3} u^3  + f u \; d \textrm{x}, 
\end{equation*}
where $f (x,y) =  s \, \sin(\pi x) \sin (\pi y)$. Note that since the Laplacian, the source term $f$ and the homogeneous Dirichlet boundary data are invariant under the symmetries of the square, if $u$ is a solution of \eqref{eq:multiple}, then $u \circ R$ is also a solution, for any $R$ a symmetry (rotation or reflection) of the unit square. Hence, by essentially distinct solutions we mean solutions that are not related via $R$. In the case of $s=800$, it was shown in \cite{BreuerKennaPlum03} that \eqref{eq:multiple} has at least 4 essentially distinct solutions, which we also find (see \Cref{fig:multiple s800}) and the authors of \cite{BreuerKennaPlum03} conjecture that more solutions show up as $s$ increases to infinity. Indeed, in the case of $s= 8000$, we find 22 essentially distinct approximate solutions (see \Cref{fig:multiple s8000}).

\begin{figure}[H]
    \centering
    \begin{minipage}{0.5\textwidth}
        \centering
        \includegraphics[width=\textwidth]{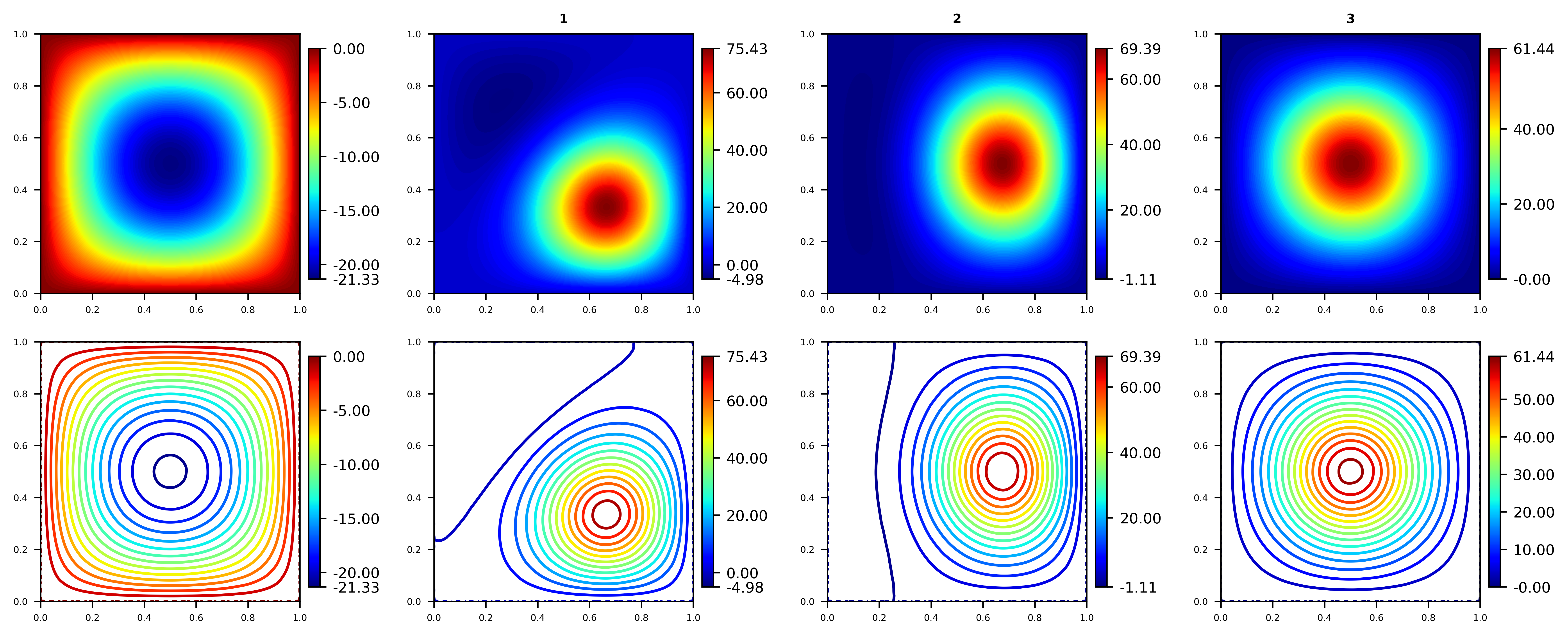}
        % \caption{$n = 400$}
    \end{minipage}
    \caption{4 essentially distinct approximate solutions of $ \Delta u + u^2 = 800 \, \sin(\pi x) \sin (\pi y)$ obtained with the proposed method, namely with our PDE discretization through LOT, in combination with the p-HISD landscape exploration technique \cite{huang2026computingsaddlepointsstiff}. Heatmap plot (top) and contour plot (bottom) of the  interpolated point values. First approximate solution found by minimization of the functional and 3 approximate solutions found by saddle point search with index $k \in \{1,2,3 \}$ (see further numerical details in Appendix \ref{appendix numerical}).}
    \label{fig:multiple s800}
\end{figure}

\noindent We note that the obtained approximate solutions shown in \Cref{fig:multiple s800} are consistent (qualitatively and quantitatively) with the ones in \cite{BreuerKennaPlum03} (see also \cite[Section 4.2]{LinYangyiHuiyuan2024}, \cite[Section B]{Pinn2025} for recent numerical methods) and the peak values of the 4 approximate solutions are consistent with the ones outlined in \cite[Section 4.3]{Allgoweretal09}.

\begin{figure}[H]
    \centering
    \begin{minipage}{0.95\textwidth}
        \centering
        \includegraphics[width=\textwidth]{multiple_solutions/Figure_2026-08-06_183749.png}
    \end{minipage}
    \begin{minipage}{0.95\textwidth}
        \centering
        \includegraphics[width=\textwidth]{multiple_solutions/Figure_2026-08-06_183807.png}
    \end{minipage}
% \end{figure}
% \begin{figure}[H]
    % \centering
    \begin{minipage}{0.95\textwidth}
        \centering
        \includegraphics[width=\textwidth]{multiple_solutions/Figure_2026-08-06_183519.png}
        % \caption{$n = 400$}
    \end{minipage}
    \begin{minipage}{0.95\textwidth}
        \centering
        \includegraphics[width=\textwidth]{multiple_solutions/Figure_2026-08-06_183617.png}
        % \caption{$n = 400$}
    \end{minipage}
    \caption{22 essentially distinct approximate solutions of $ \Delta u + u^2 = 8000 \, \sin(\pi x) \sin (\pi y)$ obtained with the proposed method, namely with our PDE discretization through LOT, in combination with the p-HISD landscape exploration technique \cite{huang2026computingsaddlepointsstiff}. Heatmap plot (top) and contour plot (bottom) of the  interpolated point values. First approximate solution found by minimization of the functional and the others found by saddle point search with index $k \in \{1,2,3,4,5,6,7,8 \}$, index shown at the top of each figure (see further numerical details in Appendix \ref{appendix numerical}).}
    \label{fig:multiple s8000}
\end{figure}

%% file: Aikawa_fem.tex
\begin{figure}[H]
    \centering
    \begin{minipage}{0.7\textwidth}
        \centering
        \includegraphics[width=\textwidth]{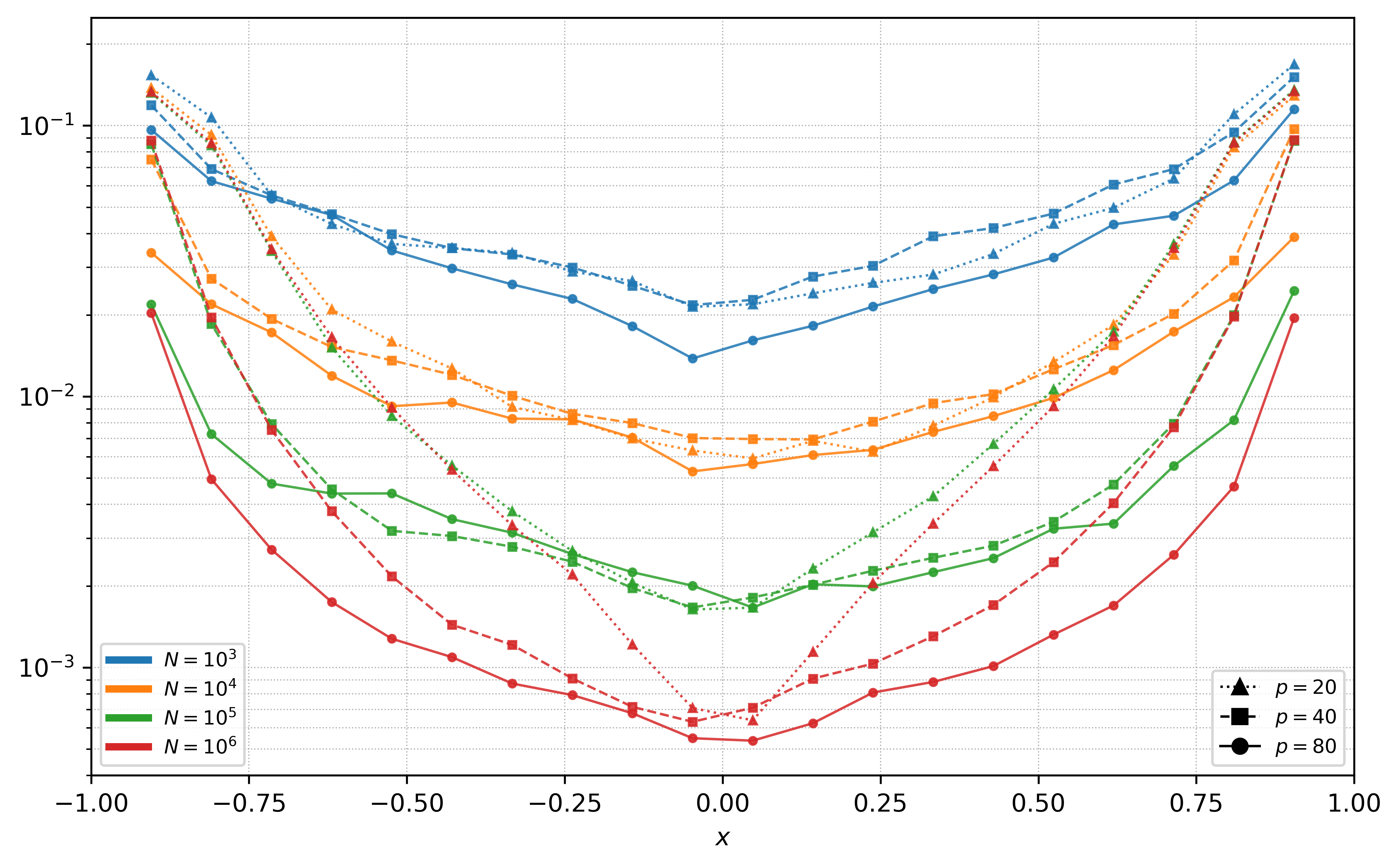}
    \end{minipage}
%    \begin{minipage}{0.7\textwidth}
%        \centering
%        \includegraphics[width=\textwidth]{Aikawa_overleaf/Figure 2026-03-31 214531.png}
%    \end{minipage}
%    \begin{minipage}{0.7\textwidth}
%        \centering
%        \includegraphics[width=\textwidth]{Aikawa_overleaf/Figure 2026-03-31 214533.png}
%    \end{minipage}
%    \begin{minipage}{0.7\textwidth}
%        \centering
%        \includegraphics[width=\textwidth]{Aikawa_overleaf/Figure 2026-03-31 214536.png}
%    \end{minipage}
    \caption{ (LOT) Combined error plots: Mean of pointwise absolute errors $\frac{1}{30}\sum_{i=1}^{30}|\partial_x u_{\text{ex}} - \widehat{\partial_x u}_i|$, 30 estimates $\widehat{\partial_x u}_i$ obtained with the proposed method LOT \eqref{eq:LOT} for computing $\nabla H$, with $N \in \{10^3, 10^4, 10^5, 10^6 \}$ Monte Carlo samples, $M=30$ $WoS$ steps, $p \in \{20, 40, 80\}$ equidistant boundary points. The plotted values correspond to $n=20$ interior points $(x,0) \in D, x\in (-1,1)$ equidistant.}
\end{figure}

\begin{figure}[H]
    \centering
    \begin{minipage}{0.22\textwidth}
        \centering
        \includegraphics[width=\textwidth]{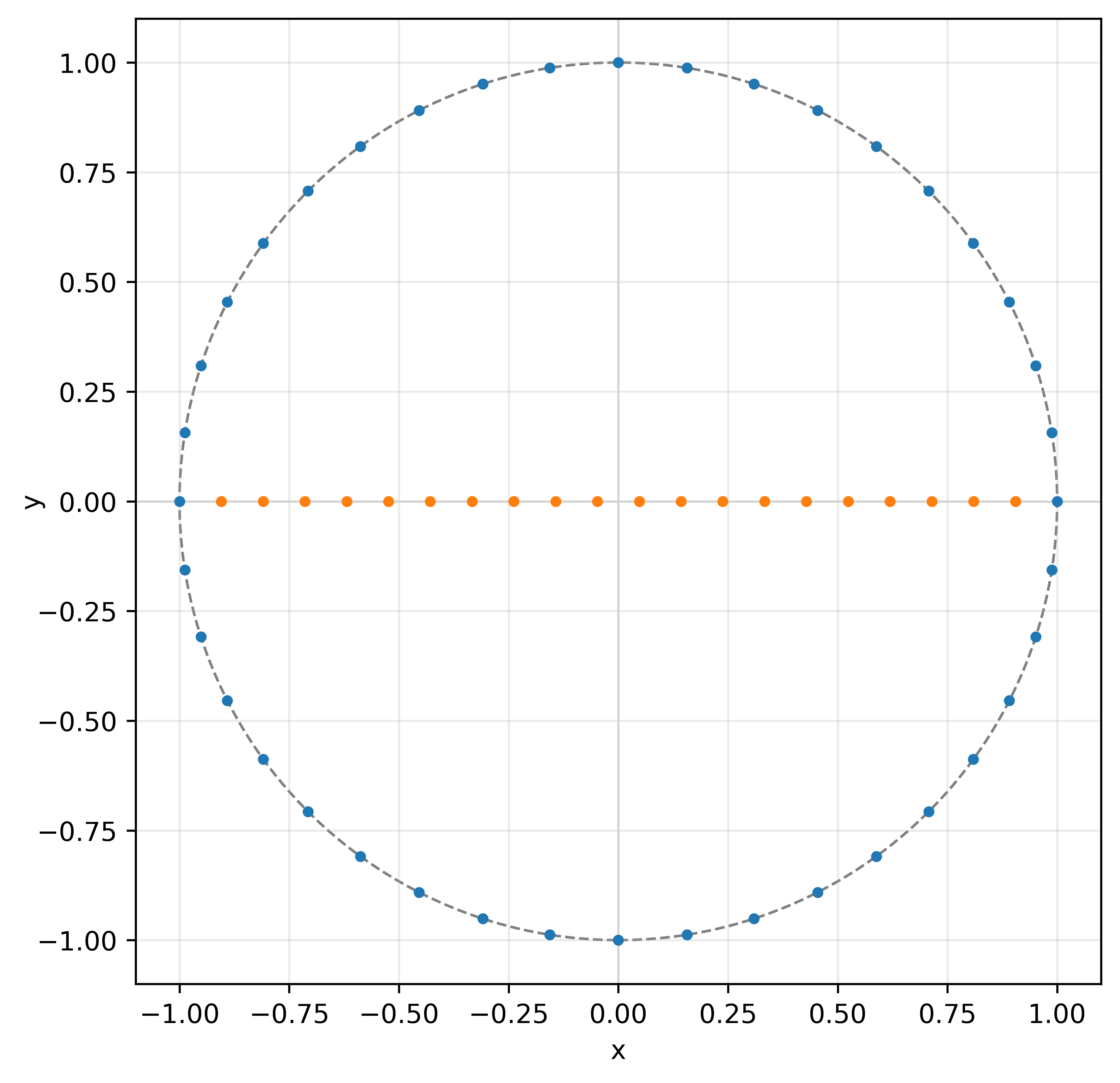}
        \captionof{subfigure}{$n=20$, $p=40$}
    \end{minipage}
    \begin{minipage}{0.22\textwidth}
        \centering
        \includegraphics[width=\textwidth]{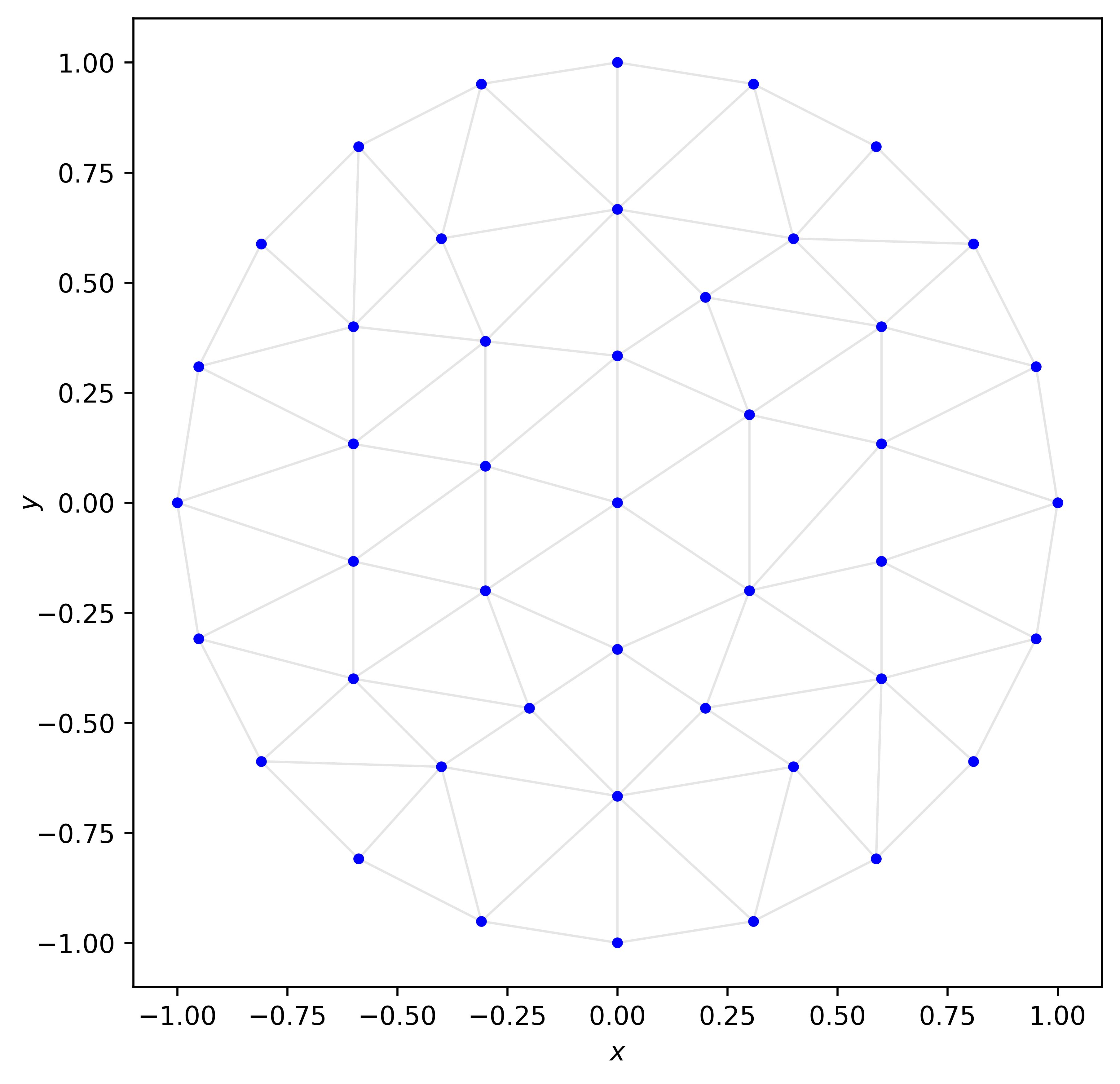}
        \captionof{subfigure}{FEM mesh 1 \\ (45 nodes, $p=20$)}
    \end{minipage}
    \begin{minipage}{0.22\textwidth}
        \centering
        \includegraphics[width=\textwidth]{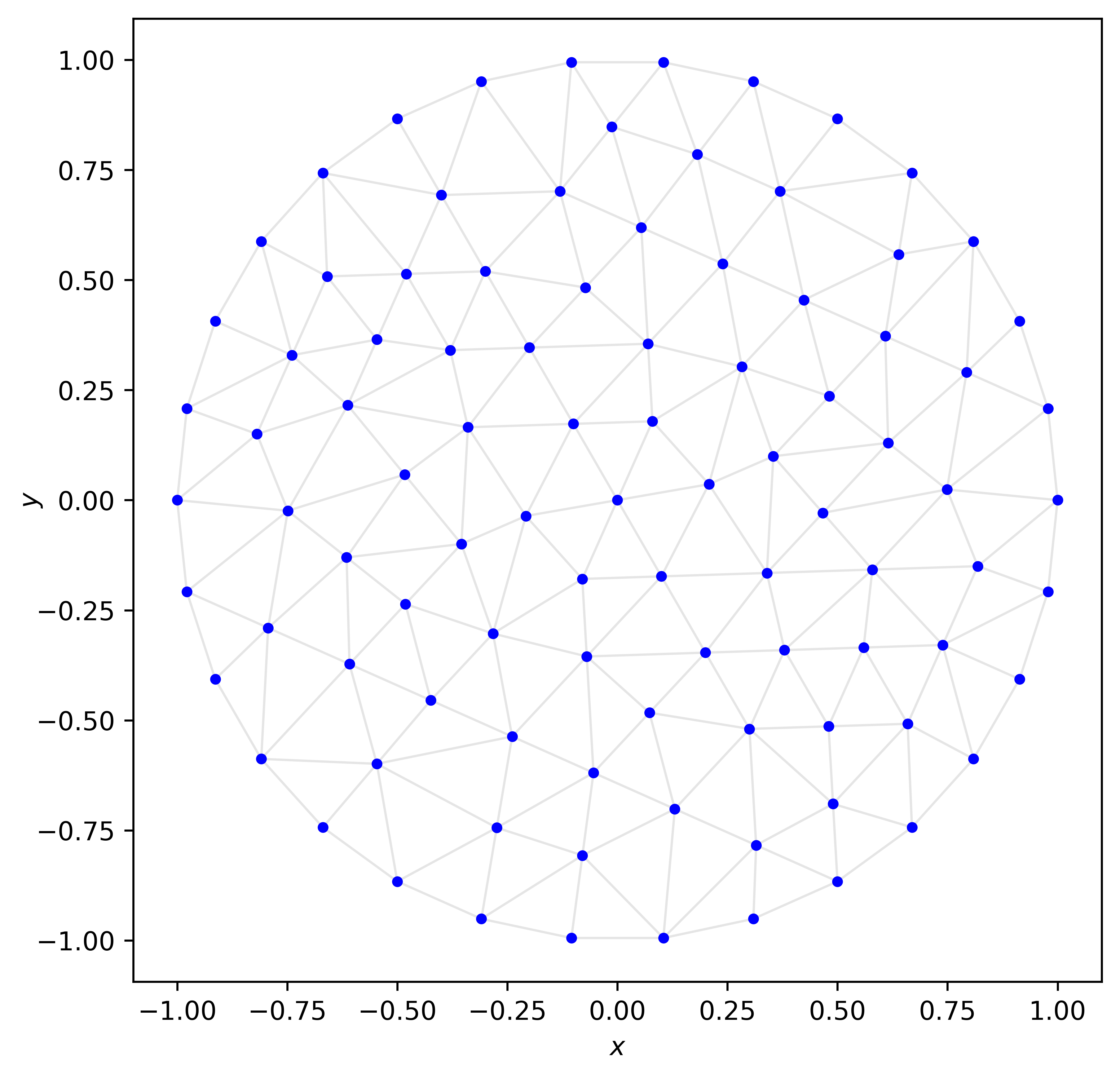}
        \captionof{subfigure}{FEM mesh 2\\
        (95 nodes, $p=30$)}
    \end{minipage}
    \begin{minipage}{0.22\textwidth}
        \centering
        \includegraphics[width=\textwidth]{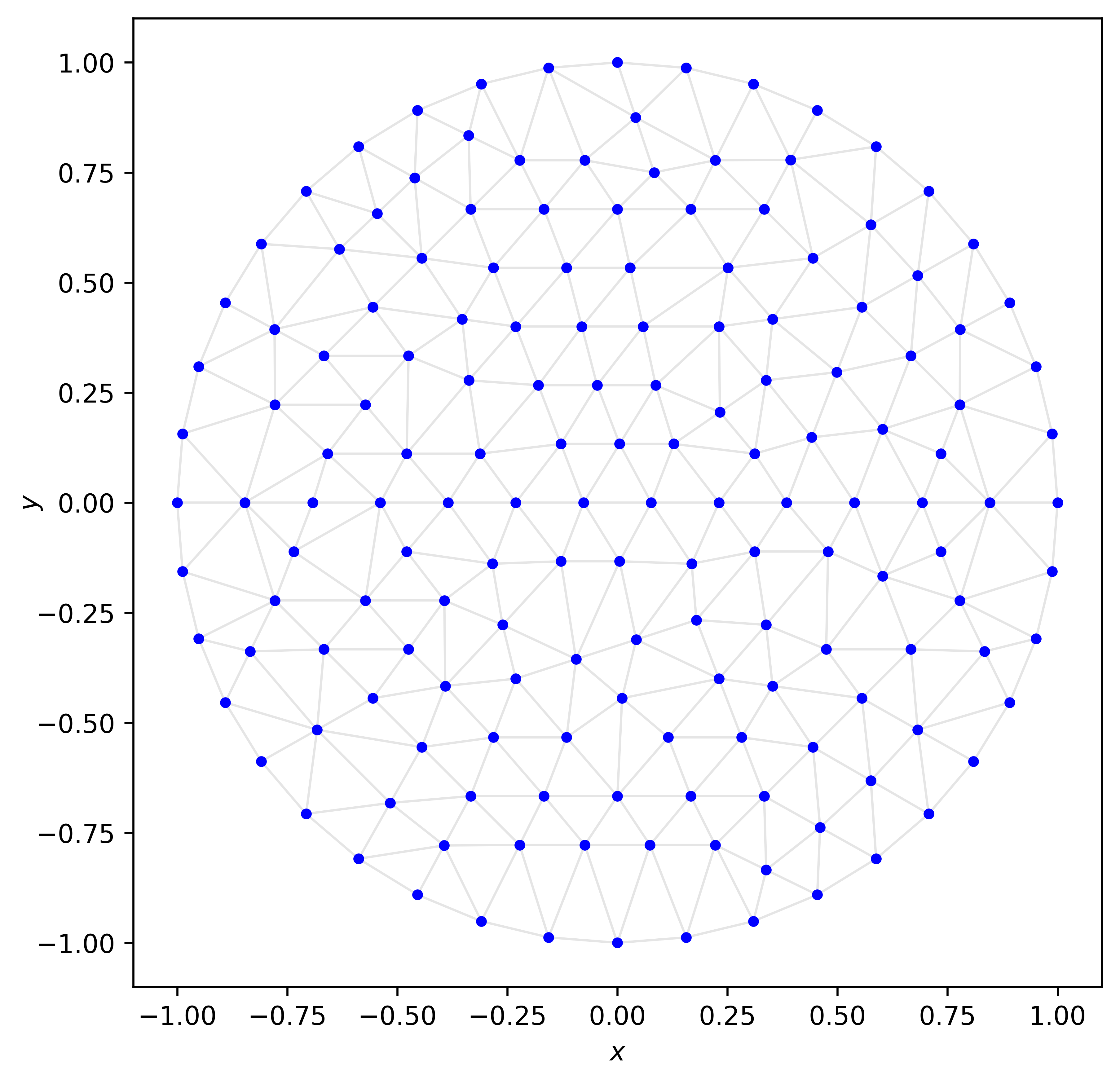}
        \captionof{subfigure}{FEM mesh 3\\
        (162 nodes, $p=40$)}
    \end{minipage}
    \caption{Discretizations: (a) $n=20$ equidistant interior points in $(-1,1)$ where the solution shall be approximated (orange) and $p=40$ equidistant boundary points (blue); (b)-(d): FEM generated discretizations of $\overline{D}$ after fixing $p \in \{20,30,40 \}$ equidistant boundary points.}
\end{figure}

\begin{figure}[H]
    \centering
    \begin{minipage}{0.75\textwidth}
        \centering
        \includegraphics[width=\textwidth]{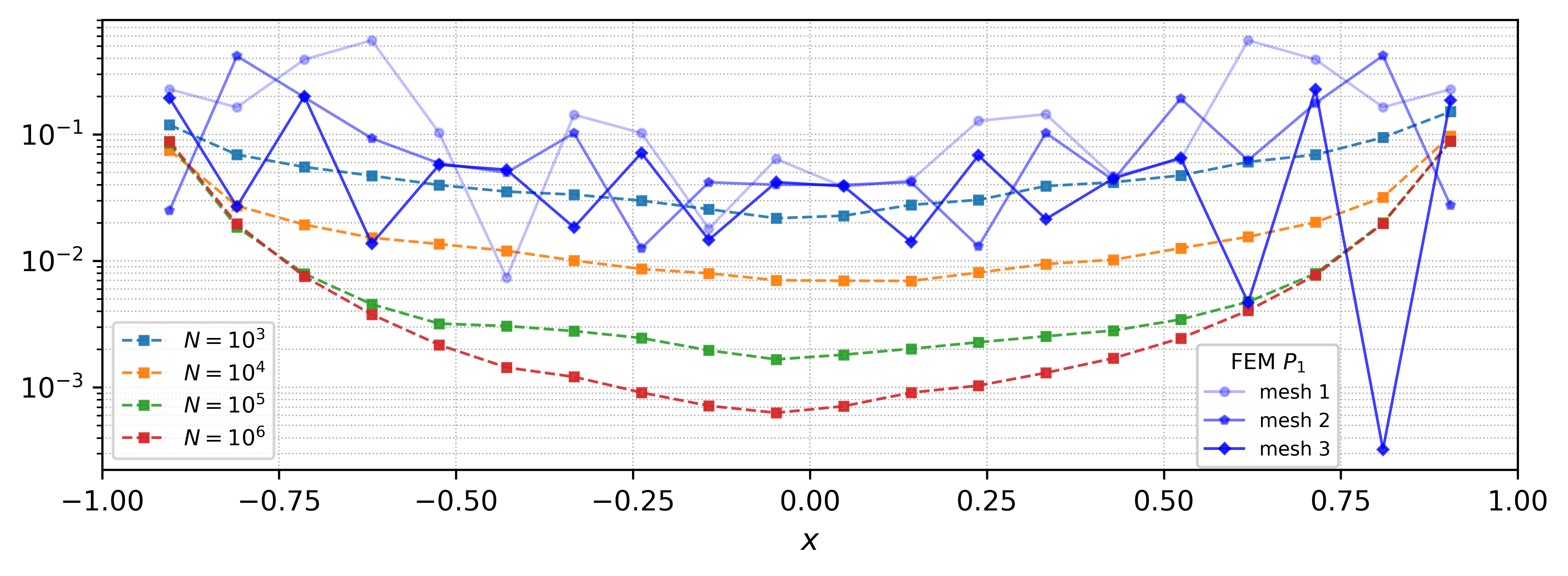}
    \end{minipage}
    \begin{minipage}{0.75\textwidth}
        \centering
        \includegraphics[width=\textwidth]{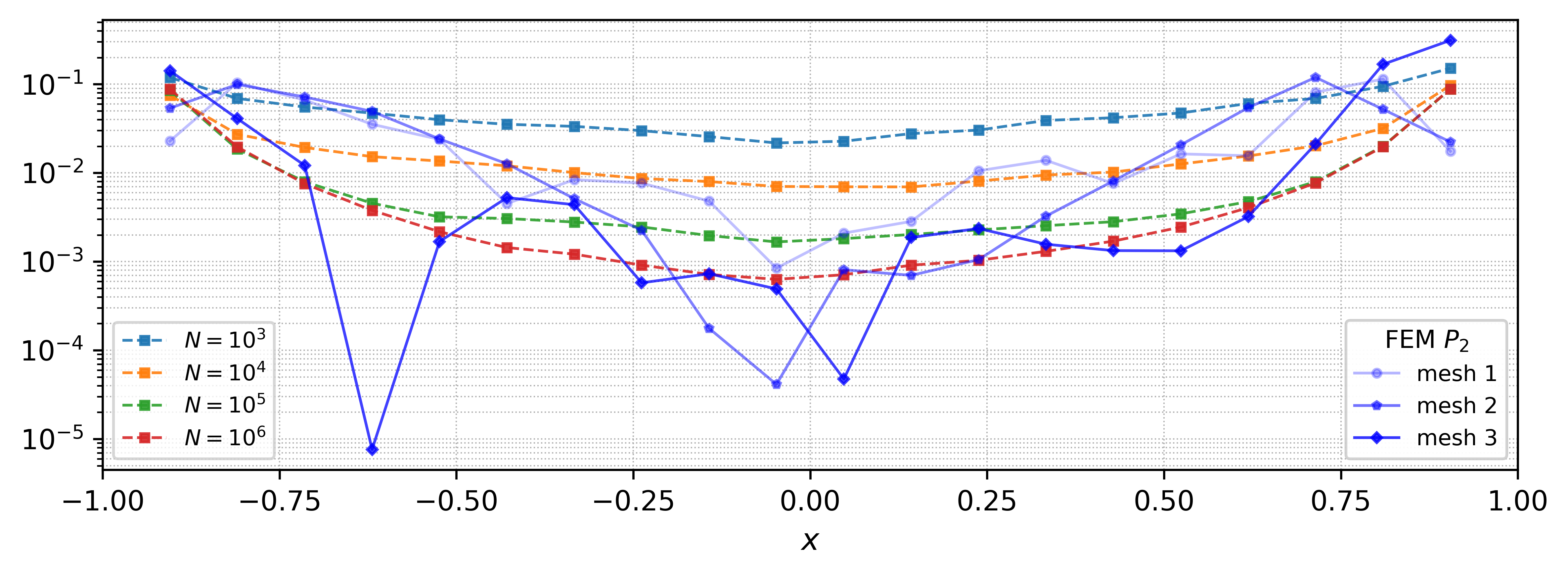}
    \end{minipage}
    \caption{ (LOT and FEM) Combined error plots: Comparison of pointwise absolute error $|\partial_x u_{\text{ex}} - \partial_x u_{fem}|$ between the partial derivative of the exact solution and of the numerical solution obtained using FEM with $P_1$ elements (top) and $P_2$ elements (bottom), respectively (basic blue), and mean of pointwise absolute errors $\frac{1}{30}\sum_{i=1}^{30}|\partial_x u_{\text{ex}} - \widehat{\partial_x u}_i|$, 30 estimates $\widehat{\partial_x u}_i$ obtained with the proposed method LOT \eqref{eq:LOT} for computing $\nabla H$, with $N \in \{10^3, 10^4, 10^5, 10^6 \}$ Monte Carlo samples, $M=30$ $WoS$ steps, $p=40$ equidistant boundary points. The plotted values correspond to interior points $(x,0) \in D, x\in (-1,1)$ equidistant.}
\end{figure}

\begin{figure}[H]
    \centering
    \begin{minipage}{0.6\textwidth}
        \centering
        \includegraphics[width=\textwidth]{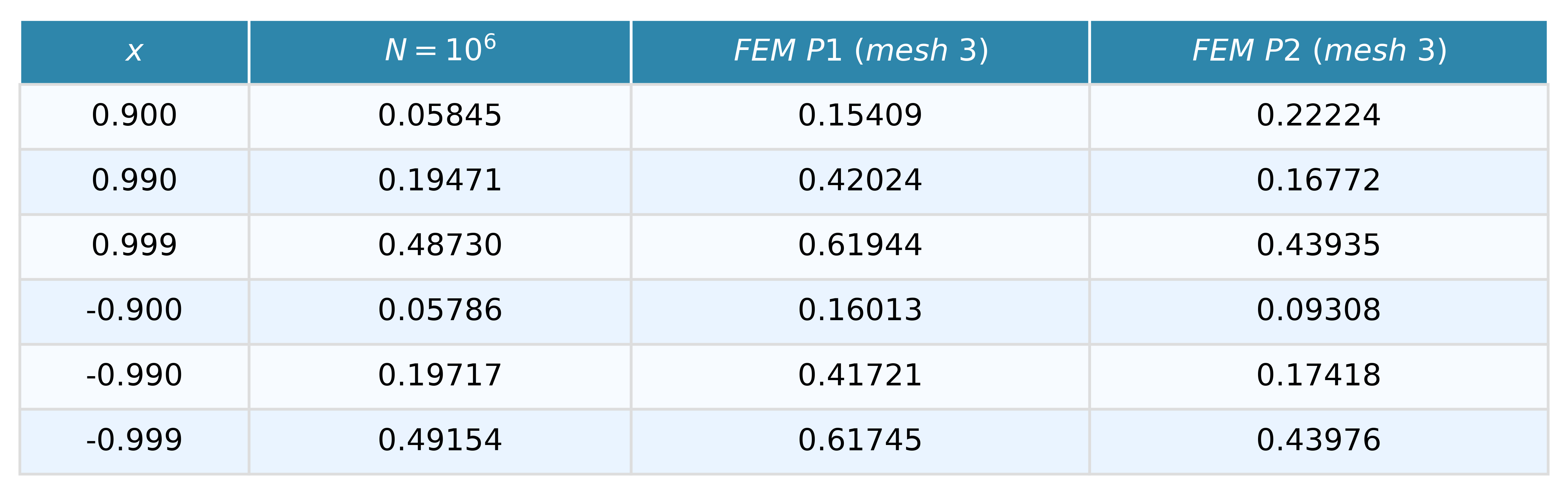}
    \end{minipage}
    \captionof{table}{(LOT and FEM) Comparison of pointwise relative error $|\partial_x u_{\text{ex}} - \partial_x u_{fem}|/ |\partial_x u_{\text{ex}}|$ between the partial derivative of the exact solution and of the numerical solution obtained using FEM with $P_1$ and $P_2$ elements, respectively, and mean of pointwise relative errors $\frac{1}{30}\sum_{i=1}^{30}|\partial_x u_{\text{ex}} - \widehat{\partial_x u}_i|/|\partial_x u_{\text{ex}}|$, 30 samples $\widehat{\partial_x u}_i$ obtained with the proposed method LOT \eqref{eq:LOT} for computing $\nabla H$, with $N=10^6$ Monte Carlo samples, $M=30$ $WoS$ steps, $p=40$ equidistant boundary points. The shown values correspond to interior points $(x,0) \in D$.} \label{table Aikawa}
\end{figure}

%% file: comparison_fem.tex
\begin{figure}[H]
    \centering
    \begin{minipage}{0.25\textwidth}
        \centering
        \includegraphics[width=\textwidth]{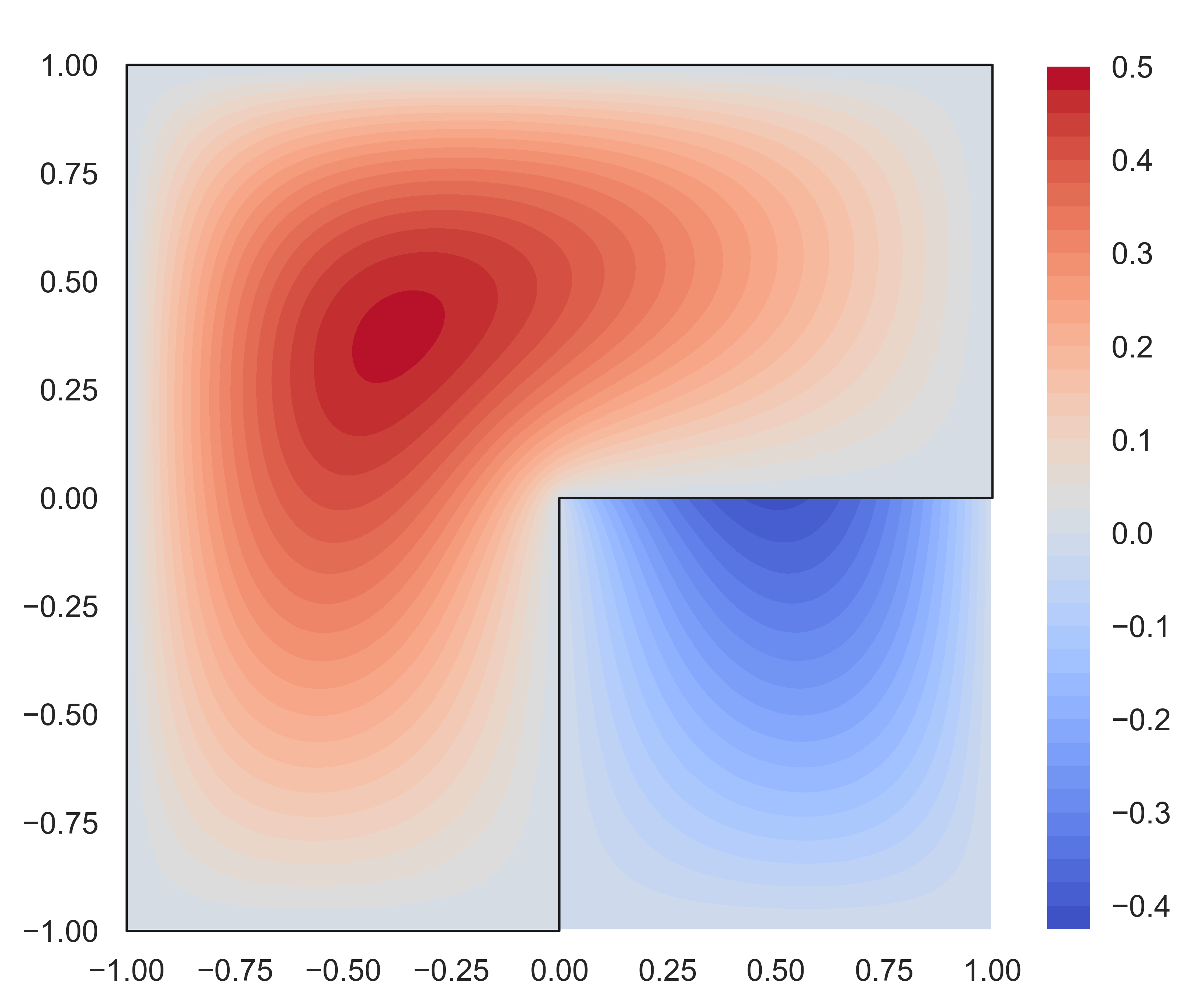}
    % \caption{(Top) exact solution $u$; (bottom) approximate solution $\widetilde{u}$}
    \end{minipage}
    \begin{minipage}{0.22\textwidth}
        \centering
        \includegraphics[width=\textwidth]{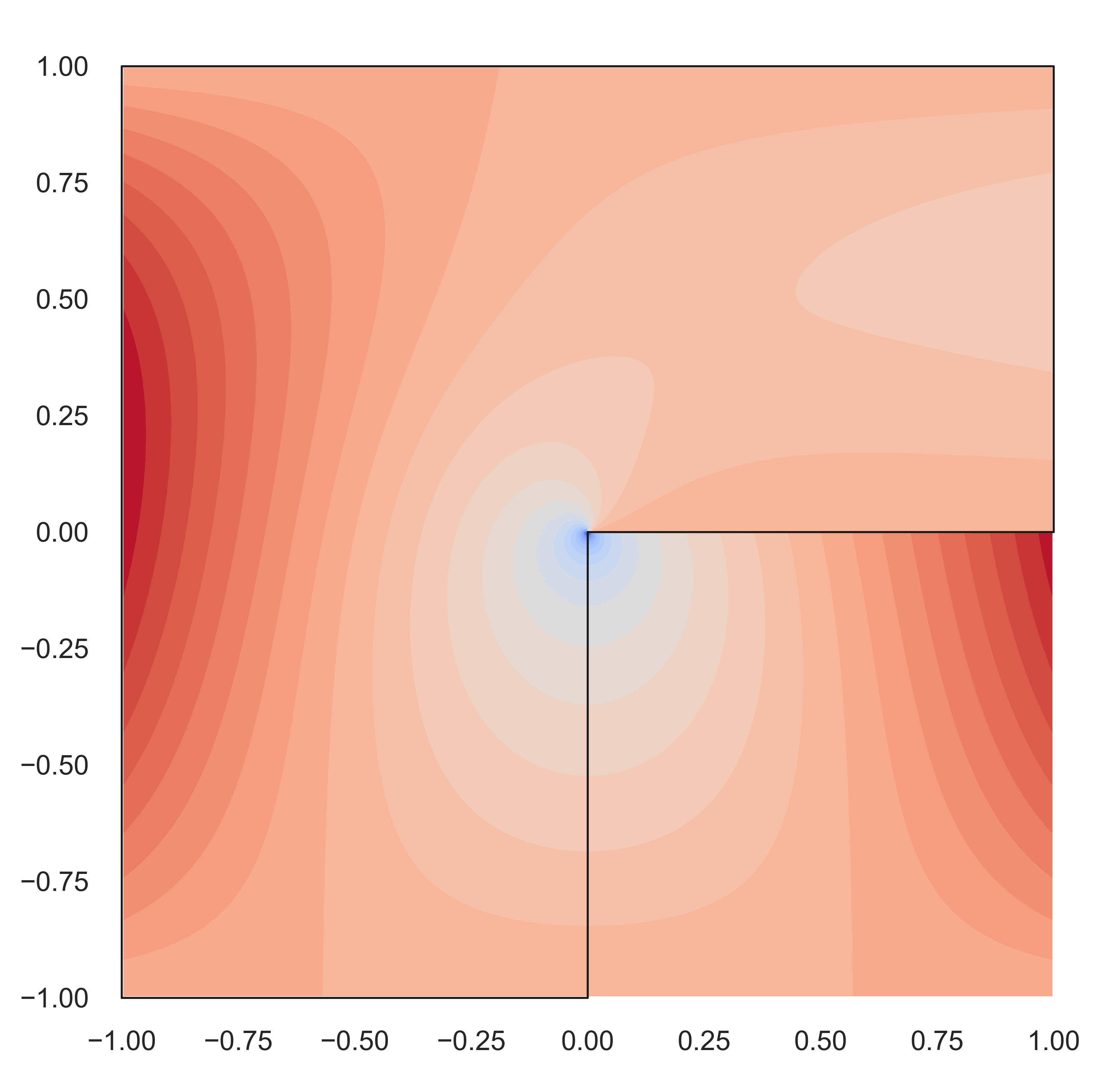}
    % \caption{(Top) exact solution $u$; (bottom) approximate solution $\widetilde{u}$}
    \end{minipage}
    \begin{minipage}{0.22\textwidth}
        \centering
        \includegraphics[width=\textwidth]{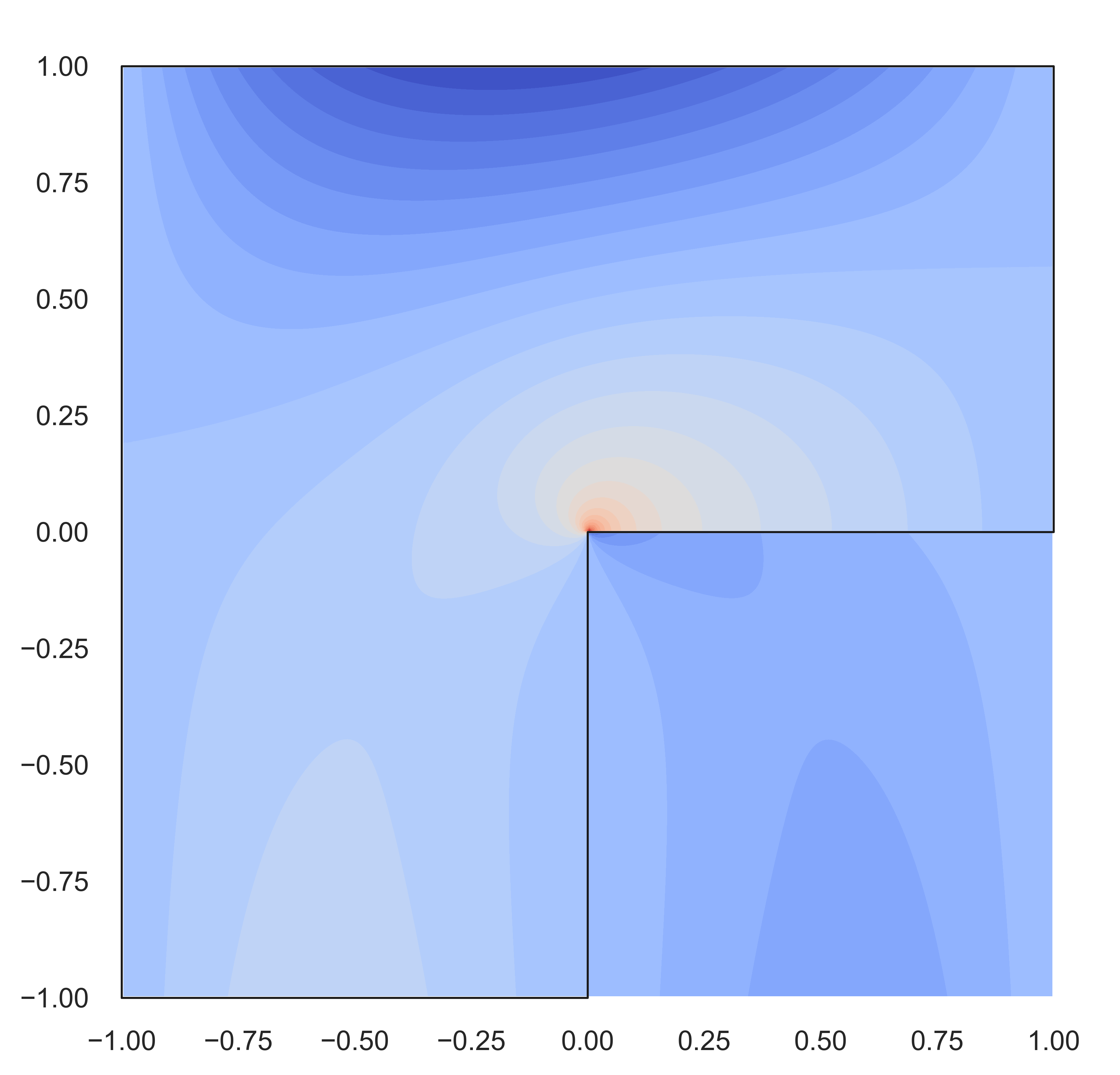}
    % \caption{(Top) exact solution $u$; (bottom) approximate solution $\widetilde{u}$}
    \end{minipage}
    \caption{Exact solution $u_{ex}$ (left), exact partial derivatives $\partial_x u_{ex}$ (middle), $\partial_y u_{ex}$ (right), plotted in $(-1,1)^2$.}
\end{figure}

\begin{figure}[H]
    \centering
    \begin{minipage}{0.24\textwidth}
        \centering
        \includegraphics[width=\textwidth]{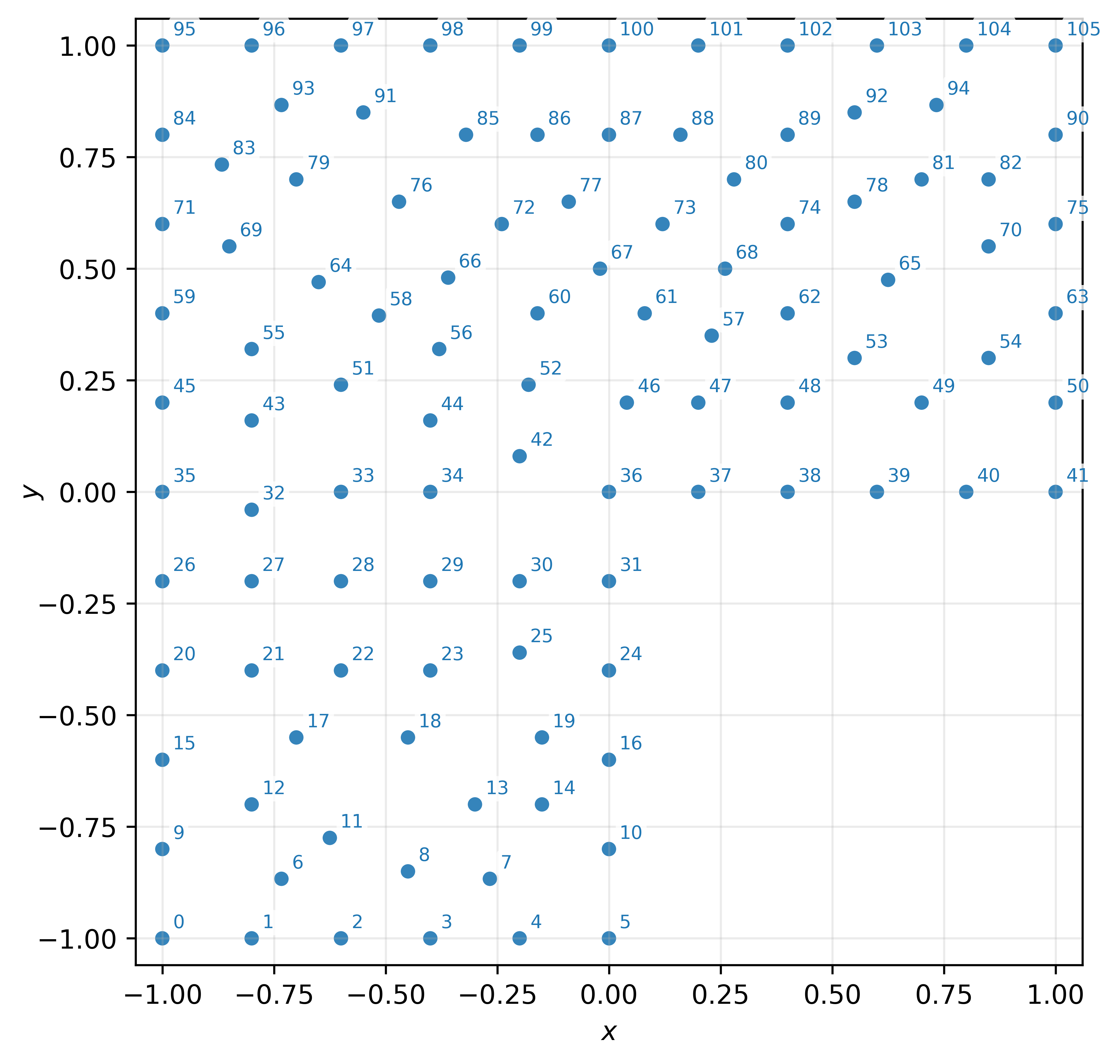}
    \end{minipage}
    \begin{minipage}{0.24\textwidth}
        \centering
        \includegraphics[width=\textwidth]{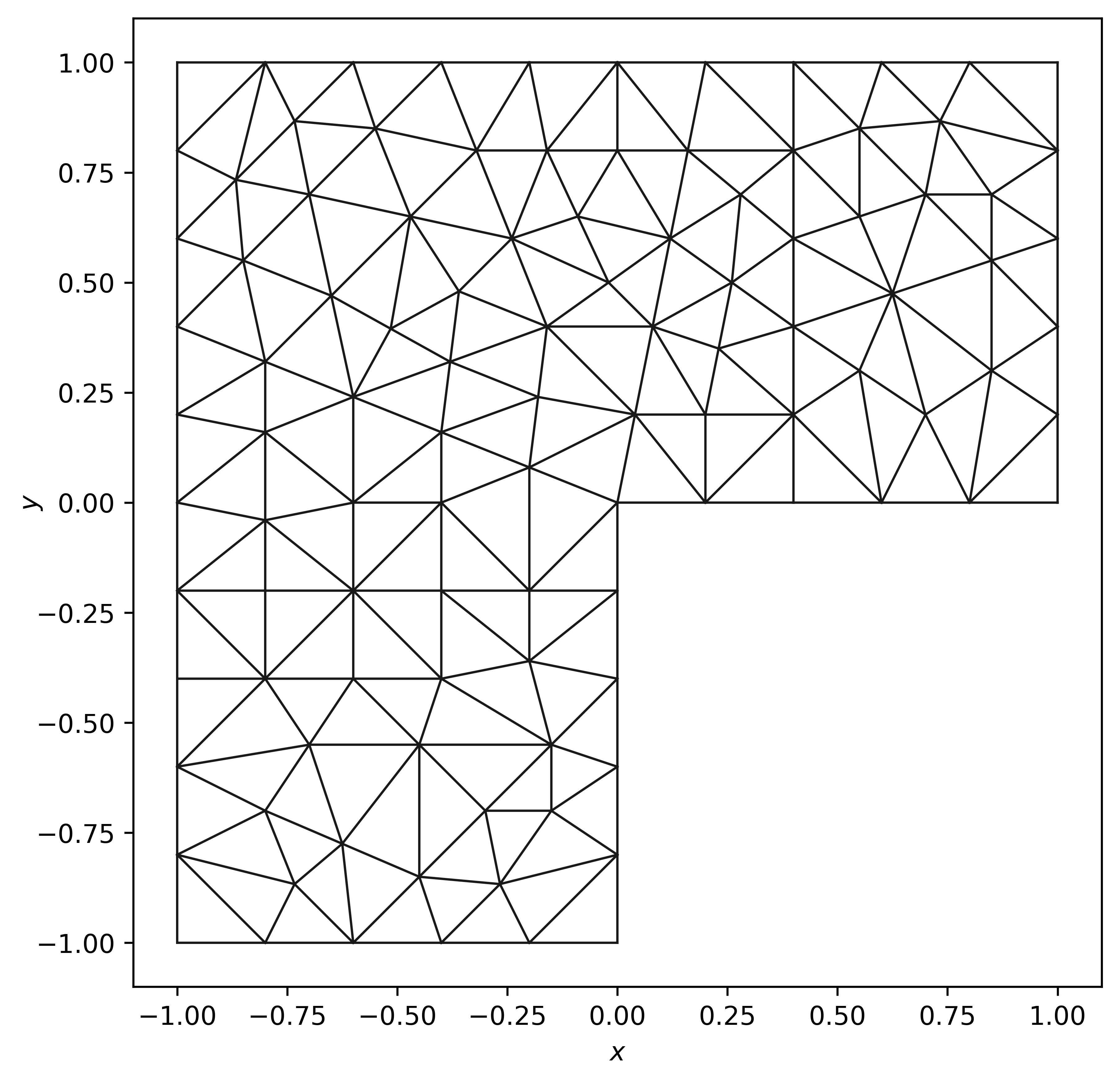}
    \end{minipage}
    \caption{Discretization of $\overline D$ with $106$ points (left) as vertices of the triangulation (right) generated by FEM after fixing equidistant boundary points; points ordered by $x$, then $y$ coordinate, annotated with their corresponding indices.} \label{fig:Lshaped indices}
\end{figure}

\begin{figure}[H]
    \centering
    \begin{minipage}{0.75\textwidth}
        \centering
        \includegraphics[width=\textwidth]{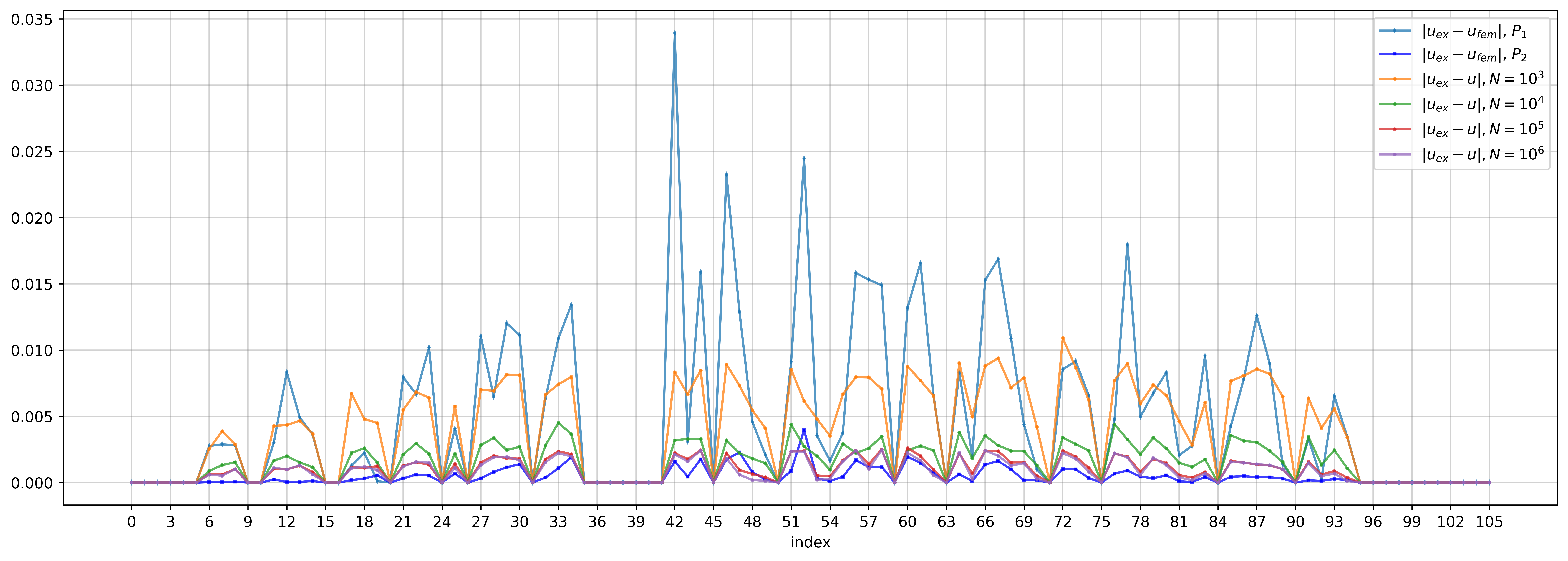}
    \end{minipage}
    \caption{(LOT and FEM) Combined error plots: Comparison of pointwise absolute error $|u_{ex} - u_{fem}|$ between the exact solution and the numerical solution obtained using FEM with $P_1$ and $P_2$ elements (light and basic blue, respectively), and mean of pointwise absolute errors $\frac{1}{30}{\sum_{i=1}^{30}|u_{ex} - \widehat{u}_i|}$, 30 estimates $\widehat{u}_i$ obtained with the proposed method LOT \eqref{eq:LOT} for computing $\nabla U$, with $N \in \{10^3, 10^4, 10^5, 10^6 \}$ Monte Carlo samples, $M=60$ $WoS$ steps. See \Cref{fig:Lshaped indices} for the point indices.}
\end{figure}

\begin{figure}[H]
    \centering
    \begin{minipage}{0.75\textwidth}
        \centering
        \includegraphics[width=\textwidth]{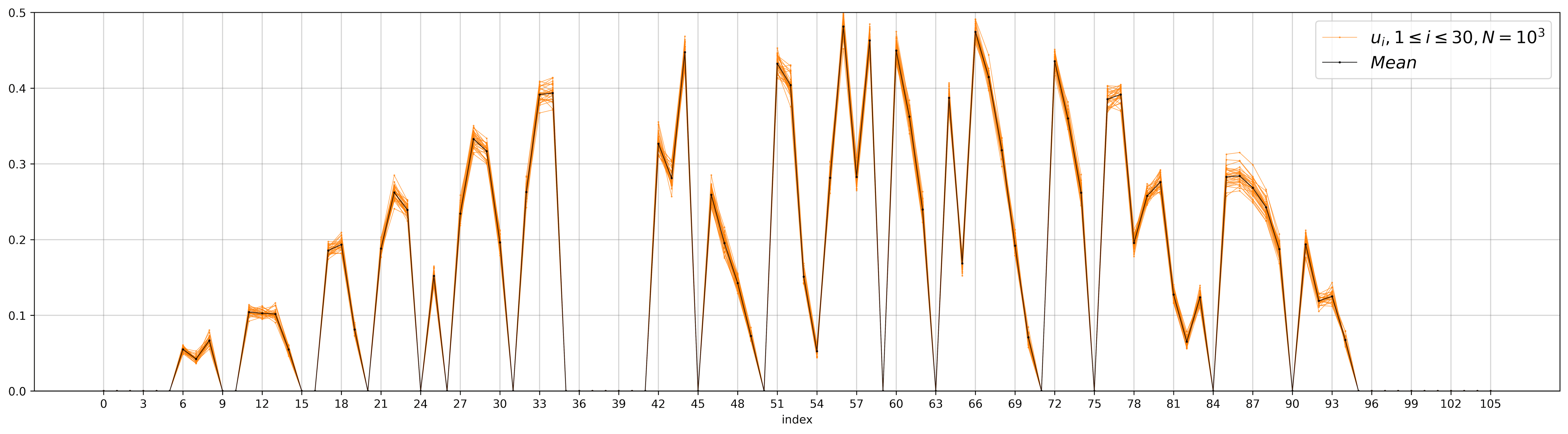}
        % \caption{$n = 400$}
    \end{minipage}
    % \hspace{20pt}
\end{figure}
\begin{figure}[H]
    \centering
    \begin{minipage}{0.75\textwidth}
        \centering
        \includegraphics[width=\textwidth]{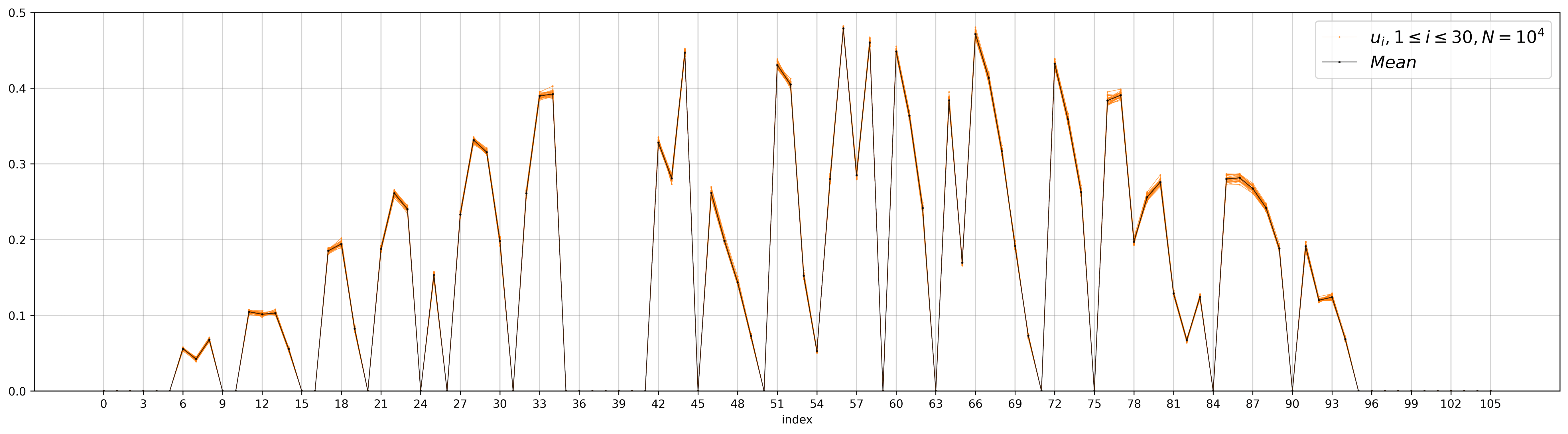}
    \end{minipage}
    \caption{Numerical solution: $30$ estimates $\widehat{u}_i$ obtained with the proposed method (orange) together with their mean (black), using $N=10^3$ Monte Carlo samples (top) and $N=10^4$ (bottom), with $M=60$ $WoS$ steps. See \Cref{fig:Lshaped indices} for the point indices.}
\end{figure}

\begin{figure}[H]
    \centering
    \begin{minipage}{0.75\textwidth}
        \centering
        \includegraphics[width=\textwidth]{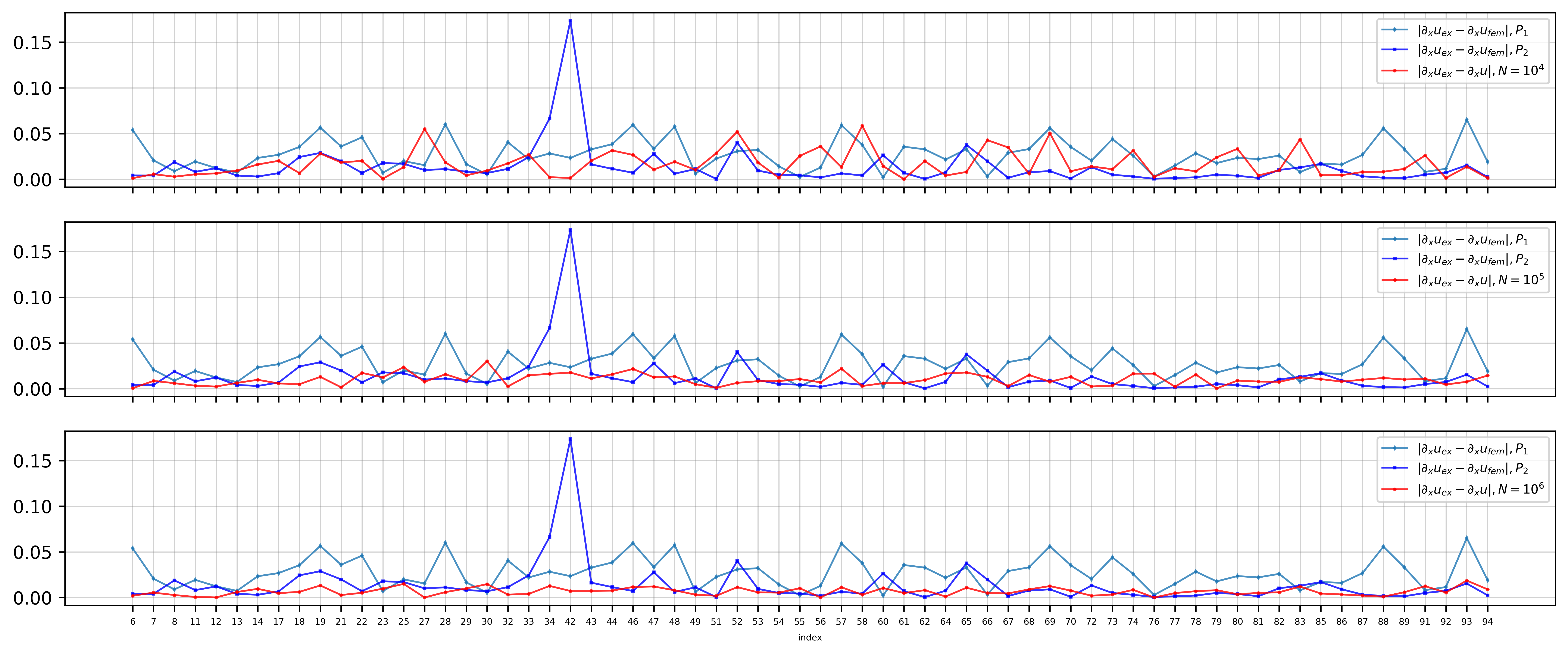}
        % \caption{$n = 400$}
    \end{minipage}
    % \hspace{20pt}
\end{figure}
\begin{figure}[H]
    \centering
    \begin{minipage}{0.75\textwidth}
        \centering
        \includegraphics[width=\textwidth]{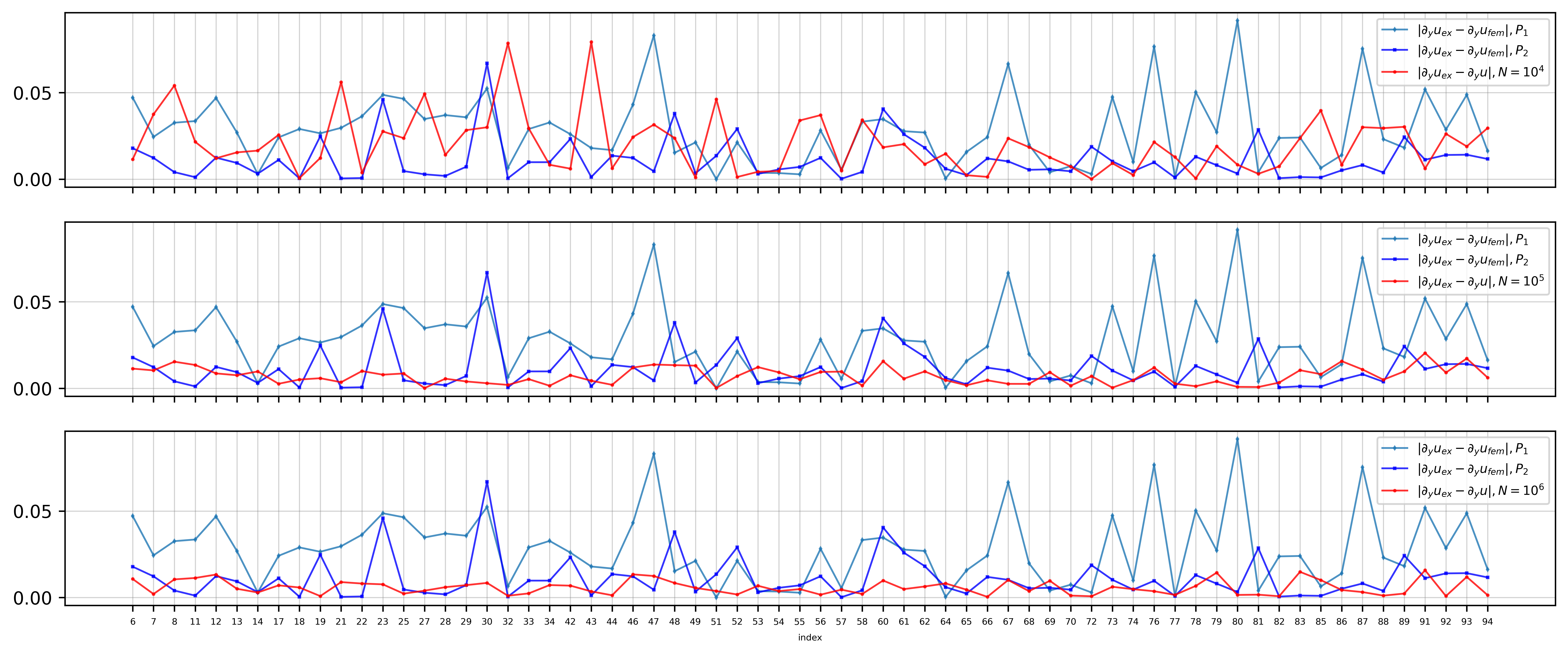}
    \end{minipage}
        \caption{(LOT and FEM) Combined error plots: Comparison of pointwise absolute error $|\partial_x u_{ex} - \partial_x u_{fem}|$ (top), $|\partial_y u_{ex} - \partial_y u_{fem}|$ (bottom) between the partial derivatives of the exact solution and of the numerical solution obtained using FEM with $P_1$ and $P_2$ elements, respectively (light and basic blue), and pointwise absolute errors $|\partial_x u_{ex} - \widehat{\partial_x u}|$ (top), $|\partial_y u_{ex} - \widehat{\partial_y u}|$ (bottom), for one sample obtained with the proposed method LOT \eqref{eq:LOT} for computing $\nabla U$, with $N \in \{ 10^4, 10^5, 10^6 \}$ Monte Carlo samples, $M=60$ $WoS$ steps (red). The plotted values are evaluated at the interior points of the discretization; see \Cref{fig:Lshaped indices} for the point indices.}
\end{figure}

%% file: elliptic-solution-operator.tex
\begin{figure}[H]
    \centering
    \begin{minipage}{0.2\textwidth}
        \centering
        \includegraphics[width=\textwidth]{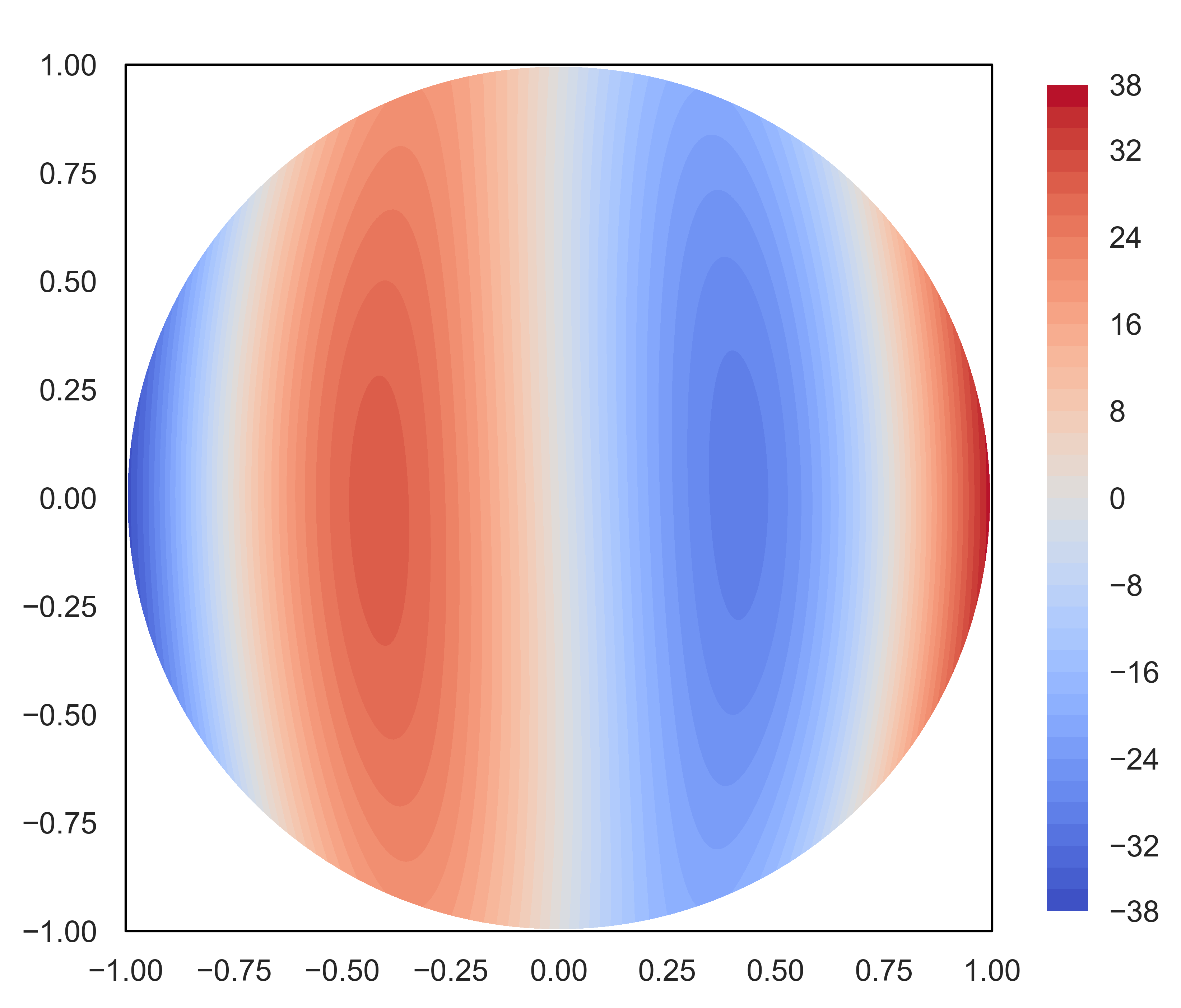}
    \end{minipage}
    % \hspace{20pt}
    \begin{minipage}{0.2\textwidth}
        \centering
        \includegraphics[width=\textwidth]{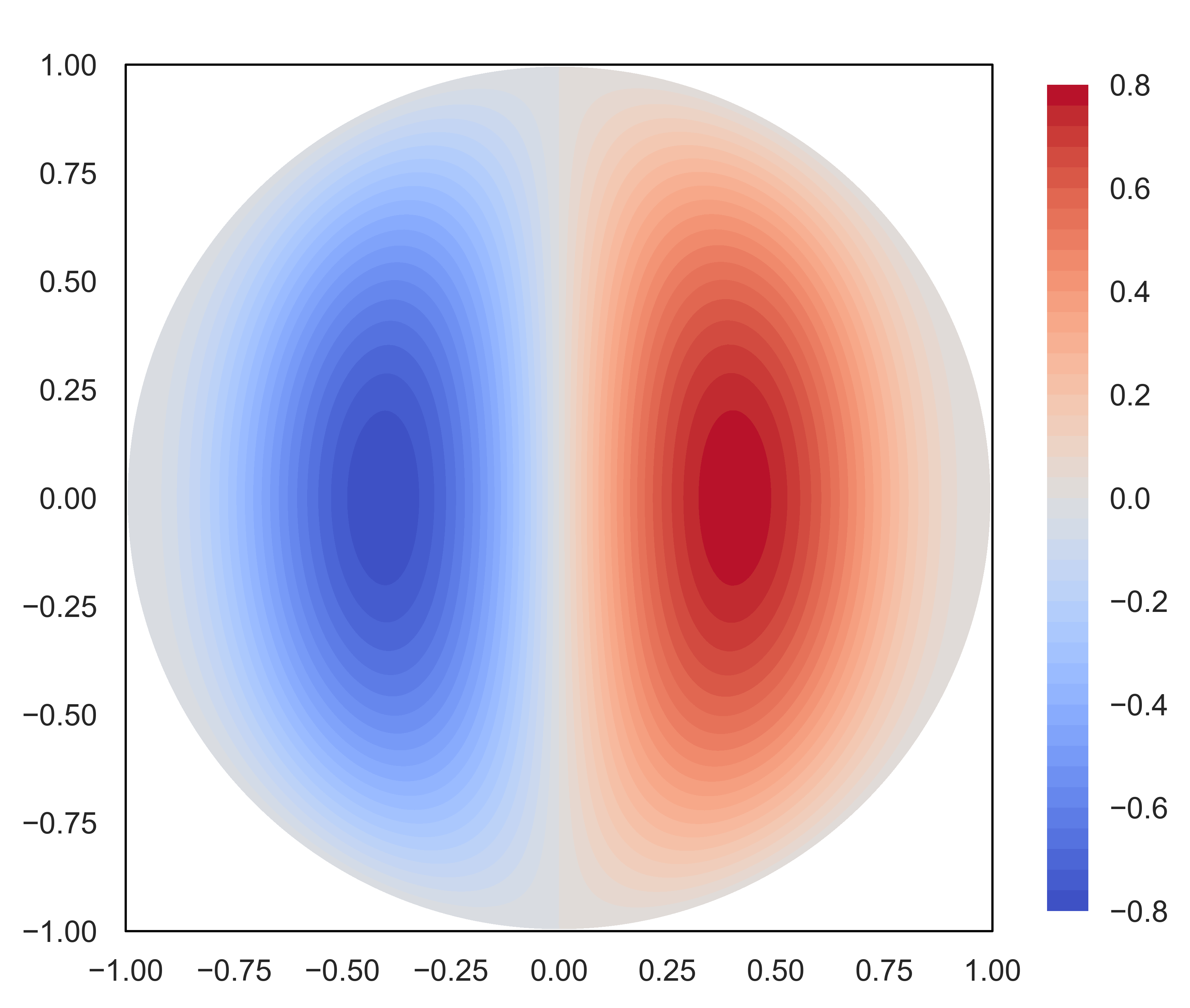}
        % \caption{$n = 400$}
    \end{minipage}
    \begin{minipage}{0.2\textwidth}
        \centering
        \includegraphics[width=\textwidth]{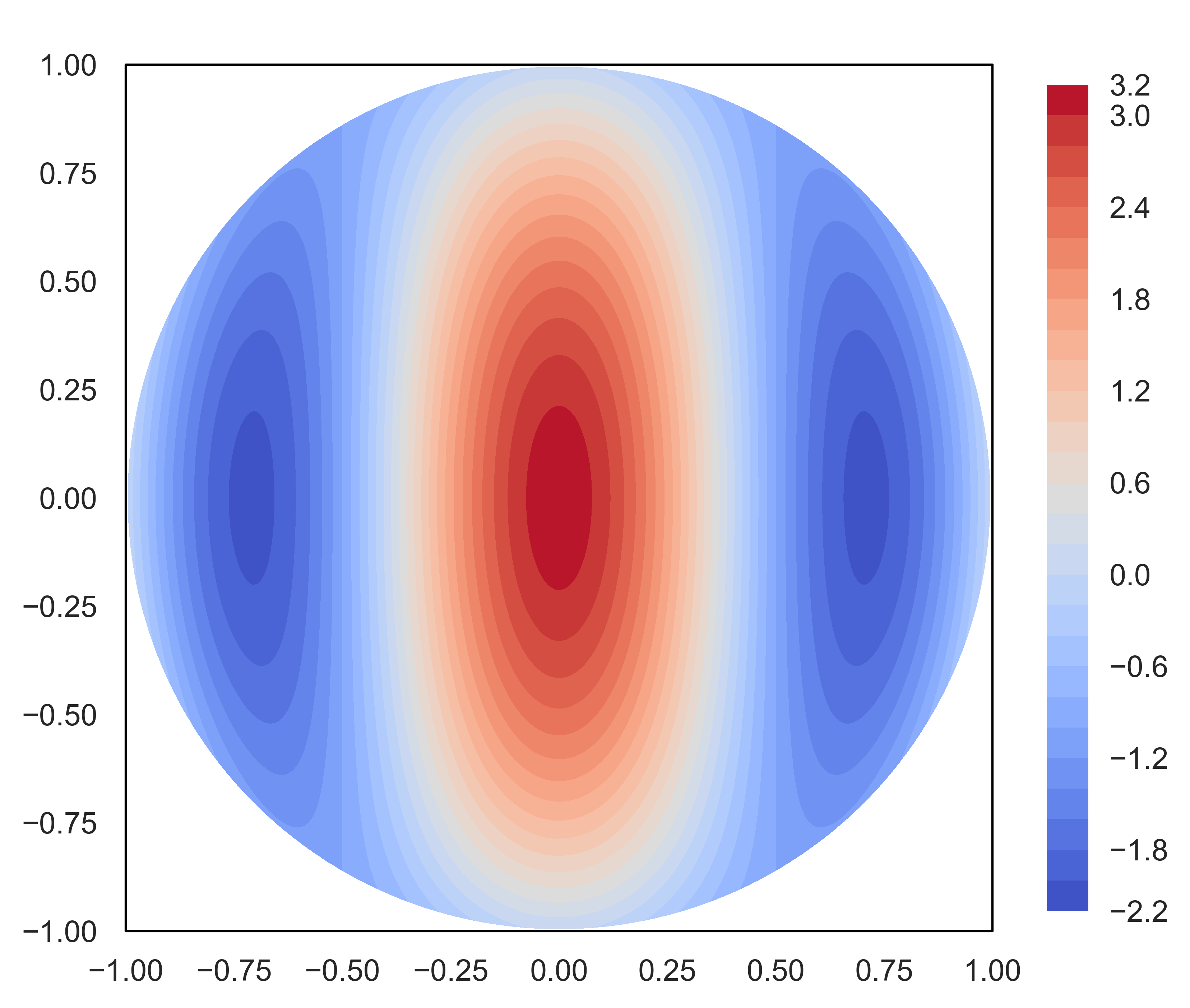}
        % \caption{$n = 400$}
    \end{minipage}
    \begin{minipage}{0.2\textwidth}
        \centering
        \includegraphics[width=\textwidth]{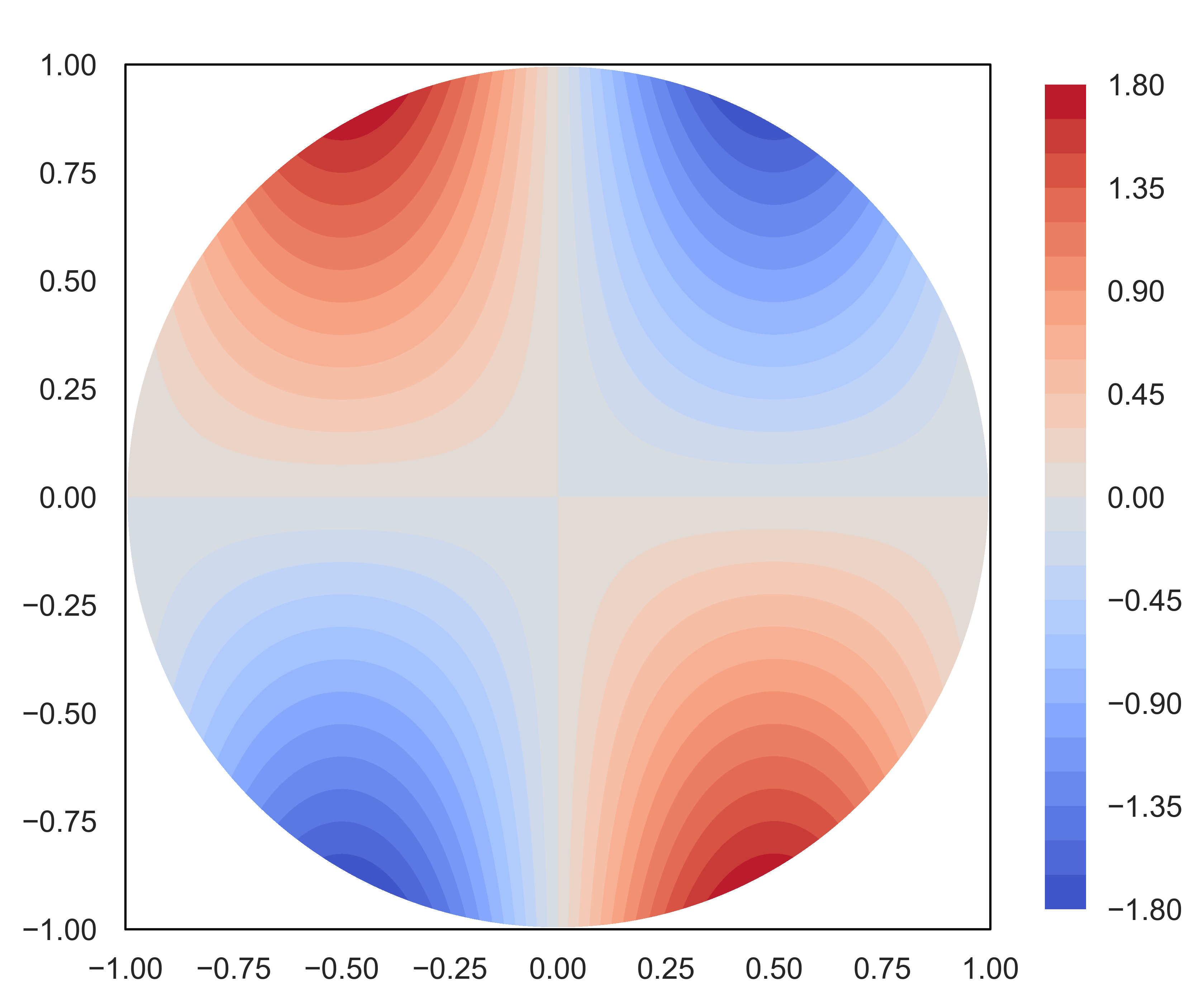}
        % \caption{$n = 400$}
    \end{minipage}
    \caption{Source function $f$ (left), exact solution $u_{ex}$ (middle-left), exact partial derivatives $\partial_x u_{ex}$ and $\partial_y u_{ex}$ (middle-right and right), plotted in $B(0,1)$.}
\end{figure}

\begin{figure}[H]
    \centering
    \begin{minipage}{0.17\textwidth}
        \centering
        \includegraphics[width=\textwidth]{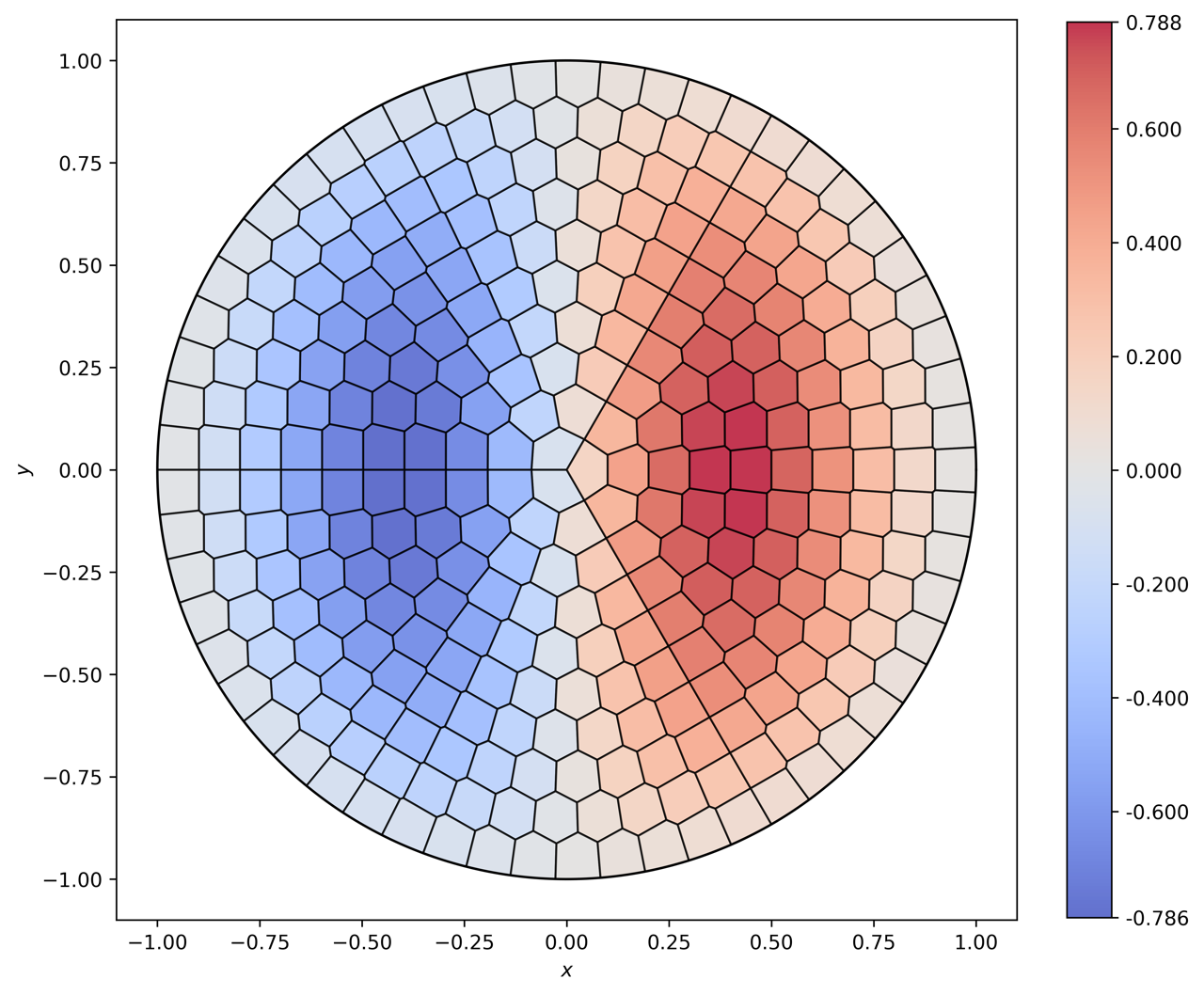}
    \end{minipage}
    % \hspace{20pt}
    \begin{minipage}{0.19\textwidth}
        \centering
        \includegraphics[width=\textwidth]{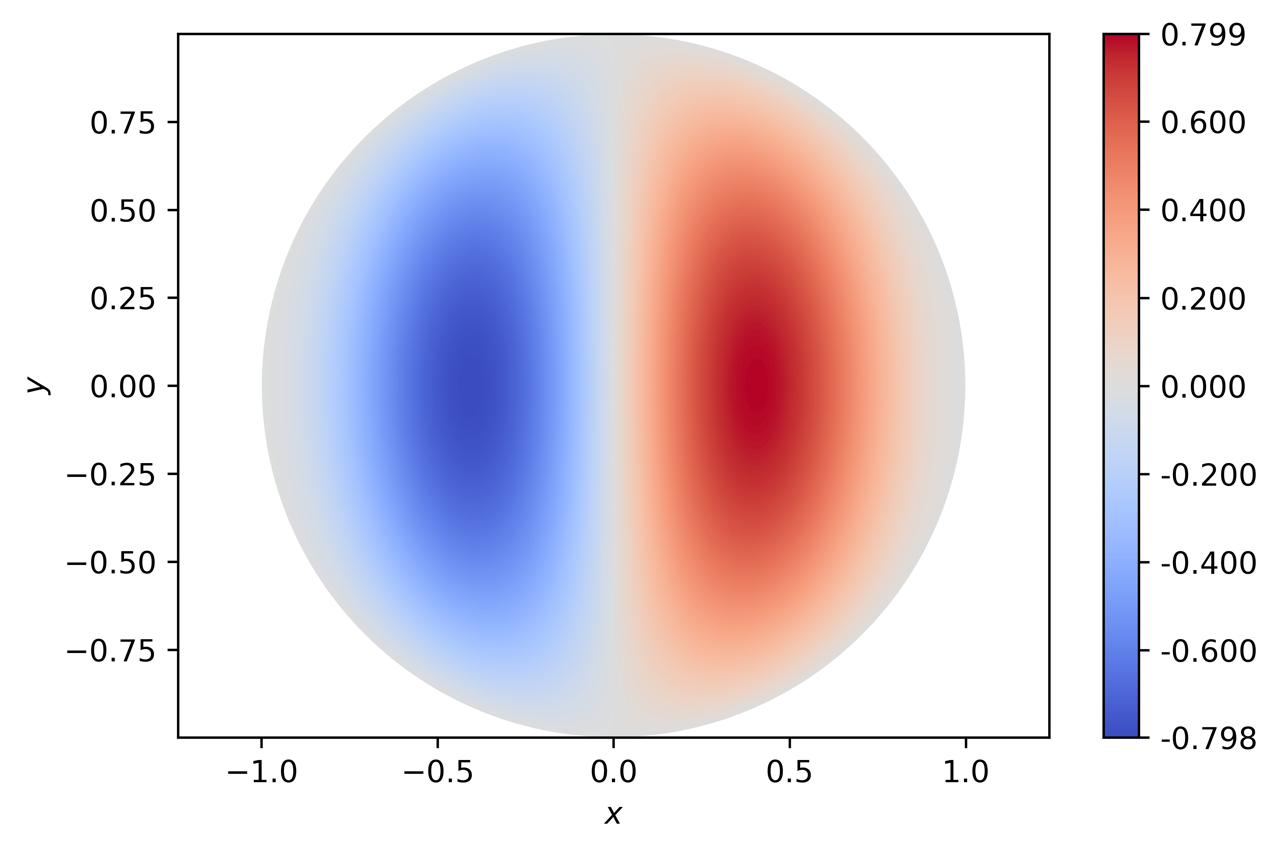}
        % \caption{$n = 400$}
    \end{minipage}
    \begin{minipage}{0.2\textwidth}
        \centering
        \includegraphics[width=\textwidth]{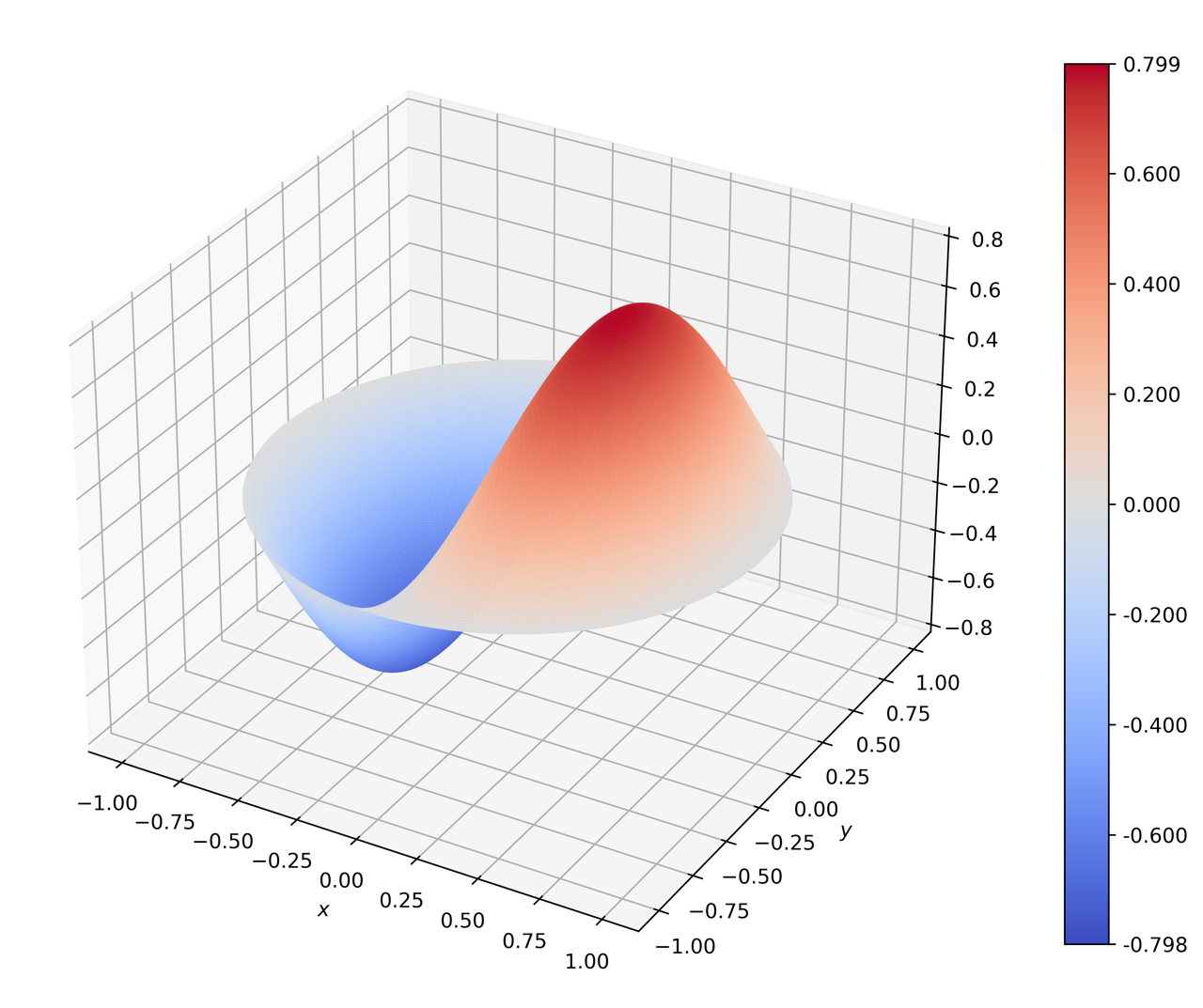}
        % \caption{$n = 400$}
    \end{minipage}
    \begin{minipage}{0.17\textwidth}
        \centering
        \includegraphics[width=\textwidth]{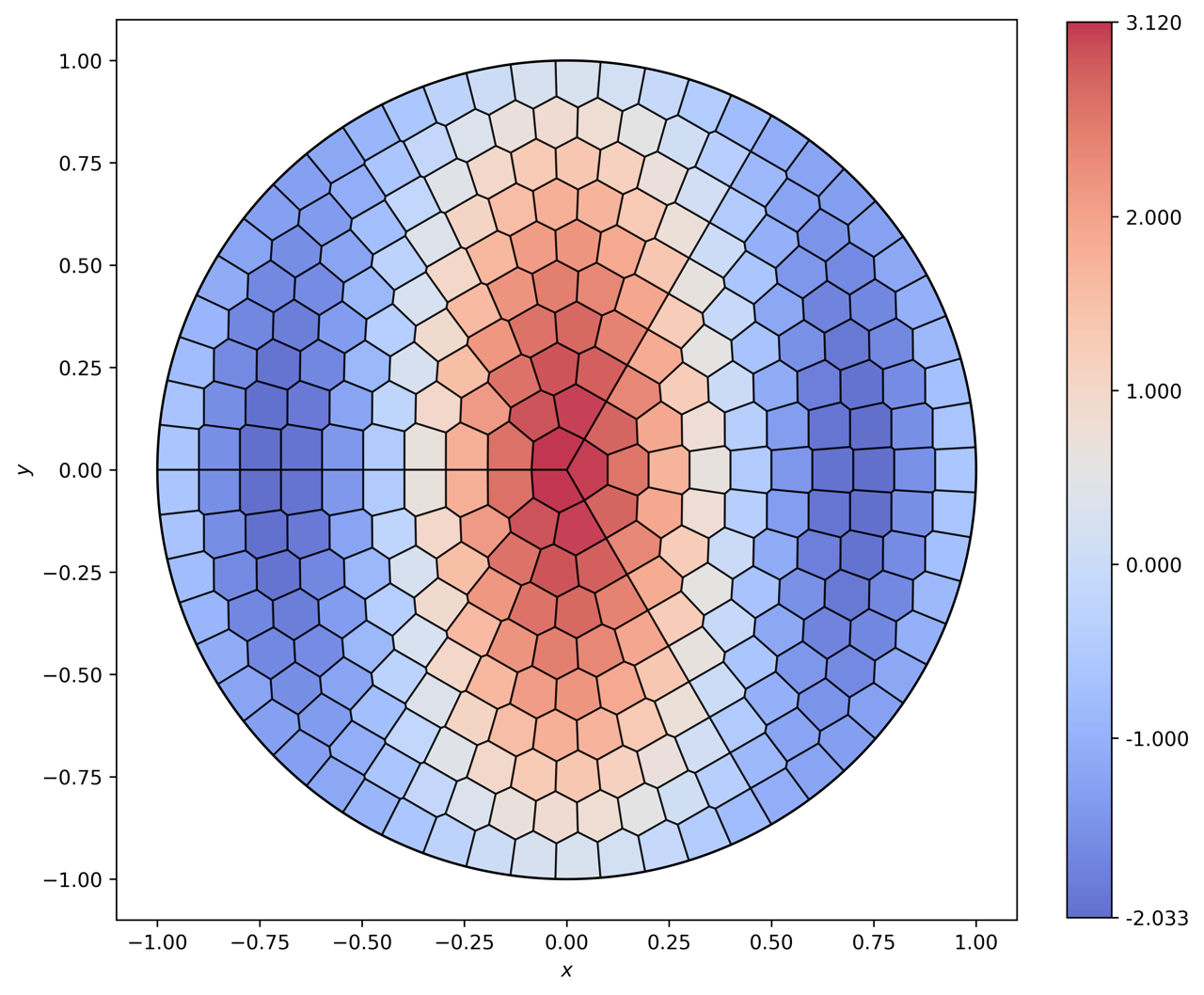}
        % \caption{$n = 400$}
    \end{minipage}
    \begin{minipage}{0.17\textwidth}
        \centering
        \includegraphics[width=\textwidth]{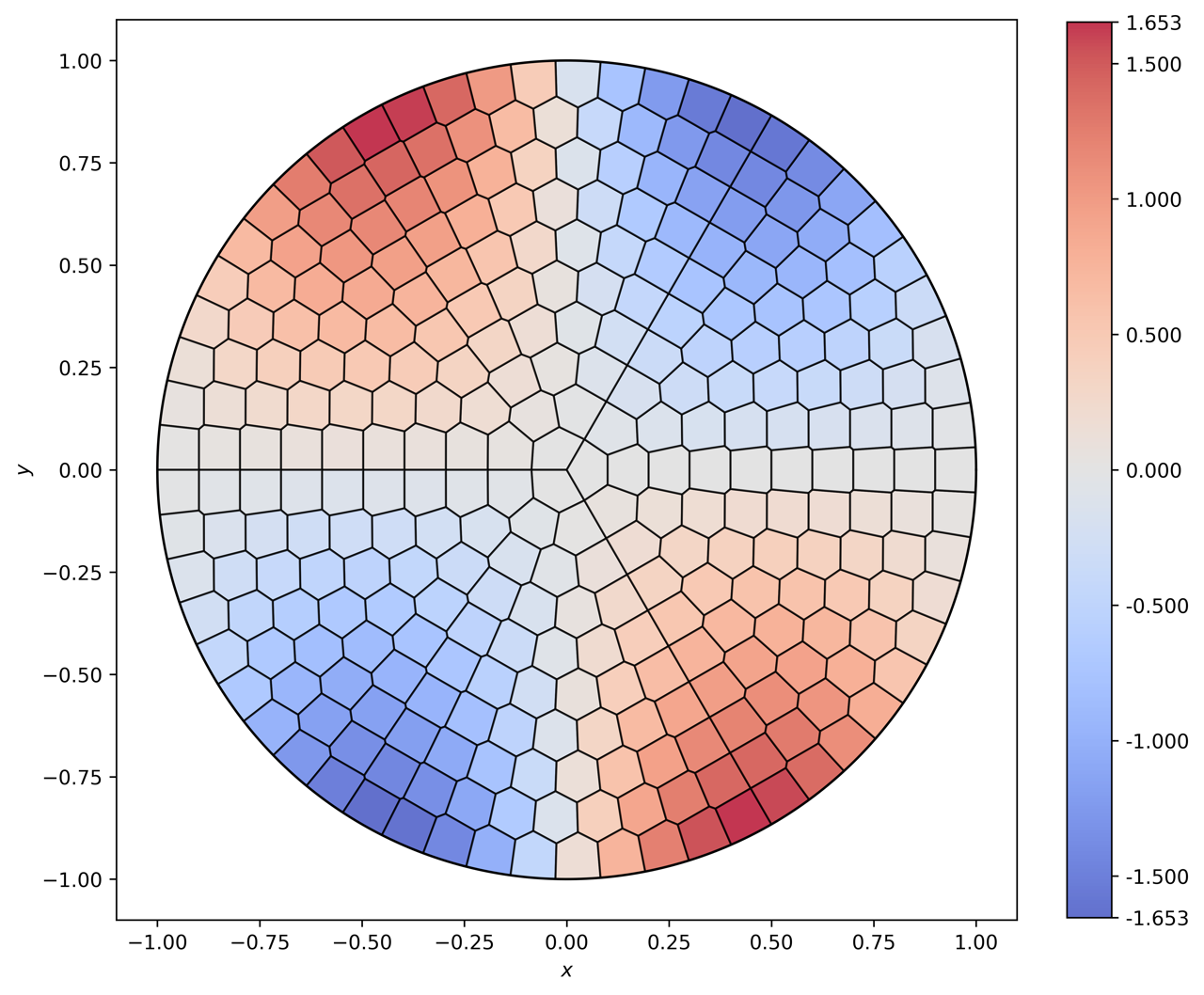}
        % \caption{$n = 400$}
    \end{minipage}\\
% \end{figure}

% \begin{figure}[h!]
    \centering
    \begin{minipage}{0.17\textwidth}
        \centering
        \includegraphics[width=\textwidth]{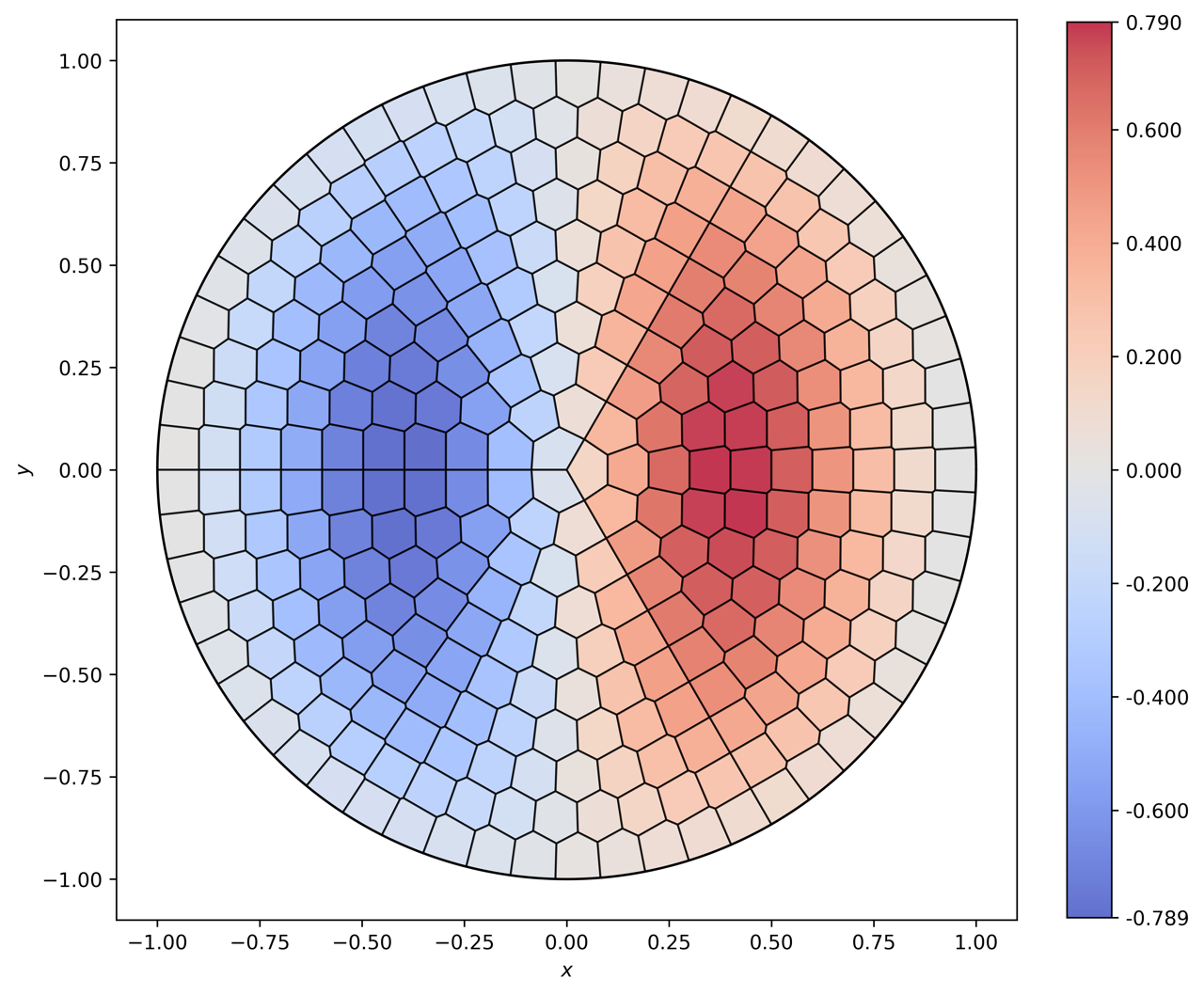}
        % \caption{$n = 400$}
    \end{minipage}
    % \hspace{20pt}
    \begin{minipage}{0.19\textwidth}
        \centering
        \includegraphics[width=\textwidth]{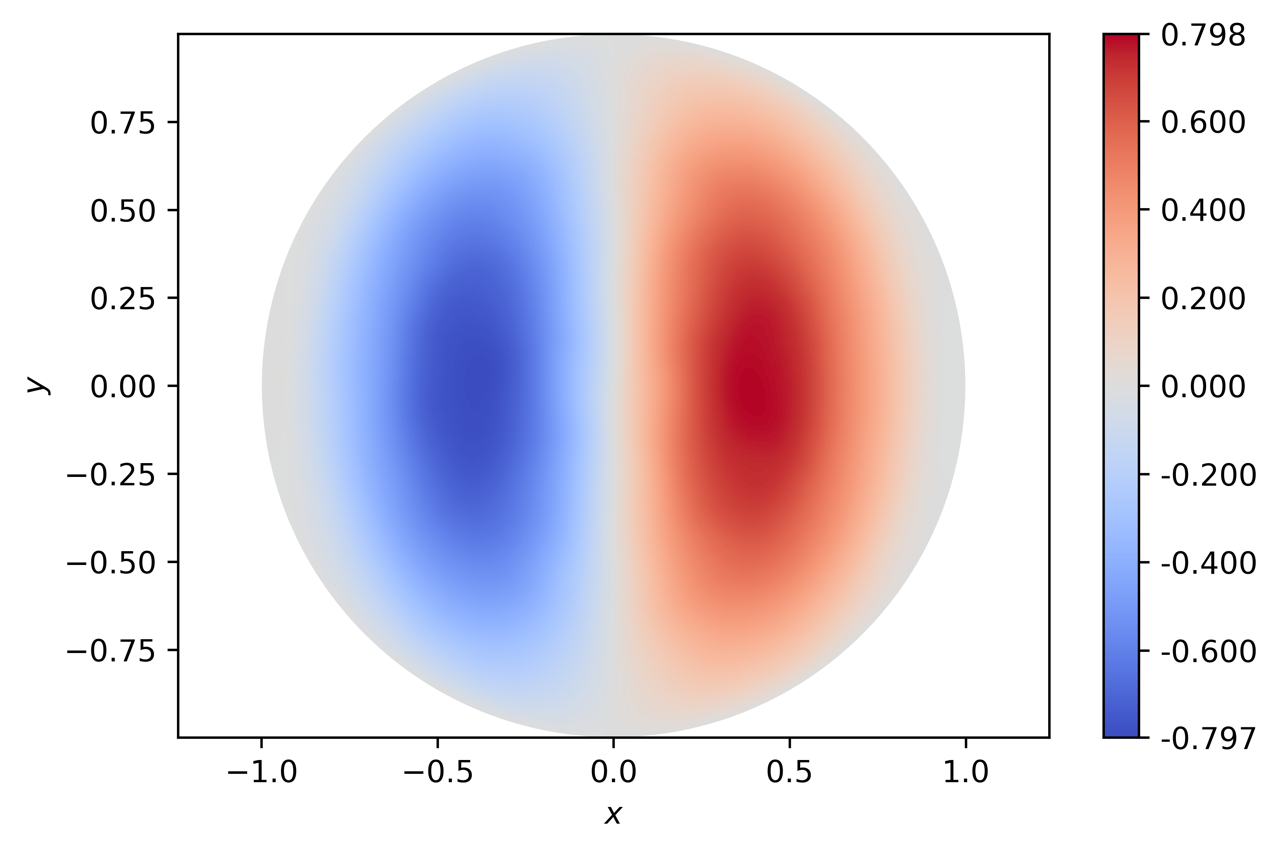}
    \end{minipage}
    \begin{minipage}{0.2\textwidth}
        \centering
        \includegraphics[width=\textwidth]{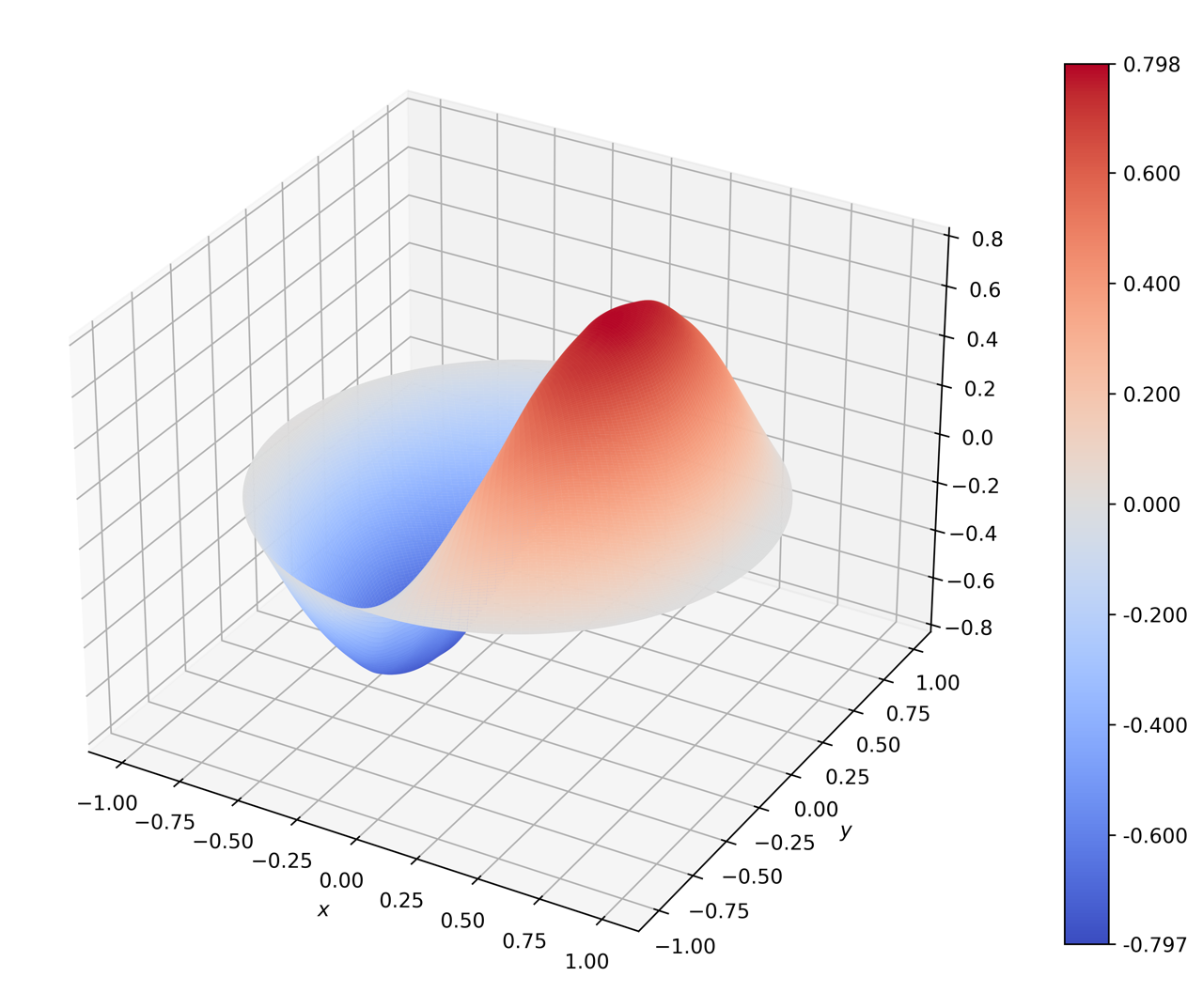}
    \end{minipage}
    \begin{minipage}{0.17\textwidth}
        \centering
        \includegraphics[width=\textwidth]{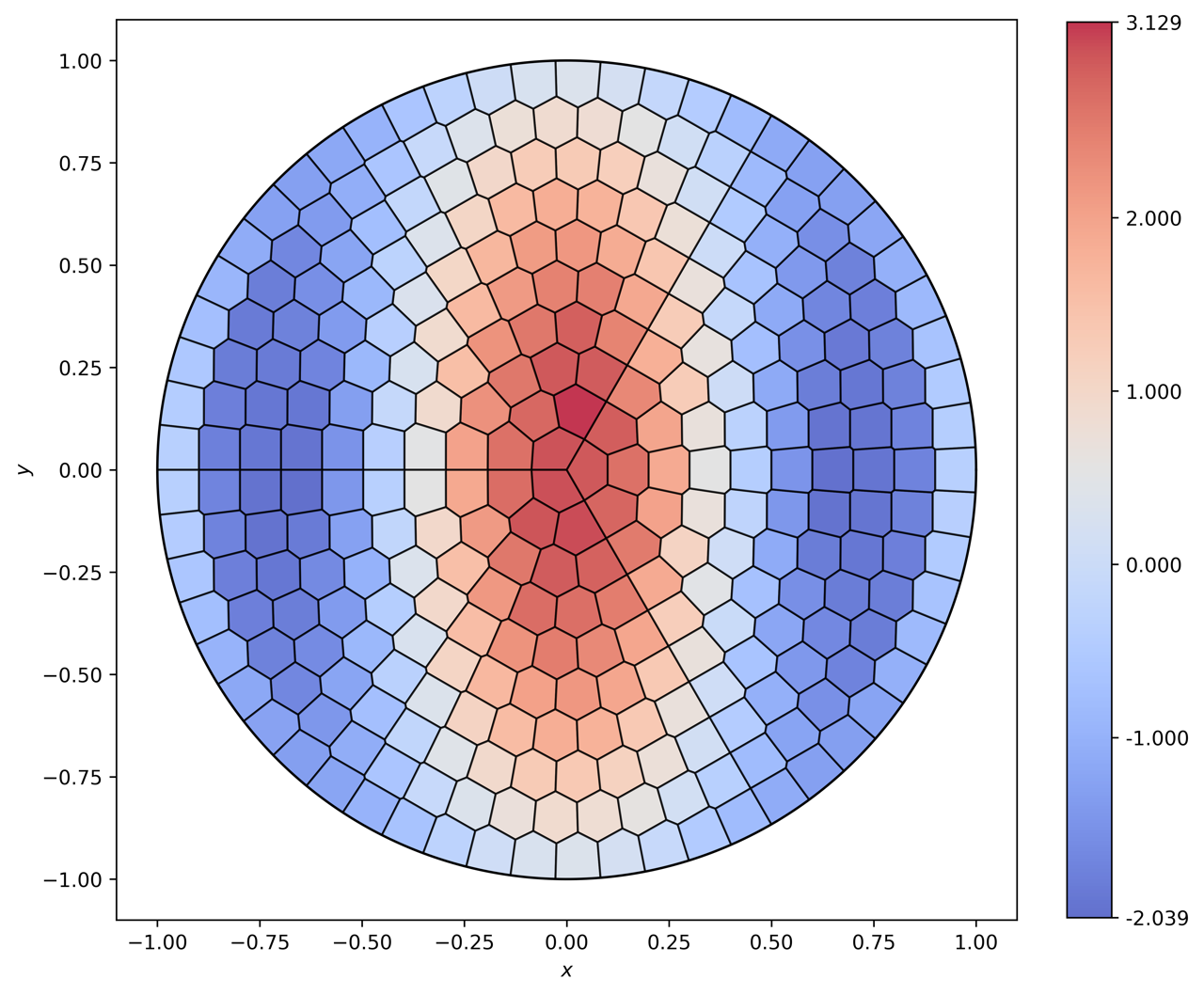}
    \end{minipage}
    \begin{minipage}{0.17\textwidth}
        \centering
        \includegraphics[width=\textwidth]{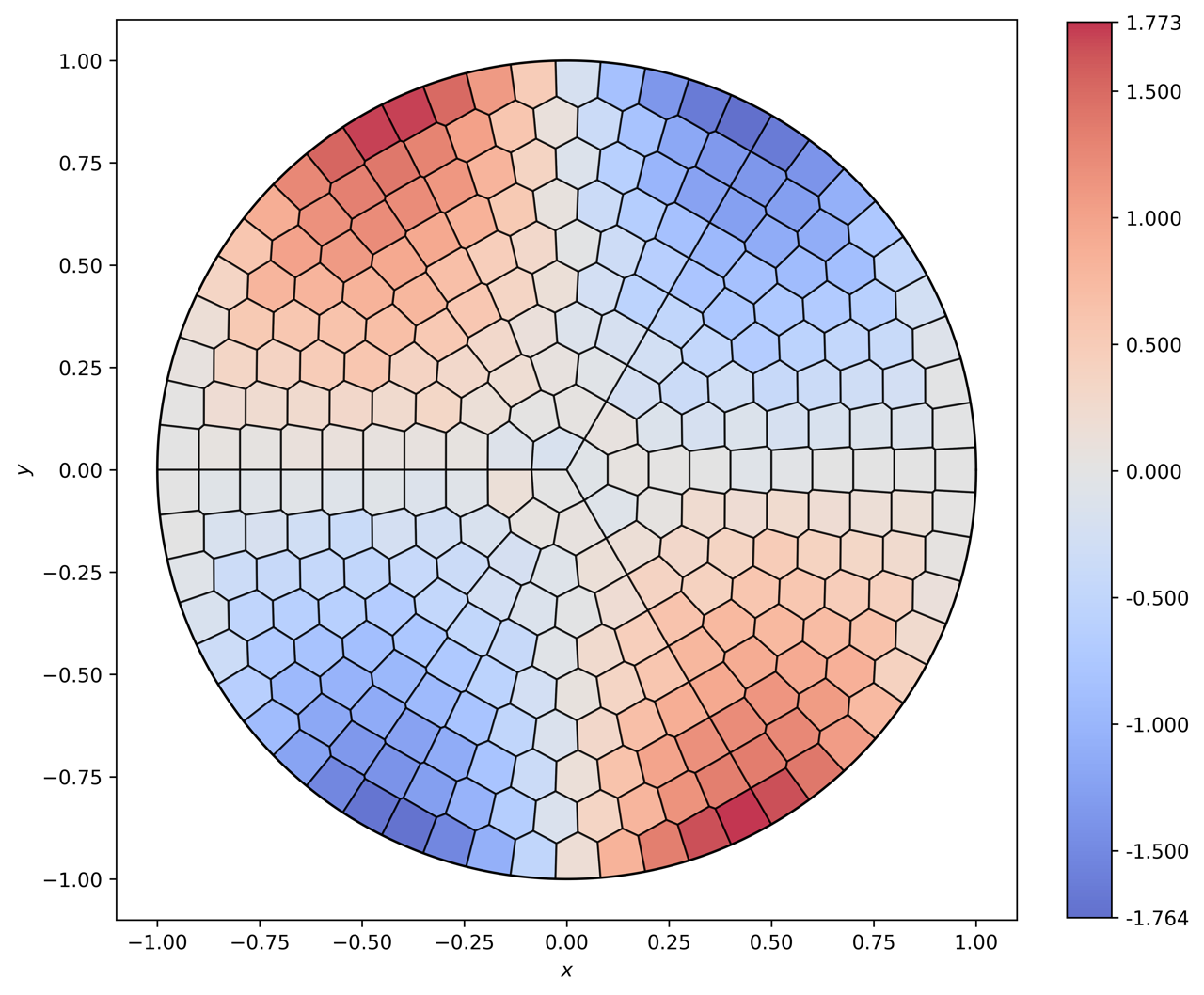}
        % \caption{$n = 400$}
    \end{minipage}
    \caption{(Top) Exact solution $u_{\text{ex}}$ (3 plots) and partial derivatives $\partial_x u_{\text{ex}}$ and $\partial_y u_{\text{ex}}$; (Bottom) Approximate solution $\widehat{u}$ (3 plots) and approximate partial derivatives $\widehat{\partial_x u}$ and $\widehat{\partial_y u}$, obtained with the proposed method.}
\end{figure}

\begin{figure}[H]
    \centering
    \begin{minipage}{0.6\textwidth}
        \centering
         \includegraphics[width=\textwidth]{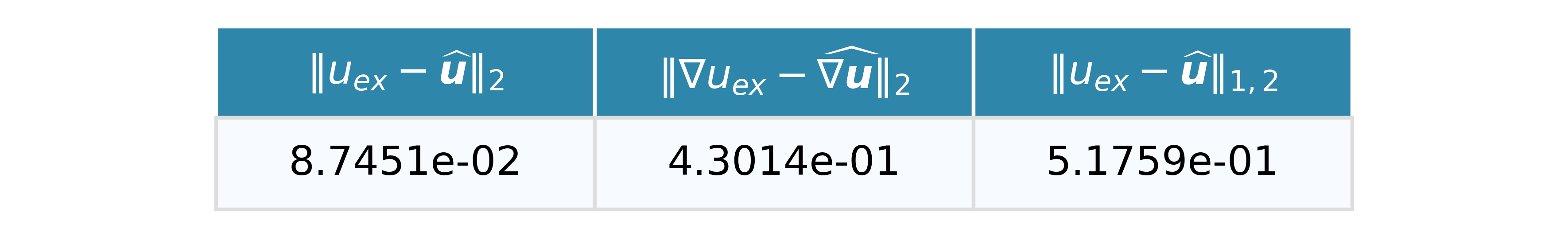}
        \includegraphics[width=\textwidth]{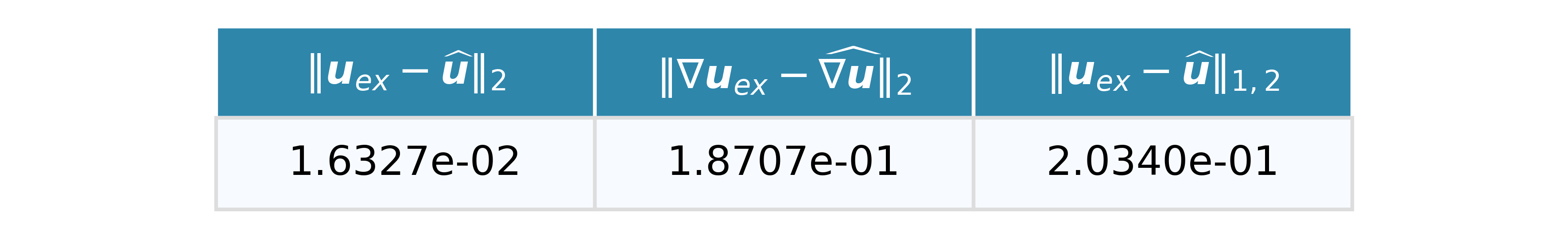}
    \end{minipage}
    \captionof{table}{Errors between exact and approximate solution: absolute errors in $L^2$ and $W^{1,2}$ norms. Note: $\| {u}_{\text{ex}} - \mathbf{u}_{\text{ex}} \|_2 \approx 8.5 \times 10^{-2}$.}
\end{figure}

%% file: elliptic-disk-new.tex
\begin{figure}[H]
    \centering
    \begin{minipage}{0.17\textwidth}
        \centering
        \includegraphics[width=\textwidth]{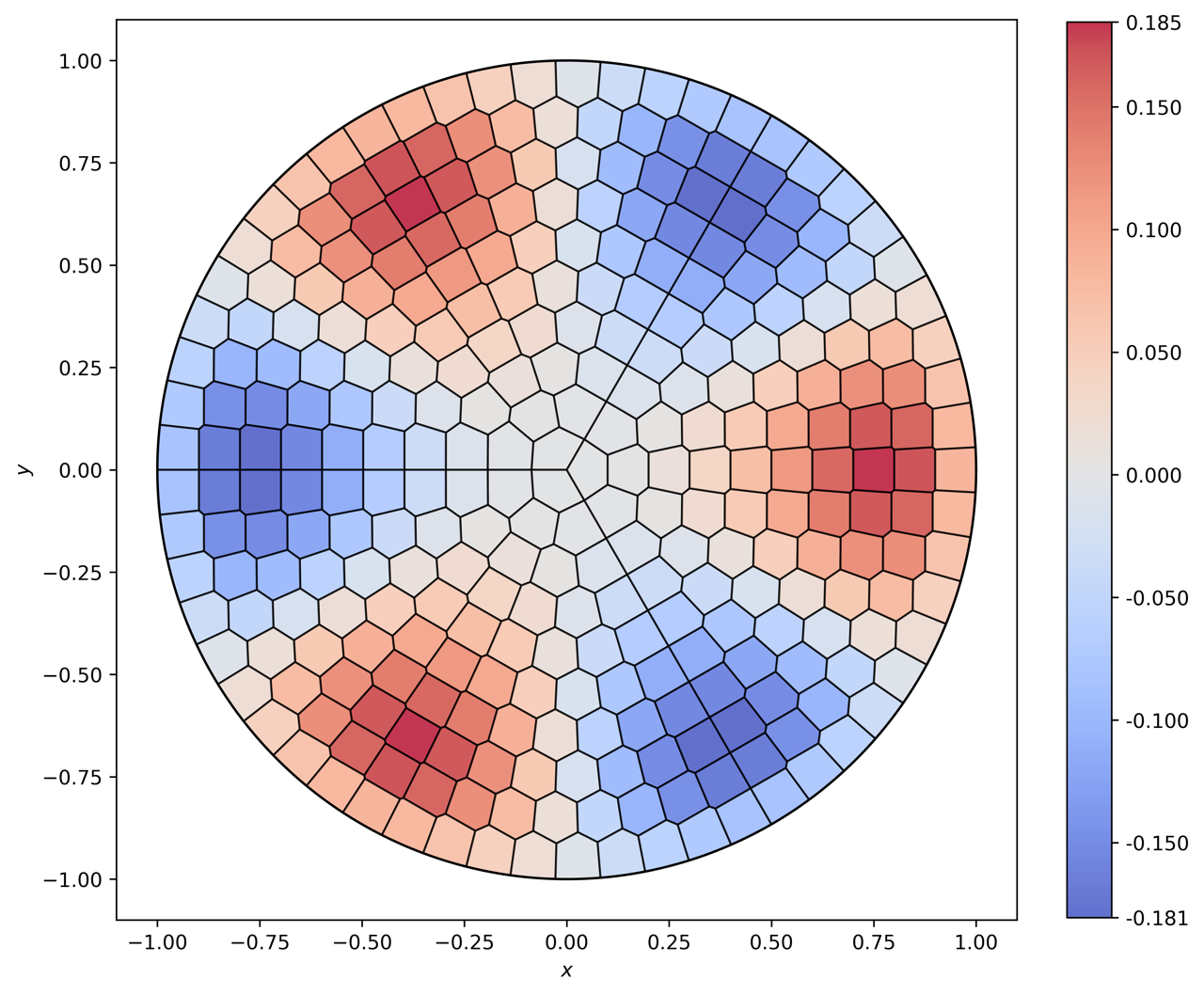}
    \end{minipage}
    % \hspace{20pt}
    \begin{minipage}{0.19\textwidth}
        \centering
        \includegraphics[width=\textwidth]{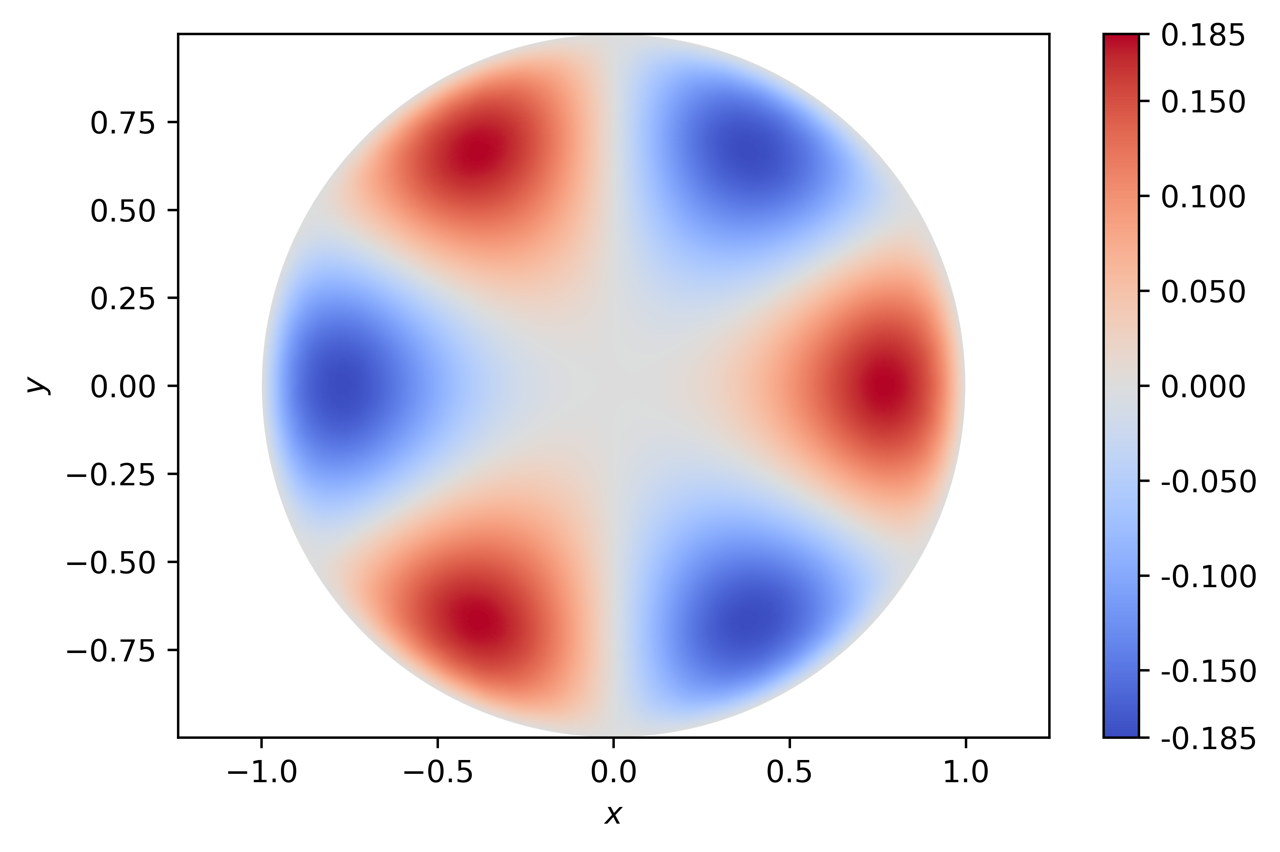}
        % \caption{$n = 400$}
    \end{minipage}
    \begin{minipage}{0.2\textwidth}
        \centering
        \includegraphics[width=\textwidth]{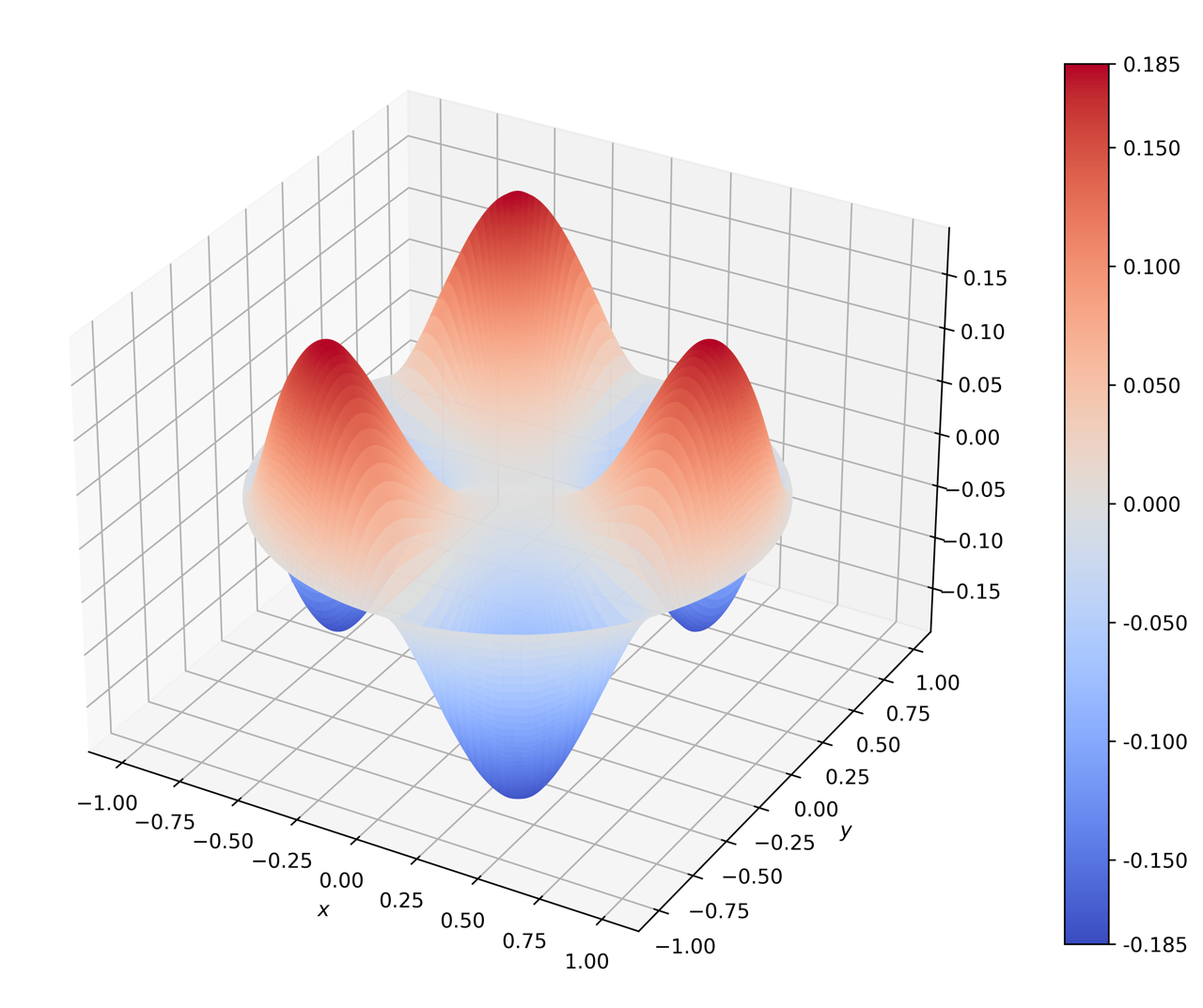}
        % \caption{$n = 400$}
    \end{minipage}
    \begin{minipage}{0.17\textwidth}
        \centering
        \includegraphics[width=\textwidth]{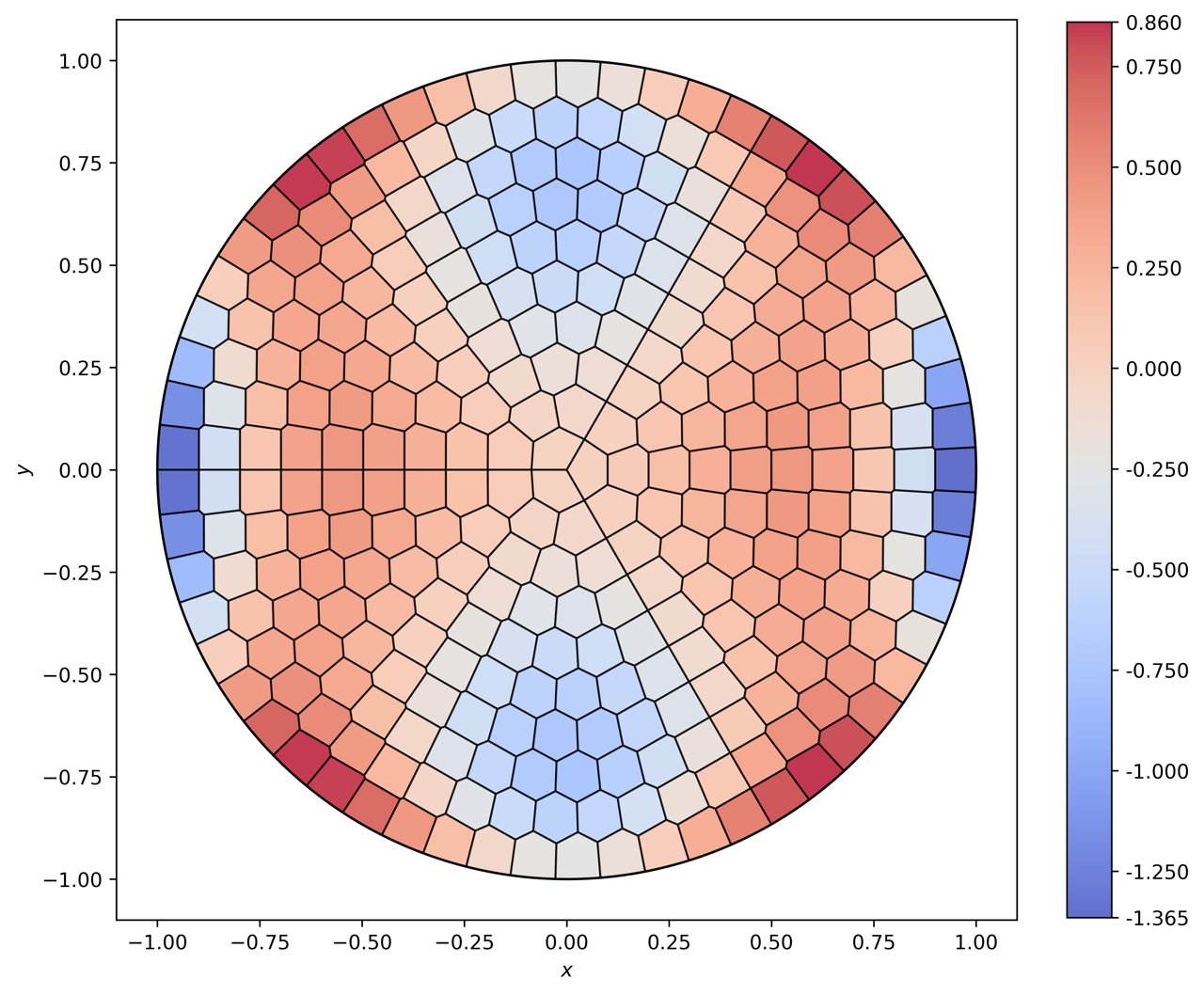}
        % \caption{$n = 400$}
    \end{minipage}
    \begin{minipage}{0.17\textwidth}
        \centering
        \includegraphics[width=\textwidth]{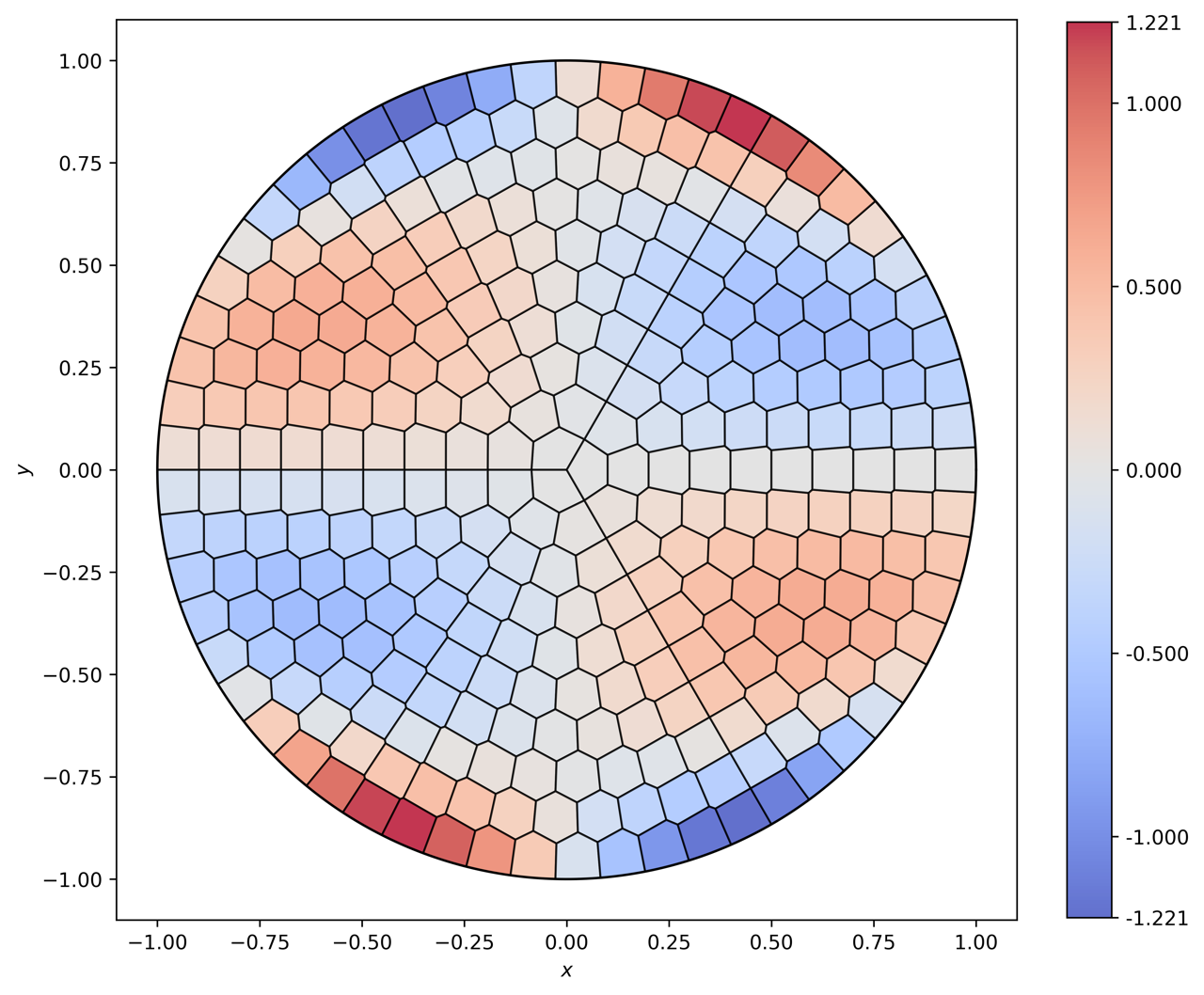}
        % \caption{$n = 400$}
    \end{minipage}\\
% \end{figure}

% \begin{figure}[h!]
    \centering
    \begin{minipage}{0.17\textwidth}
        \centering
        \includegraphics[width=\textwidth]{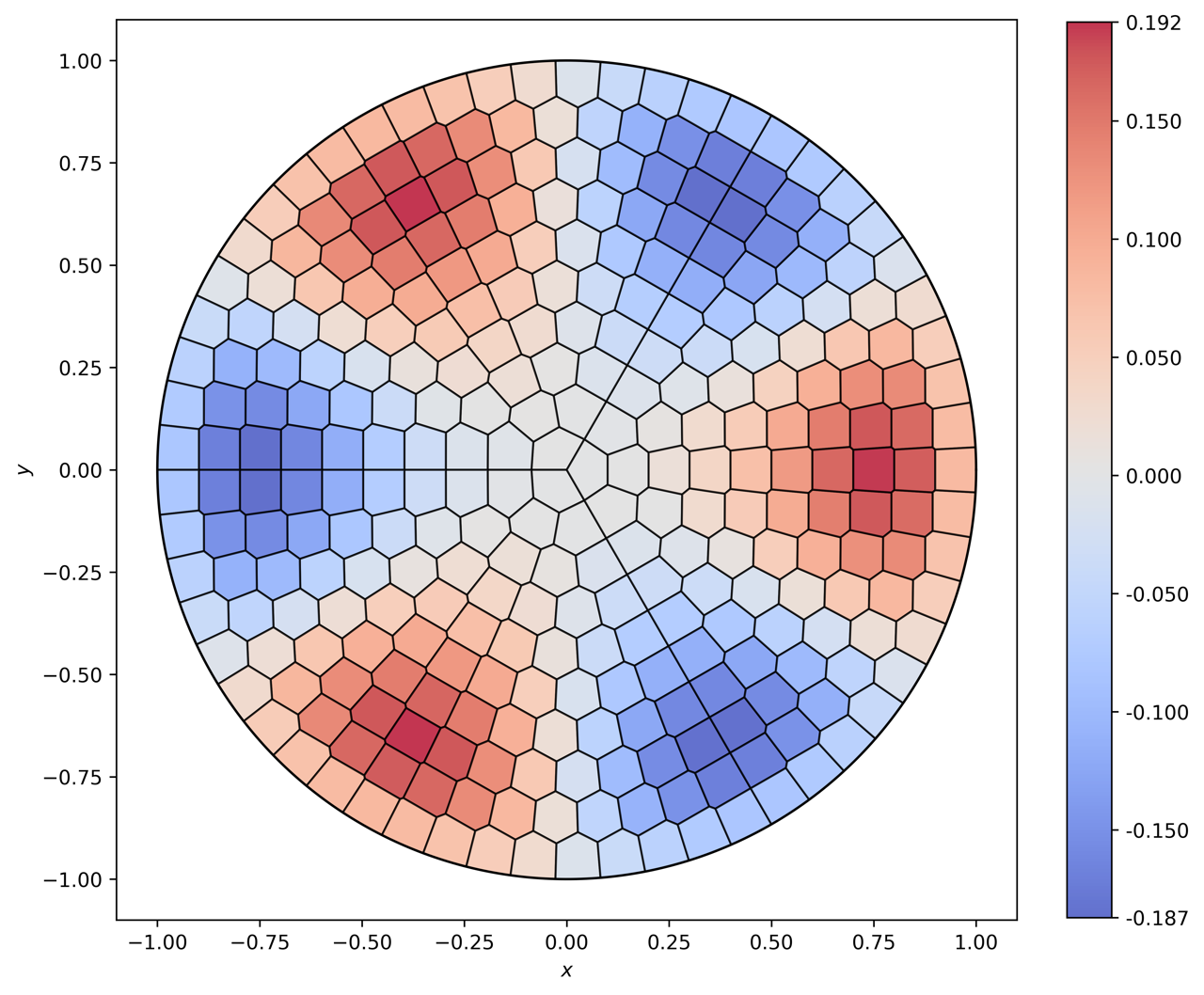}
        % \caption{$n = 400$}
    \end{minipage}
    % \hspace{20pt}
    \begin{minipage}{0.19\textwidth}
        \centering
        \includegraphics[width=\textwidth]{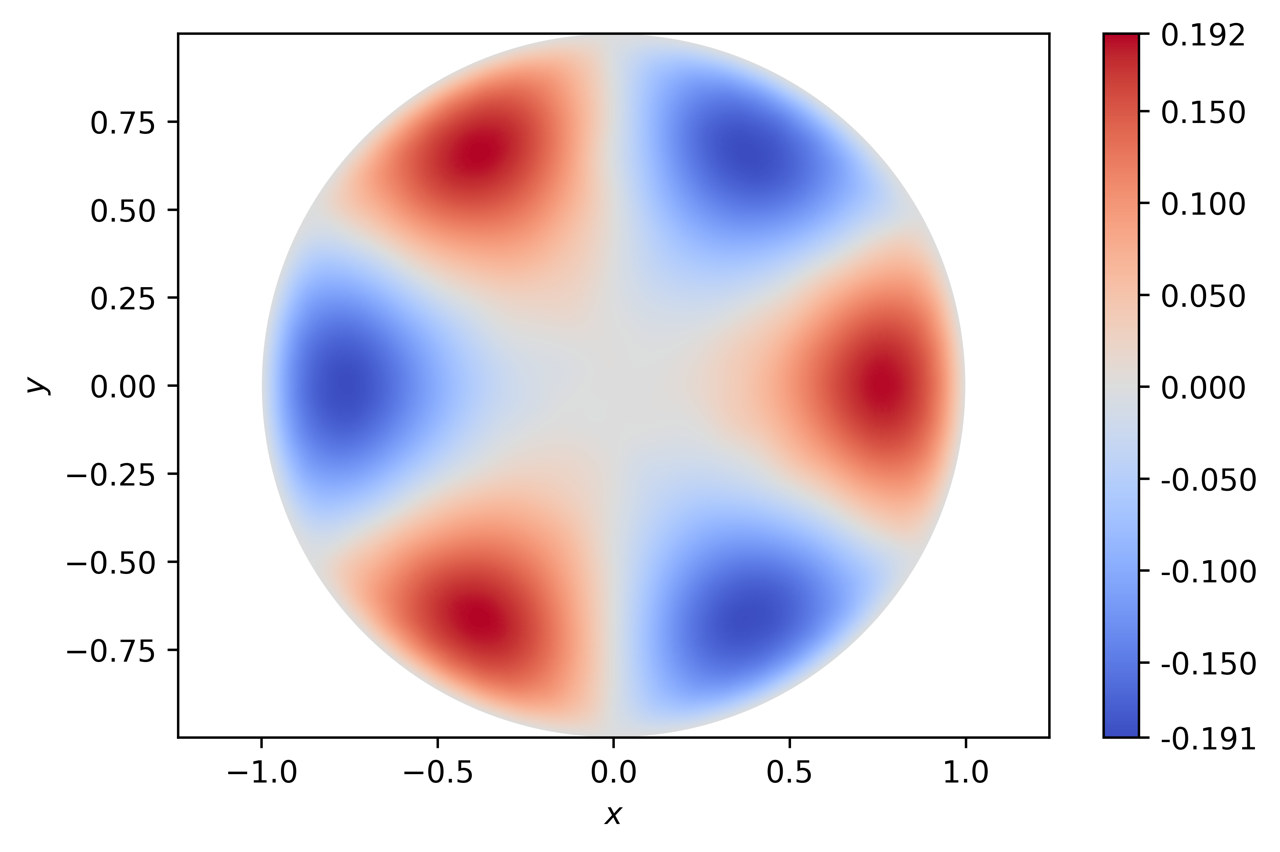}
    \end{minipage}
    \begin{minipage}{0.2\textwidth}
        \centering
        \includegraphics[width=\textwidth]{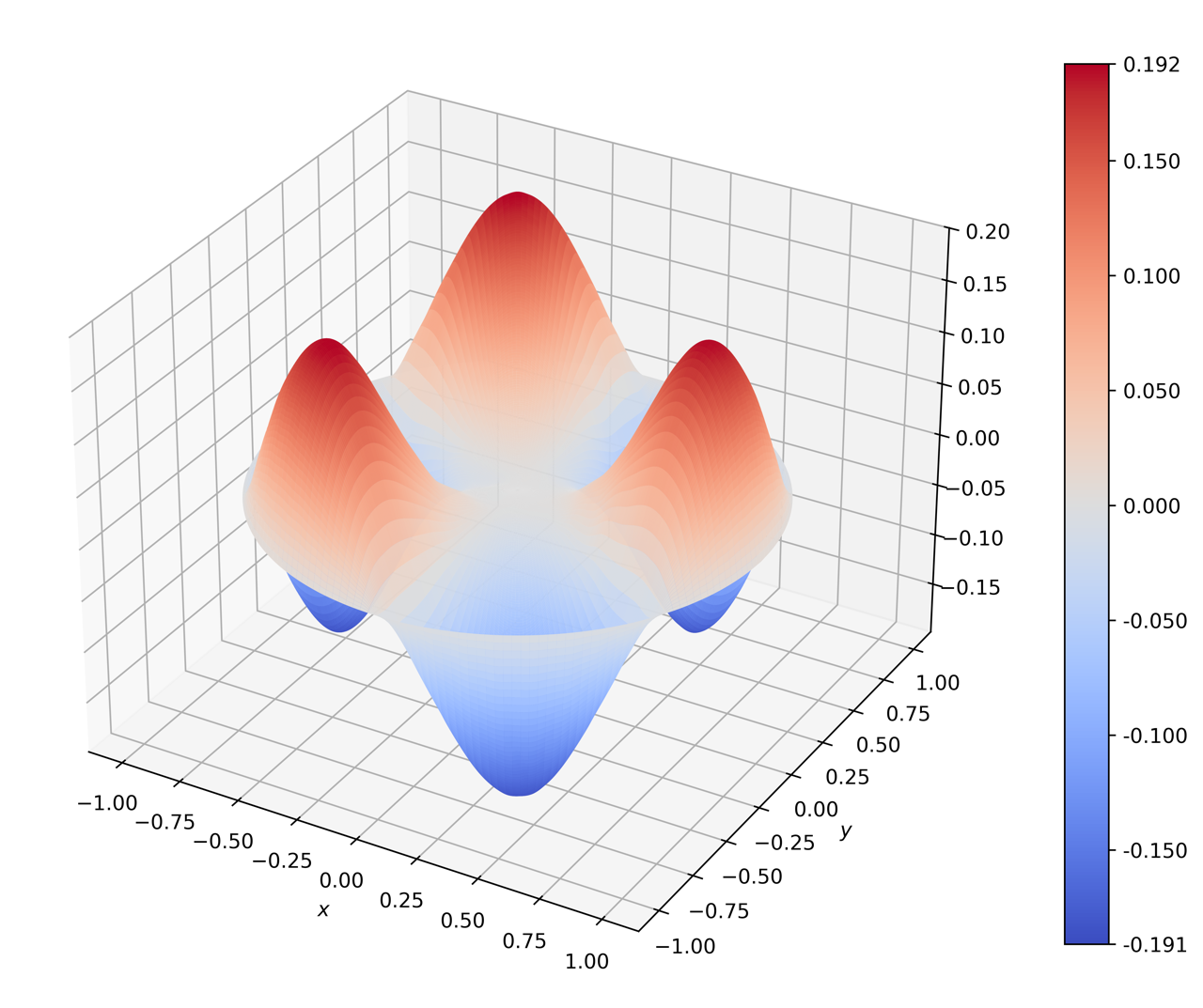}
    \end{minipage}
    \begin{minipage}{0.17\textwidth}
        \centering
        \includegraphics[width=\textwidth]{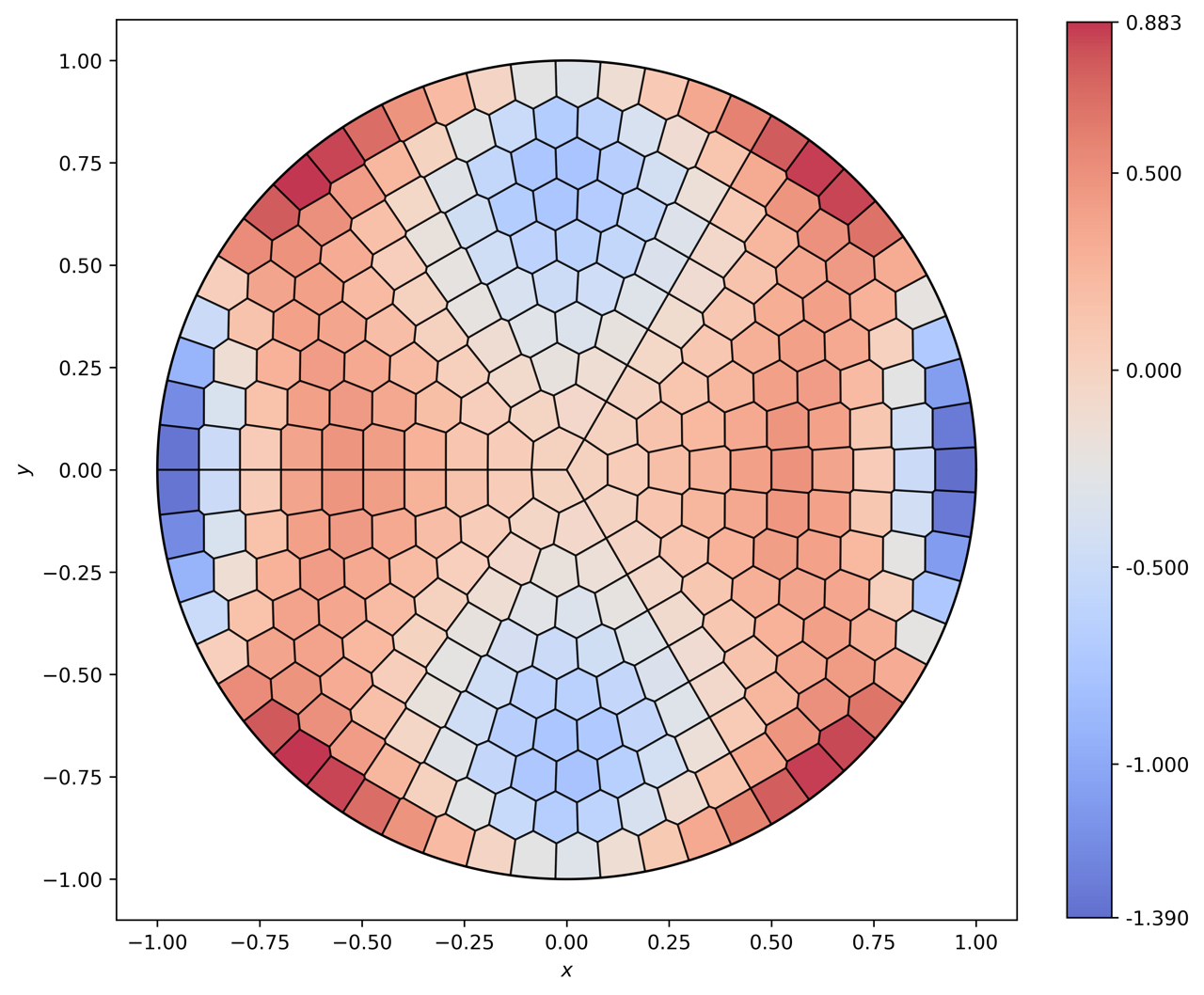}
    \end{minipage}
    \begin{minipage}{0.17\textwidth}
        \centering
        \includegraphics[width=\textwidth]{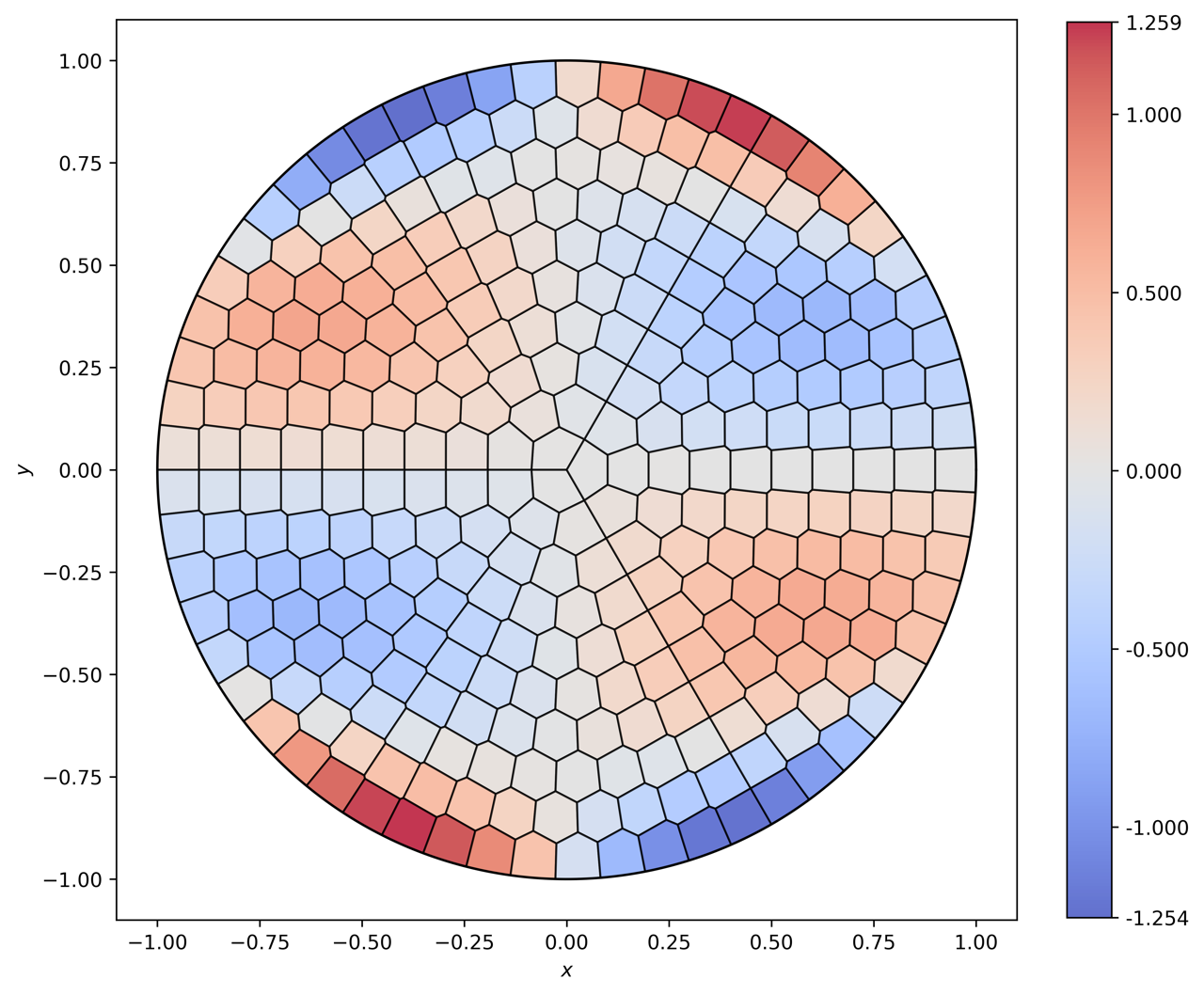}
        % \caption{$n = 400$}
    \end{minipage}
    \caption{(Top) Exact solution $u_{\text{ex}}$ (3 plots) and partial derivatives $\partial_x u_{\text{ex}}$ and $\partial_y u_{\text{ex}}$; (Bottom) Approximate solution $\widehat{u}$ (3 plots) and approximate partial derivatives $\widehat{\partial_x u}$ and $\widehat{\partial_y u}$, obtained with the proposed method.}
\end{figure}

\begin{figure}[H]
    \centering
    \begin{minipage}{0.6\textwidth}
        \centering
         \includegraphics[width=\textwidth]{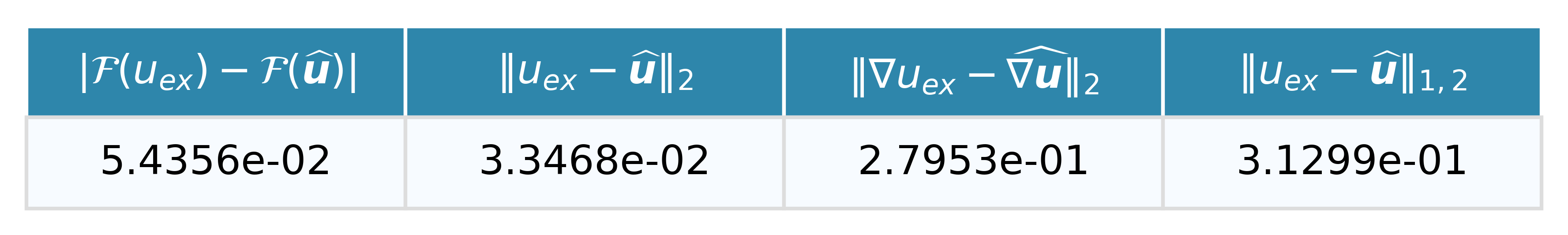}
        \includegraphics[width=\textwidth]{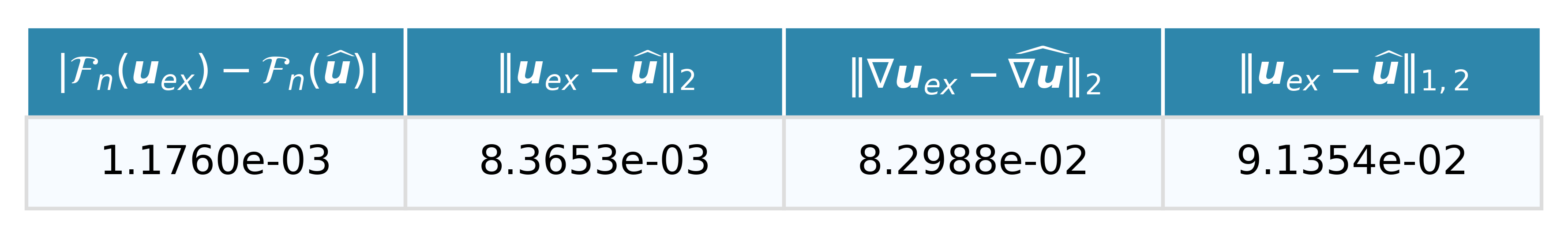}
    \end{minipage}
    \captionof{table}{Errors between exact and approximate solution: functional error and absolute errors in $L^2$ and $W^{1,2}$ norms. Note: the error between the exact and discrete functional of the exact solution is $\left| \mathcal{F}({u}_{\text{ex}}) - \mathcal{F}_n({u}_{\text{ex}}) \right| \approx 5.5\times 10^{-2} $ and $\| {u}_{\text{ex}} - \mathbf{u}_{\text{ex}} \|_2 \approx 3.1 \times 10^{-2}$.}
\end{figure}

\begin{figure}[H]
    \centering
    \begin{minipage}{0.6\textwidth}
        \centering
         \includegraphics[width=\textwidth]{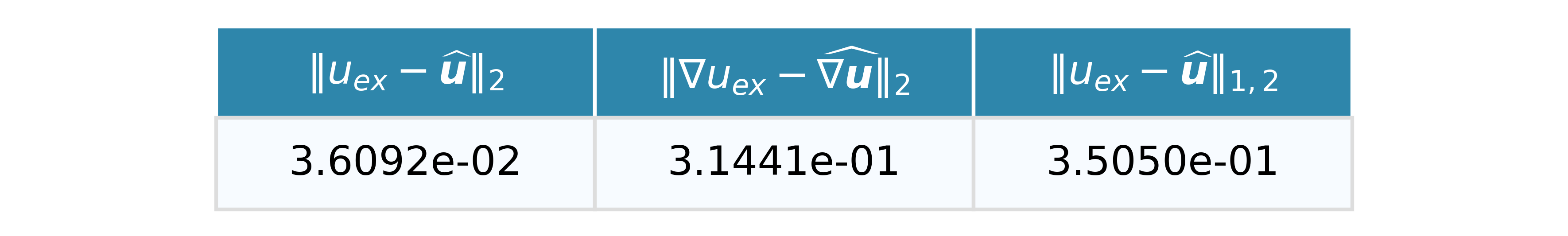}
        \includegraphics[width=\textwidth]{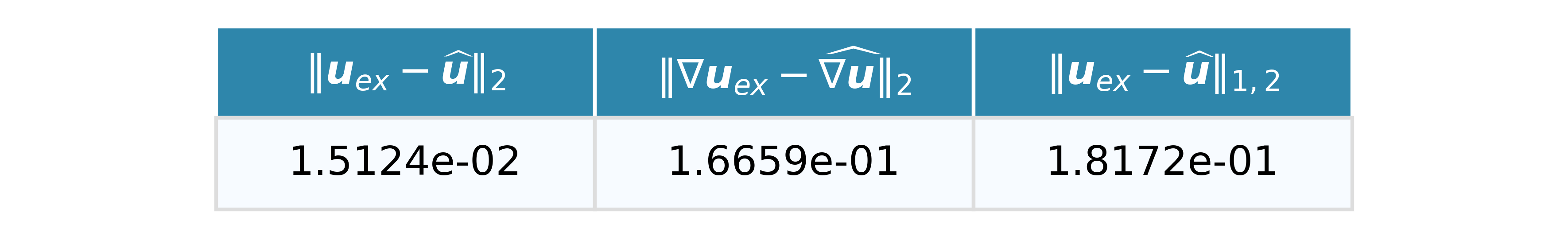}
    \end{minipage}
    \captionof{table}{(Solution operator method) Errors between exact and approximate solution obtain with the solution operator method: absolute errors in $L^2$ and $W^{1,2}$ norms.} \label{table solution operator example elliptic div min}
\end{figure}

%% file: p-laplace-disk.tex
\begin{figure}[H]
    \centering
    \begin{minipage}{0.9\textwidth}
        \centering
        \includegraphics[width=\textwidth]{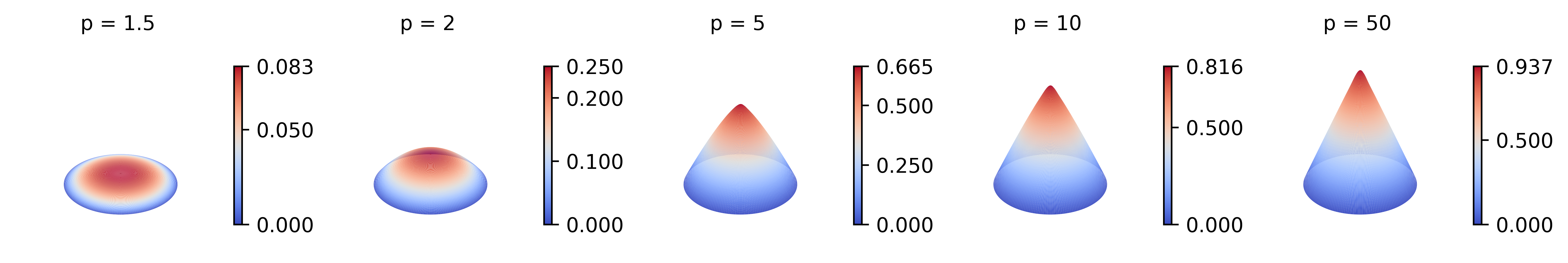}
        % \caption{$n = 400$}
    \end{minipage}
    % \hspace{20pt}
    \begin{minipage}{0.9\textwidth}
        \centering
        \includegraphics[width=\textwidth]{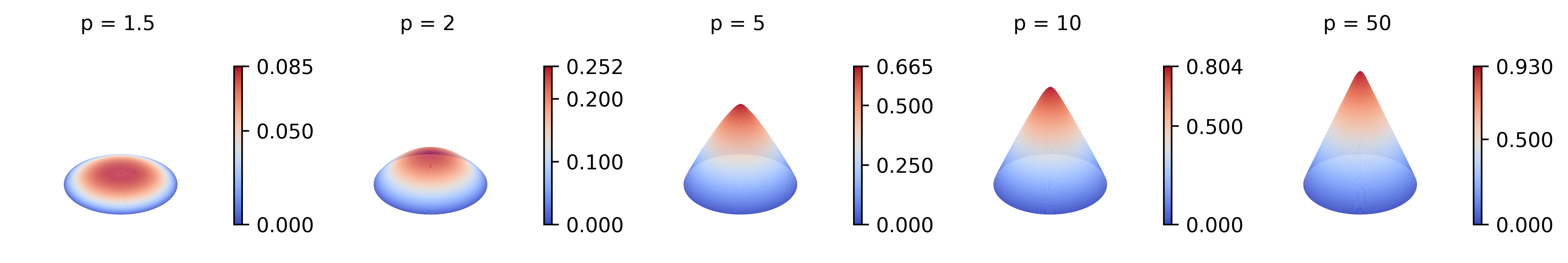}
    \caption{Exact solution $u_{\text{ex}}$ (top) and approximate solution $\widehat{u}$ obtained with the proposed method (bottom).}
    \end{minipage}
\end{figure}
\vspace{-10pt}
\begin{figure}[H]
    \centering
    \begin{minipage}{0.17\textwidth}
        \centering
        \includegraphics[width=\textwidth]{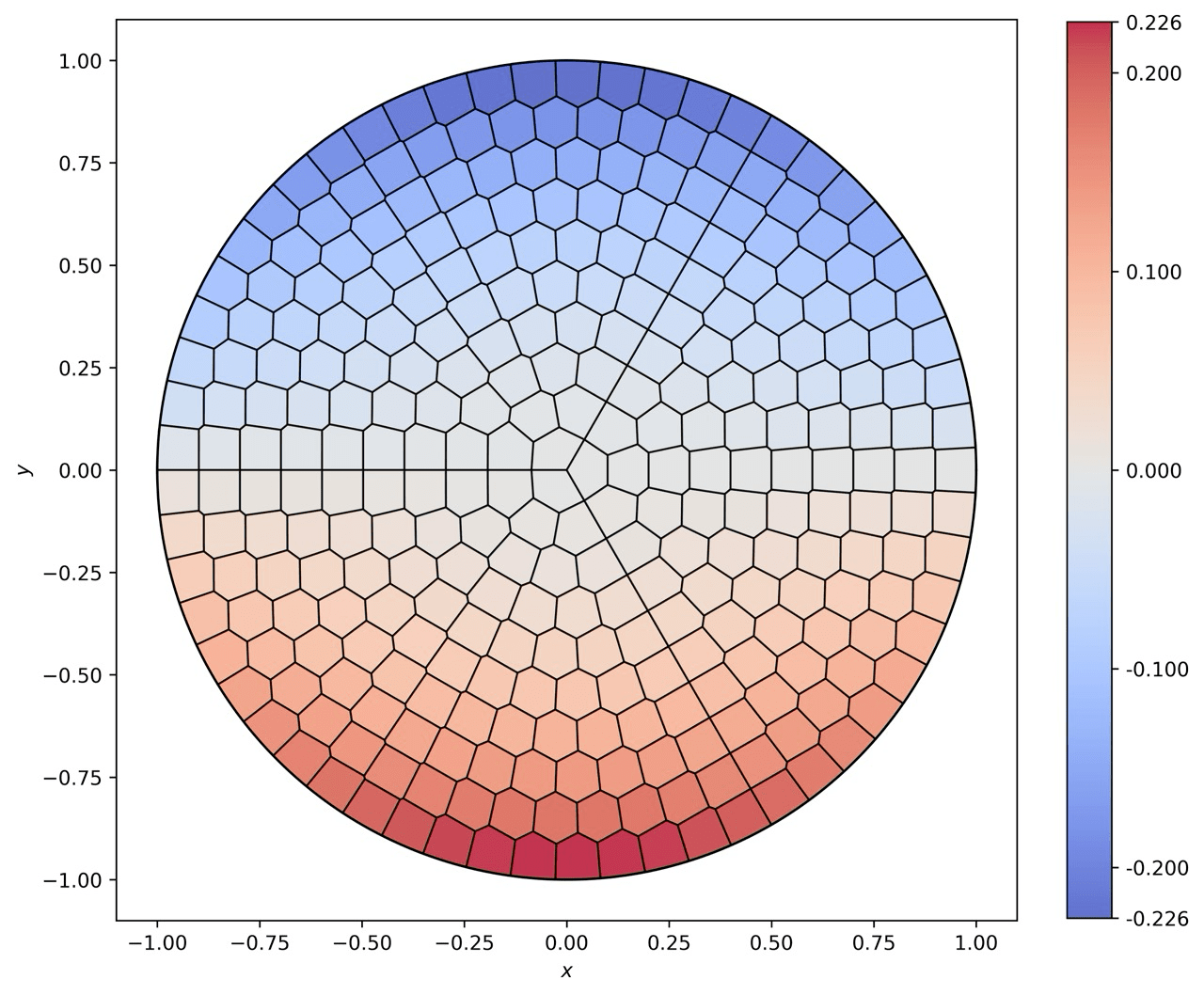}
        % \caption{$n = 400$}
    \end{minipage}
    \begin{minipage}{0.17\textwidth}
        \centering
        \includegraphics[width=\textwidth]{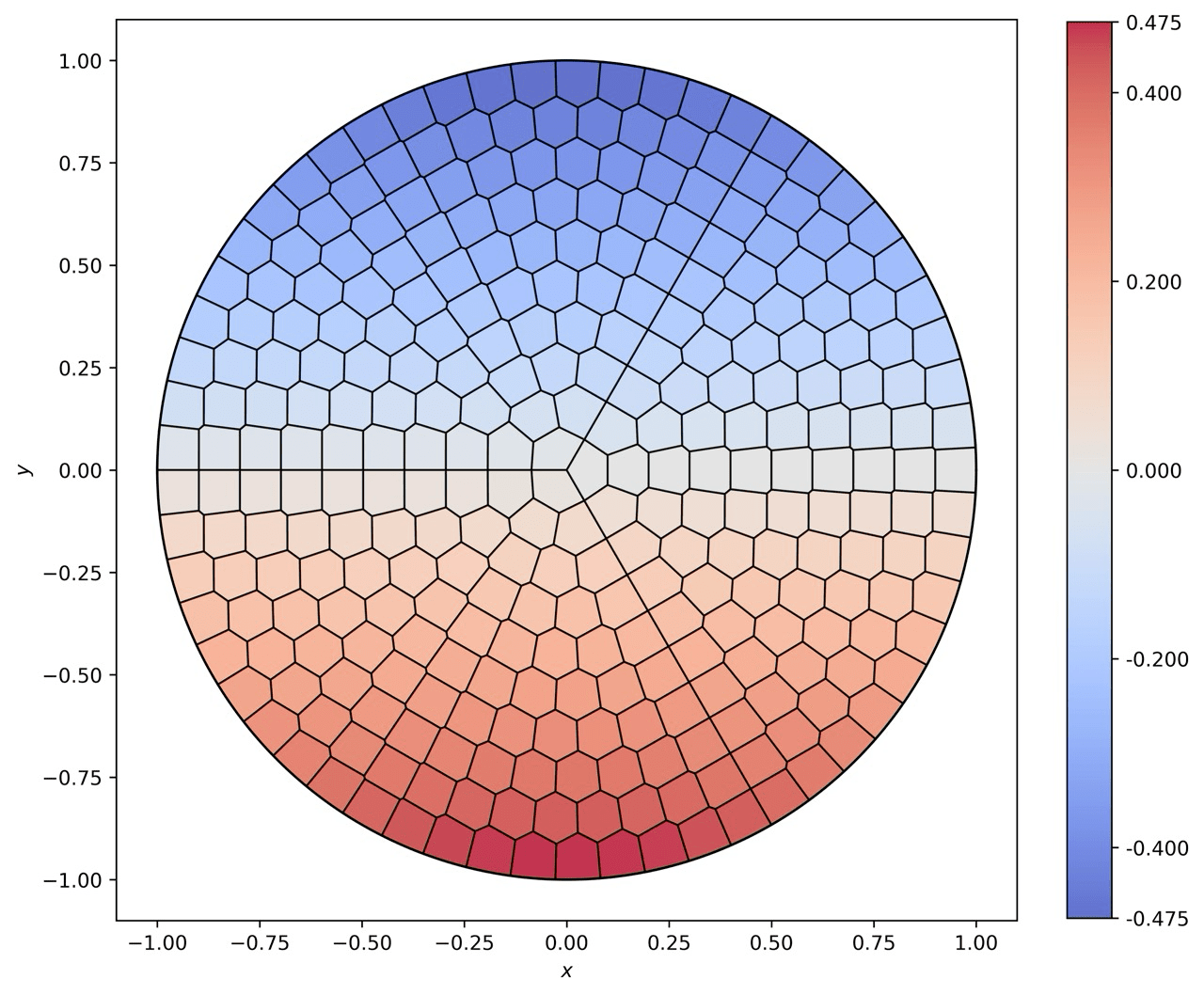}
        % \caption{$n = 400$}
    \end{minipage}
    \begin{minipage}{0.17\textwidth}
        \centering
        \includegraphics[width=\textwidth]{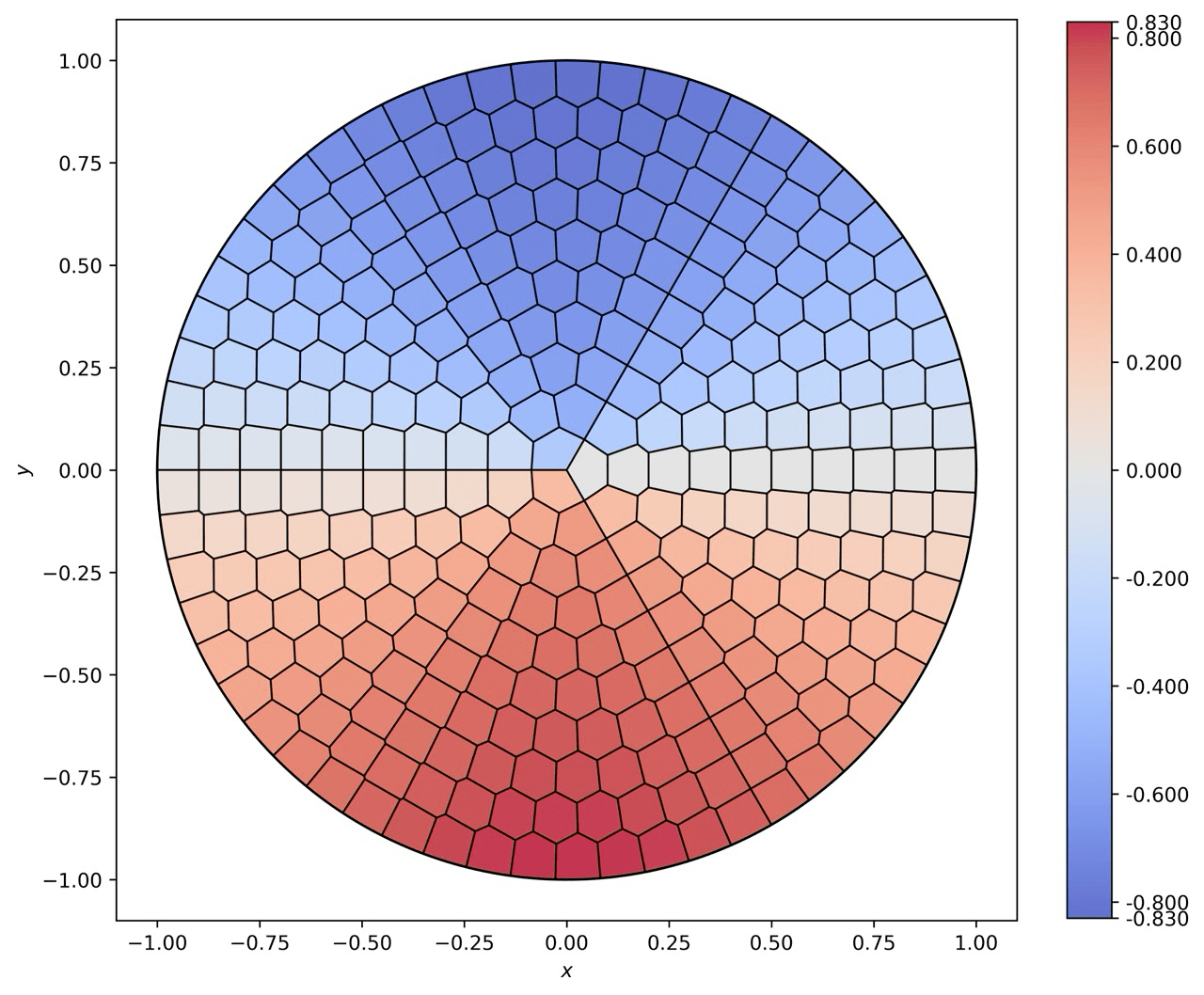}
        % \caption{$n = 400$}
    \end{minipage}
    \begin{minipage}{0.17\textwidth}
        \centering
        \includegraphics[width=\textwidth]{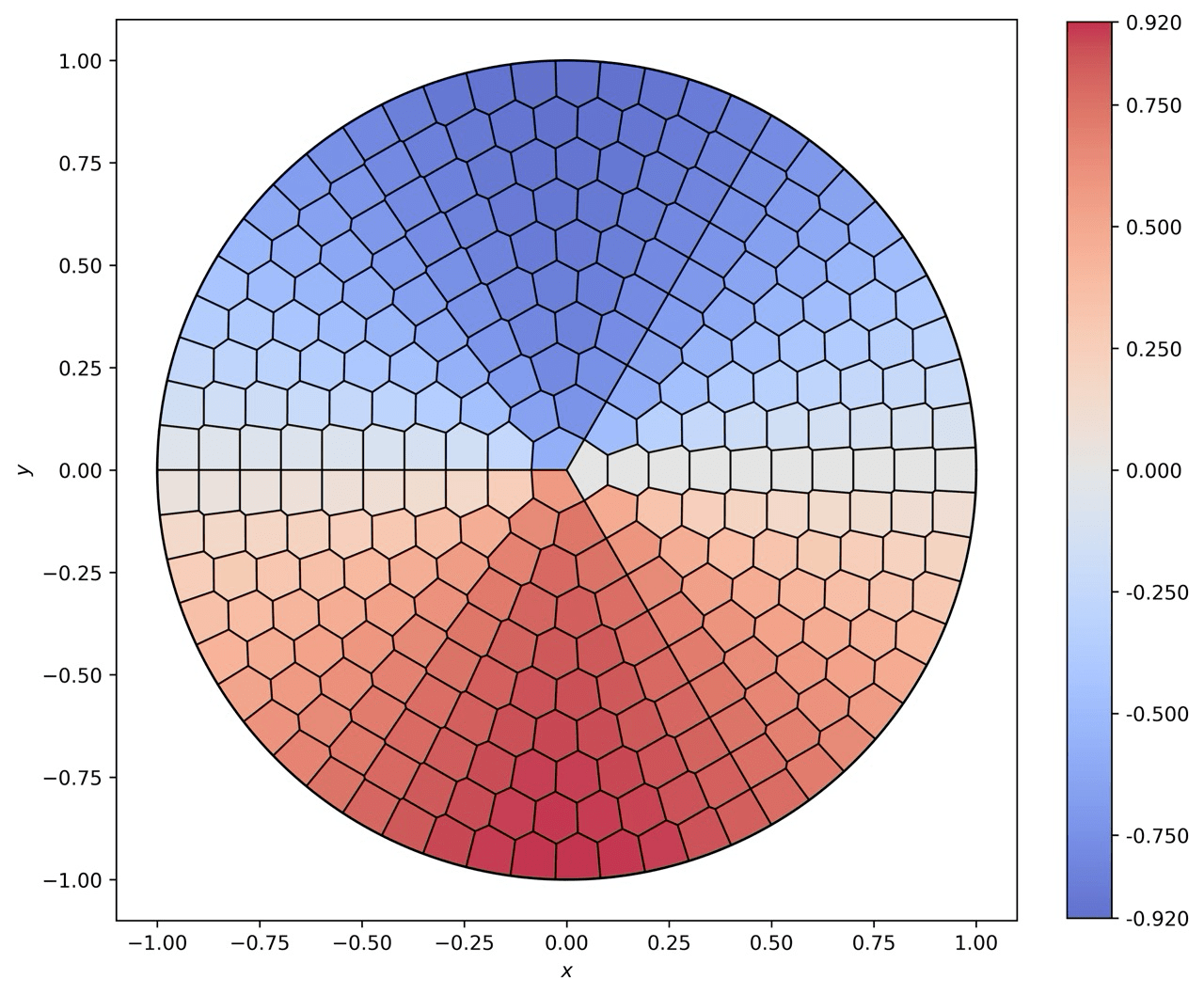}
        % \caption{$n = 400$}
    \end{minipage}
    \begin{minipage}{0.17\textwidth}
        \centering
        \includegraphics[width=\textwidth]{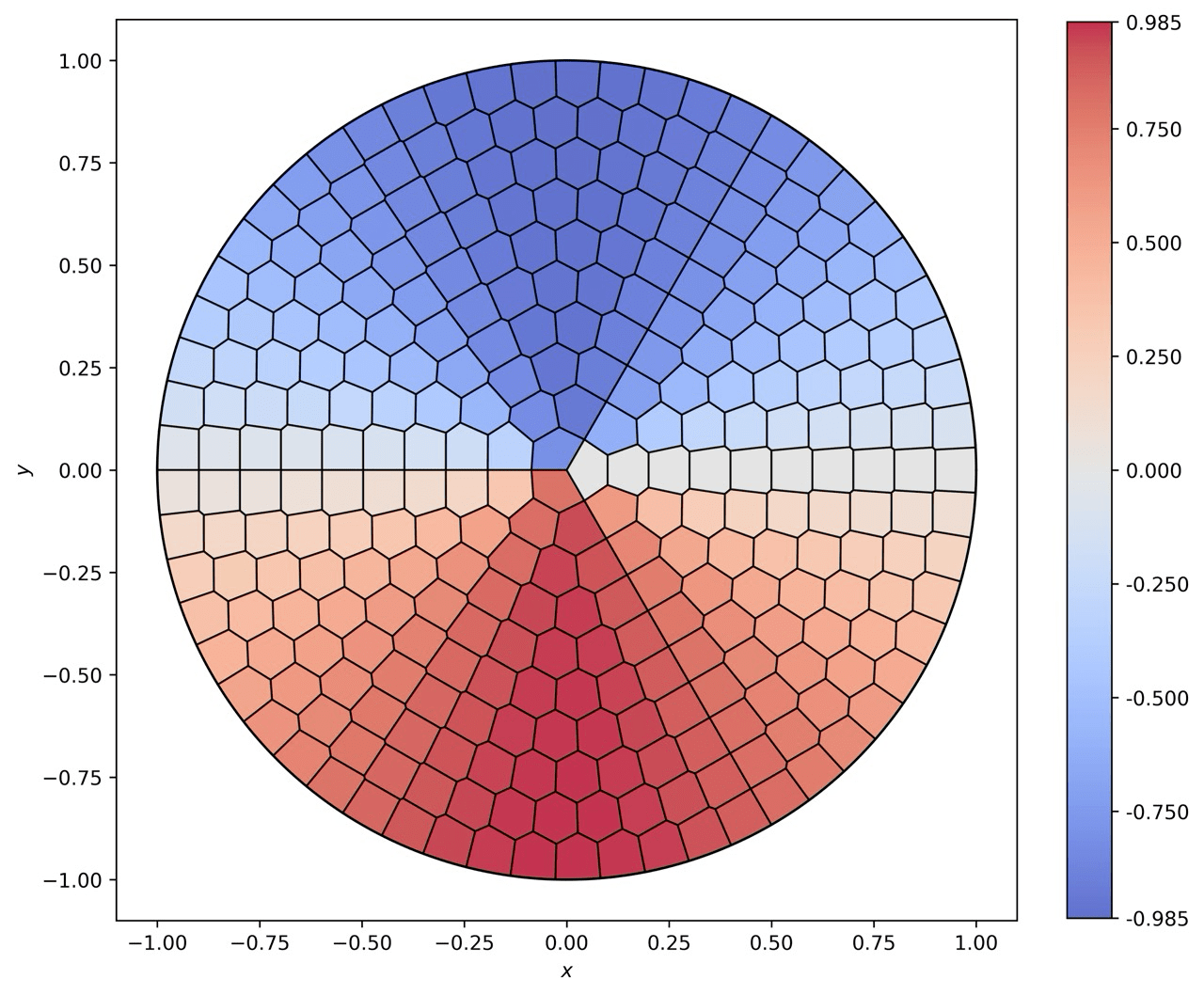}
        % \caption{$n = 400$}
    \end{minipage}\\
% \end{figure}

% \begin{figure}[h!]
    \centering
    \begin{minipage}{0.17\textwidth}
        \centering
        \includegraphics[width=\textwidth]{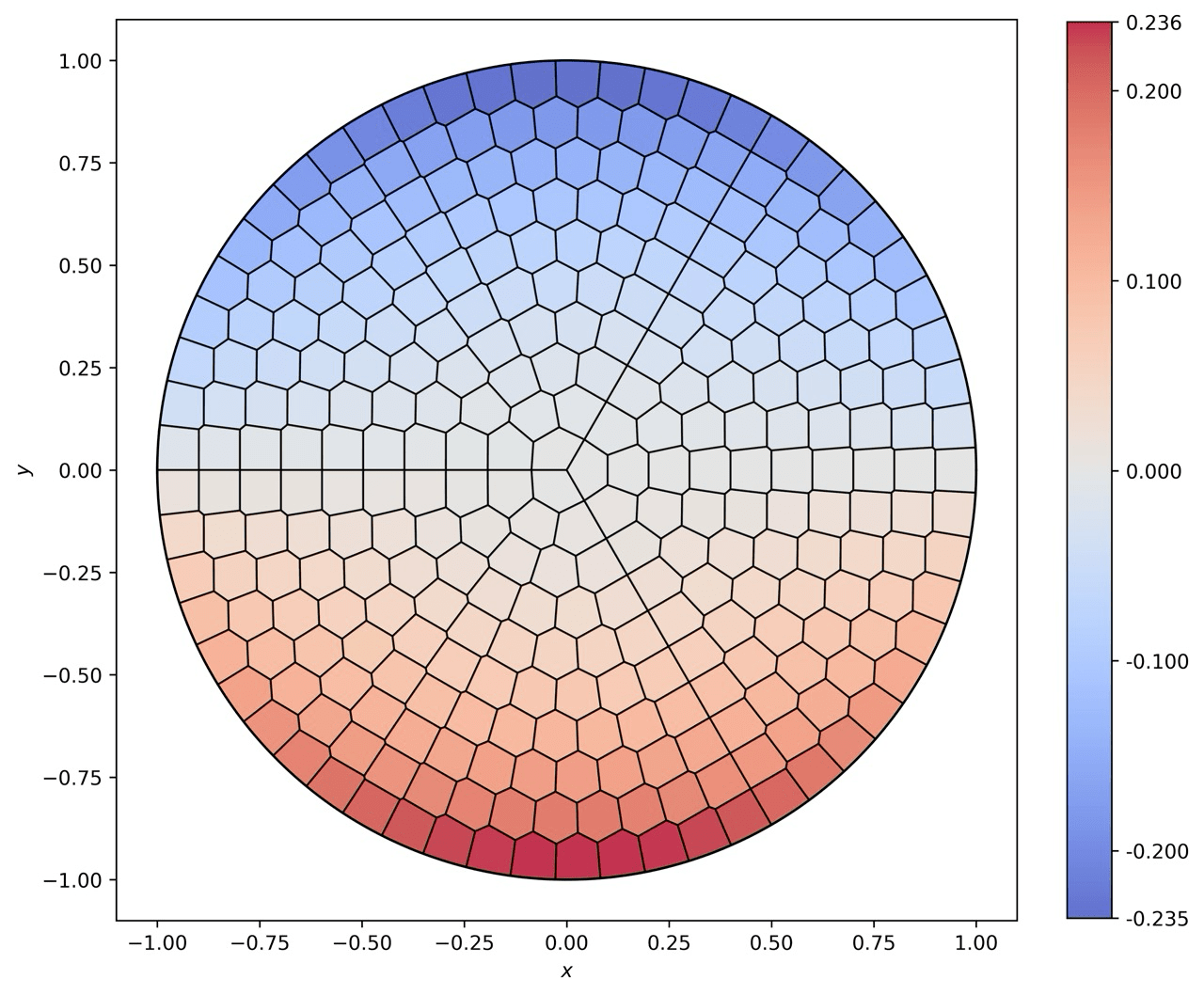}
        % \caption{$n = 400$}
    \end{minipage}
    \begin{minipage}{0.17\textwidth}
        \centering
        \includegraphics[width=\textwidth]{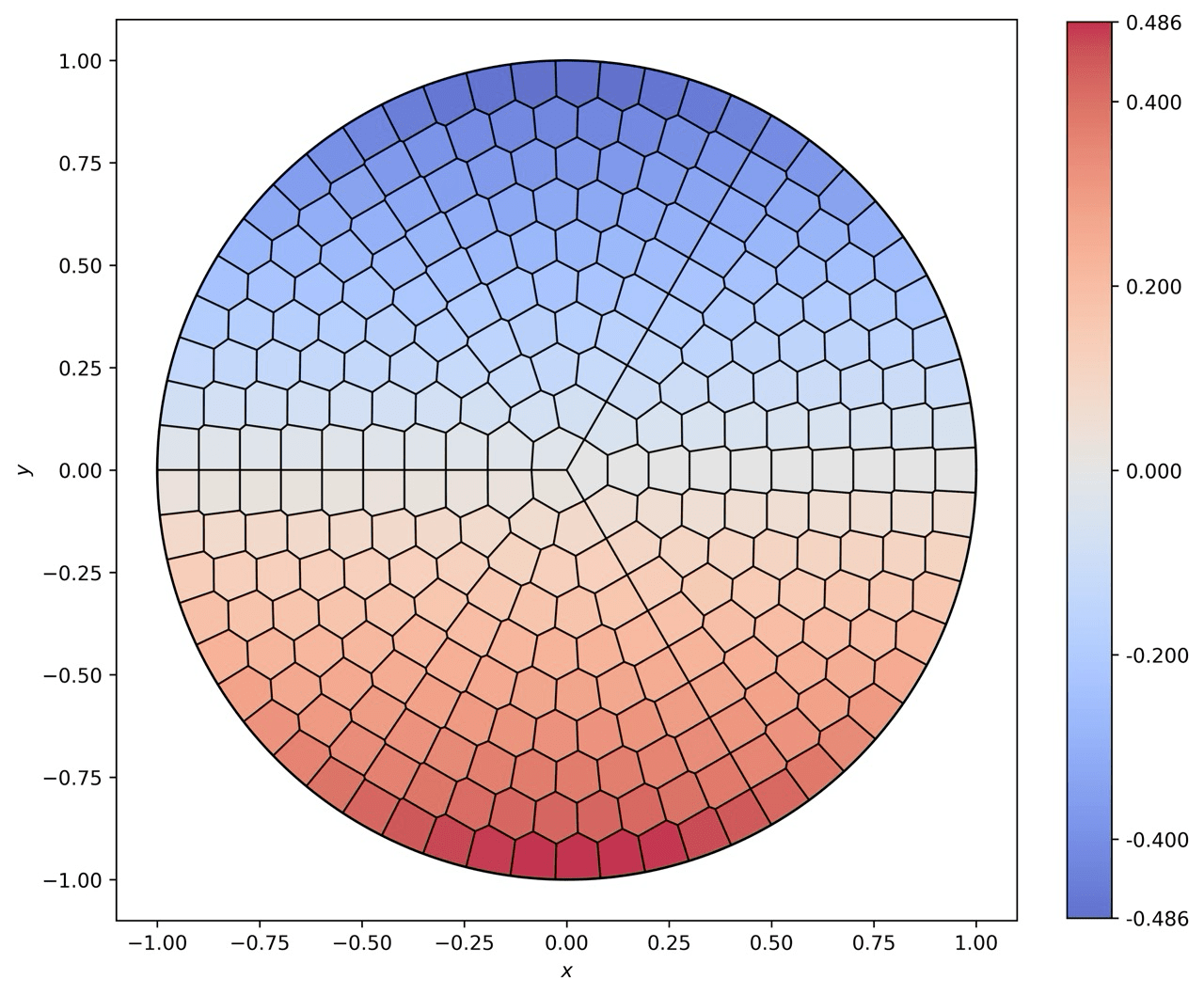}
        % \caption{$n = 400$}
    \end{minipage}
    \begin{minipage}{0.17\textwidth}
        \centering
        \includegraphics[width=\textwidth]{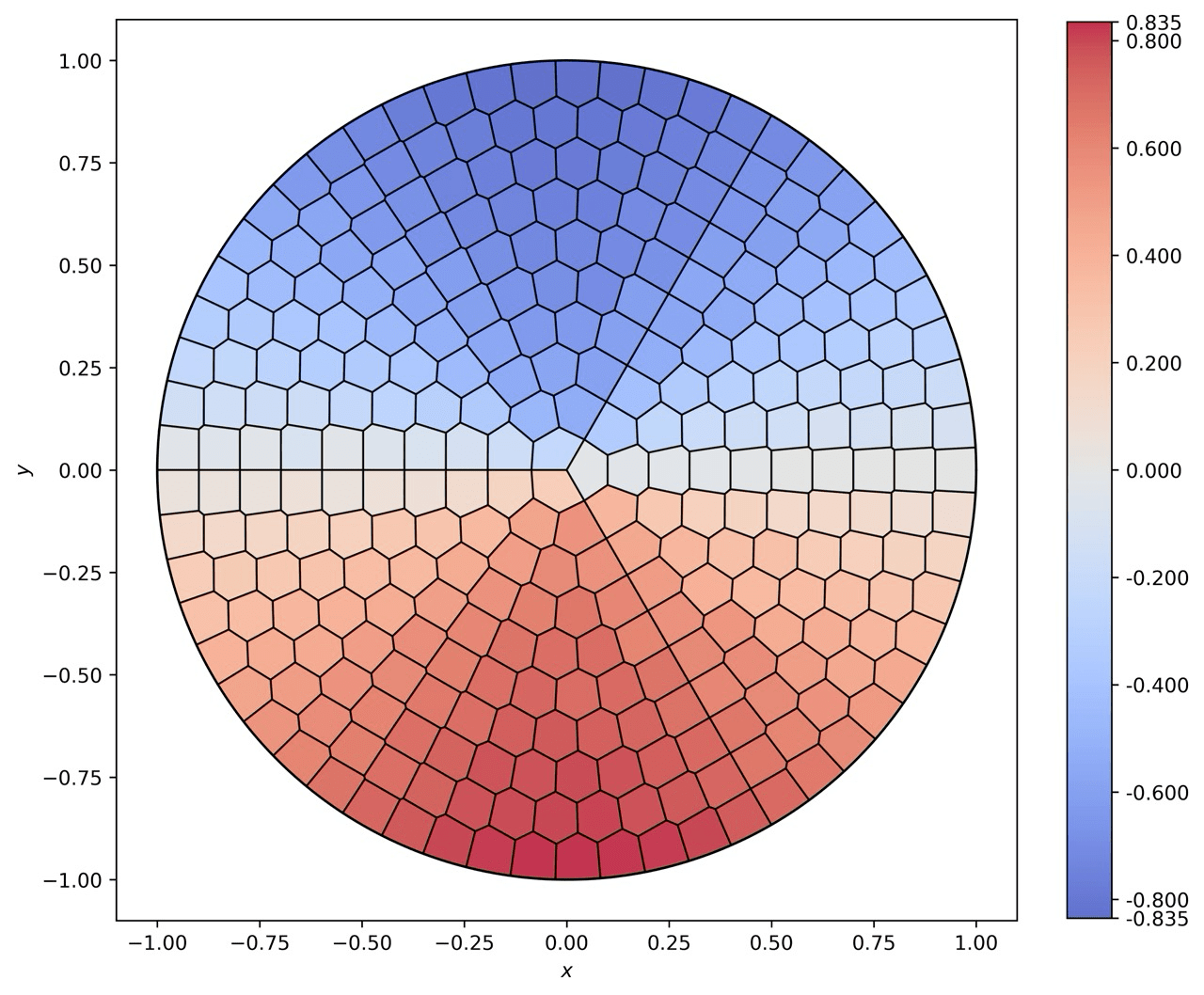}
        % \caption{$n = 400$}
    \end{minipage}
    \begin{minipage}{0.17\textwidth}
        \centering
        \includegraphics[width=\textwidth]{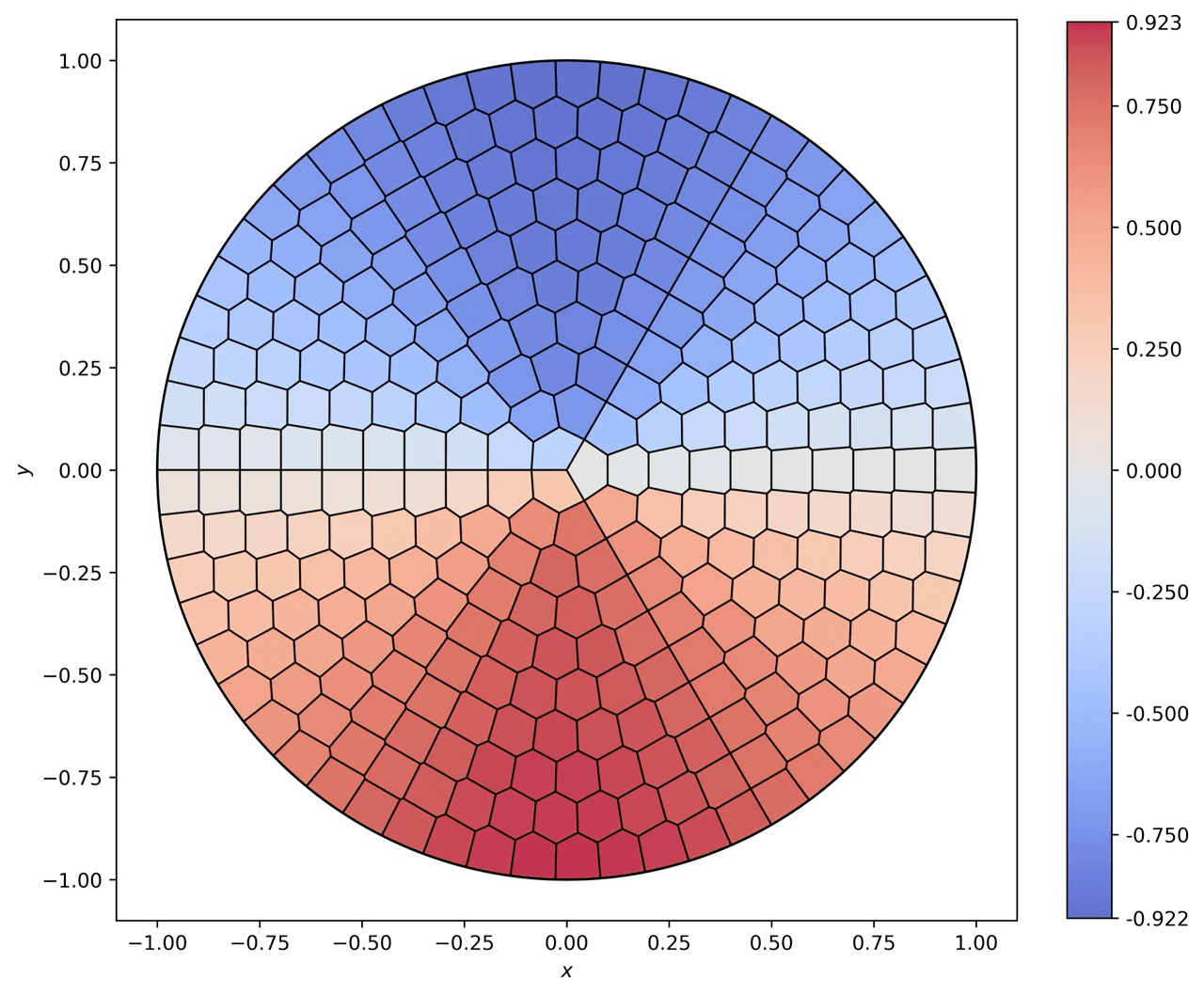}
        % \caption{$n = 400$}
    \end{minipage}
    \begin{minipage}{0.17\textwidth}
        \centering
        \includegraphics[width=\textwidth]{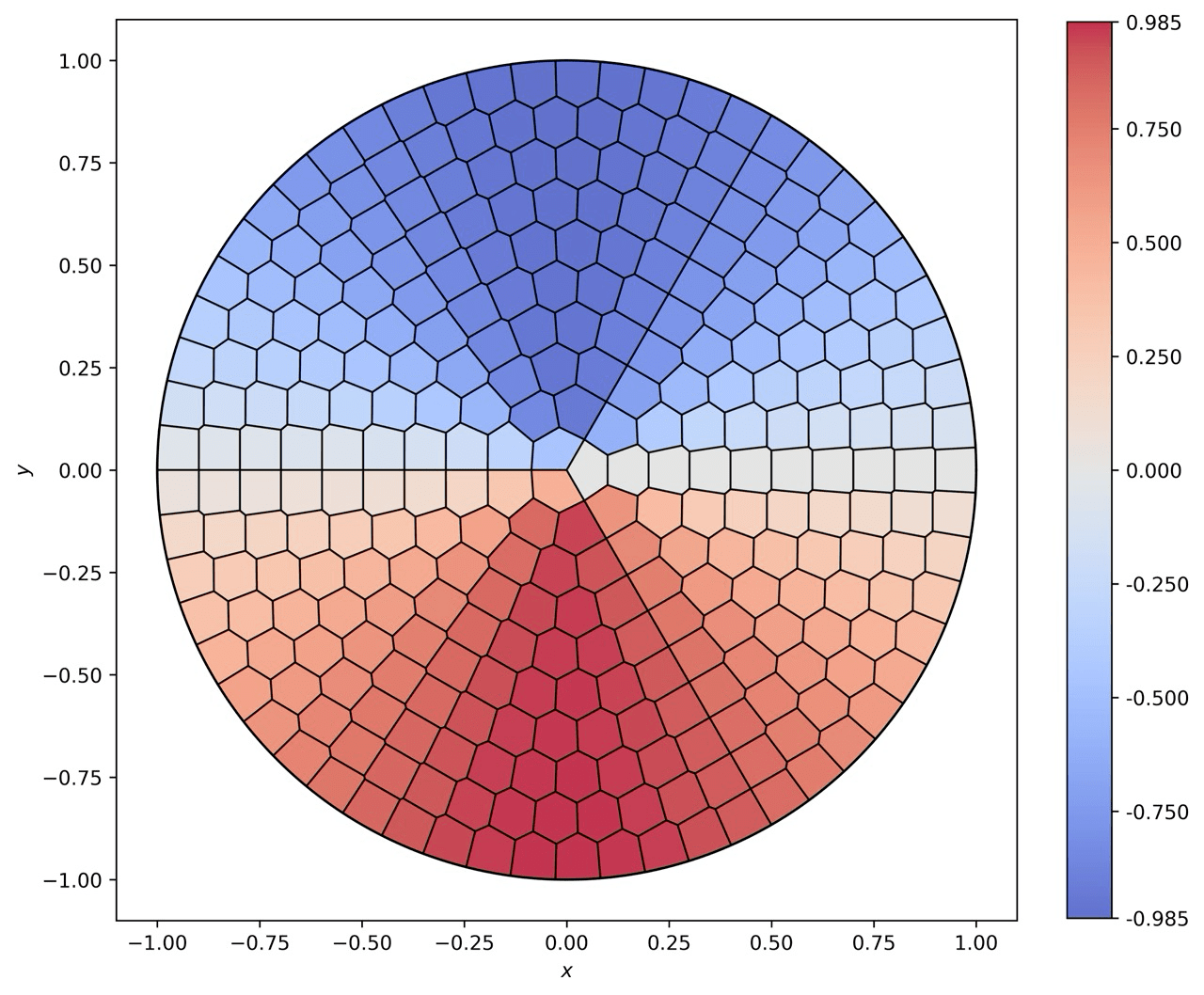}
        % \caption{$n = 400$}
    \end{minipage}
           \caption{Exact partial derivative $\partial_y u_{\text{ex}}$ of the solution (top) and its approximation $\widehat{\partial_y u}$ obtained with the proposed method (bottom), for $p-$Laplace parameter $p \in \{1.5,2,5,10,50 \}$ (from left to right).}
\end{figure}

\begin{figure}[H]
    \centering
    \begin{minipage}{0.7\textwidth}
        \centering
        \includegraphics[width=\textwidth]{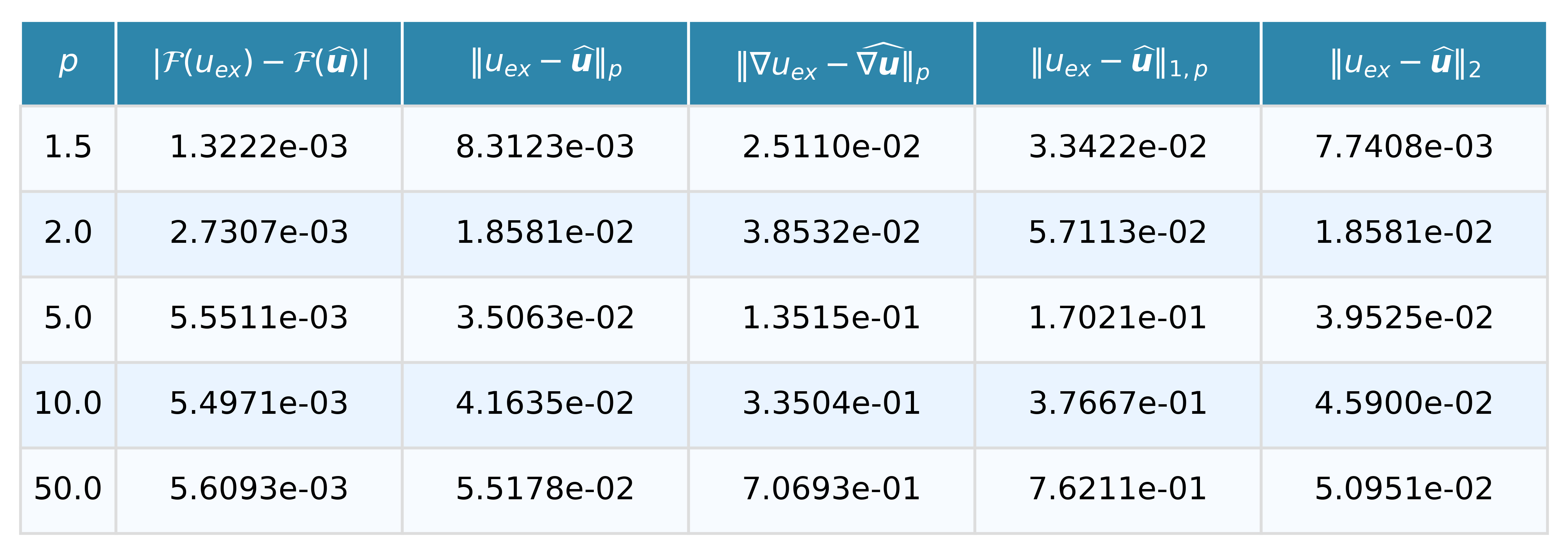}
        \includegraphics[width=\textwidth]{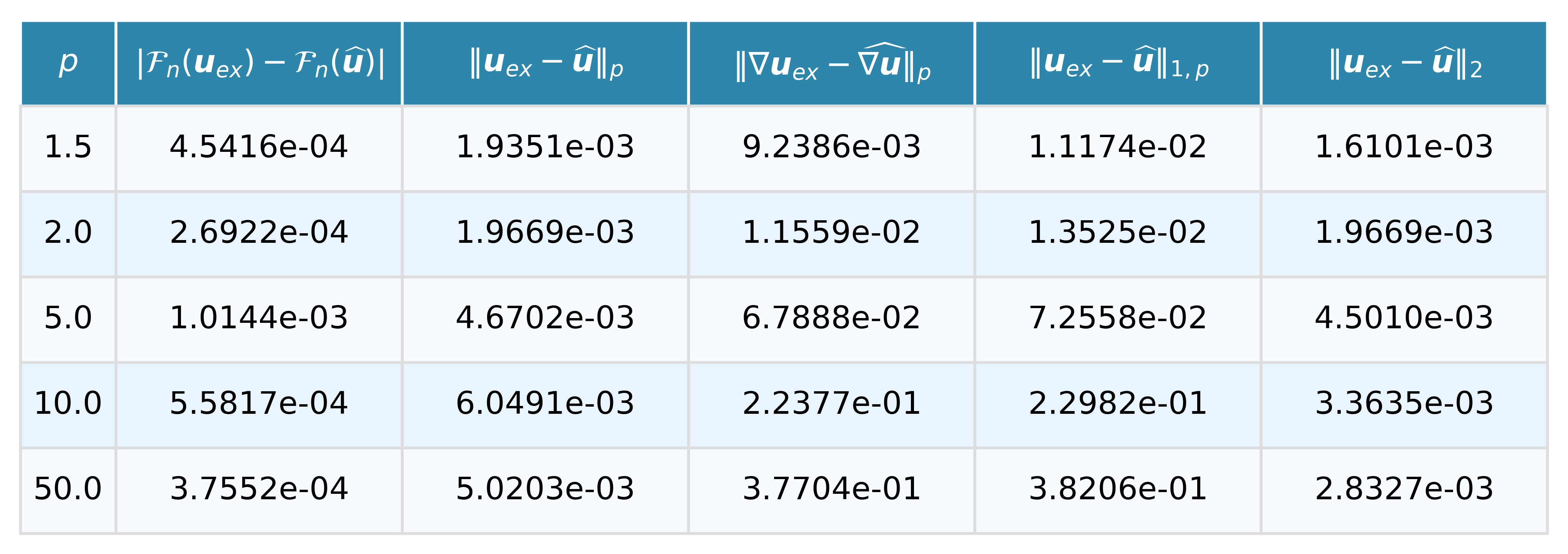}
    \end{minipage}
    \captionof{table}{Errors between exact and approximate solution: functional error and absolute errors in $L^p$, $W^{1,p}$ and $L^2$ norms.}
\end{figure}
\vspace{-20pt}
\begin{figure}[H]
    \centering
    \begin{minipage}{0.7\textwidth}
        \centering
        \includegraphics[width=\textwidth]{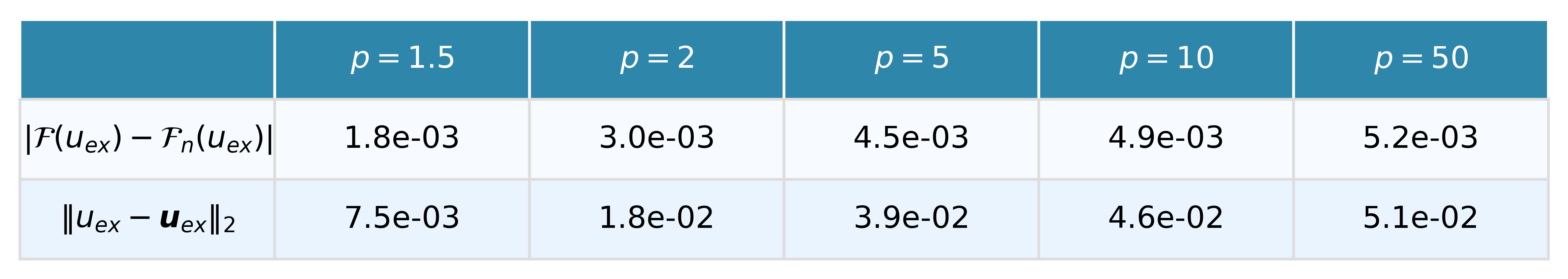}
    \end{minipage}
    \captionof{table}{Note: The error between the exact and discrete functional of the exact solution and similarly in $L^2-$norm.}
\end{figure}

%% file: p-laplace-square.tex
\begin{figure}[H]
    \centering
    \begin{minipage}{0.9\textwidth}
        \centering
        \includegraphics[width=\textwidth]{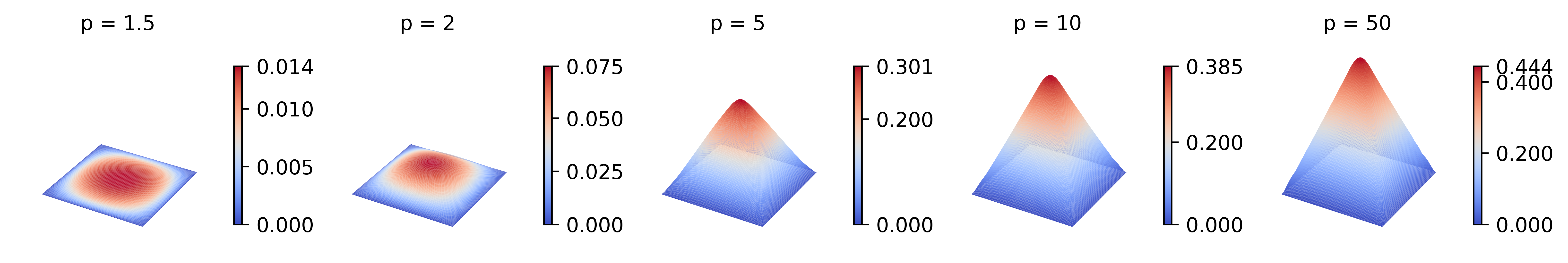}
        \caption{Approximate solution $\widehat{u}$ obtained with the proposed method.}
    \end{minipage}
    % \hspace{20pt}
\end{figure}

\begin{figure}[H]
    \centering
    \begin{minipage}{0.17\textwidth}
        \centering
        \includegraphics[width=\textwidth]{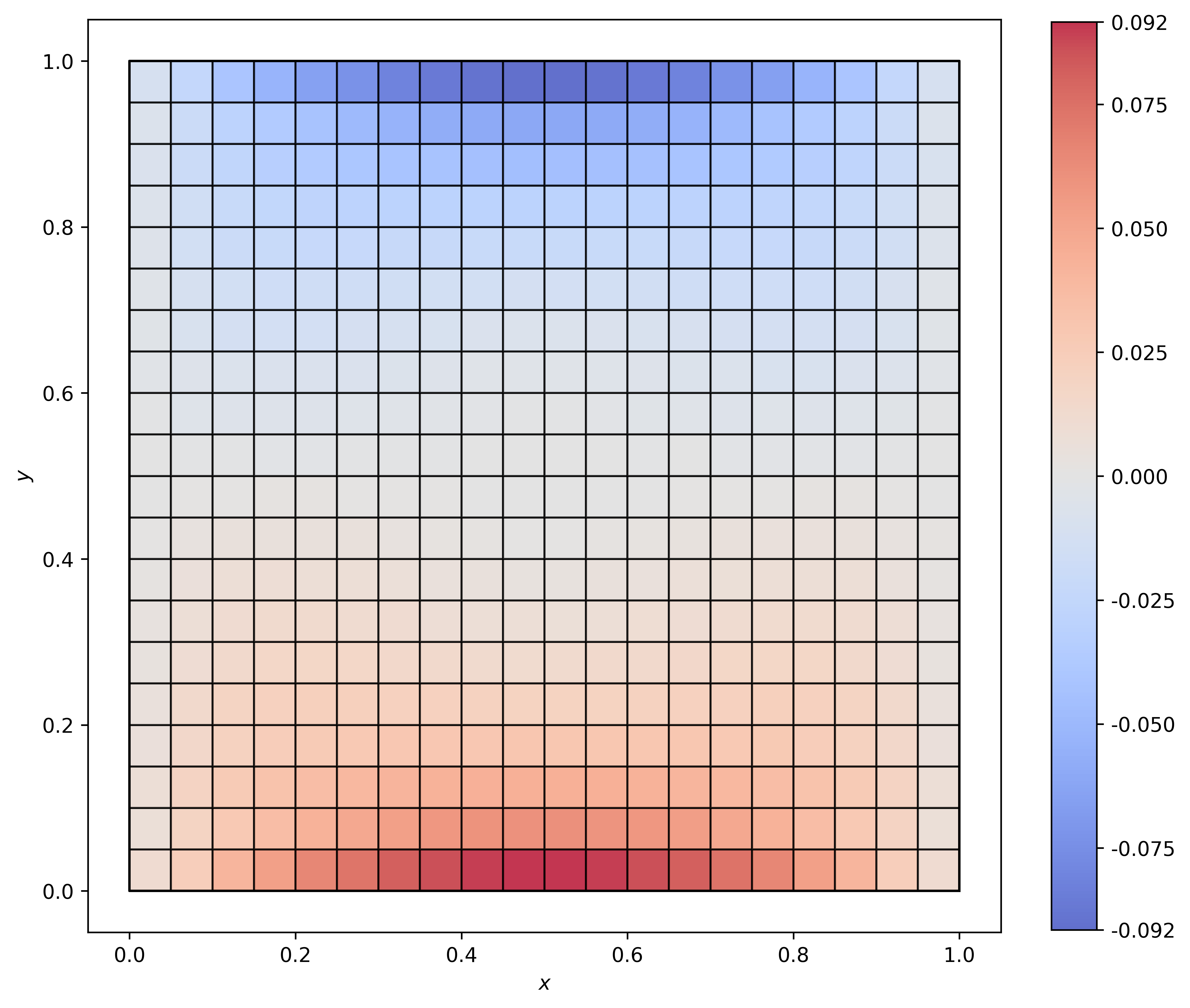}
        % \caption{$n = 400$}
    \end{minipage}
    \begin{minipage}{0.17\textwidth}
        \centering
        \includegraphics[width=\textwidth]{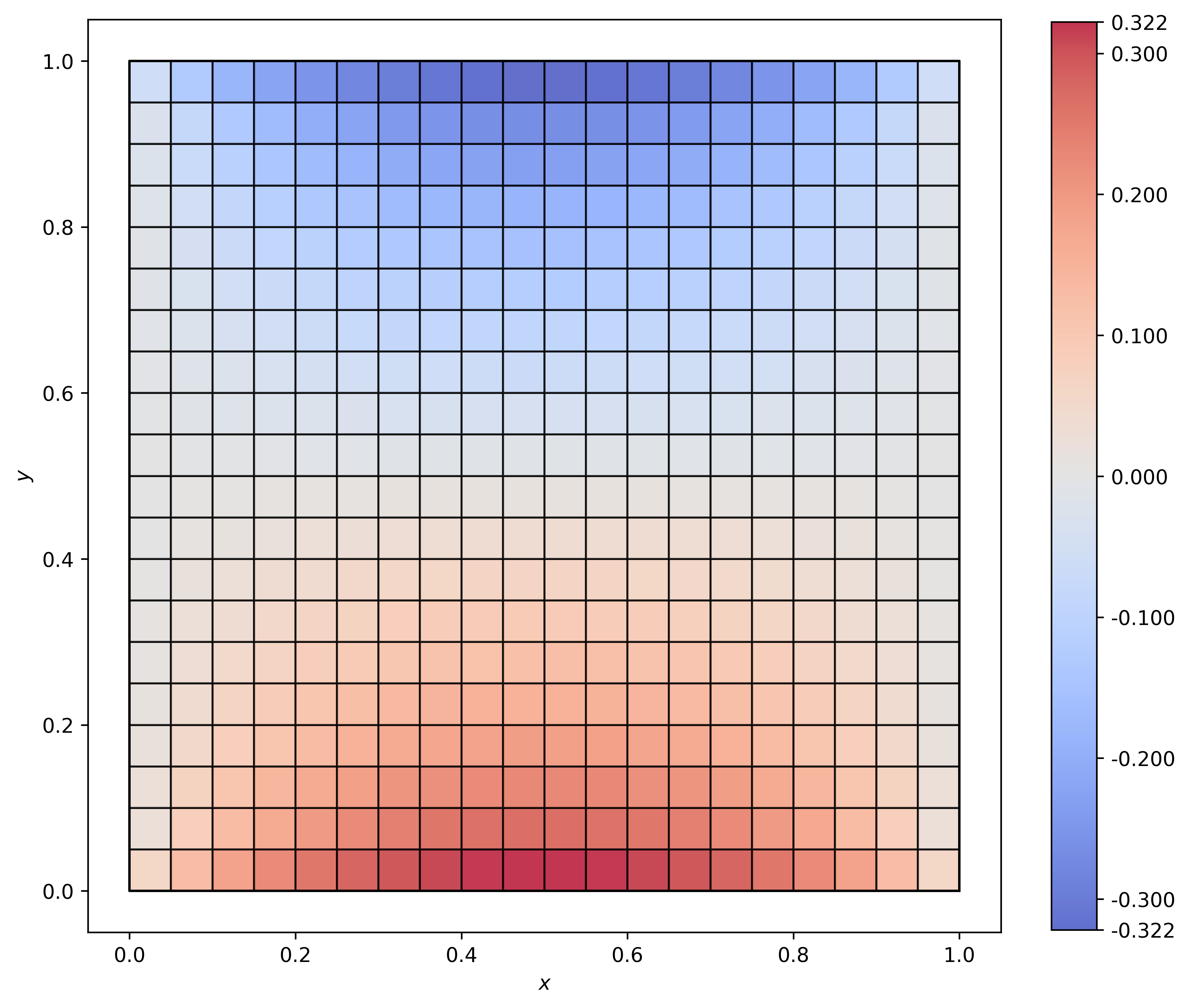}
        % \caption{$n = 400$}
    \end{minipage}
    \begin{minipage}{0.17\textwidth}
        \centering
        \includegraphics[width=\textwidth]{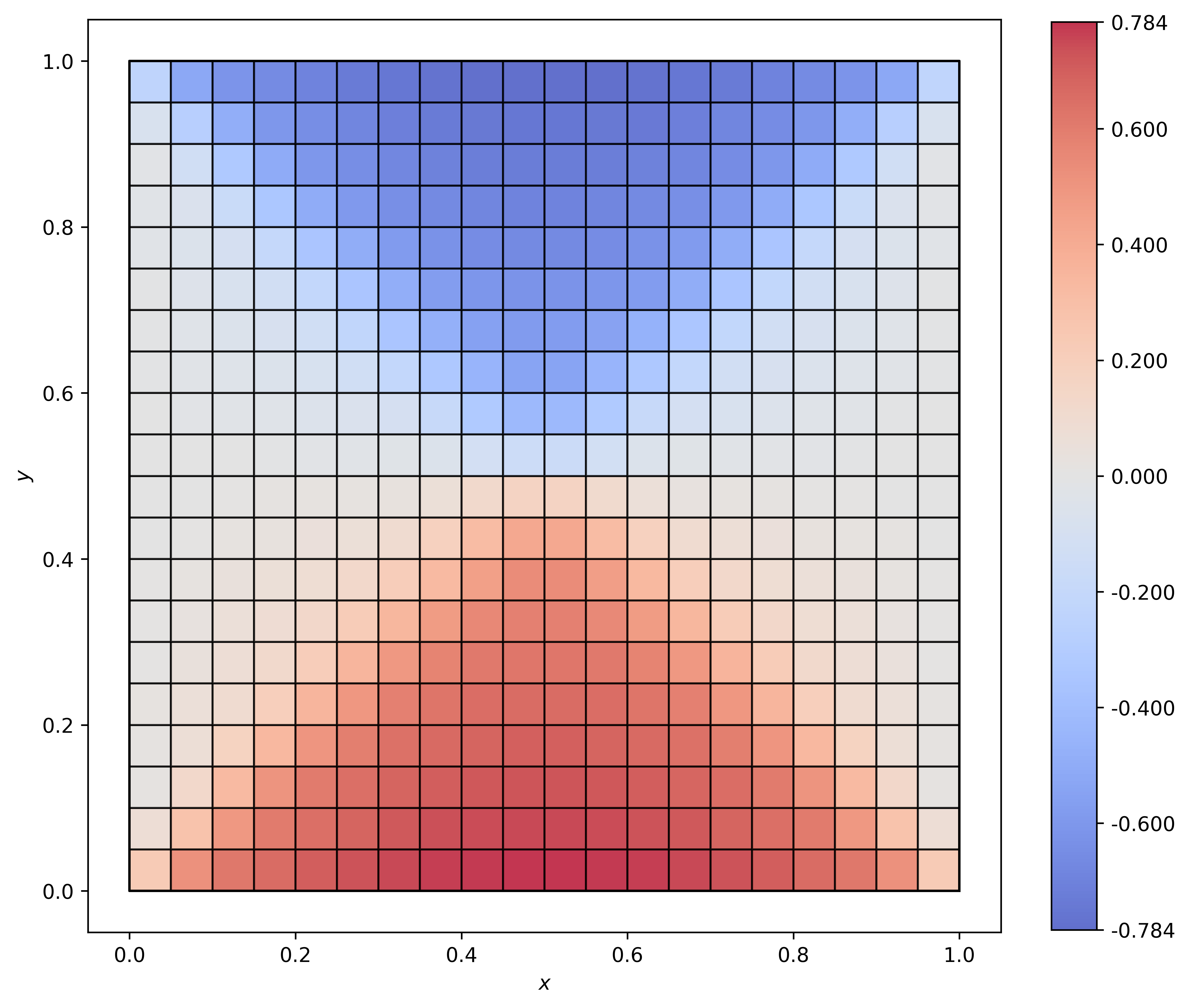}
        % \caption{$n = 400$}
    \end{minipage}
    \begin{minipage}{0.17\textwidth}
        \centering
        \includegraphics[width=\textwidth]{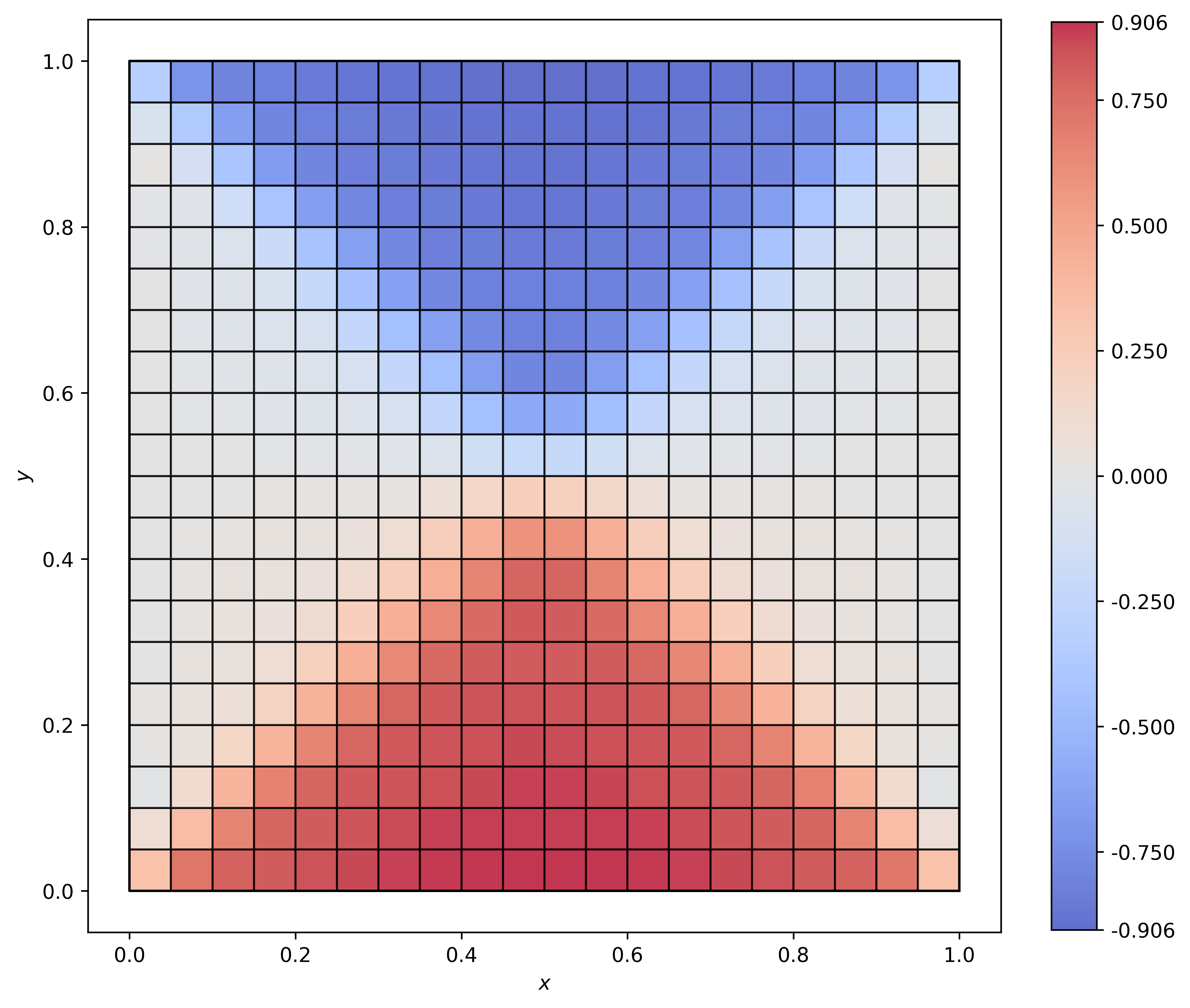}
        % \caption{$n = 400$}
    \end{minipage}
    \begin{minipage}{0.17\textwidth}
        \centering
        \includegraphics[width=\textwidth]{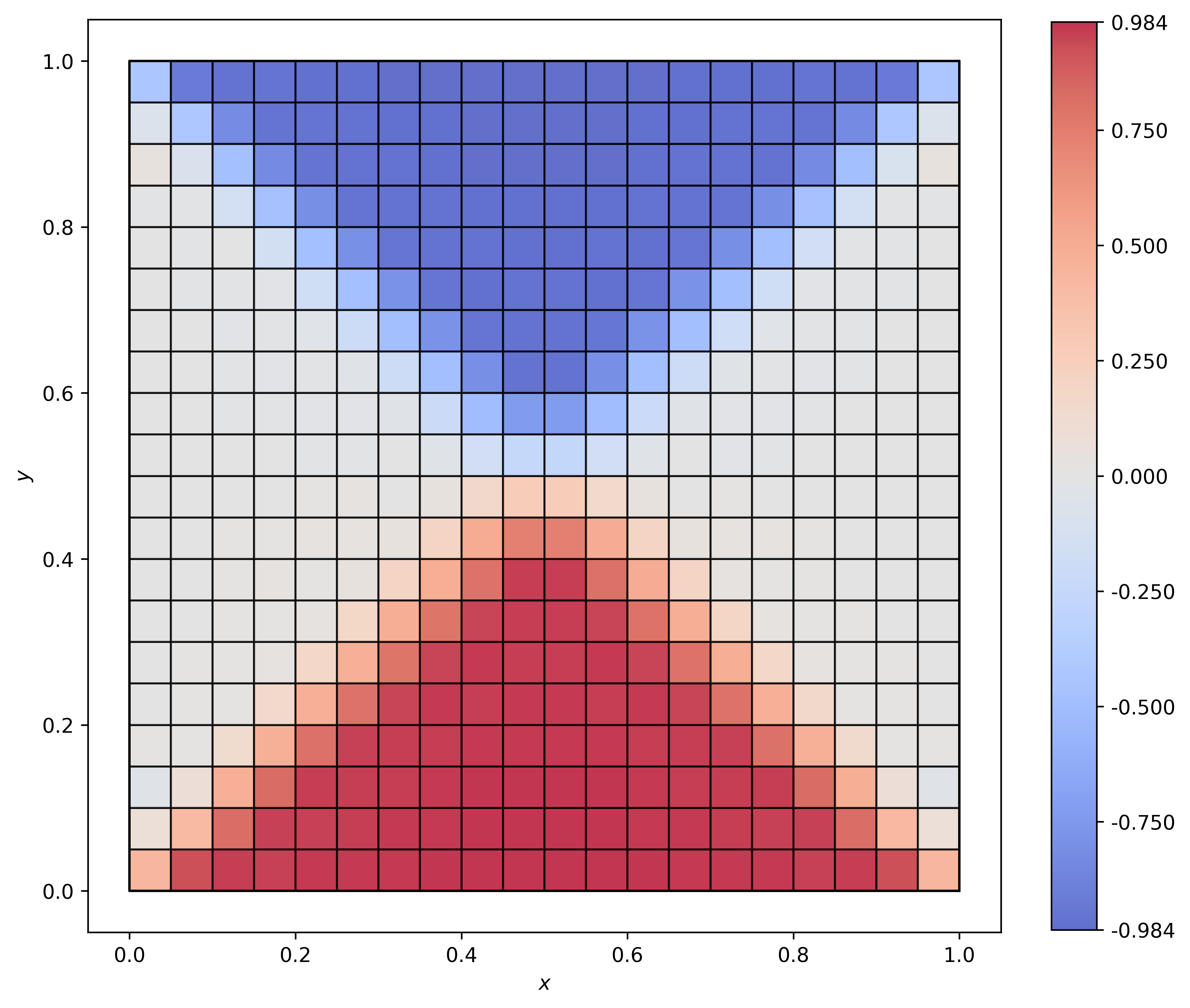}
        % \caption{$n = 400$}
    \end{minipage}
        \caption{Approximate of the partial derivative of the solution, $\widehat{\partial_y {u}}$, obtained with the proposed method, for $p-$Laplace parameter $p \in \{1.5,2,5,10,50 \}$ (from left to right).}
\end{figure}

We note that, for example, for $p=50$, the results are consistent (qualitatively and quantitatively) with the ones obtained in \cite[Example 4]{Aragon2023} on the square $(-1,1)\times(-1,1)$, taking into account the solutions relation $u_{{(0,1)\times(0,1)}}(x,y) = 2^{-\frac{p}{p-1}} u_{(-1,1)\times(-1,1)}(2x-1,2y-1)$, $x,y \in (0,1)$.

%% file: ms-disk-2.tex
\begin{figure}[H]
    \centering
    \begin{minipage}{0.17\textwidth}
        \centering
        \includegraphics[width=\textwidth]{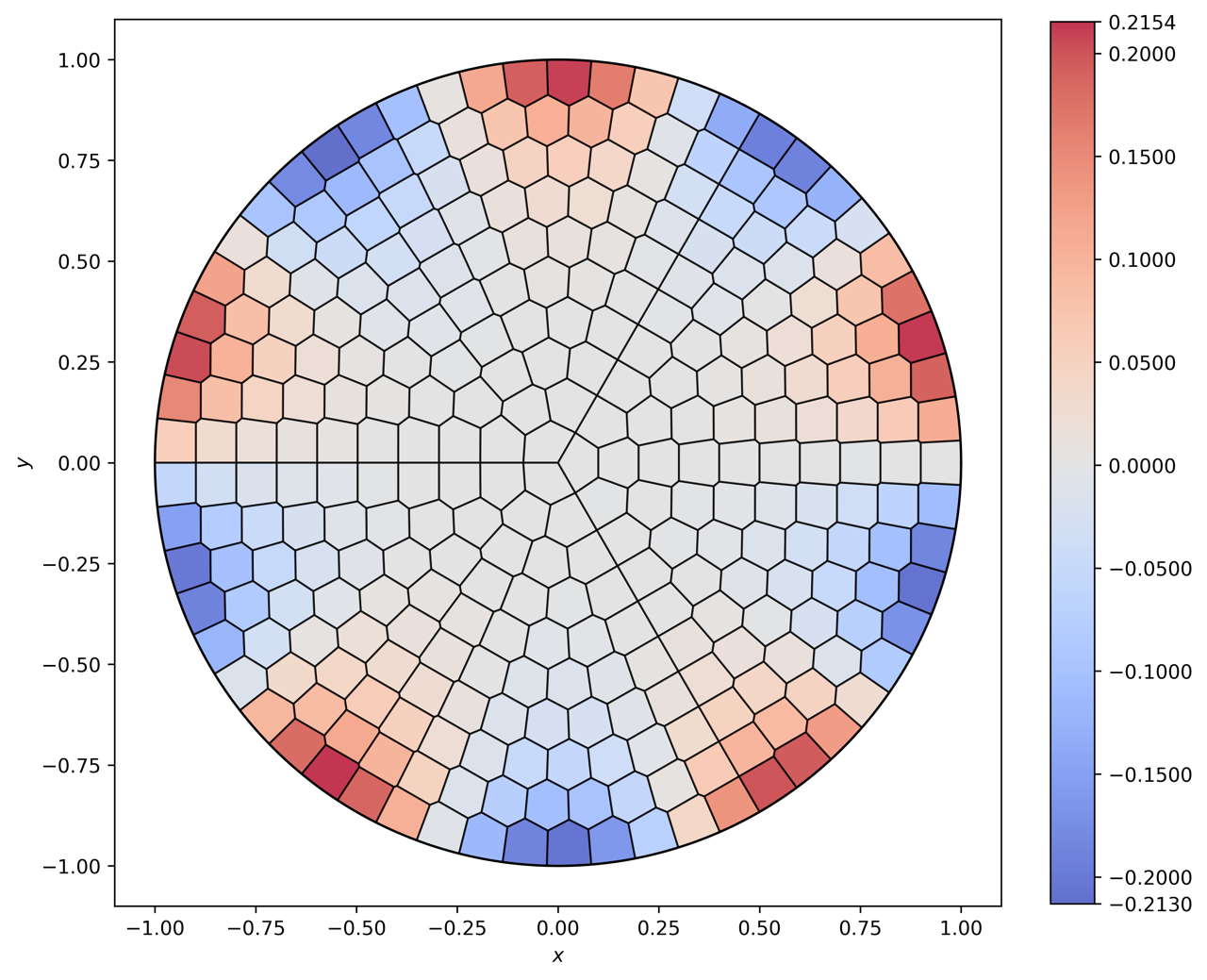}
        % \caption{$n = 400$}
    \end{minipage}
    % \hspace{20pt}
    \begin{minipage}{0.19\textwidth}
        \centering
        \includegraphics[width=\textwidth]{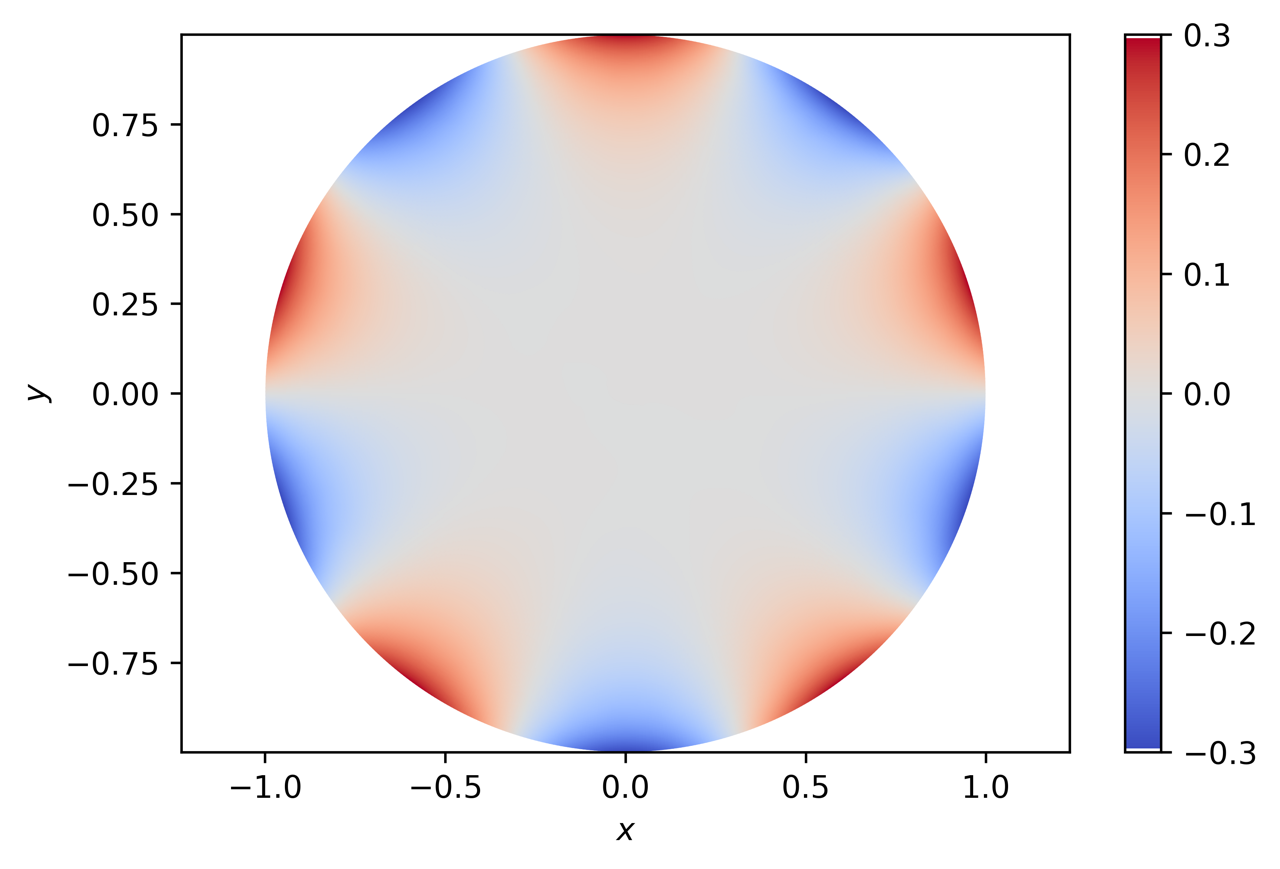}
        % \caption{$n = 400$}
    \end{minipage}
    \begin{minipage}{0.2\textwidth}
        \centering
        \includegraphics[width=\textwidth]{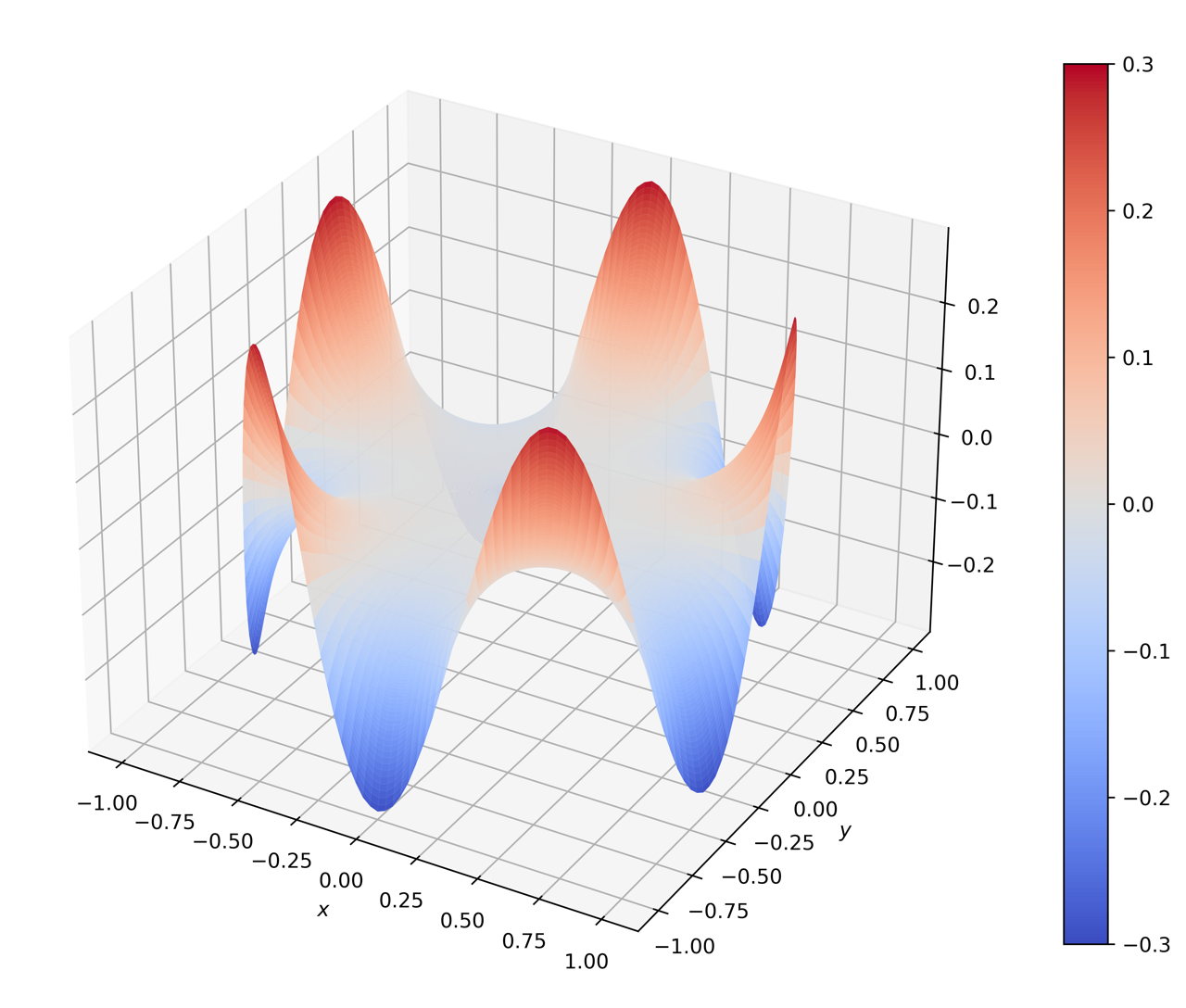}
        % \caption{$n = 400$}
    \end{minipage}
        \begin{minipage}{0.17\textwidth}
        \centering
        \includegraphics[width=\textwidth]{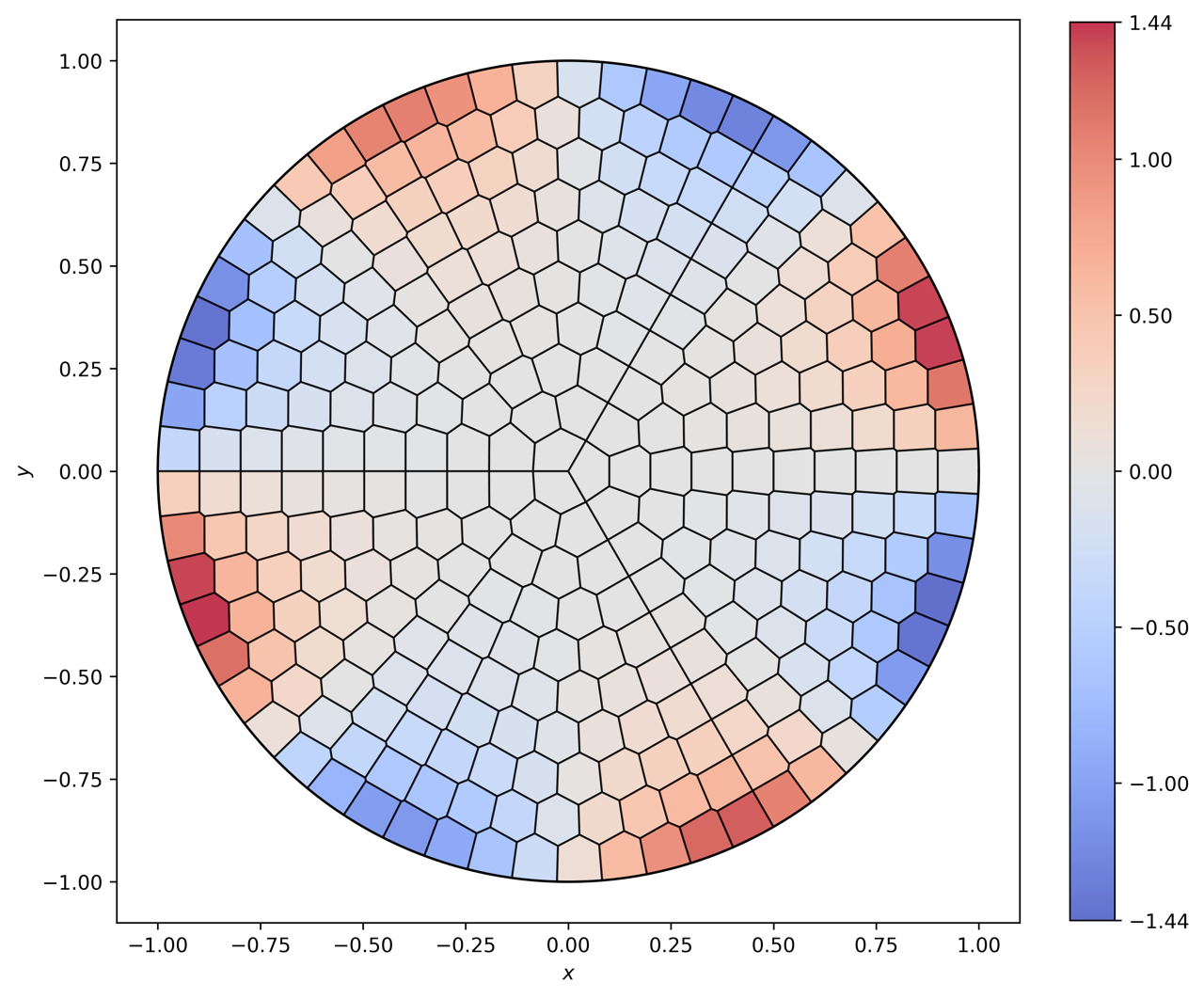}
        % \caption{$n = 400$}
    \end{minipage}
        \begin{minipage}{0.17\textwidth}
        \centering
        \includegraphics[width=\textwidth]{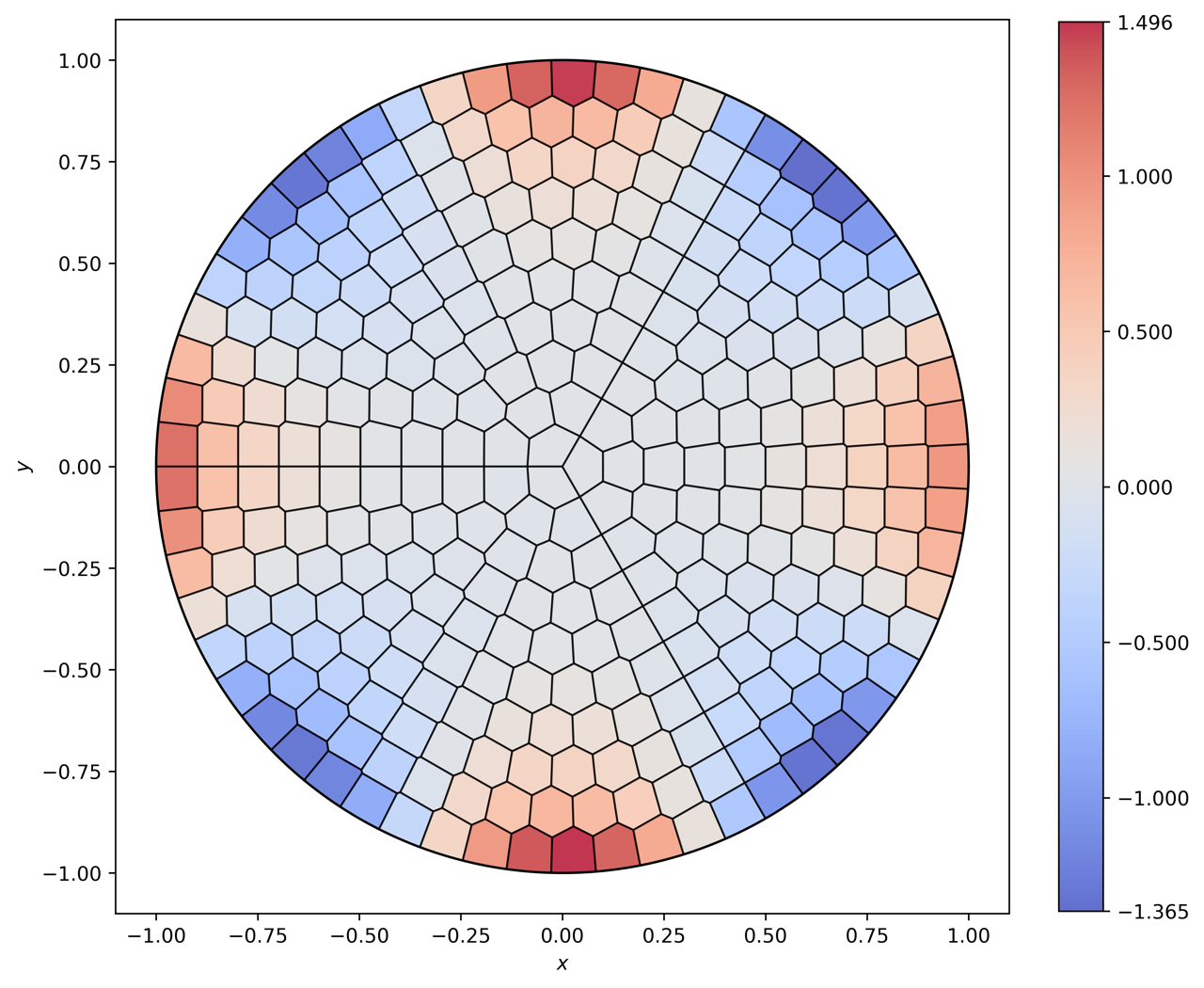}
        % \caption{$n = 400$}
    \end{minipage}
        \caption{Approximate solution $\widehat{u}$ (3 plots) and approximate partial derivatives $\widehat{\partial_x u}$ and $\widehat{\partial_y u}$, obtained with the proposed method. We note that $\widehat{u}$ is qualitatively consistent with the numerical solution in \cite[Section 5.3]{Sakakibara2024}}.
\end{figure}

%% file: Appendix.tex
\appendix
\section{Appendix}

% \subsection{Supplementary numerical details for \Cref{sec:numerical examples}}\label{appendix numerical}
\subsection{Supplementary numerical details for Section \ref{sec:numerical examples}}\label{appendix numerical}

\noindent \textbf{Estimate matrices $\mathbf{A}_{M,N}, \,\, \mathbf{B}_{M,N}, \,\, \mathbf{H}_{M,N}, \,\, \mathbf{K}_{M,N}$, when $D=B(0,1)$, $D=(0,1)\times(0,1)$}. The estimate matrices $\mathbf{A}_{M,N} \subset \mathbb{R}^{n,m}, \,\, \mathbf{B}_{M,N} \subset \mathbb{R}^{d,n,m}, \,\, \mathbf{H}_{M,N} \subset \mathbb{R}^{n,p}, \,\, \mathbf{K}_{M,N} \subset \mathbb{R}^{d, n,p},$ associated to points in \eqref{eq:unit disk points} in case of $D=B(0,1)$ and \eqref{eq:unit square points} in case of $D=(0,1) \times (0,1)$, respectively, are available at \url{https://drive.google.com/drive/folders/1QOxEojmX6tDFwwdR4Y61af0dk3xVH6Ns?}. These can be downloaded and used also for other examples of interest where $D=B(0,1)$ or $D=(0,1)\times(0,1)$. \\

\noindent \textbf{TSVD for \eqref{eq:linear_q}}. In Example \ref{example elliptic div} and in Example \ref{example elliptic div min} (Table \ref{table solution operator example elliptic div min}) we solved the corresponding system of linear algebraic equations \eqref{eq:linear_q} using the truncated singular value decomposition method. Are discarded the singular values $\sigma_i$ of $\Lambda$ in \eqref{eq:linear_q} for which $\sigma_i \leq rcond \, \sigma_{max}(\Lambda)$, with relative singular value threshold $rcond$ listed in Table \ref{table numerical parameters}. We note that for Example \ref{example elliptic div}, $\sigma_{max}(\Lambda) \approx 10^{-2}$, $\sigma_{min}(\Lambda) \approx 10^{-7}$, for Example \ref{example elliptic div min}, $\sigma_{max}(\Lambda) \approx 10^{-3}$, $\sigma_{min}(\Lambda) \approx 10^{-7}$. \\

\noindent \textbf{Optimization of $\mathcal{J}_{n,m,p,N,M}$}. 
As mentioned at the beginning of \Cref{sec:numerical examples}, for the (local) minimization of $\mathcal{J}_{n,m,p,N,M}$ \eqref{eq:min_el_discrete} we used Adam optimization method \cite[Algorithm~1 (Section 2)]{Adam17}, and when searching saddle points of $\mathcal{J}_{n,m,p,N,M}$ we relied on the p-HiSD (preconditioned high saddle dynamics) method proposed in \cite[(3.3)]{huang2026computingsaddlepointsstiff}. For completeness, we recall below the Adam and the discrete preconditioned high saddle algorithms we used, and we list the numerical parameters used in Examples \ref{example elliptic div min} - \ref{example semilinear}. 

\noindent \textbf{ADAM}. Initialize: $\bar{f}_0$, $\alpha>0$,
$\beta_1,\beta_2 \in [0,1)$, $\epsilon>0$,
$m_0$, $v_0$;

\hspace{25pt} Compute: 
\begin{equation*}
    \begin{aligned}
    m_j & = \beta_1 m_{j-1} + (1-\beta_1) \nabla \mathcal{J}(\bar{f}_{j-1}) \\
    v_j & = \beta_2 v_{j-1} + (1-\beta_2) \left[ \nabla \mathcal{J}(\bar{f}_{j-1}) \right]^{\odot 2} \\
    \alpha_j & = \alpha \frac{\sqrt{1-\beta_2^j}}{1-\beta_1^j}\\
    \bar{f}_j & = \bar{f}_{j-1} - \alpha_j \frac{m_j}{\sqrt{v_j}+\epsilon}, \quad j \geq 1.
    \end{aligned}
\end{equation*}
By $[\cdot]^{\odot2}$ we understand element-wise square. For examples \ref{example elliptic div min} - \ref{example semilinear}, we use $\beta_1=0.9,\beta_2=0.999$, $\epsilon = 10^{-10}$, $m_0 = 0 \in \mathbb{R}^m$, $v_0 =0 \in \mathbb{R}^m$ and we list $\bar{f}_0 \in \mathbb{R}^m$ and $\alpha >0$ in Table \ref{table numerical parameters}. 

\newpage
\noindent \textbf{p-HiSD}. Initialize: $\bar{f}_0$, $\eta>0, \tau>0$, $k \in \mathbb{N}^*$, $V^{(0)}= \{ v_1^{(0)}, v_2^{(0)}, \ldots , v_k^{(0)}\}$, $M$;

\hspace{75pt} $M-$orthonormalize $V^{(0)}$;

\hspace{25pt} Compute: 
\begin{equation*}
    \begin{aligned}
    \bar{f}_j & = \bar{f}_{j-1} - \eta R^M_{V^{(j-1)}} M^{-1} \nabla \mathcal{J}(\bar{f}_{j-1}) \\
    v_i^{(j)} & =  v_i^{(j-1)} - \tau P^M_{i,V^{(j-1)}} M^{-1} H(\bar{f}_j) v_i^{(j-1)}, \quad i \in \{1,\ldots,k \},\\
    M&-\text{orthonormalize } V^{(j)} = \{ v_1^{(j)}, v_2^{(j)}, \ldots , v_k^{(j)}\}, \quad j \geq 1,
    \end{aligned}
\end{equation*}
where $R^M_{V^{(j)}} = I - 2 \displaystyle\sum_{i=1}^k v_i^{(j)} (v_i^{(j)})^T M $, $P^M_{i,V^{(j)}} = I - v_i^{(j)} (v_i^{(j)})^T M   - 2 \displaystyle\sum_{l<i} v_l^{(j)} (v_l^{(j)})^T M$ and $H(\bar{f})$ denotes the Hessian of $\mathcal{J}$ at $\bar{f}$. To avoid the computation of the Hessian matrix $H(\bar{f})$, except at initial step, we approximate $H(\bar{f}) v \approx \dfrac{\nabla \mathcal{J}(\bar{f}+h v) - \nabla \mathcal{J}(\bar{f}-h v)}{2h}$, with $h = 10^{-6} \max (1, \| \bar{f} \|_2) $. For Example \ref{example semilinear}, we use $M = \mathbf{A}_{M,N}$ (in view of our proposed method and the comments in \cite[Section 6]{huang2026computingsaddlepointsstiff}), at the initial step we consider $V^{(0)}$ given by the first $k$ eigenvectors of $H(\bar{f}_0)$ and we list $\bar{f}_0 \in \mathbb{R}^m$ and $\eta, \tau >0$ in Table \ref{table numerical parameters}. \\

We include in the following table the norm $\| \nabla J(\bar{f}_{opt}) \|_{2,V} $ with $\bar{f}_{opt}$ at final step, where we denote by $\| \nabla \mathcal{J}(\bar{f}) \|_{2,V} \coloneqq \big(\displaystyle\sum_{i=1}^m \lambda (V_i) \left[\nabla \mathcal{J}(\bar{f})\right]^2_i\big)^{1/2}$. Recall that $\lambda(V_i) = \int_{V_i} \,d \lambda$, with $V_i$ the Voronoi cells associated to points $y_i$ in \eqref{eq:unit disk points}, if $D=B(0,1)$, or in \eqref{eq:unit square points}, if $D=(0,1)\times(0,1)$. More precisely, in the case \eqref{eq:unit disk points} $\lambda(V_i) \approx 10^{-2}$, $1 \leq i \leq m$, and in the case \eqref{eq:unit square points} $\lambda(V_i) = 2.5 \times10^{-3}$, $1 \leq i \leq m$. \\

\begin{table}[H]
\centering
\begin{tabular}{lllll}
\hline
Problem & Method & $\bar{f}_0 \in \mathbb{R}^m$ & Parameter(s) & $\| \nabla J(\bar{f}_{opt}) \|_{2,V} $ \\ \hline
Example \ref{example elliptic div} & TSVD &  & $rcond = 10^{-3}$ &  \\
Example \ref{example elliptic div min} & ADAM & $\bar{f}_0=0.1$ & $\alpha = 0.01$ & $\approx 10^{-4}$ \\
 & TSVD & & $rcond = 10^{-2}$ & \\
Example \ref{example plaplace disk} $p \in \{1.5,2\}$ & ADAM & $\bar{f}_0=1$ & $\alpha = 0.001$ & $\approx 10^{-5}$ \\
\qquad \qquad  \quad $p \in \{5,10\}$ & ADAM & $\bar{f}_0=1$ & $\alpha = 0.01$ & $\approx 5 \times 10^{-6}$ \\
\qquad \qquad  \quad  $p \in \{50\}$ & ADAM & $\bar{f}_0=1$ & $\alpha = 0.1$ & $\approx 5 \times 10^{-6}$ \\
Example \ref{example plaplace square} $p \in \{1.5,2,5,10 \}$ & ADAM & $\bar{f}_0=1$ & $\alpha = 1$ & $\approx 10^{-6}$ \\
\qquad \qquad  \quad  $p =50 $ & ADAM & $\bar{f}_0=1$ & $\alpha = 0.1$ & $\approx 10^{-6}$ \\
Example \ref{example minimal surface} & ADAM & $\bar{f}_0=1$ & $\alpha = 0.01$ & $\approx 10^{-6}$ \\
Example \ref{example semilinear} $s=800$ & ADAM & $\bar{f}_0=1$ & $\alpha = 1$ & $\approx 10^{-6}$ \\
  & p-HiSD $k \in \{1,2,3 \}$ & $\bar{f}_0=1000$ & $\eta = 20$, $\tau = 10$ & $\approx 10^{-6}$ \\
Example \ref{example semilinear} $s=8000$ & ADAM & $\bar{f}_0=1$ & $\alpha = 10$ & $\approx 10^{-6}$ \\
  & p-HiSD $k \in \{1,2,\ldots,8 \}$ & $\bar{f}_0=1000$ & $\eta = 20$, $\tau = 10$ & $\approx 10^{-6}$ \\
  & p-HiSD $k =5 $ & $\bar{f}_0=1000$ & $\eta = 5$, $\tau = 10$ & $\approx 10^{-6}$ \\
  & p-HiSD $k \in \{6,7 \}$ & $\bar{f}_0=1000$ & $\eta = 80$, $\tau = 10$ & $\approx 10^{-6}$ \\
 & p-HiSD $k \in \{2,3,5 \}$ & $\bar{f}_0= \bar{f}_{close}^{\,k=4,\eta=20}$ & $\eta = 20$, $\tau = 10$ & $\approx 10^{-6}$ \\
 & p-HiSD $k \in \{2,3,4,6 \}$ & $\bar{f}_0=\bar{f}_{close}^{\,k=5,\eta=20}$ & $\eta = 20$, $\tau = 10$ & $\approx 10^{-6}$ \\
  & p-HiSD $k = 4$ & $\bar{f}_0=\bar{f}_{close}^{\,k=5,\eta=5}$ & $\eta = 5$, $\tau = 10$ & $\approx 10^{-6}$ \\
 & p-HiSD $k \in \{5,7 \}$ & $\bar{f}_0=\bar{f}_{close}^{\,k=6,\eta=20}$ & $\eta = 20$, $\tau = 10$ & $\approx 10^{-6}$ \\
\hline
\end{tabular}
\caption{TSVD/Optimization/Saddle point search input parameters and convergence results.} \label{table numerical parameters}
\end{table}

Let us comment on the initial data $\bar{f}_0$ in the last 4 lines of Table \ref{table numerical parameters}. More precisely, we detail the last row and same considerations apply to the other 3 rows. During the p-HiSD algorithm with index $k=6$ and $\eta=20$, started from $\bar{f}_0 = 1000 \in \mathbb{R}^m$, there is a transient decrease in $\| \nabla \mathcal{J}(\cdot) \|_{2,V}$ that suggests that the trajectory may have passed close to another saddle point. Motivated by this observation, we restarted the p-HiSD algorithm from this local minimum of $\| \nabla \mathcal{J}(\cdot) \|_{2,V}$, denoted by $\bar{f}_{close}^{\,k=6,\eta=20}$, for other indices different from $6$, resulting in convergence for indices $k \in \{5, 7 \}$.

% \subsection{Proofs for \Cref{s:1}}\label{appendix proofs}
\subsection{Proofs for Section \ref{s:1}}\label{appendix proofs}

\begin{proof}[\bf{Proof of \Cref{prop:U_0 and DU_0}}]\label{proof:prop:U_0 and DU_0}
We mention that the first equality in \eqref{eq:U0_rep_intro} follows by e.g. \cite{Ge92}, whilst the second equality is a consequence of the monotone convergence theorem, the strong Markov and the scaling property for the Brownian motion: Assuming without loss that $0\leq f\in L^p(D), p>d/2$, and setting
\begin{equation}\label{eq:tau_i}
    \tau^x_0=0,\quad \tau^x_{i+1}:=\inf\left\{t>\tau_i^x:\|B_t^x - B^x_{\tau_i^x}\|=r\left(B^x_{\tau_i^x}\right)\right\}, \quad i\geq 0,
\end{equation}
we have
\begin{align}
    \mathbb{E} \left\{ \int_{0}^{\tau^x_{\partial D}} f(B^x_t) \ d t \right\}
    &=\mathbb{E} \left\{ \sum_{i\geq 0}\int_{\tau_i^x}^{\tau^x_{i+1}} f(B^x_t) \ d t \right\}
    = \sum_{i\geq 0}\mathbb{E} \left\{\int_{\tau_i^x}^{\tau^x_{i+1}} f(B^x_t) \ d t \right\}\\
    &=\sum_{i\geq 0}\mathbb{E} \left\{v(B(\tau_i^x))\right\}, \quad \mbox{where} \quad v(y)=r^2(y)\mathbb{E}\left\{\int_0^{\tau^y_{\partial B(0,1)}}f\left((y+r(y)B^0(t))\right)\right\}\\
    &\phantom{=\sum_{i\geq 1}\mathbb{E} \left\{v(B(\tau_i^x))\right\}, \quad \mbox{where} \quad v(y)\,\,}
    =r^2(y)\int_{B(0,1)} -2G_{B(0,1)}(0,z)f(y+r(y)z)\;dz\\
    &\phantom{=\sum_{i\geq 1}\mathbb{E} \left\{v(B(\tau_i^x))\right\}, \quad \mbox{where} \quad v(y)\,\,}
    =\frac{r^2(y)}{d}\mathbb{E}\left\{f(y+r(y)Y)\right\}\\
    &=\dfrac{1}{d} \mathbb{E} \left\{ \sum_{k \geq 0} r^2(X^x_k) f(X^x_k + r(X^x_k) Y)   \right\}, \quad x \in D.
    \end{align} 

The proof of assertion (ii) is based on the following lemma
\begin{lem}\label[lem]{thm:grad U0 representation} Let $D \subset \mathbb{R}^d$, $d \geq 2$, be a bounded domain, and $f \in L^\infty(D)$. 
Then, for all  $x \in D$ we have
\begin{equation}\label{eq:grad U0 representation}
\begin{aligned}
    \nabla U f (x) 
    = \dfrac{d}{r(x)^{d+1} \omega_{d-1}} \int_{\partial B(x,r(x))} (z-x) u(z) \sigma(d z) 
   + \dfrac{2 r(x)}{d \, V_d} \int_{B(0,1)}  \Tilde{y} \left( \dfrac{1}{|\Tilde{y}|^d} - 1 \right) f(x+r(x)\Tilde{y}) \ d \Tilde{y},
\end{aligned}
\end{equation}
where $\omega_{d-1} =\dfrac{2 \pi^{d/2}}{\Gamma(d/2)}$ is the surface area of the unit sphere in $\mathbb{R}^d$, whilst $V_d = \dfrac{\pi^{d/2}}{\Gamma(d/2+1)}$ is the volume of the unit $d$-ball.
\end{lem}

\begin{proof}[Proof of \Cref{thm:grad U0 representation}]
Let $x \in D$ be fixed, and $v_x,w_x:B(x,r(x))\rightarrow \mathbb{R}$ be the $C^1$-solutions to 
\begin{eqnarray}
\left\{
\begin{array}{ll}
    -\frac{1}{2} \Delta v_x = 0  &\mbox{ in } B(x,r(x))\\[4pt]
    v_x  = u|_{\partial B(x,r(x))}  &\mbox{ on } \partial B(x,r(x))
\end{array},
\right.
\qquad
\left\{
\begin{array}{rl}
    -\frac{1}{2} \Delta w_x = f & \mbox{ in } B(x,r(x)) \\[4pt]
    w_x = 0 & \mbox{ on } \partial B(x,r(x))
\end{array},
\right.
\end{eqnarray}
so that
\begin{align}
    u(y) 
    &= v_x(y) + w_x(y),\quad y\in B(x,r(x)), \label{eq:sol u v w}\\
    \nabla_x u(x) 
    &= \lim_{D \ni y \to x} \left[\nabla_y v_x(y) + \nabla_y w_x(y)\right].\label{eq:grad sol lim}
\end{align}
First, regarding $v_x$, we have that
\begin{equation*}
    v_x(y) = \int_{\partial B(x,r(x))} K_{B(x,r(x))}(y,z) u(z) \sigma(d z), \quad y \in B(x,r(x)),
\end{equation*}
where $K_{B(x,r(x))}(\cdot,\cdot)$ is the Poisson kernel for the ball $B(x,r(x))$, namely
\begin{equation*}
    K_{B(x,r(x))}(y,z) = \dfrac{r(x)^2-|y-x|^2}{r(x) \omega_{d-1} |y-z|^d}, \quad y \in B(x,r(x)), z \in \partial B(x,r(x)).
\end{equation*} 
Thus, by dominated convergence, we have that
\begin{equation}\label{eq:lim partial v}
\begin{aligned}
    \lim_{y \to x} \nabla_y v_x(y) 
    & = \int_{\partial B(x,r(x))} \lim_{y \to x} \nabla_y K_{B(x,r(x))}(y,z) u(z) \sigma(d z)\\
    &=\dfrac{d}{r(x)^{d+1} \omega_{d-1}} \int_{\partial B(x,r(x))} (z-x) u(z) \sigma(d z). 
\end{aligned}
\end{equation}
Secondly, regarding $w_x$, we have that
\begin{equation*}
    w_x(y) = -2 \int_{B(x,r(x))} G_{B(x,r(x))}(y,\Tilde{y}) f(\Tilde{y}) \ d \Tilde{y}, \quad y \in B(x,r(x)),
\end{equation*}
where $G_{B(x,r(x))}(\cdot,\cdot)$ is the Green function for the Laplace operator $\Delta$ on the ball $B(x,r(x))$. 
We recall that the latter is given by
\begin{equation}\label{eq:Green ball}
    G_{B(x,r(x))}(y,\Tilde{y}) = \phi(y-\Tilde{y}) - \phi \left( \dfrac{y-x}{r(x)} (\Tilde{y} - y^\ast) \right), \quad y,\Tilde{y} \in B(x,r(x)),
\end{equation}
where $\phi(\cdot)$ is the fundamental solution for the Laplace equation in $\mathbb{R}^d$, namely
\begin{eqnarray}
\phi(y) =
\left\{
\begin{array}{ll}
    \dfrac{1}{2 \pi} \ln{\|y\|}, & d=2\\[4pt]
     -\dfrac{1}{d (d-2) V_d} \dfrac{1}{\|y\|^{d-2}}, & d \geq 3
\end{array},
\right.
\end{eqnarray}
where $y^\ast := x + \dfrac{r(x)^2 (y-x)}{|y-x|^2}$. 
Thus, for $d \geq 2$, we have that
\begin{equation}\label{eq:lim partial w}
\begin{aligned}
    \lim_{y \to x} \nabla_y w_x(y) & = -2 \int_{B(x,r(x))} \lim_{y \to x} \nabla_y G_{B(x,r(x))}(y,\Tilde{y}) f(\Tilde{y}) \ d \Tilde{y} \\
    & = -\dfrac{2}{d \, V_d} \int_{B(x,r(x))}  (x-\Tilde{y}) \left( \dfrac{1}{|x-\Tilde{y}|^d} - \frac{1}{r(x)^d} \right) f(\Tilde{y}) \ d \Tilde{y} \\
     & = 2 \dfrac{r(x)}{d \, V_d} \int_{B(0,1)}  \Tilde{y} \left( \dfrac{1}{|\Tilde{y}|^d} - 1 \right) f(x+r(x)\Tilde{y}) \ d \Tilde{y}.
\end{aligned}
\end{equation} 
\end{proof}

\noindent{\bf{ Proof of \Cref{prop:U_0 and DU_0} (continued).}} In view of \Cref{thm:grad U0 representation} we have
\begin{equation}\label{eq:dxu proof}
    \dfrac{\partial}{\partial x_i} u (x)  = \dfrac{d}{r(x)^2} \ \mathbb{E} \left[ u(Z) (Z-x) \right] + 2 \dfrac{r(x)}{d} \ \mathbb{E} \left\{ \widetilde{Y}_i \left( \dfrac{1}{|\widetilde{Y}|^d} - 1 \right) f(x+r(x) \widetilde{Y} )\right\}, \quad x \in D,
\end{equation}
where $Z \sim \mathrm{Uniform}(\partial B(x,r(x)))$ and $\widetilde{Y} \sim \mathrm{Uniform}(B(0,1))$. Taking $Z = B^x_{\tau^x_{\partial B(x,r(x))}}$, in view of assertion (i) and by the strong Markov property we have
\begin{equation*}
    u\left(B^x_{\tau^x_{\partial B(x,r(x))}} \right) = \mathbb{E} \left\{ \int_{\tau^{x}_{\partial B(x,r(x))}}^{\tau^{x}_{\partial D}} f(B^{x}_t) \ d t \;\bigg|\; \mathcal{F}_{\tau^{x}_{\partial B(x,r(x))}}\right\},
\end{equation*}
hence
\begin{equation*}
    \mathbb{E} \left\{ u\left(B^x_{\tau^x_{\partial B(x,r(x))}} \right) \left( 
         B^x_{\tau^x_{\partial B (x,r(x))}} -x \right) \right\} = \mathbb{E} \left\{ \int_{\tau^x_{\partial B(x,r(x))}}^{\tau^x_{\partial D} }f(B^x_t) \ d t \ \left( 
         B^x_{\tau^x_{\partial B (x,r(x))}} -x \right)\right\},
\end{equation*}
which proves the first equality in assertion (ii).

The second equality in assertion (ii) is an obvious consequence of the second equality in assertion (i).
\end{proof}

\begin{proof}[\bf{Proof of \Cref{prop:H and DH}}]\label{proof:prop:H and DH}

The first equality in assertion (i) is well-known, see for example \cite{Ge92}.

For the second equality in (i) we proceed as follows. 
As in the proof of \Cref{thm:grad U0 representation} regarding $v_x(y)$ with $x\in D$ being fixed, we obtain similarly that
\begin{equation}\label{eq:dxh proof}
    \dfrac{\partial}{\partial x_i} h (x)  = \dfrac{d}{r(x)^2} \ \mathbb{E} \left[ h(Z) (Z-x) \right],
\end{equation}
where $Z \sim \mathrm{Uniform}(\partial B(x,r(x)))$. 
Taking $Z = B^x_{\tau^x_{\partial B(x,r(x))}}$, in view of the first part of (i) and the strong Markov property, we have that
\begin{equation*}
    h\left(B^x_{\tau^x_{\partial B(x,r(x))}} \right) = \mathbb{E} \left\{ g \left(B^z_{\tau^z_{\partial D}} \right) \right\}\bigg|_{z=Z}
    =\mathbb{E} \left\{ g \left(B^x_{\tau^x_{\partial D}} \right) \big| \mathcal{F}_{\tau^x_{\partial B(x,r(x))}} \right\},
\end{equation*}
hence,
\begin{equation*}
    \mathbb{E} \left\{ h \left(B^x_{\tau^x_{\partial B(x,r(x))}} \right) \left(B^x_{\tau^x_{\partial B(x,r(x))}} -x\right) \right\} 
    = \mathbb{E} \left\{ g \left(B^x_{\tau^x_{\partial D}} \right) \   \left(B^x_{\tau^x_{\partial B(x,r(x))}} -x\right) \right\},
\end{equation*}
which completes the proof of (i). 

In order to prove assertion (ii), let us first notice that for $x\in D$ and $M\geq 0$ we have that $X_M^x$ has the same distribution as $B^x_{\tau_M^x}$, where $\tau_M^x$ is given by \eqref{eq:tau_i}.
Consequently, for $g$ bounded and measurable on $\overline{D}$, we have 
\begin{equation*}
    \mathbb{E}\left\{g(X_M^x)\right\}
    =\mathbb{E}\left\{g\left(B^x_{\tau_M^x}\right)\right\}, 
    \quad \mathbb{E}\left\{g\left(X_M^{x}\right)(X_1^x-x)\right\}
    =\mathbb{E}\left\{g\left(B^x_{\tau_M^x}\right)(B_{\tau_1^x}-x)\right\}.
\end{equation*}
Now, since $(B^x_{t})_{0\leq t\leq \tau^x_{\partial D}}$ is a bounded martingale closed by $B^x_{\tau^x_{\partial D}}$, by Doob's convergence theorem and the fact that $B$ has continuous paths, we deduce that
\begin{equation*}
    \lim\limits_{M\to\infty} B^y_{\tau^x_M}=B^y_{\tau^x_{\partial D}} \quad \mbox{for all } y\in \overline{D}\quad  \mathbb{P}\mbox{-a.s}.
\end{equation*}
From this, using the condition on $g$ imposed in the statement, we get
\begin{equation*}
    \lim\limits_{M\to\infty} g\left(B^x_{\tau^x_M}\right)=g\left(B^x_{\tau^x_{\partial D}}\right)  \quad    \mathbb{P}\mbox{-a.s and boundedly}.
\end{equation*}
Now, assertion (ii) follows by the dominated convergence theorem.
\end{proof}

\begin{proof}[\bf{Proof of \Cref{prop:removing singularities}}]\label{proof:prop 2.6}

(i.1). Let $\psi : B(0,1)\rightarrow \mathbb{R}$ be bounded and measurable. 
Then
    \begin{align*}
            \mathbb{E}\left\{\psi(Y)\right\}
            &=-\frac{2}{\pi}\int_{B(0,1)}\psi(y)\ln{\|y\|}\;dy
            =-\frac{2}{\pi}\int_0^1\ln{r}\int_{S(0,r)}\psi(y)\;\sigma(dy)\;dr\\
            &=-4\int_0^1r\ln{r}\mathbb{E}\left\{\psi(Z_r)\right\}\; dr, \quad \mbox{where } Z_r\sim \mathrm{Uniform}(\partial B(0,r))\\       
            &=\int_0^1\underbrace{(-4r\ln{r})}_{=:\rho(r)}\mathbb{E}\left\{\psi(rZ)\right\}\; dr, \quad \mbox{where } Z\sim \mathrm{Uniform}(\partial B(0,1))\\
            &=\mathbb{E}\left\{\psi(RZ)\right\}, \quad \mbox{where } R\sim \rho, \, R \mbox{ and } Z \mbox{ independent}.
    \end{align*}
    Note that $R$ has the CDF $F_\rho(t) \coloneqq \int_0^t \rho (r) \, dr = t^2-t^2\ln{t^2}$, $t \in [0,1]$. 
            It then follows that $F^{-1}_\rho(t)=e^{\frac{1+W_{-1}(-t/e)}{2}}, \;t\in[0,1]$.
            Consequently, if $U\sim {\rm Uniform}([0,1])$ independently of $Z\sim { \rm Uniform}(\partial B(0,1))$, then $R:=e^{\frac{1+W_{-1}(-U/e)}{2}}\sim \rho$ and
    \begin{equation*}
                \mathbb{E}\left\{\psi(Y)\right\}=\mathbb{E}\left\{\psi(RZ)\right\}.
    \end{equation*}

    Let us now focus on the second representation claimed in (i.1), for which we use the above computations and continue as follows:
    \begin{align}
            \mathbb{E}\left\{\psi(Y)\right\}        &=-4\int_0^1r\ln{r}\mathbb{E}\left\{\psi(rZ)\right\}\; dr, \quad \mbox{where } Z\sim {\rm Uniform}(\partial B(0,1))\\
            &=-4\beta(2,2)\int_0^1\frac{r\ln{r}}{r(1-r)}\beta(r;2,2)\mathbb{E}\left\{\psi(rZ)\right\}\; dr, \quad \beta(r;2,2):=\frac{r(1-r)}{\beta(2,2)}\\
            &=-\frac{2}{3}\mathbb{E}\left\{\frac{\ln{R}}{(1-R)}\psi(RZ)\right\},
    \end{align}
    where $\beta$ is the usual $\beta$ function, whilst $R\sim {\sf Beta}(2,2)$ is independent of $Z$.

(i.2). We have
    \begin{align*}
                \mathbb{E}\{\psi(Y)\}
                &=\frac{2}{(d-2)\lambda(B(0,1))}\int_{B(0,1)}\psi(y) \left(\frac{1}{\|y\|^{d-2}}-1\right)  \;dy\\
                &=\frac{2\sigma(S(0,1))}{(d-2)\lambda(B(0,1))}\int
                _0^1\left(\frac{1}{r^{d-2}}-1\right)r^{d-1}\frac{1}{\sigma(S(0,r))}\int_{S(0,r)}\psi(y)   \;dy\;dr\\
                &=\frac{2d}{d-2}\int_0^1 r\left(1-r^{d-2}\right) \mathbb{E}\left\{\psi(rZ)\right\}\; dr, \quad \mbox{where } Z \sim {\rm Uniform}(\partial B(0,1))\\
                &=\frac{2d}{(d-2)^2}\int_0^1 r^{\frac{2}{d-2}-1}\left(1-r\right) \mathbb{E}\left\{\psi\left(r^{\frac{1}{d-2}}Z\right)\right\}\;dr, \quad \mbox{where } Z \sim {\rm Uniform}(\partial B(0,1))\\
                &= \mathbb{E}\left\{\psi(R^{\frac{1}{d-2}}Z)\right\}, \quad \mbox{where } R\sim \mbox{Beta}\left(\frac{2}{d-2},2\right) \mbox{ independent of } Z.
    \end{align*}    

(ii). We first set a simplifying notation:
\begin{equation*}
    \mathbb{E} \left\{ \widetilde{Y} \left( \dfrac{1}{|\widetilde{Y}|^d} - 1 \right) \underbrace{f(x+r(x) \widetilde{Y} )}_{=:F(\tilde{Y})}\right\}
    =\mathbb{E} \left\{ \widetilde{Y} \left( \dfrac{1}{|\widetilde{Y}|^d} - 1 \right) F(\tilde{Y})\right\}.
\end{equation*}
Then,
\begin{align*}
\mathbb{E} \left\{ \widetilde{Y} \left( \dfrac{1}{|\widetilde{Y}|^d} - 1 \right) F(\tilde{Y})\right\}
&=\frac{1}{\lambda(B(0,1))}\int_{B(0,1)} y \left( \dfrac{1}{|y|^d} - 1 \right) F(y)\; dy\\
&=\frac{1}{\lambda(B(0,1))}\int_0^1 \left( \dfrac{1}{r^d} - 1 \right)\int_{S(0,r)}y  F(y)\; \sigma(dy)\;dr\\
&=\frac{1}{\lambda(B(0,1))}\int_0^1 \left( \dfrac{1}{r^{d}} - 1 \right)\sigma(S(0,r))\mathbb{E}\left\{ Z_r  F(Z_r)\right\}\;dr, \quad \mbox{where } Z_r \sim \mbox{Uniform}(S(0,r))\\
&=\frac{1}{\lambda(B(0,1))}\int_0^1 \left( \dfrac{1}{r^{d}} - 1\right)\sigma(S(0,r))\mathbb{E}\left\{ rZ  F(rZ)\right\}\;dr, \quad \mbox{where } Z \sim \mbox{Uniform}(S(0,1))\\
&=\frac{\sigma(S(0,1))}{\lambda(B(0,1))}\int_0^1 \left( \dfrac{1}{r^{d}} - 1\right)r^d\mathbb{E}\left\{ Z  F(rZ)\right\}\;dr\\
&=d\int_0^1 \left(1 - r^d\right)\mathbb{E}\left\{ Z  F(rZ)\right\}\;dr
=d\int_0^1 \left(1 - r\right)\sum\limits_{0\leq k\leq d-1}r^k\mathbb{E}\left\{ Z  F(rZ)\right\}\;dr\\
&=\frac{d}{2}\int_0^1 \underbrace{2\left(1 - r\right)}_{\rho(r)}\sum\limits_{0\leq k\leq d-1}r^k\mathbb{E}\left\{ Z  F(rZ)\right\}\;dr\\
&=\frac{d}{2}\mathbb{E}\left\{\sum\limits_{0\leq k\leq d-1}R^k  Z  F(RZ)\right\}, \quad \mbox{where } R\sim \rho, R \mbox{ independent of } Z.
\end{align*}
Since $F_\rho(t):=\int_0^t \rho(r)\;dr=2t-t^2$ we get that $F_\rho^{-1}(t)=1-\sqrt{1-t}, t\in[0,1]$, hence
\begin{equation*}
    1-\sqrt{1-U}\sim \rho \quad \mbox{and thus}\quad 1-\sqrt{U}\sim \rho,\quad \mbox{where } U\sim {\rm Uniform}([0,1]). \qedhere
\end{equation*}
\end{proof}

\begin{proof}[\bf{Proof of  Lemma~ \ref{lem:weak_convergence} }]\label{proof:lem 2.10}
    Let us first treat the weak convergence.
    To this end, note that 
    \begin{equation*}
        \|f_{\mathbf{V}_m}\|_{L^\infty(D)}\leq \|f\|_{L^\infty(D)}, \quad m\geq 1,
    \end{equation*}
    which implies that the sequence $\left(f_{\mathbf{V}_m}\right)_{m\geq 1}$ is bounded in $L^r(D), r\geq 1$.
    Since $L^r(D)$ is reflexive for every $1<r<\infty$, by the Banach-Alaoglu theorem, the sequence $\left(f_{\mathbf{V}_m}\right)_{m\geq 1}$ is $L^r(D)$-weakly relatively compact. 
    Thus, the first convergence in the statement follows once we identify the limit of any weakly convergent subsequence with $f$.
    Let $f_{\mathbf{V}_{m_k}}\rightharpoonup \tilde{f}$ in $L^r(D)$ for some $1<r<\infty$.
    Then, for every $h$ Lipschitz on $\mathbb{R}^d$  we have on one hand that
    \begin{equation*}
        \int_D h_{\mathbf{V}_{m_k}} f \;dx
        =\int_D f_{\mathbf{V}_{m_k}} h \;dx
        \mathop{\longrightarrow}\limits_{k\to\infty} \int_D \tilde{f} h \; dx
    \end{equation*}
    On the other hand, for $h$ as above we clearly have that
    \begin{equation*}
        \lim\limits_{m\to\infty} h_{\mathbf{V}_{m_k}}=h\quad \mbox{uniformly on } D. 
    \end{equation*}
    Thus, for all $h$ as above we get
    \begin{equation*}
        \int_D h f \;dx=\int_D h \tilde{f} \;dx,
    \end{equation*}
    hence $\tilde{f}=f$ a.e., and the proof of the first convergence is now complete.

    Let us now turn to the convergence that regards the boundary data $g$.
    First of all, it is clear that
    \begin{equation*}
        \sup\limits_{z\in \partial D} |g_{\mathbf{\Gamma}_{p}}(z)|\leq \sup\limits_{z\in \partial D} |g(z)|.
    \end{equation*}
    Furthermore, if $z\in \partial D $ is a continuity point for $g$, then, since the set $A$ of discontinuities of $g$ is finite, there exists $p_0\geq 1$ such that for every $p\geq p_0$: there exists a unique $i_p\in \{1,\dots,p\}$ such that $z\in \Gamma^{i_p}_p$ and $\overline{\Gamma^{i_p}_p} \cap A=\emptyset$. 
    Thus, if $(z^i_p)_{1\leq i\leq p}$ denote the centers of the partition $\mathbf{\Gamma}_p$, $p\geq 1$, then by the uniform continuity of $g$ in any compact subset of $\partial D\setminus A$ we get
    \begin{equation*}
        |g(z)-g_{\mathbf{\Gamma}_p}(z)|=|g(z)-g(z^{i_p}_p)|\mathop{\longrightarrow}\limits_{p\to\infty} 0.
    \end{equation*}
    The proof is now complete.
\end{proof}

\begin{proof}[\bf{Proof of Proposition~\ref{prop:lim_mp}}]\label{proof:prop 2.11}
As it is well known, see e.g. \cite{GrWi82},  the Green function $G_D$ satisfies 
\begin{equation*}
    G_D(x,\cdot)\in L^q(D),\quad \forall q<\frac{d}{d-2},\quad \nabla_x G_D(x,\cdot)\in L^r(D),\quad \forall r<\frac{d}{d-1},\quad x\in D.
\end{equation*}
Consequently, for every such $r$ and $q$, by \Cref{lem:weak_convergence} we obtain that for every $x\in D$
\begin{equation*}
   Uf_{\mathbf{V}_m}(x)=\int_D G_D(x,y)f_{\mathbf{V}_m}(y)\;dy\mathop{\longrightarrow}\limits_{m\to\infty} \int_D G_D(x,y)f(y)\;dy=U_0f(x),
\end{equation*}
as well as
\begin{equation*}
   \nabla_x Uf_{\mathbf{V}_m}(x)
   =\int_D \nabla_x G_D(x,y)f_{\mathbf{V}_m}(y)\;dy\mathop{\longrightarrow}\limits_{m\to\infty} \int_D \nabla_x G_D(x,y)f(y)\;dy
   =Uf(x).
\end{equation*}
Finally, the second convergence follows by applying the dominated convergence theorem and Lemma~\ref{lem:weak_convergence} to the representation given by Proposition~\ref{prop:H and DH}, taking into account that the Brownian motion never hits a finite set of points, in particular
\begin{equation*}
\mathbb{P}\left\{ B^Z_{\tau_{\partial D}} \in A\right\}=0=\mathbb{P}\left\{ B^{x}_{\tau^Z_{\partial D}} \in A\right\}=0, \quad x\in D,   
\end{equation*}
where $A$ is the finite set of discontinuities of $g$, whilst $Z = B^x_{\tau^x_{\partial B(x,r(x))}}$.
\end{proof}